\documentclass[10pt,leqno]{article}

\usepackage[%
  tmargin=1in,bmargin=1in,%
  lmargin=1.3in,rmargin=1.3in,%
  marginparsep=0.2in, marginparwidth=1.0in%
]{geometry}
\AtBeginDocument{%
  \setlength{\abovedisplayskip}{1.5ex plus 0.3ex minus 0.3ex}%
  \setlength{\abovedisplayshortskip}{1.0ex plus 0.3ex minus 0.3ex}%
  \setlength{\belowdisplayskip}{1.5ex plus 0.3ex minus 0.3ex}%
  \setlength{\belowdisplayshortskip}{1.0ex plus 0.3ex minus 0.3ex}%
}

\usepackage{datetime2}
\usepackage{fancyhdr}
\usepackage[hyphens]{url}
\usepackage[style=alphabetic,backend=bibtex8]{biblatex}
\usepackage{titlesec}
\usepackage[dotinlabels]{titletoc}
\usepackage{appendix}
\usepackage[letterspace=50]{microtype}
\usepackage{stmaryrd}
\usepackage[bottom]{footmisc}
\usepackage[inline]{enumitem}
\usepackage{xspace}
\usepackage{calc}
\usepackage{etoolbox}
\usepackage{ifthen}
\usepackage{tikz}
\usepackage{tikz-cd}

\definecolor{cite}{rgb}{0.5,0.5,0.8}
\definecolor{link}{rgb}{0.6,0.3,0.5}
\definecolor{url}{rgb}{0.0,0.2,0.5}
\usepackage[%
  hyperfootnotes=false,
  colorlinks,%
  linkcolor=link,%
  citecolor=cite,%
  urlcolor=url,%
]{hyperref}
\renewbibmacro{in:}{}
\DeclareFieldFormat*{title}{\mkbibitalic{#1}}
\DeclareFieldFormat*{journaltitle}{#1\isdot}
\DeclareFieldFormat*[online,misc]{date}{(#1)}
\DeclareFieldInputHandler{pages}{}

\newbibmacro*{addendum+pubstate+date}{%
  \printfield{addendum}%
  \newunit\newblock
  \printfield{pubstate}
  \usebibmacro{date}}

\newbibmacro*{doi+eprint+url+date}{%
  \iftoggle{bbx:doi}
  {\printfield{doi}}
  {}%
  \newunit\newblock
  \iftoggle{bbx:eprint}
  {\usebibmacro{eprint}}
  {}%
  \iftoggle{bbx:url}
  {\usebibmacro{url+urldate}}
  {}
  \usebibmacro{date}}

\DeclareBibliographyDriver{online}{%
  \usebibmacro{bibindex}%
  \usebibmacro{begentry}%
  \usebibmacro{author/editor+others/translator+others}%
  \setunit{\printdelim{nametitledelim}}\newblock
  \usebibmacro{title}%
  \newunit
  \printlist{language}%
  \newunit\newblock
  \usebibmacro{byauthor}%
  \newunit\newblock
  \usebibmacro{byeditor+others}%
  \newunit\newblock
  \printfield{version}%
  \newunit
  \printfield{note}%
  \newunit\newblock
  \printlist{organization}%
  \newunit\newblock
  \usebibmacro{doi+eprint+url+date}%
  \newunit\newblock
  \usebibmacro{addendum+pubstate}%
  \setunit{\bibpagerefpunct}\newblock
  \usebibmacro{pageref}%
  \newunit\newblock
  \iftoggle{bbx:related}
    {\usebibmacro{related:init}%
     \usebibmacro{related}}
    {}%
  \usebibmacro{finentry}}

\DeclareBibliographyDriver{misc}{%
  \usebibmacro{bibindex}%
  \usebibmacro{begentry}%
  \usebibmacro{author/editor+others/translator+others}%
  \setunit{\printdelim{nametitledelim}}\newblock
  \usebibmacro{title}%
  \newunit
  \printlist{language}%
  \newunit\newblock
  \usebibmacro{byauthor}%
  \newunit\newblock
  \usebibmacro{byeditor+others}%
  \newunit\newblock
  \printfield{howpublished}%
  \newunit\newblock
  \printfield{type}%
  \newunit
  \printfield{version}%
  \newunit
  \printfield{note}%
  \newunit\newblock
  \usebibmacro{doi+eprint+url}%
  \newunit\newblock
  \usebibmacro{addendum+pubstate+date}%
  \setunit{\bibpagerefpunct}\newblock
  \usebibmacro{pageref}%
  \newunit\newblock
  \iftoggle{bbx:related}
  {\usebibmacro{related:init}%
    \usebibmacro{related}}
  {}%
  \usebibmacro{finentry}}

\DeclareFieldFormat{eprint:ktheory}{%
  K-theory preprint archives: \href{https://faculty.math.illinois.edu/K-theory/#1}{ #1}%
}

\usepackage{amsmath,amsthm}

\usepackage[T1]{fontenc}
\usepackage{lmodern}
\usepackage{MnSymbol}

\DeclareMathSymbol{A}{\mathalpha}{operators}{65}
\DeclareMathSymbol{B}{\mathalpha}{operators}{66}
\DeclareMathSymbol{C}{\mathalpha}{operators}{67}
\DeclareMathSymbol{D}{\mathalpha}{operators}{68}
\DeclareMathSymbol{E}{\mathalpha}{operators}{69}
\DeclareMathSymbol{F}{\mathalpha}{operators}{70}
\DeclareMathSymbol{G}{\mathalpha}{operators}{71}
\DeclareMathSymbol{H}{\mathalpha}{operators}{72}
\DeclareMathSymbol{I}{\mathalpha}{operators}{73}
\DeclareMathSymbol{J}{\mathalpha}{operators}{74}
\DeclareMathSymbol{K}{\mathalpha}{operators}{75}
\DeclareMathSymbol{L}{\mathalpha}{operators}{76}
\DeclareMathSymbol{M}{\mathalpha}{operators}{77}
\DeclareMathSymbol{N}{\mathalpha}{operators}{78}
\DeclareMathSymbol{O}{\mathalpha}{operators}{79}
\DeclareMathSymbol{P}{\mathalpha}{operators}{80}
\DeclareMathSymbol{Q}{\mathalpha}{operators}{81}
\DeclareMathSymbol{R}{\mathalpha}{operators}{82}
\DeclareMathSymbol{S}{\mathalpha}{operators}{83}
\DeclareMathSymbol{T}{\mathalpha}{operators}{84}
\DeclareMathSymbol{U}{\mathalpha}{operators}{85}
\DeclareMathSymbol{V}{\mathalpha}{operators}{86}
\DeclareMathSymbol{W}{\mathalpha}{operators}{87}
\DeclareMathSymbol{X}{\mathalpha}{operators}{88}
\DeclareMathSymbol{Y}{\mathalpha}{operators}{89}
\DeclareMathSymbol{Z}{\mathalpha}{operators}{90}

\usepackage[%
  cal=euler,    calscaled=1.0,%
  bb=boondox,   bbscaled=1.03,%
]{mathalfa}  

\makeatletter
\g@addto@macro\bfseries{\boldmath}
\makeatother

\usepackage[textsize=footnotesize,backgroundcolor=orange!50]{todonotes}
\usepackage[compress,capitalize]{cleveref}

\let\theoldbibliography\thebibliography
\renewcommand{\thebibliography}[1]{%
  \theoldbibliography{#1}%
  \setlength{\parskip}{0ex}
  \setlength{\itemsep}{0.5ex plus 0.2ex minus 0.2ex}
  \small
}
\apptocmd{\thebibliography}{\raggedright}{}{}

\renewcommand{\title}[1]{\newcommand{\thetitle}{#1}}
\renewcommand{\author}[1]{\newcommand{\theauthor}{#1}}

\renewcommand{\maketitle}{%
  \begin{center}
    {\linespread{1.15}%
      \bfseries\MakeUppercase%
      \thetitle\par} \vspace{4.0ex}
    \fontsize{8.5}{12.0}\selectfont
    {\MakeUppercase \theauthor}
  \end{center}
  \vspace{3.0ex}
  \thispagestyle{fancy}
}

\renewenvironment{abstract}{\noindent\begin{center}\begin{minipage}{0.8\linewidth}\small{\scshape Abstract.}}{\end{minipage}\end{center}}

\newlength{\tagsep}
\makeatletter
\def\fullwidthdisplay{\displayindent\z@ \displaywidth\columnwidth}
\edef\@tempa{\noexpand\fullwidthdisplay\the\everydisplay}
\everydisplay\expandafter{\@tempa}
\makeatother

\titleformat{\section}{\centering}{\S\thesection.}{1.5\tagsep}{\scshape}
\titleformat{\subsection}[runin]{}{\fontseries{b}\selectfont\S\bfseries\thesubsection.}{1.5\tagsep}{\bfseries}[.]

\titlespacing*{\section}{0pt}{8ex}{2.5ex}
\titlespacing*{\subsection}{0pt}{4ex}{0.5em}

\dottedcontents{section}[1.3em]{\vspace{1ex}}{1.3em}{0.7em}
\dottedcontents{subsection}[3.4em]{}{2.0em}{0.7em}

\newcommand{\crefeqfmt}[1]{
  \crefformat{#1}{(##2##1##3)}
  \Crefformat{#1}{(##2##1##3)}
  \crefrangeformat{#1}{(##3##1##4--##5##2##6)}
  \Crefrangeformat{#1}{(##3##1##4--##5##2##6)}
  \crefmultiformat{#1}{(##2##1##3}{, ##2##1##3)}{, ##2##1##3}{, ##2##1##3)}
  \Crefmultiformat{#1}{(##2##1##3}{, ##2##1##3)}{, ##2##1##3}{, ##2##1##3)}
  \crefrangemultiformat{#1}{(##3##1##4--##5##2##6}{, ##3##1##4--##5##2##6)}{, ##3##1##4--##5##2##6}{, ##3##1##4--##5##2##6)}
  \Crefrangemultiformat{#1}{(##3##1##4--##5##2##6}{, ##3##1##4--##5##2##6)}{, ##3##1##4--##5##2##6}{, ##3##1##4--##5##2##6)}
}
\newcommand{\crefsecfmt}[1]{%
  \crefformat{#1}{\S##2##1##3}
  \Crefformat{#1}{\S##2##1##3}
  \crefrangeformat{#1}{\S\S##3##1##4--##5##2##6}
  \Crefrangeformat{#1}{\S\S##3##1##4--##5##2##6}
  \crefmultiformat{#1}{\S\S##2##1##3}{--##2##1##3}{, ##2##1##3}{ and~##2##1##3}
  \Crefmultiformat{#1}{\S\S##2##1##3}{--##2##1##3}{, ##2##1##3}{ and~##2##1##3}
  \crefrangemultiformat{#1}{\S\S##3##1##4--##5##2##6}{ and~##3##1##4--##5##2##6}{, ##3##1##4--##5##2##6}{ and~##3##1##4--##5##2##6}
  \Crefrangemultiformat{#1}{\S\S##3##1##4--##5##2##6}{ and~##3##1##4--##5##2##6}{, ##3##1##4--##5##2##6}{ and~##3##1##4--##5##2##6}
}

\numberwithin{equation}{block}

\crefeqfmt{equation}
\crefeqfmt{enumi}
\crefeqfmt{enumii}
\crefsecfmt{section}
\crefsecfmt{subsection}
\crefsecfmt{appendix}
\crefname{part}{Part}{Parts}
\crefname{chapter}{Chapter}{Chapters}
\crefname{figure}{Figure}{Figures}

\makeatletter

\newcommand{\blocknumfont}{\bfseries}
\newcommand{\blockheadfont}{\bfseries}
\newcommand{\blocknotefont}{\normalfont}
\newcommand{\blockspecialfont}{\itshape}
\newcommand{\blockhorizspace}{0.4em}
\newcommand{\blocknotespace}{0.25em}
\newcommand{\blocknumsep}{}
\newcommand{\blocksep}{.}
\newcommand{\blockvertspace}{1.75ex plus 0.1ex minus 0.1ex}

\newtheoremstyle{block}%
  {\blockvertspace}
  {\blockvertspace}
  {}
  {}
  {} 
  {}
  {0em}
  {\thmname{\blockheadfont#1}%
    \@ifnotempty{#1}{\@ifnotempty{#2}{ }}%
    \thmnumber{{\blocknumfont #2\blocknumsep}}%
    \thmnote{\hspace{\blocknotespace}\blocknotefont(#3)}%
    \@ifnotempty{#1#2#3}{\blockheadfont\blocksep}\hspace{\blockhorizspace}%
  }

\newtheoremstyle{blockspecial}%
  {\blockvertspace}
  {\blockvertspace}
  {\blockspecialfont}
  {}
  {} 
  {}
  {0em}
  {\thmname{\blockheadfont#1}%
    \@ifnotempty{#1}{\@ifnotempty{#2}{ }}%
    \thmnumber{{\blocknumfont #2\blocknumsep}}%
    \thmnote{\hspace{\blocknotespace}\blocknotefont(#3)}%
    \@ifnotempty{#1#2#3}{\blockheadfont\blocksep}\hspace{\blockhorizspace}%
  }
  
\renewenvironment{proof}[1][Proof]{\par
  \pushQED{\qed}%
  \normalfont%
  \topsep\blockvertspace\relax%
  \labelsep \blockhorizspace\relax%
  \trivlist
  \item[\@ifnotempty{#1}{\hskip\labelsep\blockheadfont#1\@addpunct{\blocksep}}]\ignorespaces
}{%
  \popQED\endtrivlist\@endpefalse%
}

\makeatother

\newcommand{\defblock}[2]{%
  \theoremstyle{block}
  \newtheorem{#1}[block]{#2}%
  \Crefname{#1}{#2}{{#2}s}
  \newtheorem*{#1*}{#2}%
}

\newcommand{\defblockspecial}[2]{%
  \theoremstyle{blockspecial}
  \newtheorem{#1}[block]{#2}%
  \Crefname{#1}{#2}{{#2}s}
  \newtheorem*{#1*}{#2}%
}

\setlist{%
  parsep=0ex, listparindent=\parindent,%
  itemsep=0.75ex, topsep=0.75ex,%
  leftmargin=2.5em,%
}

\setlist[enumerate, 1]{%
  label={\upshape(\emph{\alph*})},%
  ref={{\itshape\alph*}},%
  widest=f,
}
\setlist[enumerate, 2]{%
  leftmargin=*,%
  label={\upshape(\theenumi.\roman*)},%
  ref=\theenumi.\roman*,%
}
\setlist[itemize, 1]{%
  label=$\bullet$,%
}
\setlist[itemize, 2]{%
  label=--,%
}

\usetikzlibrary{calc,decorations.pathmorphing,shapes,arrows}
\tikzcdset{
  arrow style=tikz,
  diagrams={>={stealth}},
}
  
\newcommand{\from}{\leftarrow}
\newcommand{\fromto}{\rightleftarrows}
\newcommand{\lblto}[1]{\xrightarrow{#1}}
\newcommand{\isoto}{\lblto{\raisebox{-0.3ex}[0pt]{$\scriptstyle \sim$}}}
\newcommand{\inj}{\hookrightarrow}

\newcommand{\iso}{\simeq}
\newcommand{\lbliso}[1]{\overset{#1}{\simeq}}

\newcommand{\cosimpl}[3]{
  \begin{tikzcd}[ampersand replacement=\&, column sep=small]
    #1 \ar[r, shift right=0.35ex]
       \ar[r, shift left=0.35ex] \&
    #2 \ar[r, shift right=0.70ex]
       \ar[r, shift left=0.70ex]
       \ar[r] \&
    #3 \ar[r, shift right=0.35ex]
       \ar[r, shift left=0.35ex]
       \ar[r, shift right=1.05ex]
       \ar[r, shift left=1.05ex] \&
    \cdots
  \end{tikzcd}
}

\title{Hochschild homology and the derived de Rham complex revisited}
\author{Arpon Raksit}

\bibliography{ref}

\numberwithin{block}{subsection}

\defblock{construction}{Construction}
\defblockspecial{corollary}{Corollary}
\defblock{definition}{Definition}
\defblock{example}{Example}
\defblockspecial{lemma}{Lemma}
\defblock{notation}{Notation}
\defblockspecial{proposition}{Proposition}
\defblock{recollection}{Recollection}
\defblock{remark}{Remark}
\defblockspecial{theorem}{Theorem}
\defblock{variant}{Variant}
\defblockspecial{variantspecial}{Variant}
\defblock{warning}{Warning}

\let\slasho\o

\newcommand{\num}[1]{\mathbb{#1}}
\newcommand{\cat}[1]{\mathcal{#1}}

\newcommand{\Alg}{\mathrm{Alg}}
\newcommand{\Assoc}{\mathrm{Assoc}}

\renewcommand{\Bar}{\operatorname{Bar}}
\newcommand{\BMod}[2]{{}_{#1}\mathrm{BMod}_{#2}}
\newcommand{\CAlgMod}{\mathrm{CAlgMod}}
\newcommand{\CAlg}{\mathrm{CAlg}}
\newcommand{\CSym}{\operatorname{CSym}}
\newcommand{\Cat}{\mathrm{Cat}}
\newcommand{\Cobar}{\mathrm{Cobar}}
\newcommand{\DAlgMod}{\mathrm{DAlgMod}}
\newcommand{\DAlg}{\mathrm{DAlg}}
\newcommand{\DG}{\operatorname{DG}}
\newcommand{\Einfty}{\mathrm{E}_\infty}
\newcommand{\End}{\operatorname{End}}
\newcommand{\Exc}{\mathrm{Exc}}
\newcommand{\Fil}{\mathrm{Fil}}
\newcommand{\Fun}{\mathrm{Fun}}
\newcommand{\Gm}{\num{G}_{\mathrm{m}}}

\newcommand{\Gr}{\mathrm{Gr}}
\renewcommand{\H}{\mathrm{H}}
\newcommand{\HC}{\mathrm{HC}}
\newcommand{\HH}{\mathrm{HH}}
\newcommand{\HN}{\mathrm{HC}^-}
\newcommand{\HP}{\mathrm{HP}}
\newcommand{\Hom}{\operatorname{Hom}}

\newcommand{\LMod}{\mathrm{LMod}}
\newcommand{\LOmega}{\mathrm{L}\Omega}
\newcommand{\LSym}{\operatorname{LSym}}
\newcommand{\LM}{\mathrm{LM}}
\newcommand{\Map}{\operatorname{Map}}
\newcommand{\Mod}{\mathrm{Mod}}
\newcommand{\Nm}{\mathrm{Nm}}
\newcommand{\Op}{\mathrm{Op}}
\newcommand{\PrL}{\mathrm{Pr^L}}
\newcommand{\PSh}{\cat{P}}
\newcommand{\Picspt}{\operatorname{pic}}
\newcommand{\Pic}{\operatorname{Pic}}
\newcommand{\Poly}{\mathrm{Poly}}
\newcommand{\RMod}{\mathrm{RMod}}

\newcommand{\SSeq}{\mathrm{SSeq}}
\newcommand{\Set}{\mathrm{Set}}
\newcommand{\Spc}{\mathrm{Spc}}
\newcommand{\Spt}{\mathrm{Spt}}
\newcommand{\Sym}{\operatorname{Sym}}
\newcommand{\Tor}{\operatorname{Tor}}

\newcommand{\apoly}{{\mathrm{a}\text{-}\mathrm{poly}}}
\newcommand{\aug}{\mathrm{aug}}
\newcommand{\bAlg}{\mathrm{bAlg}}
\newcommand{\beil}{\mathrm{B}}
\newcommand{\bul}{{\raisebox{0.1ex}{\scalebox{0.6}{\hspace{0.08em}$\bullet$}}}}
\newcommand{\cAlg}{\mathrm{cAlg}}
\newcommand{\cCAlg}{\mathrm{cCAlg}}
\newcommand{\cBMod}[2]{{}_{#1}\mathrm{cBMod}_{#2}}
\newcommand{\cLMod}{\mathrm{cLMod}}
\newcommand{\cMod}{\mathrm{cMod}}
\newcommand{\cRMod}{\mathrm{cRMod}}
\newcommand{\ccn}{\mathrm{ccn}}
\newcommand{\ce}{\mathrel{:=}}
\newcommand{\cir}{{\mathrm{S}^1}}
\newcommand{\clspc}{\mathrm{B}}
\newcommand{\cn}{\mathrm{cn}}
\newcommand{\cofib}{\operatorname*{cofib}}
\newcommand{\coins}{\operatorname{coins}}
\newcommand{\coker}{\operatorname*{coker}}
\newcommand{\colim}{\operatorname*{colim}}
\newcommand{\comm}{\mathrm{C}}
\renewcommand{\cot}{\mathrm{L}}
\newcommand{\cpl}[1]{#1^\wedge}
\newcommand{\cplge}[2]{#1^{\wedge, \ge #2}}
\newcommand{\cplle}[2]{#1^{\wedge, \le #2}}
\newcommand{\cube}{\mathbb{P}}
\newcommand{\dcomm}{\mathrm{D}}
\newcommand{\dR}{\mathrm{dR}}
\renewcommand{\day}{{\otimes_\mathrm{D}}}
\newcommand{\ds}{\mathrm{ds}}
\newcommand{\dual}[1]{#1^\vee}
\newcommand{\epoly}{{\mathrm{e}\text{-}\mathrm{poly}}}
\newcommand{\ev}{\mathrm{ev}}
\newcommand{\fd}{\mathrm{fd}}
\newcommand{\fib}{\operatorname*{fib}}
\newcommand{\fil}{\mathrm{fil}}
\newcommand{\filop}{\operatorname{fil}}
\newcommand{\fin}{\mathrm{fin}}
\newcommand{\grop}{\operatorname{gr}}
\newcommand{\gr}{\mathrm{gr}}

\newcommand{\heart}{\heartsuit}
\newcommand{\ho}{\mathrm{h}}
\newcommand{\hminus}{\text{h}\textsubscript{$-$}}
\newcommand{\hplus}{\text{h}\textsubscript{$+$}}
\newcommand{\hpm}{\text{h}\textsubscript{$\pm$}}
\newcommand{\id}{\mathrm{id}}
\newcommand{\im}{\operatorname{im}}
\newcommand{\ind}{\mathrm{ind}}
\newcommand{\ins}{\operatorname{ins}}
\newcommand{\koz}{{\otimes_\mathrm{K}}}
\renewcommand{\l}{\left}
\newcommand{\lax}{\mathrm{lax}}
\newcommand{\lder}{\mathrm{L}}
\newcommand{\lincir}{\mathbb{T}}
\renewcommand{\o}{\overline}
\newcommand{\op}{\mathrm{op}}
\newcommand{\post}{\mathrm{P}}
\newcommand{\pt}{\mathrm{pt}}
\newcommand{\qf}{\mathrm{qf}}
\renewcommand{\r}{\right}

\newcommand{\spl}{\operatorname{spl}}
\newcommand{\tate}{\mathrm{t}}

\renewcommand{\u}{\underline}
\newcommand{\uMap}{\u{\smash{\Map}}}
\newcommand{\und}{\mathrm{res}}
\newcommand{\unit}{\mathbb{1}}
\newcommand{\zeroone}{{\{0,1\}}}

\begin{document}

\maketitle

\begin{abstract}
  We characterize two objects by universal property: the derived de Rham complex and Hochschild homology together with its Hochschild--Kostant--Rosenberg (HKR) filtration. This involves endowing these objects with extra structure, built on notions of ``homotopy-coherent cochain complex'' and ``filtered circle action'' that we study here. We use these universal properties to give a conceptual proof that the associated graded of the HKR filtration identifies with the derived de Rham complex, as well as to give a new construction of the filtrations on cyclic, negative cyclic, and periodic cyclic homology that relate these invariants to derived de Rham cohomology.
\end{abstract}

{
  \let\clearpage\relax
  \small
  \tableofcontents
}


\section{Introduction}
\label{in}

One of the basic cohomological invariants in algebraic geometry is \emph{algebraic de Rham cohomology}: for $A$ a commutative ring and $B$ a commutative $A$-algebra, this is given by the cohomology of the \emph{algebraic de Rham complex}
\[
  \Omega^\bul_{B/A} = \{\ \Omega^0_{B/A} \to \Omega^1_{B/A} \to \Omega^2_{B/A} \to \cdots\ \}.
\]
This cochain complex has the structure of a strictly commutative differential graded $A$-algebra\footnote{The modifier ``strictly'' refers to the property that $x^2 = 0$ whenever $x \in \Omega^i_{B/A}$ for $i$ odd.}, and is characterized up to isomorphism as such by the following universal property:
\begin{enumerate}[label=($\dagger$),ref=$\dagger$]
\item \label{in--dR-univ}
  For any strictly commutative differential graded $A$-algebra $X^\bul$, a map of commutative $A$-algebras $B \iso \Omega^0_{B/A} \to X^0$ extends uniquely to a map of strictly commutative differential graded $A$-algebras $\Omega^\bul_{B/A} \to X^\bul$.
\end{enumerate}

This paper proves two analogues of \cref{in--dR-univ}: we characterize the \emph{derived de Rham complex} and the \emph{HKR-filtered Hochschild homology} of $B$ over $A$ by similar universal properties. The meaning of this terminology and the features of these universal properties will be explained in the remainder of this introduction: in \cref{in--dR}, we overview our result on the derived de Rham complex (\cref{in--dR--thm}); in \cref{in--HH}, we overview our result on HKR-filtered Hochschild homology (\cref{in--HH--thm}) and explain how these universal properties together supply a conceptual proof of the known relationship between Hochschild (and negative and periodic cyclic) homology and derived de Rham cohomology (\cref{in--HH--fils}), which is what motivated this work.

The universal properties of the derived de Rham complex and HKR-filtered Hochschild homology, while analogous to \cref{in--dR-univ} in a natural manner, involve significantly more intricate types of algebraic structure, which in particular are $\infty$-categorical in nature. Formulating these structures is a primary task of the paper. This necessitates two notices regarding the remainder of the introduction: we will immediately begin using $\infty$-categorical language, and we will content ourselves with giving somewhat informal sketches of the main definitions and results, pointing to the body of the paper for their precise articulation.


\subsection{The derived de Rham complex}
\label{in--dR}

In this paper, we will be interested in \emph{Hodge-completed derived de Rham cohomology}, a variant of algebraic de Rham cohomology that naturally lives in the context of derived algebraic geometry: given a simplicial commutative ring $A$ and a simplicial commutative $A$-algebra $B$, the Hodge-completed derived de Rham cohomology of $B$ over $A$ is an $\Einfty$-$A$-algebra $\cpl\dR_{B/A}$, equipped with a complete $\num{Z}_{\ge 0}$-indexed decreasing filtration $\cplge\dR\star_{B/A}$ called the \emph{Hodge filtration}.

If $A$ is an ordinary commutative ring and $B$ is a \emph{smooth} $A$-algebra, then $\cpl\dR_{B/A}$ is equivalent to the $E_\infty$-$A$-algebra underlying the commutative differential graded $A$-algebra $\Omega^\bul_{B/A}$, and the Hodge filtration $\cplge\dR\star_{B/A}$ is induced by the \emph{brutal filtration} $\Omega^{\bul\ge\star}_{B/A}$, given by
\[
  \Omega^{\bul \ge i}_{B/A} = \{\ \cdots \to 0 \to \Omega^i_{B/A} \to \Omega^{i+1}_{B/A} \to \cdots\ \}
  \qquad\text{for}\ i \in \num{Z}_{\ge 0}.
\]
In particular, in this smooth situation, the associated graded pieces of the Hodge filtration are given by shifts of the modules of differential forms, $\Omega^i_{B/A}[-i]$. In general, the associated graded pieces of $\cplge\dR\star_{B/A}$ are given by shifts of modules of ``derived differential forms'', $\LOmega^i_{B/A}[-i]$, with $\LOmega^1_{B/A}$ being an alternative notation for the cotangent complex $\cot_{B/A}$, and more generally $\LOmega^i_{B/A}$ denoting the derived $i$-th exterior power of $\cot_{B/A}$.

Here we give an alternative perspective on Hodge-completed derived de Rham cohomology and its Hodge filtration: namely, we formulate a derived version of the strictly commutative differential graded $A$-algebra $\Omega^\bul_{B/A}$, with terms $\LOmega^i_{B/A}$ in place of $\Omega^i_{B/A}$, and characterized by a universal property analogous to \cref{in--dR-univ} above. As with $\Omega^\bul_{B/A}$, this object has an ``underlying $\Einfty$-$A$-algebra'' that carries the information of its ``cohomology'' and is equipped with a ``brutal filtration'', and now these recover Hodge-completed derived de Rham cohomology and its Hodge filtration for all maps of simplicial commuative rings $A \to B$.

To realize the vague idea suggested in the previous paragraph, we define the following over any simplicial commutative ring $A$:
\begin{itemize}
\item The notion of an \emph{\hplus-differential graded (\hplus-dg) $A$-module} $X^\bul$ (\cref{dg--pm--dg}). This is a type of ``homotopy-coherent cochain complex'', consisting of an underlying graded $A$-module $X^* = \{X^i\}_{i \in \num{Z}}$, together with differentials $d : X^i[1] \to X^{i+1}$ for $i \in \num{Z}$ that square to zero up to coherent homotopy (the shift in the differentials is the source of the ``\hplus'' in the terminology). These are precisely defined as graded modules over a certain split square-zero graded algebra $\num{D}_+$, whose underlying $A$-module is equivalent to $A \oplus A[1]$. When $A$ is an ordinary commutative ring, any ordinary cochain complex of $A$-modules $M^\bul$ determines an \hplus-dg $A$-module $X^\bul$, with underlying graded object given by $X^i \iso M^i[i]$.
\item The \emph{cohomology type} $|X^\bul|$ of an \hplus-dg $A$-module $X^\bul$ (\cref{dg--co--cohomology-type}). This is an $A$-module equipped with a complete filtration $|X^\bul|^{\ge \star}$ whose associated graded pieces are given by $X^i[-2i]$. In the case that $A$ is an ordinary commutative ring and $X^\bul$ comes from an ordinary cochain complex of $A$-modules $M^\bul$ in the manner described above, the cohomology type $|X^\bul|$ is simply the object in the derived category represented by $M^\bul$, equipped with the brutal filtration (whose associated graded pieces are $M^i[-i]$).
\item The notion of a \emph{simplicial commutative algebra structure} on a \hplus-dg $A$-module (\cref{dg--pm--dga}). We refer to \hplus-dg $A$-modules equipped with this structure as \emph{\hplus-dg simplicial commutative $A$-algebras}. If $X^\bul$ is an \hplus-dg simplicial commutative $A$-algebra, then the zeroth graded piece $X^0$ canonically carries the structure of a simplicial commutative $A$-algebra and the filtered cohomology type $|X^\bul|^{\ge\star}$ canonically carries the structure of a filtered $\Einfty$-$A$-algebra.
\end{itemize}
With these notions in place, we prove the following result:

\begin{theorem}[see \cref{dg--dr}]
  \label{in--dR--thm}
  Let $A$ be a simplicial commutative ring and let $B$ be a simplicial commutative $A$-algebra. Then:
  \begin{enumerate}
  \item \label{in--dR--thm--def}
    There is an initial object $L\Omega_{B/A}^{+\bul}$ in the $\infty$-category of \hplus-dg simplicial commutative $A$-algebras equipped with a map of simplicial commutative $A$-algebras $B \to X^0$.
  \item \label{in--dR--thm--terms}
    For $i \ge 0$, there are equivalences $\LOmega^{+i}_{B/A} \iso \LOmega^i_{B/A}[i]$, and for $i < 0$, we have $\LOmega^{+i}_{B/A} \iso 0$. Under these equivalences, the first differential $\LOmega_{B/A}^{+0} \to \LOmega_{B/A}^{+1}[-1]$ is given by the universal derivation $d : B \to L_{B/A}$.
  \item \label{in--dR--thm--cohomology}
    There is an equivalence of filtered $\Einfty$-$A$-algebras $|\LOmega_{B/A}^{+\bul}|^{\ge\star} \iso \cplge{\dR}{\star}_{B/A}$.
  \end{enumerate}
\end{theorem}

The object $\LOmega^{+\bul}_{B/A}$ of \cref{in--dR--thm} is what we refer to in this paper as the \emph{derived de Rham complex} of $B$ over $A$. Its definition in statement \cref{in--dR--thm--def} of the theorem is the promised universal property analogous to \cref{in--dR-univ}, and statement \cref{in--dR--thm--cohomology} in the theorem explains the way in which this object recovers Hodge-completed derived de Rham cohomology and its Hodge filtration.


\subsection{Hochschild homology}
\label{in--HH}

Another family of cohomological invariants, related to the de Rham invariants discussed above, arises from the theory of Hochschild homology. For $A$ a simplicial commutative ring and $B$ a simplicial commutative $A$-algebra, the \emph{Hochschild homology} of $B$ over $A$ is a simplicial commutative $A$-algebra with $\cir$-action, $\HH(B/A)$. To this object with $\cir$-action, we may apply the homotopy fixed points construction to obtain \emph{negative cyclic homology} $\HN(B/A)$ and the Tate construction to obtain \emph{periodic cyclic homology} $\HP(B/A)$. The relationship between these invariants and the de Rham invariants mentioned above manifests in the form of certain filtrations, as described in the following result.

\begin{theorem}
  \label{in--HH--fils}
  Let $A$ be a simplicial commutative ring and let $B$ be a simplicial commutative $A$-algebra. Then:
  \begin{enumerate}
  \item \label{in--HH--fils--HH}
    There is a complete $\num{Z}_{\ge 0}$-indexed decreasing filtration $\filop^\star \HH(B/A)$ on $\HH(B/A)$ with associated graded pieces given by
    \[
      \grop^i \HH(B/A) \iso \LOmega^i_{B/A}[i] \qquad (i \in \num{Z}_{\ge 0}).
    \]
  \item \label{in--HH--fils--HP}
    There are complete decreasing $\num{Z}$-indexed filtrations $\filop^\star \HN(B/A)$ and $\filop^\star \HP(B/A)$ on $\HN(B/A)$ and $\HP(B/A)$, respectively, with associated graded pieces given by
    \[
      \grop^i \HN(B/A) \iso \cplge\dR{i}_{B/A}[2i]
      \quad\text{and}\quad
      \grop^i \HP(B/A) \iso \cpl\dR_{B/A}[2i]
      \qquad(i \in \num{Z}).
    \]
    These filtrations are exhaustive under the assumptions that $B$ is truncated (i.e. $\pi_k(B) \iso 0$ for $k \gg 0$) and that the $\Tor$-amplitude of $\cot_{B/A}$ over $B$ is contained in $[0,1]$.
  \end{enumerate}
\end{theorem}

Part \cref{in--HH--fils--HH} of \cref{in--HH--fils} goes back to the classical theorem of Hochschild--Kostant--Rosenberg: the latter is equivalent to the special case of the former where $A$ is an ordinary commutative ring and $B$ is a smooth $A$-algebra; one can then bootstrap to the general statement of \cref{in--HH--fils--HH} using formal/categorical arguments, as explained by Bhatt--Morrow--Scholze in \cite[\S2.2]{bms2}. For this reason, the filtration in \cref{in--HH--fils} is referred to as the \emph{HKR filtration} on Hochschild homology.

Part \cref{in--HH--fils--HP} was proved by Antieau \cite{antieau--periodic}, after being conjectured by Bhatt--Morrow--Scholze, who proved a $p$-adic variant of the result in \cite{bms2}.

In this paper, we shall give a new proof of \cref{in--HH--fils}, i.e. we'll find new constructions of the filtrations appearing in \cref{in--HH--fils}. Firstly, we construct HKR-filtered Hochschild homology, refined by additional $\cir$-equivariant and multiplicative structures, via a universal characterization (which crucially refers to this extra structure). Secondly, we use this additional $\cir$-equivariant structure to construct the filtrations on negative and periodic cyclic homology.

Before elaborating on our universal property for HKR-filtered Hochschild homology, let us recall that Hochschild homology itself is characterized by a universal property. Namely, for $A$ a simplicial commutative ring and $B$ a simplicial commutative $A$-algebra, we have:
\begin{enumerate}[label=($\ddagger$),ref=$\ddagger$]
\item \label{in--HH--univ}
  For any simplicial commutative $A$-algebra with $\cir$-action $X$, a map of simplicial commutative $A$-algebras $B \to X$ extends uniquely to an $\cir$-equivariant map of simplicial commutative $A$-algebras $\HH(B/A) \to X$.\footnote{It is perhaps worth emphasizing here that, when we refer to simplicial commutative rings, we always mean to work with the $\infty$-category of such objects. In particular, the universal property \cref{in--HH--univ} is a properly $\infty$-categorical one, characterizing $\HH(B/A)$ up to equivalence, in contrast with the universal property \cref{in--dR-univ}.}
\end{enumerate}

Our universal property for HKR-filtered Hochschild homology is an interpolation between \cref{in--HH--univ} and the universal property of the derived de Rham complex described in \cref{in--dR}. To this end, we define the following over any base simplicial commutative ring $A$:
\begin{itemize}
\item The notion of a ($\num{Z}$- or $\num{Z}_{\ge0}$-indexed) \emph{filtered $A$-module with filtered $\cir$-action} (\cref{fc--fc--fc-action}). Roughly speaking, a filtered $\cir$-action on a filtered $A$-module is ``an $\cir$-action that increases the filtration degree'' (analogous to the differential of a cochain complex increasing the grading degree). Rigorously, these are defined as filtered modules over a filtered $A$-algebra $\lincir_\fil$, which we refer to as the \emph{$A$-linear filtered circle}. The underlying $A$-algebra of $\lincir_\fil$ is the group algebra $A[S^1]$, and the associated graded object of $\lincir_\fil$ is the graded algebra $\num{D}_+$ appearing in the definition of \hplus-dg objects. It follows that a filtered $\cir$-action on a filtered $A$-module determines an $\cir$-action on its underlying $A$-module and an \hplus-dg structure on its associated graded $A$-module.
\item The notion of a \emph{simplicial commutative algebra structure} on a filtered $A$-module with filtered $\cir$-action (\cref{fc--fc--dalg-fc-action}). We refer to filtered $A$-modules with filtered $\cir$-action equipped with this structure as \emph{filtered simplicial commutative $A$-algebras with filtered $\cir$-action}. If $X$ is a filtered simplicial commutative $A$-algebra with filtered $\cir$-action, then the underlying object of $X$ canonically carries thes structure of a simplicial commutative $A$-algebra with $\cir$-action and the associated graded object $\gr(X)$ canonically carries the structure of a \hplus-dg simplicial commutative $A$-algebra.
\end{itemize}

\begin{remark}[The filtered circle]
  \label{in--HH--filcir}
  Our ability to define the notion of a filtered simplicial commutative $A$-algebra with filtered $\cir$-action---the key notion for the universal property of HKR-filtered Hochschild homology---boils down to exhibiting a sufficient amount of structure on the filtered circle $\lincir_\fil$. We need this object not just as a filtered algebra but as a suitably structured filtered bialgebra. This construction of the filtered circle is a key component in this paper. (The same comments apply to the object $\num{D}_+$ in the graded setting, though that construction is weaker in the sense that it may be recovered from the filtered circle but not vice-versa.)
\end{remark}

With these notions in place, we prove the following result:

\begin{theorem}[see \cref{fc--hh}]
  \label{in--HH--thm}
  Let $A$ be a simplicial commutative ring and let $B$ be a simplicial commutative $A$-algebra. Then:
  \begin{enumerate}
  \item \label{in--HH--thm--def}
    There is an initial object $\HH_\fil(B/A)$ in the $\infty$-category of $\num{Z}_{\ge0}$-indexed filtered simplicial commutative rings with filtered $\cir$-action $X^\star$ equipped with a map of simplicial commutative $A$-algebras $B \to X^0$.
  \item \label{in--HH--thm--und}
    There is an equivalence of simplicial commutative $A$-algebras with $\cir$-action $\HH_\fil(B/A)^0 \iso \HH(B/A)$, i.e. the underlying object of $\HH_\fil(B/A)$ is the Hochschild homology of $B$ over $A$.
  \item \label{in--HH--thm--gr}
    There is an equivalence of \hplus-dg simplicial commutative $A$-algebras $\gr(\HH_\fil(B/A)) \iso \LOmega^{+\bul}_{B/A}$, i.e. the associated graded object of $\HH_\fil(B/A)$ is the derived de Rham complex of $B$ over $A$.
  \item \label{in--HH--thm--cpl}
    As a filtered object, $\HH_\fil(B/A)$ is complete.
  \end{enumerate}
\end{theorem}

Statements \cref{in--HH--thm--def,in--HH--thm--und,in--HH--thm--gr} of \cref{in--HH--thm} are completely formal once the definitions have been made, conceptualizing the relationship between Hochschild homology and the derived de Rham complex. Statement \cref{in--HH--thm--cpl} requires the same basic computation going into the proof of the HKR theorem. From all the statements together we certainly in particular recover \cref{in--HH--fils}\cref{in--HH--fils--HH}, and in fact one can immediately deduce that the filtration on $\HH(B/A)$ determined by the object $\HH_\fil(B/A)$ agrees with the HKR filtration. Properly speaking, it is the object $\HH_\fil(B/A)$ that we mean to refer to by the terminology \emph{HKR-filtered Hochschild homology}.

To furthermore recover \cref{in--HH--fils}\cref{in--HH--fils--HP}, we also define, for $A$ a simplicial commutative ring and $X$ a filtered $A$-module with filtered $\cir$-action, the \emph{filtered fixed points} $X^{\lincir_\fil}$ and \emph{filtered Tate construction} $X^{\tate \lincir_\fil}$. These are both filtered $A$-modules, which can be regarded as filtrations on the usual fixed points and Tate construction of the $A$-module with $\cir$-action underlying $X$. We then prove the following result:

\begin{proposition}[see \cref{fc--ta}]
  \label{in--HH--fil-fix}
  Let $A$ be a simplicial commutative ring and let $X$ be a filtered $A$-module with filtered $\cir$-action. Then, for $i \in \num{Z}$, there are equivalences
  \[
    \gr^i(X^{\lincir_\fil}) \iso |\gr(X)|^{\ge i}[2i], \qquad
    \gr^i(X^{\tate \lincir_\fil}) \iso |\gr(X)|[2i]
  \]
  (where we are regarding $\gr(X)$ as an \hplus-dg $A$-module, so that $|\gr(X)|$  denotes the cohomology type of $\gr(X)$ and $|\gr(X)|^{\ge\star}$ the brutal filtration thereon).
\end{proposition}

It follows from \cref{in--HH--thm,in--HH--fil-fix} (together with a bit of extra analysis regarding completeness and exhaustiveness) that applying filtered fixed points and filtered Tate construction to HKR-filtered Hochschild homology $\HH_\fil(B/A)$ produces filtrations on negative and periodic cyclic homology as specified by \cref{in--HH--fils}\cref{in--HH--fils--HP}. It is possible to show that these filtrations agree with those constructed in \cite{antieau--periodic}, but we do not spell this out in the paper.


\subsection{Outline}
\label{in--ol}

In general, we refer to the beginning of each section and subsection for description and motivation of its contents. However, this paper being somewhat lengthy and at times technical, let us give a brief guide to the material. The main definitions and results are contained in \cref{dg,fc}; this was summarized in \cref{in--dR,in--HH}, so we say no more about it here. The earlier \cref{bi,gf,dc} are devoted to preliminary categorical and algebraic material that are needed to formulate and work with the structures studied later. Two pieces of this preliminary material may be of particular interest:
\begin{itemize}
\item In \cref{gf--kd}, we discuss the relationship between filtered objects in a stable presentable symmetric monoidal $\infty$-category $\cat{C}$ and \emph{\hminus-differential graded objects} in $\cat{C}$; the latter is another flavor of ``homotopy-coherent cochain complexes'' playing a role in this paper (in addition to the $\hplus$-dg objects discussed in \cref{in--dR}). This relationship underlies our notion of ``cohomology type'' for homotopy-coherent cochain complexes and hence the proof of \cref{in--HH--fil-fix}. Additionally, it supplies a clean construction of the Beilinson t-structure on filtered objects (\cref{gf--t--beilinson}).
\item In \cref{dc}, we discuss the theory of \emph{derived commutative rings}, an extension of the theory of simplicial commutative rings to the nonconnective setting, which we use to formulate the structure of interest on the filtered circle. This theory (though not our choice of terminology) is due to Mathew; it involves ideas of Brantner, which are also used in their joint work \cite{brantner-mathew--deformation}, and is related to earlier work of Illusie \cite[\textsection I.4]{illusie--cotangent-i}. But it has not previously appeared in writing, so we give a complete account of the aspects of the theory needed here.
\end{itemize}


\subsection{Related work}
\label{pr--rw}

The relationship between Hochschild homology and the de Rham complex simplifies substantially in characteristic zero. For example, the filtrations of \cref{in--HH--fils} are canonically split when $A$ is a $\num{Q}$-algebra. In this characteristic zero setting, results essentially equivalent to the ones of this paper appear in earlier work of To\"en--Vezzosi \cite{toen-vezzosi--hkr}, and there are related results due to Ben-Zvi--Nadler \cite{benzvi-nadler--loop-conn}. However, a crucial ingredient in these previous treatments is a certain formality result of the circle over $\num{Q}$ that breaks down in positive and mixed characteristic. In this paper, we accommodate this non-formality and prove integral results. Our perspective should be viewed as building on and refining the viewpoint of To\"en--Vezzosi in particular.

Some of the main ideas in this paper were developed independently and concurrently by Moulinos--Robalo--To\"en. In \cite{moulinos-robalo-toen--hkr}, they construct a ``filtered circle'' essentially equivalent to ours, and use it just as we do here to construct the filtrations on Hochschild homology and negative cyclic homology discussed above. Of course, there are also differences between their paper and this one; let us highlight a few:
\begin{itemize}
\item There is some difference in formalism: \cite{moulinos-robalo-toen--hkr} employs the language of derived algebraic geometry, while here we stick to (derived) algebra. For example, in \cite{moulinos-robalo-toen--hkr}, the derived stacks $\clspc\Gm$ and $\num{A}^1/\Gm$ are used to work with graded and filtered objects, whereas here we do so directly. Accordingly, the filtered circle takes the form of a derived stack in \cite{moulinos-robalo-toen--hkr}; that object should be the spectrum, in an appropriate derived algebro-geometric sense, of the filtered circle studied in this paper (more precisely, of the object $\dual\lincir_\fil$ defined in \cref{fc--fc}).
\item Most importantly, the two papers give rather different constructions of the filtered circle. Here we use the Postnikov filtration construction and demonstrate that this preserves the necessary structure, while \cite{moulinos-robalo-toen--hkr} uses To\"en's description of the affinization of $\cir$ (regarded as a derived stack), which relates it at each prime $p$ to the scheme of $p$-typical Witt vectors. The latter is closer in spirit to the work of Ben-Zvi--Nadler cited above, and offers its own insights into the nature of the filtered circle.

  The approach here has the advantages of involving no computations and working directly over $\num{Z}$ (the results in \cite{moulinos-robalo-toen--hkr} are proved only over $\num{Z}_{(p)}$, though they do indicate how to make the desired construction over $\num{Z}$ from their perspective; see also the related further work of To\"en \cite{toen--schematization-revisited}). Our construction also allows for comparison, as the universal property of the Postnikov filtration construction gives uniqueness results (see \cref{fc--fc--fc-unique}).
\item The results of \cite{moulinos-robalo-toen--hkr} involve an object equivalent to our derived de Rham complex $\LOmega^{+\bul}_{A/R}$, but we discuss the relationship between this object and the usual notion of (Hodge-completed) derived de Rham cohomology in more detail (\cref{dg--dr--classical-compare}). This is afforded by our study in \cref{dg--co} of the ``cohomology'' of homotopy-coherent cochain complexes, or in other words in our results relating filtered objects with \hminus\ cochain complexes and \hminus\ cochain complexes with \hplus\ cochain complexes.
\item There are applications of these ideas discussed in \cite{moulinos-robalo-toen--hkr} that are not mentioned at all here, e.g. to the theory of shifted sympletic structures in positive characteristic.
\end{itemize}

Finally, an alternative formulation of the notion of ``homotopy-coherent cochain complex'' was considered by Joyal \cite[\S35]{joyal--notes}, and is studied in recent work of Walde \cite{walde--dold-kan} and of Ariotta \cite{ariotta--complex}. In particular, Ariotta gives another perspective on the relationship between cochain complexes and filtered objects, discussed here in \cref{gf--kd}.


\subsection{Acknowledgements}
\label{pr--ac}

It is a pleasure to thank Ben Antieau, Tony Feng, S\slasho ren Galatius, Joj Helfer, Lars Hesselholt, Jacob Lurie, Akhil Mathew, Thomas Nikolaus, and Allen Yuan for enlightening and encouraging conversations related to this work, and Ben Antieau, William Balderrama, Lukas Brantner, S\slasho ren Galatius, Akhil Mathew, Bertrand To\"en, and a referee for feedback on earlier versions of the paper. I would also like to acknowledge Tasos Moulinos and Marco Robalo for discussing some of the material presented here and its relation to their joint work with To\"en. Finally, I am grateful for the support provided by the NSF Graduate Research Fellowship and the hospitality of the institutions where this work was carried out: Stanford University, the University of Copenhagen (including support by the Danish National Research Foundation through the Centre for Symmetry and Deformation (DNRF92)), and my parents' home.


\subsection{Notation and terminology}
\label{pr--nt}

Throughout the paper, we use the language of higher category theory and higher algebra developed by Lurie \cite{lurie--topos,lurie--algebra}. All statements and constructions should be understood in a homotopy-invariant sense. We will often cite the \emph{adjoint functor theorem}, by which we mean \cite[Corollary 5.5.2.9]{lurie--topos}, and the \emph{monadicity theorem}, by which we mean \cite[Theorem 4.7.3.5]{lurie--algebra}.

We let $\Spc$ denote the $\infty$-category of spaces and $\Spt$ the $\infty$-category of spectra; we let $\Cat_\infty$ denote the $\infty$-category of small $\infty$-categories; we let $\PrL$ denote the $\infty$-category of presentable $\infty$-categories and colimit-preserving\footnote{When we refer to limits and colimits, we by default mean those with small indexing diagrams.} functors, which we often regard as a symmetric monoidal $\infty$-category by \cite[\S4.8.1]{lurie--algebra}, so that commutative algebra objects of $\PrL$ may be identified with presentable symmetric monoidal $\infty$-categories\footnote{We follow the convention of \cite[Definition 3.4.4.1]{lurie--algebra} that, when applied to a symmetric monoidal $\infty$-category, the adjective ``presentable'' includes the condition that the tensor product commute with colimits in each variable.}. We implicitly regard every ordinary category as an $\infty$-category (via the nerve construction). We also often implicitly regard abelian groups as spectra and commutative rings as simplicial commutative rings or $\Einfty$-rings; in particular, for $A$ a commutative ring, $\Mod_A$ denotes the $\infty$-category of $A$-module spectra (this is equivalent to the derived $\infty$-category $\mathrm{D}(A)$ of complexes over $A$).

For $\cat{C}$ a small $\infty$-category, we let $\PSh(\cat{C})$ denote the $\infty$-category $\Fun(\cat{C}^\op,\Spc)$ of presheaves (of spaces) on $\cat{C}$; the Yoneda embedding $\cat{C} \inj \PSh(\cat{C})$ exhibits $\PSh(\cat{C})$ as the free $\infty$-category admitting colimits generated by $\cat{C}$ \cite[Theorem 5.1.5.6]{lurie--topos}. For $\cat{C}$ a small $\infty$-category admitting finite coproducts, we let $\PSh_\Sigma(\cat{C})$ denote the full subcategory of $\PSh(\cat{C})$ spanned by the finite product--preserving functors $\cat{C}^\op \to \Spc$; the Yoneda embedding $\cat{C} \inj \PSh_\Sigma(\cat{C})$ exhibits $\PSh_\Sigma(\cat{C})$ as the free $\infty$-category admitting sifted colimits generated by $\cat{C}$ \cite[Proposition 5.5.8.15]{lurie--topos}.

For $\cat{C}$ an $\infty$-category, we let $\End(\cat{C})$ denote the $\infty$-category of endofunctors on $\cat{C}$. We often regard $\End(\cat{C})$ as a monoidal $\infty$-category via the \emph{composition monoidal structure}, over which $\cat{C}$ is left-tensored (via evaluation). A \emph{(co)monad on $\cat{C}$} is a (co)algebra object $T$ of $\End(\cat{C})$. Given a (co)monad $T$ on $\cat{C}$, we may consider \emph{$T$-(co)module objects of $\cat{C}$} (note that others often refer to these as $T$-(co)algebras).

At various points, we will invoke the following notion of \emph{endomorphism objects}, introduced in \cite[\S4.7.1]{lurie--algebra}. Let $\cat{E}$ be a monoidal $\infty$-category and let $\cat{C}$ an $\infty$-category left tensored over $\cat{E}$. Given an object $X \in \cat{C}$, an \emph{endomorphism object for $X$ in $\cat{E}$} is an object $\End(X) \in \cat{E}$ equipped with a map $\epsilon : \End(X) \otimes X \to X$ in $\cat{C}$, such that, for every $E \in \cat{E}$, the map
\[
  \Map_{\cat{E}}(E,\End(X)) \to \Map_{\cat{C}}(E \otimes X,X)
\]
induced by $\epsilon$ is an equivalence. It is shown in loc. cit. that, in this situation, $\End(X)$ canonically promotes to an algebra in $\cat{E}$ and the map $\epsilon$ canonically promotes to a left $E$-module structure on $X$, such that, for any $E \in \Alg(\cat{E})$, the induced map
\[
  \Map_{\Alg(\cat{E})}(E,\End(X)) \to \LMod_{E}(\cat{C}) \times_{\cat{C}} \{X\}
\]
is an equivalence. There is an evident dual notion of \emph{coendomorphism objects}, which we will also make use of.

Finally, throughout, we let $\cir$ denote the circle, regarded as a group object in $\Spc$, and we let $\clspc\cir$ denote its classifying space.


\section{Bialgebras}
\label{bi}

Let $\num{Z}[\cir]$ denote the group algebra over $\num{Z}$ on $\cir$. Recall that there is a canonical equivalence of $\infty$-categories
\[
  \Mod_{\num{Z}[\cir]} \iso \Fun(\clspc\cir,\Mod_{\num{Z}}).
\]
(see \cref{bi--tn--rep}). The symmetric monoidal structure on $\Mod_{\num{Z}}$ induces a pointwise symmetric monoidal structure on the right-hand side (encoding the tensor product of $\cir$-representations), which transports under the above equivalence to a symmetric monoidal structure on $\Mod_{\num{Z}[\cir]}$. This is not given by the relative tensor product over $\num{Z}[\cir]$, but there is a way of producing this symmetric monoidal structure on $\Mod_{\num{Z}[\cir]}$ purely in terms of structure on the group algebra. The important point is that $\num{Z}[\cir]$ is not just an algebra, but moreover a \emph{cocommutative bialgebra}: that is, it carries a cocommutative coalgebra structure compatible with its algebra structure. For example, the diagonal map $\cir \to \cir \times \cir$ induces a map of algebras $\Delta : \num{Z}[\cir] \to \num{Z}[\cir \times \cir] \iso \num{Z}[\cir] \otimes_\num{Z} \num{Z}[\cir]$. This leads to the desired tensor product of $\num{Z}[\cir]$-modules as follows: given two $\num{Z}[\cir]$-modules $M$ and $N$, the $\num{Z}$-linear tensor product $M \otimes_\num{Z} N$ canonically has the structure of a $\num{Z}[\cir] \otimes_\num{Z} \num{Z}[\cir]$-module, and we restrict this to a $\num{Z}[\cir]$-module structure using the map $\Delta$.

In \cref{dg,fc}, we will be dealing with structures that are analogous to $\cir$-actions---namely, homotopy-coherent cochain complexes and filtered $\cir$-actions---but that cannot be formulated using functor categories, rather only as module categories in the linear setting. The purpose of this section is to prepare for those situations, by studying bialgebras in the general context of a (stable) presentable symmetric monoidal $\infty$-category. Our aim is to explain how certain features of representation categories like $\Fun(\clspc\cir,\Mod_{\num{Z}})$, e.g. the tensor product discussed above, carry through to categories of modules (or comodules) over bialgebras in this generality.

This section is organized as follows: in \cref{bi--df}, we state the basic definitions of bialgebras in the $\infty$-categorical setting; in \cref{bi--tn}, we construct the symmetric monoidal structure on the category of modules over a cocommutative bialgebra informally sketched above; in \cref{bi--du}, we discuss what happens when one dualizes a bialgebra; and in \cref{bi--ta}, we formulate generalizations to this setting of the orbits, fixed points, and Tate construction for group actions.


\subsection{Definitions}
\label{bi--df}

Fix a symmetric monoidal $\infty$-category $\cat{C}$. Let us begin by recalling the definition of coalgebra and comodule objects of $\cat{C}$. In what follows, recall that the symmetric monoidal structure on $\cat{C}$ determines a symmetric monoidal structure on $\cat{C}^\op$.

\begin{definition}
  \label{bi--df--co}
  A \emph{coalgebra object of $\cat{C}$} is an algebra object of $\cat{C}^\op$. We let $\cAlg(\cat{C})$ denote the $\infty$-category $\Alg(\cat{C}^\op)^\op$, and refer to this as the \emph{$\infty$-category of coalgebra objects of $\cat{C}$}. Given a coalgebra object $A$ of $\cat{C}$, a \emph{left} \emph{$A$-comodule object of $\cat{C}$} is a left $A$-module object of $\cat{C}^\op$. We let $\cLMod_A(\cat{C})$ denote the $\infty$-category $\LMod_A(\cat{C}^\op)^\op$, and refer to this as the \emph{$\infty$-category of left $A$-comodule objects of $\cat{C}$}.

  Similarly, a \emph{cocommutative coalgebra object of $\cat{C}$} is a commutative algebra object of $\cat{C}^\op$. We let $\cCAlg(\cat{C})$ denote the $\infty$-category $\CAlg(\cat{C}^\op)^\op$, and refer to this as the \emph{$\infty$-category of cocommutative coalgebra objects of $\cat{C}$}. Given a cocommutative coalgebra object $A$ of $\cat{C}$, an \emph{$A$-comodule object of $\cat{C}$} is an $A$-module object of $\cat{C}^\op$. We let $\cMod_A(\cat{C})$ denote the $\infty$-category $\Mod_A(\cat{C}^\op)^\op$, and refer to this as the \emph{$\infty$-category of $A$-comodule objects of $\cat{C}$}.
\end{definition}

Now, there are two natural options for encoding the idea of an object of $\cat{C}$ equipped with compatible algebra and coalgebra structures. Noting that the $\infty$-categories $\Alg(\cat{C})$ and $\cAlg(\cat{C})$ inherit symmetric monoidal structures from $\cat{C}$ \cite[Example 3.2.4.4]{lurie--algebra}, we may consider both coalgebra objects of $\Alg(\cat{C})$ and algebra objects of $\cAlg(\cat{C})$. In this paper, we will only deal with the situation where at least one of the algebra and coalgebra structures is commutative. In these cases, we can show that the two options are equivalent:

\begin{proposition}
  \label{bi--df--swap}
  There are canonical equivalences of symmetric monoidal $\infty$-categories
  \[
    \cAlg(\CAlg(\cat{C})) \isoto \CAlg(\cAlg(\cat{C})), \quad
    \Alg(\cCAlg(\cat{C})) \isoto \cCAlg(\Alg(\cat{C})),
  \]
  \[
    \CAlg(\cCAlg(\cat{C})) \isoto \cCAlg(\CAlg(\cat{C}))
  \]
  commuting with the (symmetric monoidal) forgetful functors to $\cat{C}$.
\end{proposition}

\begin{remark}
  \label{bi--df--swap-motivation}
  For motivation, we point out that \cref{bi--tn--calg} below invokes one of the identifications of \cref{bi--df--swap}.
\end{remark}

\begin{proof}[Proof of \cref{bi--df--swap}]
  The third equivalence is \cite[Corollary 3.3.4]{lurie--elliptic-i}.

  Let us address the first equivalence. The symmetric monoidal forgetful functor $\CAlg(\cat{C}) \to \cat{C}$ induces a symmetric monoidal functor $G : \cAlg(\CAlg(\cat{C})) \to \cAlg(\cat{C})$. We also have a symmetric monoidal forgetful functor $G' : \CAlg(\cAlg(\cat{C})) \to \cAlg(\cat{C})$. Since the symmetric monoidal structure on $\CAlg(\cat{C})$ is cocartesian \cite[Corollary 3.2.4.7]{lurie--algebra}, and the forgetful functor $\cAlg(\CAlg(\cat{C})) \to \CAlg(\cat{C})$ is symmetric monoidal and creates finite coproducts \cite[Corollary 3.2.2.4]{lurie--algebra}, the symmetric monoidal structure on $\cAlg(\CAlg(\cat{C}))$ is also cocartesian. It then follows from \cite[Theorem 2.4.3.18]{lurie--algebra} that there is a unique symmetric monoidal functor $U: \cAlg(\CAlg(\cat{C})) \to \CAlg(\cAlg(\cat{C}))$ equipped with an equivalence $G' \circ U \iso G$. We wish to show that $U$ is an equivalence.

  Let us first consider the case that $\cat{C}$ is presentable. Then the functors $G$ and $G'$ admit left adjoints, both given by formation of symmetric algebras in $\cat{C}$: for $G$, note that the symmetric algebra functor $\Sym_{\cat{C}} : \cat{C} \to \CAlg(\cat{C})$ obtains an oplax symmetric monoidal structure as the left adjoint to the symmetric monoidal forgetful functor $\CAlg(\cat{C}) \to \cat{C}$, and hence induces a functor $\cAlg(\cat{C}) \to \cAlg(\CAlg(\cat{C}))$ left adjoint to $G$; for $G'$, note that the forgetful functor $\cAlg(\cat{C}) \to \cat{C}$ creates colimits, so that $G'$ admits a left adjoint given by formation of symmetric algebras in $\cAlg(\cat{C})$, which agrees with formation of symmetric algebras in $\cat{C}$ after forgetting coalgebra structures. Moreover, $G$ and $G'$ are both monadic (see \cite[Example 4.7.3.11]{lurie--algebra}, and note once more that the forgetful functor $\cAlg(\cat{C}) \to \cat{C}$ creates colimits), so it follows from \cite[Corollary 4.7.3.16]{lurie--algebra} that $U$ is an equivalence.

  Suppose now that $\cat{C}$ is not presentable. Enlarging the universe if necessary, we may assume that $\cat{C}$ is small. Then we may consider the symmetric monoidal Yoneda embedding $\cat{C} \inj \PSh(\cat{C})$, where $\PSh(\cat{C})$ is equipped with the Day convolution symmetric monoidal structure (see \cite[Corollary 4.8.1.12]{lurie--algebra} or \cite[Proposition 2.14]{glasman--day}). As $\PSh(\cat{C})$ is a presentable symmetric monoidal $\infty$-category, the preceding argument gives us an equivalence $\cAlg(\CAlg(\PSh(\cat{C}))) \isoto \CAlg(\cAlg(\PSh(\cat{C})))$. By the evident naturality of the functor $U$ above, this restricts to the desired equivalence $\cAlg(\CAlg(\cat{C})) \isoto \CAlg(\cAlg(\cat{C}))$ upon pulling back from $\PSh(\cat{C})$ to $\cat{C}$.

  Establishing the second equivalence is similar, except that, for $\cat{C}$ presentable, we invoke cartesianness instead of cocartesianness, the right adjoints to the forgetful functors $\Alg(\cCAlg(\cat{C})) \to \Alg(\cat{C})$ and $\cCAlg(\Alg(\cat{C})) \to \Alg(\cat{C})$ (which exist by \cite[Corollary 3.1.5]{lurie--elliptic-i}) instead of left adjoints, and the dual comonadic form of \cite[Corollary 4.7.3.16]{lurie--algebra}.
\end{proof}

\begin{definition}
  \label{bg--bi--bi}
  We set $\bAlg_\comm(\cat{C}) := \CAlg(\cAlg(\cat{C}))$, and implicitly identify this with the $\infty$-category $\cAlg(\CAlg(\cat{C}))$ whenever convenient using the equivalence of \cref{bi--df--swap}. We refer to objects of $\bAlg_\comm(\cat{C})$ as \emph{commutative bialgebra objects of $\cat{C}$}. Similarly, we set $\bAlg^\comm(\cat{C}) := \cCAlg(\Alg(\cat{C}))$, which we implicitly identify with $\Alg(\cCAlg(\cat{C}))$, and whose objects we refer to as  \emph{cocommutative bialgebra objects of $\cat{C}$}; and we set $\bAlg^\comm_\comm(\cat{C}) := \CAlg(\cCAlg(\cat{C}))$, which we implicitly identify with $\cCAlg(\CAlg(\cat{C}))$, and whose objects we refer to as \emph{bicommutative bialgebra objects of $\cat{C}$} (see also \cite[\textsection 3.3]{lurie--elliptic-i} for this definition).
\end{definition}

\begin{example}
  \label{bi--df--monoid}
  Suppose that the symmetric monoidal structure on $\cat{C}$ is cartesian. Then the forgetful functor $\cCAlg(\cat{C}) \to \cat{C}$ is a (symmetric monoidal) equivalence (this is dual to \cite[Corollary 2.4.3.10]{lurie--algebra}), so the same is true of the forgetful functor $\Alg(\cCAlg(\cat{C})) \to \Alg(\cat{C})$. As $\Alg(\cat{C})$ is equivalent to the $\infty$-category of monoid objects in $\cat{C}$, we deduce that any monoid object of $\cat{C}$ canonically determines a cocommutative bialgebra object of $\cat{C}$.
\end{example}


\subsection{Tensor products}
\label{bi--tn}

In this subsection, we construct the canonical tensor product of modules over a cocommutative bialgebra, generalizing the tensor product of group representations, making precise the informal discussion from the beginning of the section. It will actually be more convenient to discuss the construction in the dual context of comodules over a commutative bialgebra.

Out of convenience, we will make most statements in this subsection in the setting of small $\infty$-categories; however, everything goes through relative to an arbitrary universe, so we will feel free to apply the results of this subsection to large $\infty$-categories. We regard $\Cat_\infty$ as a symmetric monoidal $\infty$-category via the cartesian symmetric monoidal structure. We identify small (symmetric) monoidal $\infty$-categories with (commutative) algebra objects of $\Cat_\infty$; furthermore, for $\cat{C}$ a small monoidal $\infty$-category, we identify small $\infty$-categories right tensored over $\cat{C}$ with right $\cat{C}$-modules in $\Cat_\infty$ \cite[Example 2.4.2.4 and Proposition 2.4.2.5]{lurie--algebra}. If $\cat{C}$ is a small symmetric monoidal $\infty$-category, then $\RMod_{\cat{C}}(\Cat_\infty)$ carries a canonical symmetric monoidal structure given by the relative tensor product construction, so that commutative algebra objects of $\RMod_{\cat{C}}(\Cat_\infty)$ may be identified with small symmetric monoidal $\infty$-categories $\cat{D}$ under $\cat{C}$ (see \cite[\textsection 3.4.1 and \textsection 4.5]{lurie--algebra}, and note that cartesian product in $\Cat_\infty$ commutes with colimits in each variable, as $\Cat_\infty$ is cartesian closed).

The following result encapsulates the main content of this subsection.

\begin{proposition}
  \label{bi--tn--main}
  Let $\cat{C}$ be a small symmetric monoidal $\infty$-category. Then the construction $A \mapsto \cLMod_A(\cat{C})$ canonically extends to a symmetric monoidal functor $\mu_{\cat{C}} : \cAlg(\cat{C}) \to \RMod_{\cat{C}}(\Cat_\infty)$.
\end{proposition}

Before proving this, let us explain how it supplies the symmetric monoidal structures we are looking for.

\begin{construction}
  \label{bi--tn--main-cor}
  Let $\cat{C}$ be a small symmetric monoidal $\infty$-category and let $A$ be a commutative bialgebra object of $\cat{C}$. Using \cref{bi--tn--main}, we obtain a symmetric monoidal structure on the $\infty$-category $\cLMod_A(\cat{C})$, together with one on the forgetful functor $\cLMod_A(\cat{C}) \to \cat{C}$, as follows.

  Note that $\cAlg(\cat{C})$ admits a final object, namely the trivial coalgebra $\unit$, with $\cLMod_\unit(\cat{C}) \iso \cat{C}$. It follows that the symmetric monoidal functor $\mu_{\cat{C}}$ of \cref{bi--tn--main} factors canonically through a symmetric monoidal functor $\mu'_{\cat{C}} : \cAlg(\cat{C}) \to \RMod_{\cat{C}}(\Cat_\infty)_{/\cat{C}}$. This induces a functor on commutative algebra objects
  \[
    \bAlg_\comm(\cat{C}) \iso \CAlg(\cAlg(\cat{C})) \to \CAlg(\RMod_{\cat{C}}(\Cat_\infty)_{/\cat{C}}) \iso \CAlg(\Cat_\infty)_{\cat{C}//\cat{C}}. \qedhere
  \]
  Applying this functor to $A \in \bAlg_\comm(\cat{C})$ gives the desired symmetric monoidal functor $\cLMod_A(\cat{C}) \to \cat{C}$.
\end{construction}

\begin{variant}
  \label{bi--tn--main-cor-op}
  Let $\cat{C}$ be a small symmetric monoidal $\infty$-category and let $A$ be a cocommutative bialgebra object of $\cat{C}$. Applying \cref{bi--tn--main-cor} with $\cat{C}^\op$ in place of $\cat{C}$, we obtain a symmetric monoidal structure on the $\infty$-category $\LMod_A(\cat{C})$, together with one on the forgetful functor $\LMod_A(\cat{C}) \to \cat{C}$.
\end{variant}

We now move on to the proof of \cref{bi--tn--main}. The key constructions needed for this are very similar to those of \cite[\S\S4.8.3--4.8.5]{lurie--algebra}; we will omit details here.

\begin{notation}
  \label{bi--tn--fib}
  The construction $\cat{C} \mapsto \cAlg(\cat{C})$ determines a functor $\cAlg : \Alg(\Cat_\infty) \to \Cat_\infty$ (on maps, this sends a monoidal functor $\cat{C} \to \cat{D}$ to the induced functor $\cAlg(\cat{C}) \to \cAlg(\cat{D})$), classified by a cocartesian fibration that we denote by $\Cat_\infty^{\cAlg} \to \Alg(\Cat_\infty)$. We identify objects of $\Cat_\infty^{\cAlg}$ with pairs $(\cat{C},A)$ where $\cat{C}$ is a small monoidal $\infty$-category and $A$ is a coalgebra object of $\cat{C}$
  
  For the sake of consistent notation, let $\Cat_\infty^\RMod$ denote the $\infty$-category $\RMod(\Cat_\infty)$ of pairs $(\cat{C},\cat{M})$ where $\cat{C}$ is a small monoidal $\infty$-category and $\cat{M}$ is a small $\infty$-category right tensored over $\cat{C}$. We have a cocartesian fibration $\Cat_\infty^\RMod \to \Alg(\Cat_\infty)$ sending $(\cat{C},\cat{M}) \mapsto \cat{C}$. 
\end{notation}

\begin{remark}
  \label{bi--tn--fib-cartesian}
  The $\infty$-categories $\Cat_\infty^{\cAlg}$ and $\Cat_\infty^\RMod$ admit finite products, given by the formulas
  \[
    \textstyle{\prod_{i \in I} (\cat{C}_i,A_i) \iso (\prod_{i \in I} \cat{C}_i, \{A_i\}_{i \in I})}
    \quad\text{and}\quad
    \textstyle{\prod_{i \in I} (\cat{C}_i,\cat{M}_i) \iso (\prod_{i \in I} \cat{C}_i, \prod_{i \in I} \cat{M}_i)}
  \]
  respectively. We regard $\Cat_\infty^{\cAlg}$ and $\Cat_\infty^\RMod$ as symmetric monoidal $\infty$-categories via the cartesian symmetric monoidal structures.
\end{remark}

\begin{construction}
  \label{bi--tn--key}
  Given a monoidal $\infty$-category $\cat{C}$ and a coalgebra object $A$ of $\cat{C}$, the $\infty$-category $\cLMod_A(\cat{C})$ is canonically right tensored over $\cat{C}$ (the dual construction for modules is explained in \cite[\textsection 4.3.2]{lurie--algebra}). Moreover, this construction canonically extends to a functor $\mu : \Cat_\infty^{\cAlg} \to \Cat_\infty^\RMod$ over $\Alg(\Cat_\infty)$ sending $(\cat{C},A) \mapsto (\cat{C},\cLMod_A(\cat{C}))$. The functor $\mu$ evidently preserves finite products, and we thus regard it as a symmetric monoidal functor.
\end{construction}

\begin{proof}[Proof of \cref{bi--tn--main}]
  The small symmetric monoidal $\infty$-category $\cat{C}$ defines a commutative algebra object of $\Cat_\infty$. Consider the span of functors
  \[
    \CAlg(\Cat_\infty) \from \CAlg(\CAlg(\Cat_\infty)) \to \CAlg(\Alg(\Cat_\infty)),
  \]
  induced by the forgetful functors $\CAlg(\Cat_\infty) \to \Alg(\Cat_\infty) \to \Cat_\infty$. By \cite[Example 3.2.4.5]{lurie--algebra}, the left arrow in this span is an equivalence. Hence, $\cat{C}$ may also be regarded as a commutative algebra object of $\Alg(\Cat_\infty)$. This determines symmetric monoidal structures on the fibers
  \[
    \Cat_\infty^{\cAlg} \times_{\Alg(\Cat_\infty)} \{\cat{C}\} \iso \cAlg(\cat{C})
    \quad\text{and}\quad
    \Cat_\infty^\RMod \times_{\Alg(\Cat_\infty)} \{\cat{C}\} \iso \RMod_\cat{C}(\Cat_\infty),
  \]
  which identify canonically with the usual ones (cf. \cite[Remark 4.8.5.19]{lurie--algebra}). Thus, the symmetric monoidal functor $\mu$ of \cref{bi--tn--key} restricts to the desired symmetric monoidal functor $\mu_{\cat{C}} : \cAlg(\cat{C}) \to \RMod_\cat{C}(\Cat_\infty)$, with $\mu_{\cat{C}}(A) \iso \cLMod_A(\cat{C})$.
\end{proof}

\begin{remark}
  \label{bi--tn--functorial}
  Let $\cat{C}$ be a small symmetric monoidal $\infty$-category. We note the following functoriality properties of \cref{bi--tn--main-cor}:
  \begin{enumerate}
  \item \label{bi--tn--functorial--algebra} Let $A \to A'$ be a map of commutative bialgebras in $\cat{C}$. Then there is a canonical symmetric monoidal structure on the corestriction functor $\cLMod_A(\cat{C}) \to \cLMod_{A'}(\cat{C})$, compatible with the symmetric monoidal forgetful functors to $\cat{C}$. This is visible from the construction.
  \item \label{bi--tn--functorial--category} Let $\cat{D}$ be another small symmetric monoidal $\infty$-category and $F : \cat{C} \to \cat{D}$ a symmetric monoidal functor. Let $A$ be a commutative bialgebra in $\cat{C}$. Then there is an induced commutative bialgebra structure on $F(A)$, and the functor $F' : \cLMod_A(\cat{C}) \to \cLMod_{F(A)}(\cat{D})$ determined by $F$ is canonically symmetric monoidal. This follows from the proof of \cref{bi--tn--main}, more precisely the functoriality of $\mu$.
  \end{enumerate}
  Analogous statements hold in the dual situation of \cref{bi--tn--main-cor-op}.
\end{remark}

\begin{remark}
  \label{bi--tn--presentable}
  Let $\cat{C}$ be a presentable symmetric monoidal $\infty$-category and let $A$ be a cocommutative bialgebra object of $\cat{C}$. Applying the above constructions in a larger universe, we obtain a symmetric monoidal structure on $\LMod_A(\cat{C})$. We note that this symmetric monoidal structure is a presentable one: that is, the tensor product commutes with colimit separately in each variable. This follows from the fact that the symmetric monoidal forgetful functor $\LMod_A(\cat{C}) \to \cat{C}$ preserves colimits. (The same can be said if we replace $A$ by a commutative bialgebra object of $\cat{C}$ and $\LMod_A(\cat{C})$ with $\cLMod_A(\cat{C})$.)

  In addition, if $F : \cat{C} \to \cat{D}$ is a morphism in $\CAlg(\PrL)$, then the symmetric monoidal functor $F' : \LMod_A(\cat{C}) \to \LMod_{F(A)}(\cat{D})$ of \cref{bi--tn--functorial}\cref{bi--tn--functorial--category} induces an map $\cat{D} \otimes_\cat{C} \LMod_A(\cat{C}) \to \LMod_{F(A)}(\cat{D})$ in $\CAlg(\PrL)$, which is an equivalence by \cite[Theorem 4.8.4.6]{lurie--algebra}.
\end{remark}

\begin{example}
  \label{bi--tn--rep}
  Let $G$ be a group object of $\Spc$; by \cref{bi--df--monoid}, we may regard $G$ as a cocommutative bialgebra object of $\Spc$. Let $\cat{C}$ be a presentable symmetric monoidal $\infty$-category. Then we have a unique colimit-preserving symmetric monoidal functor $F : \Spc \to \cat{C}$, giving us a cocommutative bialgebra $F(G)$ in $\cat{C}$. From \cref{bi--tn--main-cor-op}, we obtain a symmetric monoidal structure on $\LMod_{F(G)}(\cat{C})$. We claim that there is a canonical symmetric monoidal equivalence between this and $\Fun(\clspc G, \cat{C})$ (with the pointwise tensor product, and where $\clspc G$ denotes the classifying space of $G$).

  To see this, first observe that we have canonical equivalences $\LMod_{F(G)}(\cat{C}) \iso \cat{C} \otimes \LMod_G(\Spc)$ and $\Fun(\clspc G,\cat{C}) \iso \cat{C} \otimes \Fun(\clspc G,\Spc)$ in $\CAlg(\PrL)$ (the former as mentioned in \cref{bi--tn--presentable} and the latter following from \cite[Proposition 4.8.1.17]{lurie--algebra}). It therefore suffices to treat the case $\cat{C} = \Spc$. In this case, the symmetric monoidal structures on $\LMod_G(\Spc)$ and $\Fun(\clspc G,\Spc)$ are cartesian, so the standard equivalence $\LMod_G(\Spc) \iso \Fun(\clspc G,\Spc)$ promotes uniquely to a symmetric monoidal equivalence.

  For completeness, let us review one description of the comparison functor $\LMod_G(\Spc) \to \Fun(\clspc G, \Spc)$. It is given by the composition
  \[
    \LMod_G(\Spc) \to \Spc_{/\clspc G} \to \Fun(\clspc G, \Spc),
  \]
  where the first functor is given by relative product (in the sense of \cite[\textsection 4.4]{lurie--algebra}) with the final object, i.e. $- \times_G *$ (together with the observation that this sends the final object $*$ to the bar construction/classifying space $\clspc G$), and the second functor is the straightening equivalence.
\end{example}

We close this subsection with an alternative characterization of commutative algebra objects for the symmetric monoidal structures constructed above.

\begin{proposition}
  \label{bi--tn--calg}
  In the situation of \cref{bi--tn--main-cor}, there is a canonical equivalence of $\infty$-categories
  \[
    \cLMod_A(\CAlg(\cat{C})) \iso \CAlg(\cLMod_A(\cat{C}))
  \]
  commuting with the forgetful functors to $\cat{C}$. Here, on the left-hand side, we regard the commutative bialgebra $A$ as a coalgebra object of $\CAlg(\cat{C})$, while on the right-hand side, we regard it as a commutative algebra object of $\cAlg(\cat{C})$ (as was done in \cref{bi--tn--main-cor}).
\end{proposition}

\begin{proof}
  Using the fact that the forgetful functor $\CAlg(\CAlg(\cat{C})) \to \CAlg(\cat{C})$ is an equivalence, we may regard $A$ as a commutative bialgebra object of $\CAlg(\cat{C})$ (which is carried to the original commutative bialgebra object $A$ of $\cat{C}$ under the forgetful functor $\CAlg(\cat{C}) \to \cat{C}$). By \cref{bi--tn--functorial}\cref{bi--tn--functorial--category}, we have a canonical symmetric monoidal functor $\alpha : \cLMod_A(\CAlg(\cat{C})) \to \cLMod_A(\cat{C})$. The symmetric monoidal structure on $\CAlg(\cat{C})$ is cocartesian, from which we deduce that the same is true of $\cLMod_A(\CAlg(\cat{C}))$  (as the forgetful functor $\cLMod_A(\CAlg(\cat{C})) \to \CAlg(\cat{C})$ is symmetric monoidal and preserves coproducts). It follows that the functor $\alpha$ factors uniquely through the forgetful functor $\CAlg(\cLMod_A(\cat{C}))$, giving us a commutative diagram of $\infty$-categories
  \[
    \begin{tikzcd}
      \cLMod_A(\CAlg(\cat{C})) \ar[rr] \ar[rd] &
      &
      \CAlg(\cLMod_A(\cat{C})) \ar[ld] \\
      &
      \cLMod_A(\cat{C}).
    \end{tikzcd}
  \]

  We wish to show that the horizontal arrow in the above diagram is an equivalence. For this we can argue as in the proof of \cref{bi--df--swap}. Firstly, we may pass to the case that $\cat{C}$ is a presentable symmetric monoidal $\infty$-category. Then, the two diagonal arrows in the diagram admit left adjoints, both given by the formation of symmetric algebras in $\cat{C}$, and moreover the diagonal arrows are monadic. Finally, it follows from \cite[Corollary 4.7.3.16]{lurie--algebra} that the horizontal arrow is an equivalence.
\end{proof}


\subsection{Duality}
\label{bi--du}

Let $\cat{C}$ be a symmetric monoidal $\infty$-category and let $\unit \in \cat{C}$ denote the unit object. Recall that an object $X \in \cat{C}$ is called \emph{dualizable} if there exists an object $\dual{X} \in \cat{C}$ and maps $\epsilon : X \otimes \dual{X} \to \unit$ and $\eta : \dual{X} \otimes X \to \unit$ such that the compositions
\[
  X \iso \unit \otimes X \lblto{\eta \otimes \id} X \otimes \dual{X} \otimes X \lblto{\id \otimes \epsilon} X \otimes \unit \iso X,
\]
\[
  \dual{X} \iso \dual{X} \otimes \unit \lblto{\id \otimes \eta} \dual{X} \otimes X \otimes \dual{X} \lblto{\epsilon \otimes \id} \unit \otimes \dual{X} \iso \dual{X}
\]
are homotopic to the identity. If $X$ is dualizable, the object $\dual{X}$ is determined up to equivalence and referred to as the \emph{dual} of $X$.

Our goal in this subsection is to understand the behavior of bialgebra structures under dualization. Similar to the previous subsection, we shall work under the assumption that $\cat{C}$ is small, but this restriction is only apparent.

\begin{notation}
  \label{bi--du--fd}
  We let $\cat{C}_\fd$ denote the full subcategory of $\cat{C}$ spanned by the dualizable objects; it contains $\unit$ and is closed under tensor products, hence inherits a symmetric monoidal structure from $\cat{C}$. Dualization then determines a functor $\dual{(-)} : \cat{C}_\fd \to \cat{C}_\fd^\op$.
\end{notation}

\begin{proposition}
  \label{bi--du--main}
  There exists a symmetric monoidal functor
  \[
    \beta : \cAlg(\cat{C}_\fd) \to \Alg(\cat{C}_\fd)^\op \iso \cAlg(\cat{C}_\fd^\op)
  \]
  uniquely determined by the commuting of the following diagram of symmetric monoidal $\infty$-categories:
  \[
    \begin{tikzcd}
      \cAlg(\cat{C}_\fd) \ar[r, hook] \ar[d, "\beta"] &
      \cAlg(\cat{C}) \ar[r, "\mu'_{\cat{C}}"] &
      \RMod_\cat{C}(\Cat_\infty)_{/\cat{C}} \ar[d, "(-)^\op"] \\
      \cAlg(\cat{C}_\fd^\op) \ar[r, hook] &
      \cAlg(\cat{C}_\fd^\op) \ar[r, "\mu'_{\cat{C}^\op}"] &
      \RMod_{\cat{C}^\op}(\Cat_\infty)_{/\cat{C}^\op}
    \end{tikzcd}
  \]
  (where the symmetric monoidal functors $\mu'_{(-)}$ are as in \cref{bi--tn--main-cor}). Moreover, the diagram of $\infty$-categories
  \[
    \begin{tikzcd}
      \cAlg(\cat{C}_\fd) \ar[r, "\beta"] \ar[d] &
      \cAlg(\cat{C}_\fd^\op) \ar[d] \\
      \cat{C}_\fd \ar[r, "\dual{(-)}"] &
      \cat{C}_\fd^\op
    \end{tikzcd}
  \]
  canonically commutes (where the vertical arrows are the forgetful functors).
\end{proposition}

\begin{corollary}
  \label{bi--du--main-cor}
  Let $A$ be a dualizable coalgebra (resp. commutative bialgebra) object of $\cat{C}$. Then there exists a unique algebra (resp. cocommutative bialgebra) structure on $\dual{A}$ equipped with an equivalence of $\infty$-categories (resp. symmetric monoidal $\infty$-categories) $\cLMod_A(\cat{C}) \iso \LMod_{\dual{A}}(\cat{C})$ commuting with the forgetful functors to $\cat{C}$.
\end{corollary}

\begin{remark}
  \label{bi--du--main-equiv}
  It is immediate from \cref{bi--du--main} that the composite
  \[
    \cAlg(\cat{C}_\fd) \lblto{\beta} \Alg(\cat{C}_\fd)^\op \lblto{\beta} \cAlg(\cat{C}_\fd)
  \]
  is canonically equivalent to the identity, and hence that $\beta$ is in fact a symmetric monoidal equivalence.
\end{remark}

The proof of \cref{bi--du--main} will require some preliminaries. In the following two results, we will denote objects of the $\infty$-category $\RMod_{\cat{C}}(\Cat_\infty)_{/\cat{C}}$ by pairs $(\cat{M},U)$, where $\cat{M}$ is an $\infty$-category right tensored over $\cat{C}$ and $U$ is a $\cat{C}$-linear functor $U : \cat{M} \to \cat{C}$, and sometimes will omit $U$ when it may be understood from context. Also, note that if $U$ admits a right adjoint $G : \cat{C} \to \cat{M}$, then $G$ is canonically lax $\cat{C}$-linear, and we may consider the property that $G$ is in fact strictly $\cat{C}$-linear (\cite[Example 7.3.2.8 and Remark 7.3.2.9]{lurie--algebra}).

\begin{lemma}
  \label{bi--du--linear-comonad}
  Let $(\cat{M},U)$ be an object of $\RMod_{\cat{C}}(\Cat_\infty)_{/\cat{C}}$ such that $U$ admits a right adjoint $G : \cat{C} \to \cat{M}$ that is $\cat{C}$-linear. Then there is a coalgebra $A \in \cAlg(\cat{C})$ equipped with a map $\alpha : \cat{M} \to \cLMod_A(\cat{C})$ in $\RMod_{\cat{C}}(\Cat_\infty)_{/\cat{C}}$, such that, for any $B \in \cAlg(\cat{C})$, the map
  \[
    \Map_{\cAlg(\cat{C})}(A,B) \to \Map_{\RMod_{\cat{C}}(\Cat_\infty)_{/\cat{C}}}(\cat{M},\cLMod_B(\cat{C}))
  \]
  induced by $\alpha$ is a homotopy equivalence.
\end{lemma}

\begin{proof}
  Let $T := UG \in \End(\cat{C})$. Recall from \cite[\S4.7.3]{lurie--algebra} that:
  \begin{itemize}
  \item the natural transformation $\eta' : U \to TU = UGU$ induced by the unit transformation $\eta : \id_\cat{M} \to GU$ exhibits $T$ as a \emph{coendomorphism object} for $U \in \Fun(\cat{M},\cat{C})$, where we regard $\Fun(\cat{M},\cat{C})$ as left tensored over the monoidal $\infty$-category $\End(\cat{C})$ via postcomposition;
  \item this determines a comonad structure on $T$ and a factorization of $U$ through the forgetful functor $\cLMod_T(\cat{C}) \to \cat{C}$, such that, for any other comonad $T' \in \cAlg(\End(\cat{C}))$, the induced map
    \[
      \Map_{\cAlg(\End(\cat{C}))}(T,T') \to \Map_{(\Cat_\infty)_{/\cat{C}}}(\cat{M},\cLMod_{T'}(\cat{C}))
    \]
    is a homotopy equivalence.
  \end{itemize}
  In our situation, the functors $U,G,T$ and natural transformations $\eta,\eta'$ are all canonically $\cat{C}$-linear, and the same argument may be carried out with $\End(\cat{C})$ replaced by the monoidal $\infty$-category $\End_\cat{C}(\cat{C})$ of $\cat{C}$-linear endofunctors and $(\Cat_\infty)_{/\cat{C}}$ replaced by $\RMod_{\cat{C}}(\Cat_\infty)_{/\cat{C}}$. Since we have a canonical equivalence of monoidal $\infty$-categories $\End_\cat{C}(\cat{C}) \iso \cat{C}$, this gives the claim.
\end{proof}

\begin{proposition}
  \label{bi--du--embedding}
  The functor $\mu'_{\cat{C}} : \cAlg(\cat{C}) \to \RMod_{\cat{C}}(\Cat_\infty)_{/\cat{C}}$ is fully faithful. Its essential image consists of those objects $(\cat{M},U) \in \RMod_{\cat{C}}(\Cat_\infty)_{/\cat{C}}$ satisfying the following properties:
  \begin{enumerate}
  \item $U$ is comonadic, in particular admits a right adjoint $G : \cat{C} \to \cat{M}$;
  \item $G$ is $\cat{C}$-linear.
  \end{enumerate}
\end{proposition}

\begin{proof}
  Let $\RMod_{\cat{C}}(\Cat_\infty)_{/\cat{C}}^0$ denote the full subcategory of $\RMod_{\cat{C}}(\Cat_\infty)_{/\cat{C}}$ spanned by those objects $(\cat{M},U)$ such that $U$ admits a $\cat{C}$-linear right adjoint $G$. The functor $\mu'_{\cat{C}}$ factors through this subcategory, so let us abuse notation and now regard $\mu'_{\cat{C}}$ as a functor $\cAlg(\cat{C}) \to \RMod_{\cat{C}}(\Cat_\infty)_{/\cat{C}}^0$. It follows from \cref{bi--du--linear-comonad} that $\mu'_{\cat{C}}$ admits a left adjoint $\nu'_{\cat{C}} : \RMod_{\cat{C}}(\Cat_\infty)_{/\cat{C}}^0 \to \cAlg(\cat{C})$. It is easy to see that the counit transformation $\nu'_{\cat{C}}\mu'_{\cat{C}} \to \id_{\cAlg(\cat{C})}$ is an equivalence, implying that $\mu'_{\cat{C}}$ is fully faithful, and that the unit transformation $\id_{\RMod_{\cat{C}}(\Cat_\infty)_{/\cat{C}}} \to \mu'_{\cat{C}}\nu'_{\cat{C}}$ is an equivalence on an object $(\cat{M},U)$ if and only if $U$ is comonadic, implying that the essential image of $\Psi_{\cat{C}}'$ is as claimed.
\end{proof}

We need one more lemma to prove \cref{bi--du--main}.

\begin{lemma}
  \label{bi--du--dual-adjoint}
  Let $A \in \cAlg(\cat{C}_\fd)$ be a dualizable coalgebra object of $\cat{C}$. Then the forgetful functor $U : \cLMod_A(\cat{C}) \to \cat{C}$ is monadic, in particular admits a left adjoint $F : \cat{C} \to \cLMod_A(\cat{C})$. Moreover, there is a canonical equivalence between the composition $UF$ and the functor $\dual{A} \otimes - : \cat{C} \to \cat{C}$.
\end{lemma}

\begin{proof}
  By a reduction similar to the one made in \cref{bi--df--swap}, we may instead work in the situation that $\cat{C}$ is a presentable symmetric monoidal $\infty$-category. Since $U$ is a forgetful functor from a comodule category, it preserves colimits. As $A$ is dualizable, the functor $A \otimes - : \cat{C} \to \cat{C}$ preserves limits, implying that $U$ also preserves limits. It follows from the adjoint functor theorem that $U$ admit a left adjoint $F$ and from the monadicity theorem that $U$ is monadic.

  Let $G : \cat{C} \to \cLMod_A(\cat{C})$ denote the right adjoint of $U$, which is given by tensoring with $A$, so that $UG$ is naturally equivalent to the functor $A \otimes - : \cat{C} \to \cat{C}$. It is then immediate from $UF$ being left adjoint to $UG$ that $UF \iso \dual{A} \otimes -$.
\end{proof}

\begin{proof}[Proof of \cref{bi--du--main}]
  By \cref{bi--du--embedding}, $\mu'_{\cat{C}}$ and $\mu'_{\cat{C}^\op}$ are fully faithful. Thus, uniqueness of $\beta$ is automatic, and to prove existence, it suffices to show for $A \in \cAlg(\cat{C}_\fd)$ that $\cLMod_A(\cat{C})^\op$ lies in the essential image of $\mu'_{\cat{C}^\op}|_{\cAlg(\cat{C}_\fd^\op)}$. Using the description of the essential image in \cref{bi--du--embedding} (and unravelling opposites), we need to show that the forgetful functor $U : \cLMod_A(\cat{C}) \to \cat{C}$ admits a $\cat{C}$-linear left adjoint $F : \cat{C} \to \cLMod_A(\cat{C})$, that $U$ is monadic, and that the $\cat{C}$-linear composite $UF$ is given by tensoring with a dualizable object of $\cat{C}$. But this follows from \cref{bi--du--dual-adjoint}. Since \cref{bi--du--dual-adjoint} in fact gives a natural identification $UF \iso \dual{A} \otimes -$, we deduce a natural identification between the underlying object in $\cat{C}_\fd^\op$ of $\beta(A)$ with $\dual{A}$.
\end{proof}

We note one more elaboration on the results above that will be relevant in the next subsection.

\begin{remark}
  \label{bi--du--dual-bimodule}
  Let $A \in \Alg(\cat{C}_\fd)$ be a dualizable algebra object of $\cat{C}$. By \cref{bi--du--main}, we have a canonical coalgebra structure on $\dual{A}$, together with a $\cat{C}$-linear equivalence $\LMod_A(\cat{C}) \iso \cLMod_{\dual{A}}(\cat{C})$ over $\cat{C}$. We claim that, more generally, for any $\infty$-category $\cat{D}$ left tensored over $\cat{C}$, we have a canonical $\cat{C}$-linear equivalence $\LMod_A(\cat{D}) \iso \cLMod_{\dual{A}}(\cat{D})$ over $\cat{D}$.

  As above, it suffices to prove this in the case that $\cat{C}$ is a presentable symmetric monoidal $\infty$-category and $\cat{D}$ is a $\cat{C}$-module in $\PrL$. Now, (the categorical dual of) \cref{bi--du--dual-adjoint} tells us that the forgetful functor $U : \LMod_A(\cat{C}) \to \cat{C}$ admits a $\cat{C}$-linear right adjoint $G : \cat{C} \to \LMod_A(\cat{C})$ and is comonadic, with associated comonad given by tensoring with $\dual{A}$. Applying $- \otimes_{\cat{C}} \cat{D}$ in $\PrL$, we deduce the same for the forgetful functor $\LMod_A(\cat{D}) \to \cat{D}$, giving the claimed equivalence.
  
  As one example, we find that we have an equivalence
  \begin{align*}
    \BMod{A}{A}(\cat{C})
    &\iso \RMod_A(\LMod_A(\cat{C})) \\
    &\iso \RMod_A(\cLMod_A(\cat{C})) \\
    &\iso \cRMod_{\dual{A}}(\cLMod_A(\cat{C})) \\
    &\iso \cBMod{\dual{A}}{\dual{A}}(\cat{C}).
  \end{align*}
  Here $\BMod{A}{A}(\cat{C})$ denotes the $\infty$-category of $A$-$A$-bimodules in $\cat{C}$ and $\cBMod{\dual{A}}{\dual{A}}(\cat{C})$ is the $\infty$-category of $\dual{A}$-$\dual{A}$-bicomodules in $\cat{C}$ (again defined dually, i.e. as $\BMod{\dual{A}}{\dual{A}}(\cat{C}^\op)^\op$). The first and last identifications are by \cite[Theorem 4.3.2.7]{lurie--algebra}. The middle two identifications come from the preceding discussion (invoking also the obvious variant for right modules and comodules).

  In the context of this example, note that $\dual{A}$ is canonically a bicomodule over itself, and so the preceding equivalence exhibits a canonical $A$-$A$-bimodule structure on $\dual{A}$.
\end{remark}


\subsection{The Tate construction}
\label{bi--ta}

Let $X$ be a spectrum with $\cir$-action, i.e. a diagram $\clspc\cir \to \Spt$. The \emph{homotopy orbits} $X_{\ho \cir}$ and \emph{homotopy fixed points} $\smash{X^{\ho \cir}}$ are defined as the colimit and limit, respectively, of the diagram. There is a canonical norm map $\Nm_X : X_{\ho \cir}[1] \to X^{\ho \cir}$ relating these two constructions, and the \emph{Tate construction} $X^{\tate \cir}$ is defined as the cofiber of $\Nm_X$. In the case that $X$ arises from the Hochschild homology construction, $X_{\ho \cir}$ is cyclic homology, $X^{\ho \cir}$ is negative cyclic homology, and $X^{\tate \cir}$ is periodic cyclic homology.

The goal of this subsection is to formulate these notions in a more general context that will apply not just to objects with $\cir$-action but to filtered objects with filtered $\cir$-action (and in fact to homotopy-coherent cochain complexes as well). Throughout this subsection, we work in the following general context:
\begin{itemize}
\item We let $\cat{C}$ be a stable presentable symmetric monoidal $\infty$-category. We denote the unit object by $\unit$. Presentability ensures that the monoidal structure is closed; we denote internal mapping objects in $\cat{C}$ by $\uMap(-,-)$.
\item We let $A$ be a cocommutative bialgebra object of $\cat{C}$. We regard $\LMod_A(\cat{C})$ as a presentable symmetric monoidal $\infty$-category via \cref{bi--tn--main-cor-op} (and \cref{bi--tn--presentable}), and denote the tensor product simply by $\otimes$. Beware that we will also be using the relative tensor product over $A$, which we denote by $\otimes_A$. In $\LMod_A(\cat{C})$ too we have $\cat{C}$-valued mapping objects, which we denote by $\uMap_A(-,-)$: these are determined by natural equivalences
  \[
    \Map_{\cat{C}}(T, \uMap_A(X,Y)) \iso \Map_{\LMod_A(\cat{C})}(X \otimes T, Y).
  \]
\item We assume that $A$ is dualizable in $\cat{C}$. We regard $\dual{A}$ as an $A$-$A$-bimodule as in \cref{bi--du--dual-bimodule}.
\item We assume given an invertible object $\omega_A \in \cat{C}$ and an equivalence of $A$-$A$-bimodules
  \[
    \alpha : \dual{A} \iso \omega_A^{-1} \otimes A
  \]
  (the $A$-$A$-bimodule structure on the right side being induced by that of $A$).
\end{itemize}
The last stipulation on $\dual{A}$ may be thought of as a form of Poincar\'e/Atiyah duality. For example, in the case $\cat{C} = \Spt$ and $A = \num{S}[\cir]$, Atiyah duality for $\cir$ implies that this stipulation is satisfied with $\omega_A \iso \num{S}[1]$, and this is one way to understand the source of the shift in the norm map for spectra with $\cir$-action described above.

Let us now begin by defining generalizations of the orbits, fixed points, norm map, and Tate construction for $A$-modules in $\cat{C}$.

\begin{notation}[Orbits and fixed points]
  \label{bi--ta--orbits-fixed}
  Let $\epsilon : A \to \unit$ denote the counit of the bialgebra structure on $A$. Restriction in $\epsilon$ determines a functor $\rho : \cat{C} \to \LMod_A(\cat{C})$ that preserves limits and colimits (as its composite with the forgetful functor $\LMod_A(\cat{C}) \to \cat{C}$ is the identity, and the forgetful functor creates limits and colimits), hence admits left and right adjoints by the adjoint functor theorem. We will denote these adjoints by $(-)_A : \LMod_A(\cat{C}) \to \cat{C}$ and $(-)^A : \LMod_A(\cat{C}) \to \cat{C}$, respectively; they are given by the formulas
  \[
    X_A \iso \unit \otimes_A X
    \quad\text{and}\quad
    X^A \iso \uMap_A(\unit,X).
  \]
\end{notation}

\begin{remark}
  \label{bi--ta--fixed-formula}
  Let $B_\bullet : \Delta^\op \to \LMod_A(\cat{C})$ denote the bar resolution computing $\unit \otimes_A A \iso \unit$, so that $B_n \iso A^{\otimes n+1}$ for $n \ge 0$ \cite[\textsection 4.4.2]{lurie--algebra}. Then, for any $X \in \LMod_A(\cat{C})$, we obtain a cosimplicial diagram $\uMap_A(B_\bullet,X) : \Delta \to \LMod_A(\cat{C})$, with
  \[
    \uMap_A(B_n,X) \iso \uMap_A(A^{\otimes n+1},X) \iso \uMap(A^{\otimes n},X) \iso (\dual{A})^{\otimes n} \otimes X,
  \]
  and whose limit is $\uMap_A(\unit,X) \iso X^A$.
\end{remark}

\begin{remark}
  \label{bi--ta--fixed-monoidal}
  It follows from \cref{bi--tn--functorial} that the restriction functor $\rho : \cat{C} \to \LMod_A(\cat{C})$ is canonically symmetric monoidal, implying that the right adjoint $(-)^A : \LMod_A(\cat{C}) \to \cat{C}$ is canonically lax symmetric monoidal.
\end{remark}

\begin{construction}[Norm map and Tate construction]
  \label{bi--ta--norm}
  Let $\eta : \unit \to \dual{A}$ denote the unit map of the dual bialgebra structure on $\dual{A}$ (\cref{bi--du--main-cor}); this is in particular a map of $\dual{A}$-$\dual{A}$-bicomodules, equivalently a map of $A$-$A$-bimodules (\cref{bi--du--dual-bimodule}). For any $X \in \LMod_A(\cat{C})$, this induces a map of left $A$-modules
  \begin{align*}
    \omega_A \otimes X_A
    &\iso \omega_A \otimes (\unit \otimes_A X) \\
    &\lblto{\eta} \omega_A \otimes (\dual{A} \otimes_A X) \\
    &\lbliso{\alpha} \omega_A \otimes ((\omega_A^{-1} \otimes A) \otimes_A X) \\
    &\iso \omega_A \otimes \omega_A^{-1} \otimes X  \\
    &\iso X.
  \end{align*}
  Since the left-hand side is a module over $\unit$ in $\LMod_A(\cat{C})$, this factors uniquely through a map
  \[
    \Nm_X : \omega_A \otimes X_A \to X^A
  \]
  in $\cat{C}$, which we refer to as the \emph{norm map}. We then define the \emph{Tate construction} of $X$ as
  \[
    X^{\tate A} := \cofib(\Nm_X).
  \]
  Note that these constructions are evidently functorial in $X$, in the sense that they determine functors $\Nm : \Delta^1 \times \LMod_A(\cat{C}) \to \cat{C}$ and $(-)^{\tate A} : \LMod_A(\cat{C}) \to \cat{C}$.
\end{construction}

\begin{remark}
  \label{bi--ta--access}
  As $(-)^A : \LMod_A(\cat{C}) \to \cat{C}$ is a right adjoint functor between presentable $\infty$-categories, it is $\kappa$-accessible for some regular cardinal $\kappa$. Since $\omega_A \otimes (-)_A : \LMod_A(\cat{C}) \to \cat{C}$ preserves all colimits, it follows that the Tate construction $(-)^{\tate A} : \LMod_A(\cat{C}) \to \cat{C}$ is also $\kappa$-accessible. This will be relevant in \cref{bi--ta--univ} below.
\end{remark}

Next, imitating \cite[\S I.3]{nikolaus-scholze--tc}, we characterize the above Tate construction, or more precisely the natural transformation $(-)^A \to (-)^{\tate A}$, by a universal property involving the behavior of these functors on \emph{induced} $A$-modules (see \cref{bi--ta--induced} below). As in op. cit., this will show that the functor $(-)^{\tate A}$ inherits a lax symmetric monoidal structure from the one on the functor $(-)^A$ (\cref{bi--ta--fixed-monoidal}). It will also show that the definitions here agree with the usual ones in the setting of group actions (in addition to \cite{nikolaus-scholze--tc}, see e.g. \cite{greenlees-may--tate}, \cite[\textsection 5.2]{rognes--groups}, and \cite[\textsection 6.1.6]{lurie--algebra} for earlier accounts of the Tate construction).

\begin{definition}
  \label{bi--ta--induced}
  We let $\LMod_A^\ind(\cat{C})$ denote the smallest stable full subcategory of $\LMod_A(\cat{C})$ containing the objects $A \otimes X_0$ for $X_0 \in \cat{C}$. We refer to the objects of $\LMod_A^\ind(\cat{C})$ as \emph{induced} $A$-modules.
\end{definition}

\begin{remark}
  \label{bi--ta--induced-ideal}
  The subcategory $\LMod_A^\ind(\cat{C}) \subseteq \LMod_A(\cat{C})$ is a \emph{$\otimes$-ideal}: that is, given an induced $A$-module $X$ and an arbitrary $A$-module $Y$, the tensor product $X \otimes Y$ is also induced. It is enough to show this in the case that $X \iso A \otimes X_0$ for $X_0 \in \cat{C}$ (this follows from the definition of $\LMod_A^\ind(\cat{C})$, since, fixing $Y$, the full subcategory of $\LMod_A(\cat{C})$ spanned by those objects $X$ such that $X \otimes Y$ is induced is a stable one). In this case, the canonical map $X_0 \to X$ induces a map $X_0 \otimes Y \to X \otimes Y$ in $\cat{C}$, which extends uniquely to a map $u_Y : A \otimes X_0 \otimes Y \to X \otimes Y$ in $\LMod_A(\cat{C})$. To prove the claim, it suffices to show that $u_Y$ is an equivalence. This is true by hypothesis when $Y = \unit$, and follows in general since $u_Y \iso u_\unit \otimes \id_Y$.
\end{remark}

\begin{lemma}
  \label{bi--ta--induced-tate-vanishing}
  For $X \in \LMod_A^\ind(\cat{C})$, we have $X^{\tate A} \iso 0$.
\end{lemma}

\begin{proof}
  As $(-)^{\tate A}$ is an exact functor (being the cofiber of a map between exact functors), the full subcategory of $\LMod_A(\cat{C})$ spanned by those objects $X$ such that $X^{\tate A} \iso 0$ is a stable one. Thus, by definition of $\LMod_A^\ind(\cat{C})$, it suffices to prove that $X^{\tate A} \iso 0$ in the case that $X \iso A \otimes X_0$ for $X_0 \in \cat{C}$; let us now assume that this is the case.

  Reviewing the definition of $X^{\tate A}$ from \cref{bi--ta--norm}, what we need to show is that the map
  \[
    \omega_A \otimes (\unit \otimes_A X) \to \uMap_A(\unit, \omega_A \otimes (\dual{A} \otimes_A X)),
  \]
  induced by the unit map $\eta : \unit \to \dual{A}$, is an equivalence. Under our assumed equivalence $X \iso A \otimes X_0$, the map in question identifies with the map
  \[
    \omega_A \otimes X_0 \iso \uMap(\unit, \omega_A \otimes X_0) \to \uMap_A(\unit, \omega_A  \otimes \dual{A} \otimes X_0)
  \]
  induced by $\eta$. This is indeed an equivalence, as $\eta$ defines the unit of an adjunction $\Map_{\cat{C}}(Y,Z) \iso \Map_{\LMod_A(\cat{C})}(Y,\dual{A} \otimes Z)$.
\end{proof}

\begin{lemma}
  \label{bi--ta--induced-generate}
  For all $X \in \LMod_A(\cat{C})$, the canonical maps
  \[
    \l[\colim_{Y \in \LMod_A^\ind(\cat{C})_{/X}} Y\r] \to X,
  \]
  \[
    \l[\colim_{Y \in \LMod_A^\ind(\cat{C})_{/X}} \cofib(Y \to X)^A\r] \to
    \l[\colim_{Y \in \LMod_A^\ind(\cat{C})_{/X}} \cofib(Y \to X)^{\tate A}\r]
  \]
  are equivalences.
\end{lemma}

\begin{proof}
  The first map being an equivalence is equivalent to the statement that the identity functor $\id : \LMod_A(\cat{C}) \to \LMod_A(\cat{C})$ is left Kan extended from the full subcategory $\LMod_A^\ind(\cat{C})$. Let $G : \LMod_A(\cat{C}) \to \cat{C}$ be the forgetful functor and let $F : \cat{C} \to \LMod_A(\cat{C})$ be its left adjoint, given by the formula $F(X_0) \iso A \otimes X_0$. Set $T := FG : \LMod_A(\cat{C}) \to \LMod_A(\cat{C})$. Then the bar resolution expresses the identity functor on $\LMod_A(\cat{C})$ as a geometric realization of the iterates $T^n$ for $n \ge 1$ (see \cite[Proof of Lemma 4.7.3.13]{lurie--algebra}). Since the full subcategory of $\End(\LMod_A(\cat{C}))$ spanned by those functors that are left Kan extended from $\LMod_A^\ind(\cat{C})$ is closed under colimits, it now suffices to show that this property is satisfied by $T^n = (FG)^n$ for any $n \ge 1$. In fact, more is true: $T^n$ may be obtained by left Kan extension along the functor $F : \cat{C} \to \LMod_A(\cat{C})$ (this is stronger because $F$ factors through $\LMod_A^\ind(\cat{C})$). Indeed, this is visibly true once one recalls that left Kan extension along the left adjoint $F$ is given by precomposition with the right adjoint $G$.

  We now address the second map in the statement. Its fiber is given by
  \[
    \colim_{Y \in \LMod_A^\ind(\cat{C})_{/X}} \left(\omega_A \otimes \cofib(Y \to X)_A\right) \iso \omega_A \otimes \l(\colim_{Y \in \LMod_A^\ind(\cat{C})_{/X}} \cofib(Y \to X)\r)_A,
  \]
  the equivalence resulting from the fact that the functor $\omega_A \otimes (-)_A$ preserves colimits. But this vanishes, since
  \[
    \colim_{Y \in \LMod_A^\ind(\cat{C})_{/X}} \cofib(Y \to X) \iso
    \cofib\l(\colim_{Y \in \LMod_A^\ind(\cat{C})_{/X}} Y \to X\r) \iso 0,
  \]
  the first equivalence using that colimits commute with cofibers and that
  \[
    X \iso \colim_{Y \in \LMod_A^\ind(\cat{C})_{/X}} X
  \]
  since $\LMod_A^\ind(\cat{C})_{/X}$ is filtered, and the second equivalence following from the first part of the proposition.
\end{proof}

\begin{proposition}
  \label{bi--ta--univ}
  Let $\kappa$ be a regular cardinal such that $\cat{C}$ and $(-)^A : \LMod_A(\cat{C}) \to \cat{C}$ are $\kappa$-accessible. Then the following statements hold:
  \begin{enumerate}
  \item \label{bi--ta--univ--initial}
    The canonical transformation $(-)^A \to (-)^{\tate A}$ exhibits the Tate construction $(-)^{\tate A}$ as the initial $\kappa$-accessible, exact functor $F : \LMod_A(\cat{C}) \to \cat{C}$ under $(-)^A$ satisfying $F(X) \iso 0$ for all $X \in \LMod_A^\ind(\cat{C})$.
  \item \label{bi--ta--univ--colim}
    Suppose given a natural transformation $c : (-)^A \to F$ of functors $\LMod_A(\cat{C}) \to \cat{C}$ such that $F$ is $\kappa$-accessible and $F(X) \iso 0$ for all $X \in \LMod_A^\ind(\cat{C})$. Then the induced transformation $(-)^{\tate A} \to F$ (as per \cref{bi--ta--univ--initial}) is an equivalence if and only if $\fib(c)$ preserves colimits.
  \item \label{bi--ta--univ--monoidal-initial}
    There is a unique pair of lax symmetric monoidal structure on the Tate construction $(-)^{\tate A} : \LMod_A(\cat{C}) \to \cat{C}$ and lax symmetric monoidal structure on the natural transformation $(-)^A \to (-)^{\tate A}$. Moreover, this exhibits $(-)^{\tate A}$ as the initial $\kappa$-accessible, exact lax symmetric monoidal functor $F : \LMod_A(\cat{C}) \to \cat{C}$ under $(-)^A$ satisfying $F(X) \iso 0$ for all $X \in \LMod_A^\ind(\cat{C})$. 
  \end{enumerate}
\end{proposition}

\begin{proof}
  This follows from \cref{bi--ta--induced-ideal,bi--ta--induced-tate-vanishing,bi--ta--induced-generate} using the the results of \cite[\S I.3]{nikolaus-scholze--tc}. To be more precise, see Theorems I.3.1 and I.4.1 and their proofs in \cite{nikolaus-scholze--tc}, and note that the preceding results here just named fill in the role of Lemma I.3.8 there.
\end{proof}

\begin{remark}
  \label{bi--ta--agree}
  It follows from \cref{bi--ta--univ} that, in the case $\cat{C} = \Spt$ and $A = \num{S}[G]$ for $G$ any finite group or abelian compact Lie group, the norm map and Tate construction defined in this section agree with the usual ones under the symmetric monoidal equivalence $\LMod_A(\cat{C}) \iso \Fun(\clspc G,\Spt)$ (\cref{bi--tn--rep}), as it is shown in \cite[Theorems I.3.1 and I.4.1]{nikolaus-scholze--tc} that the latter enjoy the same universal property.
\end{remark}

\begin{remark}[Naturality]
  \label{bi--ta--natural}
  Let $\cat{D}$ be another stable presentable symmetric monoidal $\infty$-category and let $F : \cat{C} \to \cat{D}$ be a symmetric monoidal functor. Then $F(A)$ is a cocommutative bialgebra in $\cat{D}$ satisfying the same hypotheses as $A$, in particular with $\omega_{F(A)} \iso F(\omega_A)$. By \cref{bi--tn--functorial}, $F$ determines a symmetric monoidal functor $\LMod_A(\cat{C}) \to \LMod_{F(A)}(\cat{D})$, which here we will just denote by $F$. This begets a natural transformation $\phi_A : F(X)_{F(A)} \to F(X_A)$ and a lax symmetric monoidal natural transformation $\phi^A : F(X^A) \to F(X)^{F(A)}$. It is straightforward to check that the composite
  \[
    F(X)_{F(A)} \lblto{\phi_A} F(X_A) \lblto{F(\Nm_X)} F(X^A) \lblto{\phi^A} F(X)^{F(A)}
  \]
  is canonically homotopic to $\Nm_{F(X)}$.

  Suppose now that $F$ preserves colimits. Then $\phi_A$ is an equivalence, and the previous paragraph gives us a commutative diagram
  \[
    \begin{tikzcd}
      F(X_A) \ar[r, "F(\Nm_X)"] \ar[d, "\phi_A^{-1}", swap] &
      F(X^A) \ar[d, "\phi^A"] \\
      F(X)_{F(A)} \ar[r, "\Nm_{F(X)}"] &
      F(X)^{F(A)}
    \end{tikzcd}
  \]
  Taking horizontal cofibers, we obtain a natural transformation $\phi^{\tate A} : F(X^{\tate A}) \to F(X)^{\tate F(A)}$.

  We next consider the lax symmetric monoidal structures on the functors $(-)^{\tate A}$ and $(-)^{\tate F(A)}$ supplied by \cref{bi--ta--univ}. We claim that $\phi^{\tate A}$ is canonically a transformation of lax symmetric monoidal functors. To see this, let $G : \cat{D} \to \cat{C}$ denote the right adjoint to $F$, which exists by the adjoint functor theorem. The symmetric monoidal structure on $F$ determines a lax symmetric monoidal structure on $G$, and it suffices to promote the adjoint transformation $X^{\tate A} \to G(F(X)^{\tate F(A)})$ to one of lax symmetric monoidal functors. Since $F$ sends induced $A$-modules to induced $F(A)$-modules, this follows from \cref{bi--ta--univ} (where we choose our regular cardinal $\kappa$ large enough such that $G$ is $\kappa$-accessible).

  Finally, if we assume moreover that $F$ preserves limits, then we deduce from \cref{bi--ta--fixed-formula} that $\phi^A$ is an equivalence, and hence $\phi^{\tate A}$ is as well.
\end{remark}


\section{Gradings and filtrations}
\label{gf}

Graded and filtered algebraic structures are central in this paper. This section is devoted to generalities concerning graded and filtered objects in a stable presentable symmetric monoidal $\infty$-category, following \cite[\S 3]{lurie--rot} to a great extent. See \cite{moulinos--filtrations} for another account.

In \cref{gf--df}, we set out our basic framework for working with graded and filtered objects; in \cref{gf--kd}, we discuss the relationship between graded and filtered objects as mediated by the ``associated graded'' construction (an instance of Koszul duality, and the first appearance of ``homotopy-coherent cochain complexes'' in the paper); in \cref{gf--t}, we discuss various t-structures in the graded and filtered settings that will be useful later on, including a new and more general perspective on the ``Beilinson t-structure'' furnished by the results of \cref{gf--kd} (see \cite{ariotta--complex} for more on this).


\subsection{Definitions}
\label{gf--df}

\begin{notation}
  \label{gf--df--gf}
  We regard the set of integers $\num{Z}$ as a category via the standard partial order $\le$, and we let $\num{Z}^\ds$ denote the set of integers regarded as a discrete category (i.e. with only identity morphisms). We use the following notation for any presentable symmetric monoidal $\infty$-category $\cat{C}$:
  \begin{enumerate}[leftmargin=*]
  \item Let $\Gr(\cat{C})$ denote the $\infty$-category $\Fun(\num{Z}^\ds,\cat{C})$, which we regard as a (stable presentable) symmetric monoidal $\infty$-category via the Day convolution symmetric monoidal structure coming from addition on $\num{Z}$. We refer to objects of $\Gr(\cat{C})$ as \emph{graded objects of $\cat{C}$}, and sometimes denote them by $X^*$ or $\{X^n\}_{n \in \num{Z}}$. We denote the Day convolution tensor product of graded objects by $\oast$; it is given by the formula
    \[
      (X^* \oast Y^*)^n \iso \coprod_{i+j=n} X^i \otimes Y^j.
    \]
    For $n \in \num{Z}$, restriction along the inclusion $\{n\} \inj \num{Z}^\ds$ defines an \emph{evaluation functor} $\ev^n : \Gr(\cat{C}) \to \cat{C}$, sending $X^* \mapsto X^n$. Each of these admits a fully faithful left adjoint \emph{insertion functor} $\ins^n : \cat{C} \to \Gr(\cat{C})$, given by left Kan extension along the same inclusion, or concretely by the formula
    \[
      \ins^n(X)^m \iso
      \begin{cases}
        X & m=n \\
        0 & \text{otherwise}.
      \end{cases}
    \]
    Note that right Kan extension would produce the same result, i.e. $\ins^n$ is also right adjoint to $\ev^n$. We note that $\ins^0$ has a canonical symmetric monoidal structure, since $\{0\} \inj \num{Z}$ is a map of commutative monoids, and hence $\ev^0$ has a canonical lax symmetric monoidal structure.

  \item We let $\Fil(\cat{C})$ denote the $\infty$-category $\Fun(\num{Z}^\op,\cat{C})$, which we again regard as a (stable presentable) symmetric monoidal $\infty$-category via the Day convolution structure coming from addition on $\num{Z}$. We refer to objects of $\Fil(\cat{C})$ as \emph{filtered objects of $\cat{C}$}, and sometimes denote them by $X^\star$ or depict them by diagrams
    \[
      \cdots \to X^2 \to X^1 \to X^0 \to X^{-1} \to X^{-2} \to \cdots.
    \]
    We denote the Day convolution tensor product of filtered objects by $\ostar$; it is given by the formula
    \[
      (X^\star \ostar Y^\star)^n \iso \colim_{i+j \ge n} X^i \otimes Y^j.
    \]
    Just as for graded objects, for $n \in \num{Z}$, restriction along the inclusion $\{n\} \inj \num{Z}$ defines an \emph{evaluation functor} $\ev^n : \Fil(\cat{C}) \to \cat{C}$, sending $X^\star \mapsto X^n$. And again each of these admits a fully faithful left adjoint \emph{insertion functor} $\ins^n : \cat{C} \to \Fil(\cat{C})$ given by left Kan extension, here given by the formula
    \[
      \ins^n(X)^m \iso
      \begin{cases}
        X & m \le n \\
        0 & \text{otherwise}
      \end{cases}
    \]
    (when we need to refer to both the filtered and graded insertion functors in proximity, we will add extra decorations to clarify), with $\ins^0$ again having a canonical symmetric monoidal structure and $\ev^0$ a canonical lax symmetric monoidal structure.
    
  \item \label{gf--df--gf--und}
    Restriction along the evident functor $\num{Z}^\ds \to \num{Z}^\op$ defines a functor $\und : \Fil(\cat{C}) \to \Gr(\cat{C})$, extracting the \emph{underlying graded object} of a filtered object. Given a filtered object $X^\star$, we may denote its underlying graded object by $X^*$.
    
  \item We have a functor $\gr : \Fil(\cat{C}) \to \Gr(\cat{C})$ extracting the \emph{associated graded object} of filtered object; it is given by the formula $\gr(X)^n \iso \cofib(X^{n+1} \to X^n)$.

  \item \label{gf--df--gf--spl} The functor $\und : \Fil(\cat{C}) \to \Gr(\cat{C})$ admits a left adjoint, given by left Kan extension along the inclusion $\num{Z}^\ds \inj \num{Z}^\op$. We denote this left adjoint by $\spl : \Gr(\cat{C}) \to \Fil(\cat{C})$; it is described concretely by the formula $\spl(X)^i \iso \coprod_{j \ge i} X^j$. We say a filtered object $X \in \Fil(\cat{C})$ is \emph{split} if there exists an equivalence $\spl(\gr(X)) \iso X$ in $\Fil(\cat{C})$.
    
  \item If $X$ is a graded (resp. filtered) object of $\cat{C}$, then, for $n \in \num{Z}$, we let $X(n)$ denote the graded (resp. filtered) object defined by $X(n)^m \iso X^{m-n}$. Note that, in both the graded and filtered settings, for each $n \in \num{Z}$, we have natural equivalences $\ins^n(X) \iso \ins^0(X)(n)$ for $X \in \cat{C}$. We will sometimes leave the functor $\ins^0$ implicit, i.e. by default regard objects of $\cat{C}$ as graded/filtered objects in grading/filtration $0$, and by this token write $X(n)$ to mean $\ins^n(X)$ for $X \in \cat{C}$ and $n \in \num{Z}$.
  \end{enumerate}
\end{notation}

\begin{variant}
  \label{gf--df--gf-nonnegative}
  We will also consider \emph{nonnegative} graded and filtered objects. These are defined by replacing the set of integers $\num{Z}$ in \cref{gf--df--gf} with the set of nonnegative integers $\num{Z}_{\ge 0}$. If $\cat{C}$ is a presentable symmetric monoidal $\infty$-category, we will denote the resulting $\infty$-categories of nonnegatively graded and filtered objects by $\Gr^{\ge 0}(\cat{C})$ and $\Fil^{\ge 0}(\cat{C})$. The rest of the constructions described in \cref{gf--df--gf} (the symmetric monoidal structures, the evaluation and insertion functors, the underlying and associated graded functors) also go through in the nonnegative setting, and we will use the same notation for them.

  There are evident restriction functors $\Gr(\cat{C}) \to \Gr^{\ge 0}(\cat{C})$ and $\Fil(\cat{C}) \to \Fil^{\ge 0}(\cat{C})$, both of which we will denote by $\ev^{\ge 0}$. Similar to the evaluation functors discussed in \cref{gf--df--gf}, these have fully faithful left adjoints, both of which we will denote by $\ins^{\ge 0}$ (and again, in the graded setting this is also right adjoint to $\ev^{\ge 0}$). Similar to $\ins^0$, the functors $\ins^{\ge 0}$ are canonically symmetric monoidal, and hence $\ev^{\ge 0}$ canonically lax symmetric monoidal. We often implicity identify $\Gr^{\ge0}(\cat{C})$ and $\Fil^{\ge0}(\cat{C})$ with full subcategories of $\Gr(\cat{C})$ and $\Fil(\cat{C})$ via the embeddings $\ins^{\ge0}$. 
\end{variant}

\begin{notation}
  \label{gf--df--calg}
  Let $\cat{C}$ be a presentable symmetric monoidal $\infty$-category, and regard $\Gr(\cat{C})$ and $\Fil(\cat{C})$ as such as in \cref{gf--df--gf}. We will refer to commutative algebras in $\Gr(\cat{C})$ as \emph{graded commutative algebras in $\cat{C}$}, and denote $\CAlg(\Gr(\cat{C}))$ by $\Gr\CAlg(\cat{C})$. Similarly, we refer to commutative algebras in $\Fil(\cat{C})$ as \emph{filtered commutative algebras in $\cat{C}$}, and denote $\CAlg(\Fil(\cat{C}))$ by  $\Fil\CAlg(\cat{C})$. We also have \emph{nonnegatively} graded and filtered commutative algebras in $\cat{C}$, the $\infty$-categories of which we denote by $\Gr^{\ge 0}\CAlg(\cat{C})$ and $\Fil^{\ge 0}\CAlg(\cat{C})$ respectively. 
\end{notation}

\begin{remark}
  \label{gf--df--sptalg}
  In \cite{lurie--rot}, the notions in \cref{gf--df--gf} are discussed only in the case $\cat{C}=\Spt$. However, this is the universal case in the stable setting: for any stable presentable symmetric monoidal $\infty$-category $\cat{C}$, there are canonical equivalences $\Gr(\Spt) \otimes \cat{C} \isoto \Gr(\cat{C})$ and $\Fil(\Spt) \otimes \cat{C} \isoto \Fil(\cat{C})$ in $\CAlg(\PrL)$. Moreover, the functors $\und,\gr : \Fil(\cat{C}) \to \Gr(\cat{C})$ are obtained by tensoring the functors $\und,\gr : \Fil(\Spt) \to \Gr(\Spt)$ with $\cat{C}$ in $\PrL$.
\end{remark}


\subsection{Koszul duality}
\label{gf--kd}

Our goal in this subsection is to analyze the relationship between filtered and graded objects in a stable presentable symmetric monoidal $\infty$-category $\cat{C}$. The analysis can be framed by the following question: To what extent, and how, can one recover a filtered object $X \in \Fil(\cat{C})$ from its associated graded object $\gr(X) \in \Gr(\cat{C})$? The answer turns out to fit into the framework of \emph{Koszul} (or \emph{bar-cobar}) \emph{duality}. We'll begin with a discussion of the general framework, and subsequently explain how it applies to graded and filtered objects.

Let us first recall the basic features of Koszul duality between augmented algebras and coalgebras, following \cite[\S5.2.2]{lurie--algebra}.

\begin{recollection}
  \label{gf--kd--bar}
  Let $\cat{C}$ be a monoidal $\infty$-category admitting geometric realizations. Let $A$ be an augmented algebra object of $\cat{C}$. We let $\Bar(A) \in \cat{C}$ denote the \emph{bar construction} on $A$: this is the geometric realization of a canonical simplicial diagram $\Delta^\op \to \cat{C}$ that on objects sends $[n] \to A^{\otimes n}$, and under the assumption that the tensor product $\otimes : \cat{C} \times \cat{C} \to \cat{C}$ preserves geometric realizations, this computes the relative tensor product $\unit \otimes_A \unit$ (where $\unit$ denotes the unit object).

  The bar construction furthermore satisfies the following universal property \cite[Definition 5.2.2.1, Lemma 5.2.2.6, Remark 5.2.2.8]{lurie--algebra}. Let $\BMod{A}{A}(\cat{C})$ denote the $\infty$-category of $A$-$A$-bimodule objects of $\cat{C}$ and let $\rho : \cat{C} \to \BMod{A}{A}(\cat{C})$ denote the restriction functor induced by the augmentation of $A$. Then there is a canonical map $A \to \rho(\Bar(A))$ in $\BMod{A}{A}(\cat{C})$ such that the induced map
  \[
    \Map_\cat{C}(\Bar(A),X) \to \Map_{\BMod{A}{A}(\cat{C})}(A,\rho(X))
  \]
  is a homotopy equivalence for all $X \in \cat{C}$.

  We then have the formally dual notion: assume instead that $\cat{C}$ admits totalizations and let $B$ be an augmented coalgebra object of $\cat{C}$. Then we let $\Cobar(B) \in \cat{C}$ denote the \emph{cobar construction} on $B$, which is the totalization of a canonical cosimplicial diagram $\Delta^\op \to \cat{C}$ that on objects sends $[n] \to B^{\otimes n}$, and satisfies a universal property dual to the one stated for the bar construction.

  Finally, we recall that, assuming $\cat{C}$ admits both geometric realizations and totalizations, the constructions $A \mapsto \Bar(A)$ and $B \mapsto \Cobar(B)$ canonically promote to a pair of adjoint functors
  \[
    \Bar : \Alg^\aug(\cat{C}) \fromto \cAlg^\aug(\cat{C}) : \Cobar,
  \]
  where $\Alg^\aug(\cat{C})$ and $\cAlg^\aug(\cat{C})$ denote the $\infty$-categories of augmented algebra objects and augmented coalgebra objects of $\cat{C}$ respectively \cite[Remark 5.2.2.19]{lurie--algebra}. In particular, the bar (resp. cobar) construction on an augmented algebra (resp. coalgebra) carries a canonical coalgebra (resp. algebra) structure.
\end{recollection}

\begin{remark}
  \label{gf--kd--bar-bi}
  Let $\cat{C}$ be a symmetric monoidal $\infty$-category admitting geometric realizations, and assume that the tensor product $\otimes : \cat{C} \times \cat{C} \to \cat{C}$ preserves geometric realizations. Let $A$ be an augmented \emph{commutative} algebra object of $\cat{C}$. Since the forgetful functor $\Alg(\CAlg(\cat{C})) \to \CAlg(\cat{C})$ is an equivalence, we may regard $A$ as an augmented algebra object in $\CAlg(\cat{C})$, and hence perform the bar construction in $\CAlg(\cat{C})$ to obtain $\Bar(A)$ as an object in $\cAlg(\CAlg(\cat{C}))$, i.e. a commutative bialgebra object of $\cat{C}$. Since the forgetful functor $\CAlg(\cat{C}) \to \cat{C}$ preserves geometric realizations, the underlying coalgebra of $\Bar(A)$ agrees with the bar construction on underlying augmented algebra of $A$.
\end{remark}

The following proposition supplies an alternative construction of a coalgebra structure on  bar construction of an augmented algebra.\footnote{I learned this statement from lecture notes of Lurie. That this alternative coalgebra structure agrees with the one constructed in \cite[\S5.2.2]{lurie--algebra} is established in \cite[\textsection 3.4]{brantner-campos-nuiten--pd-operads}. However, this agreement will not actually be necessary for the particular situation in which we will invoke both perspectives (the proof of \cref{gf--kd--mod}).}

\begin{proposition}
  \label{gf--kd--bar-coend}
  Let $\cat{C}$ be a monoidal $\infty$-category admitting geometric realizations. Let $\unit$ denote the unit object and let $A$ be an augmented algebra object of $\cat{C}$. Let $\cat{M} := \LMod_A(\cat{C})$, which we regard as right tensored over $\cat{C}$. Then $\Bar(A) \in \cat{C}$ is a coendomorphism object for $\unit \in \cat{M}$.
\end{proposition}

\begin{proof}
  Consider the canonical map $A \to \rho(\Bar(A))$ discussed in \cref{gf--kd--bar} and the unit map $\unit \to A$. These induce homotopy equivalences
  \[
    \Map_\cat{C}(\Bar(A),X) \isoto \Map_{\BMod{A}{A}(\cat{C})}(A,\rho(X)) \isoto \Map_{\LMod_A(\cat{C})}(\unit, \rho'(X)),
  \]
  where $\rho : \cat{C} \to \BMod{A}{A}(\cat{C})$ and $\rho' : \cat{C} \to \LMod_A(\cat{C})$ denote the forgetful functors induced by the augmentation of $A$. Noting that $\rho'(X)$ is naturally equivalent to the tensoring of $\unit \in \cat{M}$ with $X \in \cat{C}$, this exhibits $\Bar(A)$ as a coendomorphism object for $\unit \in \cat{M}$.
\end{proof}

What we will actually be relevant for us is the dual statement for the cobar construction:

\begin{variantspecial}
  \label{gf--kd--cobar-end}
  Let $\cat{C}$ be a monoidal $\infty$-category admitting totalizations. Let $B$ be an augmented coalgebra in $\cat{C}$. Let $\cat{M} := \cLMod_B(\cat{C})$, which we regard as right tensored over $\cat{C}$. Then $\Cobar(B) \in \cat{C}$ is an endomorphism object for $\unit \in \cat{M}$.
\end{variantspecial}

We now come to the Koszul duality result that we are ultimately interested in here, taking place at the level of categories of modules and comodules:

\begin{proposition}
  \label{gf--kd--mod}
  Let $\cat{C}$ be a presentable symmetric monoidal $\infty$-category. Let $\unit \in \cat{C}$ denote the unit object. Let $A$ be an augmented commutative algebra object of $\cat{C}$. Assume that the augmentation exhibits $\unit$ as a dualizable $A$-module, so that $\Bar(A) \iso \unit \otimes_A \unit$ is dualizable in $\cat{C}$. We regard $\Bar(A)$ as a commutative bialgebra object of $\cat{C}$ by \cref{gf--kd--bar-bi}, and thereby regard $\cLMod_{\Bar(A)}(\cat{C})$ as a presentable symmetric monoidal $\infty$-category by \cref{bi--tn--main-cor,bi--tn--presentable}. The following statements then hold:
  \begin{enumerate}
  \item \label{gf--kd--mod--functor} The symmetric monoidal functor $\unit \otimes_A - : \Mod_A(\cat{C}) \to \cat{C}$ canonically lifts along the (symmetric monoidal) forgetful functor $U : \cLMod_{\Bar(A)}(\cat{C}) \to \cat{C}$ to a symmetric monoidal functor $F : \Mod_A(\cat{C}) \to \cLMod_{\Bar(A)}(\cat{C})$.
  \item \label{gf--kd--mod--adjoint} The functor $F$ in \cref{gf--kd--mod--functor} admits a fully faithful right adjoint $G : \cLMod_{\Bar(A)}(\cat{C}) \to \Mod_A(\cat{C})$.
  \item \label{gf--kd--mod--local} If $\cat{C}$ is stable, then the essential image of the functor $G$ in \cref{gf--kd--mod--adjoint} is the full subcategory $\Mod_A(\cat{C})_\unit \subseteq \Mod_A(\cat{C})$ spanned by the \emph{$\unit$-local} objects, in the sense of Bousfield localization.
  \end{enumerate}
\end{proposition}

\begin{proof}
  We begin with \cref{gf--kd--mod--functor}. By the results of \cite[\S4.8.5]{lurie--algebra}, the symmetric monoidal structure on $\cLMod_{\Bar(A)}(\cat{C})$ determines a commutative algebra structure on the endomorphism object $E := \End_{\cLMod_{\Bar(A)}(\cat{C})}(\unit) \in \cat{C}$ and the symmetric monoidal forgetful functor $U$ induces a map of commutative algebras $E \to \unit$. Moreover, producing the desired symmetric monoidal functor $F$ factoring $\unit \otimes_A - : \Mod_A(\cat{C}) \to \cat{C}$ is equivalent to producing a map of augmented commutative algebras $A \to E$. To do so, it suffices to exhibit an equivalence of augmented commutative algebras $E \iso \Cobar(\Bar(A))$, since there is a canonical map of augmented commutative algebras $A \to \Cobar(\Bar(A))$ determined by the bar-cobar adjunction described in \cref{gf--kd--bar} (applied in $\CAlg(\cat{C})$). 

  Let $E' := \Cobar(\Bar(A))$. Let
  \[
    \rho : \cat{C} \to \cLMod_{\Bar(A)}(\cat{C}), \quad
    \rho' : \CAlg(\cat{C}) \to \cLMod_{\Bar(A)}(\CAlg(\cat{C}))
  \]
  denote the corestriction functors induced by the augmentation $\unit \to \Bar(A)$. By \cref{gf--kd--cobar-end}, $E'$ is an endomorphism object of $\unit \in \cLMod_{\Bar(A)}(\CAlg(\cat{C}))$. Thus, the image under $\rho'$ of the augmentation $E \to \unit$ in $\CAlg(\cat{C})$ determines a map of augmented commutative algebras $\alpha : E \to E'$. Conversely, using the universal property of $E$ as an endomorphism object, we obtain a canonical map $\beta : E' \to E$ in $\cat{C}$. One can check that the map $\rho(E) \iso \unit \otimes E \to \unit$ in $\cLMod_{\Bar(A)}(\cat{C})$ exhibiting $E$ as the endomorphism object of $\unit$ is obtained by applying $\rho$ to the augmentation $E \to \unit$ in $\cat{C}$, from which we deduce that the composition $\beta \circ \alpha$ is homotopic to the identity. It now suffices to show that $\beta$ is an equivalence, as then $\alpha$ is too. This follows from the chain of equivalences, natural in $X \in \cat{C}$,
  \begin{align*}
    \Map_\cat{C}(X,E')
    &\iso \Map_{\CAlg(\cat{C})}(\Sym(X),E') \\
    &\iso \Map_{\cLMod_{\Bar(A)}(\CAlg(\cat{C}))}(\rho'(\Sym(X)),\unit) \\
    &\iso \Map_{\cLMod_{\Bar(A)}(\cat{C})}(\rho(X),\unit) \\
    &\iso \Map_\cat{C}(X,E),
  \end{align*}
  where $\Sym : \cat{C} \to \CAlg(\cat{C})$ denotes the left adjoint to the forgetful functor (the composite equivalence is indeed induced by $\beta$).

  Statement \cref{gf--kd--mod--adjoint} follows from applying (the dual form of) \cite[Corollary 4.7.3.16]{lurie--algebra} to the following commutative diagram of $\infty$-categories:
  \[
    \begin{tikzcd}
      \Mod_A(\cat{C}) \ar[rr, "F"] \ar[dr, "\unit \otimes_A -", swap] &
      &
      \cLMod_{\Bar(A)}(\cat{C}) \ar[dl, "U"] \\
      &
      \cat{C}
    \end{tikzcd}
  \]
  For statement \cref{gf--kd--mod--local}, observe that $F$ vanishes on $\unit$-acyclic objects, as $U$ is conservative, and hence factors through the Bousfield localization $\Mod_A(\cat{C})_\unit$ of $\Mod_A(\cat{C})$ at $\unit$. It follows formally that the essential image of the right adjoint $G$ is contained in $\Mod_A(\cat{C})_\unit$, and that there is an induced adjunction $\Mod_A(\cat{C})_\unit \fromto \cLMod_{\Bar(A)}(\cat{C})$. This adjunction is an equivalence, since the right adjoint remains fully faithful and now by construction the left adjoint is conservative.
\end{proof}

\begin{remark}[Naturality]
  \label{gf--kd--mod-natural}
  In the situation of \cref{gf--kd--mod}, suppose given another augmented commutative algebra $B$, whose augmentation exhibits $\unit$ as a dualizable $B$-module, so that we also obtain a symmetric monoidal functor $F' : \Mod_B(\cat{C}) \to \cLMod_{\Bar(B)}(\cat{C})$. It follows from the proof of \cref{gf--kd--mod} that, for any map of augmented commutative algebras $\phi : A \to B$, the diagram of symmetric monoidal $\infty$-categories
  \[
    \begin{tikzcd}
      \Mod_A(\cat{C}) \ar[r, "F"] \ar[d, "B \otimes_A -"] &
      \cLMod_{\Bar(A)}(\cat{C}) \ar[d] \\
      \Mod_B(\cat{C}) \ar[r, "F'"] &
      \cLMod_{\Bar(B)}(\cat{C})
    \end{tikzcd}
  \]
  commutes, where the right-hand vertical arrow is the corestriction functor of \cref{bi--tn--functorial}\cref{bi--tn--functorial--algebra} determined by the map of commutative bialgebras $\Bar(\phi) : \Bar(A) \to \Bar(B)$.
\end{remark}

We now explain how \cref{gf--kd--mod} applies to the question of recovering a filtered object from its associated graded object. For the remainder of the subsection, we let $\cat{C}$ be a stable presentable symmetric monoidal $\infty$-category and $\unit$ the unit object of $\cat{C}$.

\begin{notation}
  \label{gf--kd--gf-aug}
  Let $\num{S}^\gr \in \Gr(\Spt)$ and $\num{S}^\fil \in \Fil(\Spt)$ denote the unit objects in graded and filtered spectra, so $\num{S}^\gr \iso \ins^0_\gr(\num{S})$ and $\num{S}^\fil \iso \ins^0_\fil(\num{S})$. Set $\num{S}^\gr[t] \ce \und(\num{S}^\fil) \in \CAlg(\Gr(\Spt))$, the commutative algebra structure coming from the canonical lax symmetric monoidal structure on the functor $\und : \Fil(\cat{C}) \to \Gr(\Spc)$ (see \cite[Remark 3.1.3, Notation 3.1.4, Remark 3.1.5]{lurie--rot}). There is a canonical augmentation $\epsilon : \num{S}^\gr[t] \to \num{S}^\gr$ in $\CAlg(\Gr(\Spt))$ (\cite[Remark 3.2.6]{lurie--rot}).
  
  Now let $\unit^\gr \in \Gr(\cat{C})$ and $\unit^\fil \in \Fil(\cat{C})$ denote the analogous graded and filtered unit objects of $\cat{C}$, and let $\unit^\gr[t] \ce \und(\unit^\fil) \in \CAlg(\Gr(\cat{C}))$. These are the images of $\num{S}^\gr,\num{S}^\fil,\num{S}^\gr[t]$ under the canonical symmetric monoidal functors $\Gr(\Spt) \to \Gr(\cat{C})$ and $\Fil(\Spt) \to \Fil(\cat{C})$. We let $\epsilon : \unit^\gr[t] \to \unit^\gr$ denote the image of the map $\epsilon : \num{S}^\gr[t] \to \num{S}^\gr$ of the previous paragraph under the canonical functor $\Gr(\Spt) \to \Gr(\cat{C})$.
\end{notation}

\begin{remark}
  \label{gf--kd--gfdual}
  Observe that the fiber of the augmentation $\epsilon : \unit^\gr[t] \to \unit$ is equivalent to $\unit^\gr[t](-1)$. This implies that $\unit$ is dualizable as a $\unit^\gr[t]$-module.
\end{remark}

\begin{proposition}
  \label{gf--kd--und}
  The functor $\und : \Fil(\cat{C}) \to \Gr(\cat{C})$ induces an equivalence of symmetric monoidal $\infty$-categories
  \[
    \Fil(\cat{C}) \isoto \Mod_{\unit^\gr[t]}(\Gr(\cat{C})).
  \]
  Moreover, the compositition
  \[
    \Fil(\cat{C}) \isoto \Mod_{\unit^\gr[t]}(\Gr(\cat{C})) \lblto{\unit^\gr \oast_{\unit^\gr[t]}} \Mod_{\unit^\gr}(\Gr(\cat{C})) \iso \Gr(\cat{C})
  \]
  is canonically equivalent to the associated graded functor $\gr : \Fil(\cat{C}) \to \Gr(\cat{C})$.
\end{proposition}

\begin{proof}
  For $\cat{C}=\Spt$, this is proven in \cite[Proposition 3.1.6 and Remarks 3.2.6 and 3.2.8]{lurie--rot}. The general statement follows from this case by tensoring with $\cat{C}$ in $\PrL$ (see \cref{gf--df--sptalg}).
\end{proof}

\begin{remark}
  \label{gf--kd--rees}
  \cref{gf--kd--und} is an incarnation of the classical ``Rees algebra'' construction.
\end{remark}

\begin{remark}
  \label{gf--kd--zeta}
  The associated graded functor $\gr : \Fil(\cat{C}) \to \Gr(\cat{C})$ preserves colimits, and hence admits a right adjoint $\zeta : \Gr(\cat{C}) \to \Fil(\cat{C})$. \cref{gf--kd--und} shows that that the right adjoint $\zeta$ sends a graded object $X^*$ to the filtered object
  \[
    \cdots \to X^2 \lblto{0} X^1 \lblto{0} X^0 \lblto{0} X^{-1} \lblto{0} X^{-2} \to \cdots.
  \]
\end{remark}

It follows from \cref{gf--kd--gfdual,gf--kd--und} that $\Fil(\cat{C})$ and $\Gr(\cat{C})$ fit into an example of the situation of \cref{gf--kd--mod}. Let us unravel what \cref{gf--kd--mod} says in this example. We first analyze the relevant bar construction.

\begin{notation}
  \label{gf--kd--dminus}
  We let $\dual{\num{D}_-}$ denote the bar construction $\Bar(\unit^\gr[t])$, which we regard as a commutative bialgebra in $\Gr(\cat{C})$ (\cref{gf--kd--bar-bi}). We let $\num{D}_-$ denote the dual $\dual{\Bar(\unit^\gr[t])}$, which we regard as a cocommutative bialgebra in $\Gr(\cat{C})$ (\cref{bi--du--main-cor}).

  We can compute what these objects look like. Under the equivalence $\Fil(\cat{C}) \iso \Mod_{\unit^\gr[t]}(\Gr(\cat{C}))$ of \cref{gf--kd--und}, the $\unit^\gr[t]$-module $\unit$ corresponds to the unique filtered object $\num{A} \in \Fil(\cat{C})$ with $\num{A}^0 \iso \unit$ and $\num{A}^n \iso 0$ for $n \ne 0$ (see \cite[Notation 3.2.4]{lurie--rot}). It follows from \cref{gf--kd--und} that in $\Gr(\cat{C})$ we have
  \[
    \dual{\num{D}_-} \iso \unit^\gr \oast_{\unit^\gr[t]} \unit^\gr \iso \gr(\num{A}),
  \]
  from which we calculate
  \[
    (\dual{\num{D}_-})^n \iso
    \begin{cases}
      \unit & \text{if}\ n = 0 \\
      \unit[1] & \text{if}\ n = -1 \\
      0 & \text{otherwise},
    \end{cases}
    \qquad\text{and hence}\qquad
    \num{D}_-^n \iso
    \begin{cases}
      \unit & \text{if}\ n = 0 \\
      \unit[-1] & \text{if}\ n = 1 \\
      0 & \text{otherwise}.
    \end{cases}
  \]
  Regarding the bialgebra structures on these, the unit and counit maps on each are given by the displayed equivalences in degree $0$, and the multiplication and comultiplication maps are then uniquely determined in the homotopy category $\ho \Gr(\cat{C})$; for example, the multiplications are the standard square-zero ones. However, the further homotopy-coherent structure in $\Gr(\cat{C})$ itself is less transparent in general.
\end{notation}

We next characterize the relevant Bousfield localization.

\begin{definition}
  \label{gf--kd--cpl}
  We say a filtered object $X \in \Fil(\cat{C})$ is \emph{complete} if $\lim_{n \to \infty} X^n \iso 0$. We denote by $\cpl{\Fil}(\cat{C})$ the full subcategory of $\Fil(\cat{C})$ spanned by the complete objects.
\end{definition}

\begin{lemma}[{\cite[Lemma 2.15]{gwilliam-pavlov--filtered}}]
  \label{gf--kd--cpl-loc}
  The inclusion $\cpl{\Fil}(\cat{C}) \inj \Fil(\cat{C})$ admits a left adjoint $\cpl{(-)} : \Fil(\cat{C}) \to \cpl{\Fil}(\cat{C})$. Given a map $X \to Y$ in $\Fil(\cat{C})$, the induced map $\cpl{X} \to \cpl{Y}$ is an equivalence if and only if the induced map $\gr(X) \to \gr(Y)$ is an equivalence in $\Gr(\cat{C})$.
\end{lemma}

We finally arrive at the main result of this subsection.

\begin{theorem}
  \label{gf--kd--gf}
  There is a canonical equivalence of symmetric monoidal $\infty$-categories $\o{\gr} : \cpl{\Fil}(\cat{C}) \isoto \LMod_{\num{D}_-}(\Gr(\cat{C}))$ making the diagram
  \[
    \begin{tikzcd}
      \Fil(\cat{C}) \ar[r, "\gr"] \ar[d, "\cpl{(-)}", swap] &
      \Gr(\cat{C}) \\
      \cpl\Fil(\cat{C}) \ar[r, "\o{\gr}"] &
      \LMod_{\num{D}_-}(\Gr(\cat{C})) \ar[u, "U", swap]
    \end{tikzcd}
  \]
  commute (here $U$ is the forgetful functor).
\end{theorem}

\begin{proof}
  Noting that \cref{bi--du--main-cor} gives an equivalence $\LMod_{\num{D}_-}(\Gr(\cat{C})) \iso \cLMod_{\dual{\num{D}}_-}(\Gr(\cat{C}))$ commuting with the forgetful functors to $\Gr(\cat{C})$, this follows from combining \cref{gf--kd--mod,gf--kd--und}, \cref{gf--kd--gfdual}, and \cref{gf--kd--cpl-loc}.
\end{proof}

\begin{remark}
  \label{gf--kd--unravel}
  More informally, what \cref{gf--kd--gf} says is that, given a filtered object $X \in \Fil(\cat{C})$, what can be recovered from the associated graded object $\gr(X)$ is the completion $\cpl{X}$, and the extra data necessary to perform this recovery is a $\num{D}_-$-module structure on $\gr(X)$.

  In view of the description of $\num{D}_-$ in \cref{gf--kd--dminus}, in the homotopy category $\ho\Gr(\cat{C})$, the $\num{D}_-$-module structure on $\gr(X)$ determines a type of cochain complex object, with differentials of the form $\gr^i(X)[-1] \to \gr^{i+1}(X)$. Unravelling the constructions made above, one can check that the differentials are induced by the boundary maps in the cofiber sequence $X^{i+1} \to X^i \to \gr^i(X)$, recovering the cochain complex recorded explicitly in \cite[Remark 1.2.2.3]{lurie--algebra}.

  Accordingly, we may think of a $\num{D}_-$-module structure in $\Gr(\cat{C})$ as a kind of ``homotopy-coherent cochain complex'' structure. We will return to this idea in \cref{dg}.
\end{remark}


\subsection{t-structures}
\label{gf--t}

In this subsection, we define various t-structures for graded and filtered objects that will come in handy later in the paper.\footnote{The material in this subsection grew in part out of a conversation with Ben Antieau.}

\begin{definition}
  \label{gf--t--compatible}
  Let $\cat{C}$ be a stable presentable symmetric monoidal $\infty$-category. By a \emph{compatible t-structure} on $\cat{C}$, we will mean a t-structure $(\cat{C}_{\ge 0}, \cat{C}_{\le 0})$ on the underlying stable $\infty$-category of $\cat{C}$ satisfying the following properties:
  \begin{enumerate}
  \item $\cat{C}_{\le 0}$ is closed under filtered colimits in $\cat{C}$;
  \item the unit object $\unit$ of $\cat{C}$ lies in $\cat{C}_{\ge 0}$;
  \item if $X,Y \in \cat{C}_{\ge 0}$, then $X \otimes Y \in \cat{C}_{\ge 0}$.
  \end{enumerate}
\end{definition}

Below, we will state that various constructions give compatible t-structures, but omit verifying the t-structure and compatibility axioms, as they will always be straightforward to check.

For the remainder of this subsection, we fix a stable presentable symmetric monoidal $\infty$-category $\cat{C}$, equipped with a compatible t-structure $(\cat{C}_{\ge 0},\cat{C}_{\le 0})$. We begin by discussing ways in which $\Gr(\cat{C})$ inherits a t-structure from $\cat{C}$.

\begin{construction}
  \label{gf--t--gr-t}
  We define the following three t-structures on $\Gr(\cat{C})$:
  \begin{enumerate}
  \item Let $\Gr(\cat{C})_{\ge 0}$ be the full subcategory of $\Gr(\cat{C})$ consisting of those graded objects $X^*$ for which $X^n \in \cat{C}_{\ge 0}$ for all $n \in \num{Z}$. Let $\Gr(\cat{C})_{\le 0}$ be the full subcategory of $\Gr(\cat{C})$ consisting of those graded objects $X^*$ for which $X^n \in \cat{C}_{\le 0}$ for all $n \in \num{Z}$. Then $(\Gr(\cat{C})_{\ge 0}, \Gr(\cat{C})_{\le 0})$ is a compatible t-structure on $\Gr(\cat{C})$, which we will refer to as the \emph{neutral t-structure}.
  
  \item Let $\Gr(\cat{C})_{\ge 0^+}$ be the full subcategory of $\Gr(\cat{C})$ consisting of those graded objects $X^*$ for which $X^n \in \cat{C}_{\ge n}$ for all $n \in \num{Z}$. Let $\Gr(\cat{C})_{\le 0^+}$ be the full subcategory of $\Gr(\cat{C})$ consisting of those graded objects $X^*$ for which $X^n \in \cat{C}_{\le n}$ for all $n \in \num{Z}$. Then $(\Gr(\cat{C})_{\ge 0^+}, \Gr(\cat{C})_{\le 0^+})$ is a compatible t-structure on $\Gr(\cat{C})$. We will refer to this as the \emph{positive t-structure} on $\Gr(\cat{C})$, and may write $\Gr(\cat{C})_+$ to denote $\Gr(\cat{C})$ equipped with this t-structure.

  \item Let $\Gr(\cat{C})_{\ge 0^-}$ be the full subcategory of $\Gr(\cat{C})$ consisting of those graded objects $X^*$ for which $X^n \in \cat{C}_{\ge -n}$ for all $n \in \num{Z}$. Let $\Gr(\cat{C})_{\le 0^-}$ be the full subcategory of $\Gr(\cat{C})$ consisting of those graded objects $X^*$ for which $X^n \in \cat{C}_{\le -n}$ for all $n \in \num{Z}$. Then $(\Gr(\cat{C})_{\ge 0^-}, \Gr(\cat{C})_{\le 0^-})$ is a compatible t-structure on $\Gr(\cat{C})$. We will refer to this as the \emph{negative t-structure} on $\Gr(\cat{C})$, and may write $\Gr(\cat{C})_-$ to denote $\Gr(\cat{C})$ equipped with this t-structure.
  \end{enumerate}
\end{construction}

\begin{remark}
  \label{gf--t--gr-heart}
  The hearts $\Gr(\cat{C})^\heart$, $\Gr(\cat{C})_+^\heart$ and $\Gr(\cat{C})_-^\heart$ of the three t-structures defined in \cref{gf--t--gr-t} can each be identified with the category $\Gr(\cat{C}^\heart) = \Fun(\num{Z}^\ds,\cat{C}^\heart)$, via the functors described the following formulas:
  \begin{align*}
    \{X^n\}_{n \in \num{Z}} \in \Gr(\cat{C}^\heart)
    \quad \mapsto \quad
    \begin{cases}
      \{X^n\}_{n \in \num{Z}} \in \Gr(\cat{C})^\heart \\
      \{X^n[n]\}_{n \in \num{Z}} \in \Gr(\cat{C})_+^\heart \\
      \{X^n[-n]\}_{n \in \num{Z}} \in \Gr(\cat{C})_-^\heart.
    \end{cases}
  \end{align*}

  As the three t-structures are all compatible with the symmetric monoidal structure on $\Gr(\cat{C})$, there are induced symmetric monoidal structures on each of their hearts, which under the above identifications determine symmetric monoidal structures on $\Gr(\cat{C}^\heart)$. For the positive and negative t-structures, this recovers the usual symmetric monoidal structure considered on graded objects in an abelian category, with tensor product again given by Day convolution but with the symmetry isomorphism incorporating the Koszul sign convention (for example, a commutative algebra object for this symmetric monoidal structure is a graded commutative algebra in $\cat{C}^\heart$ as usually defined). However, the appearance of this sign convention depended on our choices of shifts in the t-structures $\Gr(\cat{C})_\pm$: for the neutral t-structure, we obtain the Day convolution symmetric monoidal structure on $\Gr(\cat{C}^\heart)$ with no Koszul sign convention involved. Since these two symmetric monoidal structures on $\Gr(\cat{C}^\heart)$ will both arise later on, we will distinguish between them by denoting the former (with the Koszul sign rule) by $\Gr(\cat{C}^\heart)^\koz$, and the latter (without the Koszul sign rule) by $\Gr(\cat{C}^\heart)^\day$.
\end{remark}

Mediating between the positive and negative t-structures on graded objects will turn out to be an important aspect of the story we will tell later on. The necessary translation is encapsulated in the following result.

\begin{proposition}
  \label{gf--t--shift}
  Let $\cat{C}$ be a $\num{Z}$-linear stable presentable symmetric monoidal $\infty$-category (in the sense of \cite[Definition D.1.2.1]{lurie--sag}). Then the two constructions
  \[
    \{X^n\}_{n \in \num{Z}} \in \Gr(\cat{C})
    \quad \mapsto \quad
    \begin{cases}
      \{X^n[2n]\}_{n \in \num{Z}} \in \Gr(\cat{C}) \\
      \{X^n[-2n]\}_{n \in \num{Z}} \in \Gr(\cat{C})
    \end{cases}
  \]
  canonically extend to a pair of mutually inverse symmetric monoidal $t$-exact equivalences
  \[
    [2*] : \Gr(\cat{C})_- \fromto \Gr(\cat{C})_+ : [-2*],
  \]
  such that the induced equivalences
  \[
    \Gr(\cat{C}^\heart) \iso \Gr(\cat{C})_-^\heart \fromto \Gr(\cat{C})_+^\heart \iso \Gr(\cat{C}^\heart)
  \]
  are homotopic to the identity.
\end{proposition}

\begin{proof}
  We first extend the assignments $n \mapsto \num{Z}[2n]$ and $n \mapsto \num{Z}[-2n]$ to maps of $\Einfty$-spaces $\phi,-\phi : \num{Z}^\ds \to \Pic(\num{Z})$, where $\Pic(\num{Z})$ denote the Picard space of invertible $\num{Z}$-module spectra (see e.g. \cite[\textsection 2.2]{mathew-stojanoska--picard-tmf} for generalities on Picard spaces).\footnote{This cannot be done over the sphere spectrum $\num{S}$: the map $\num{Z}^\ds \to \Pic(\num{S})$ sending $n \mapsto \num{S}[2n]$ can be given an $\mathrm{E}_2$-structure but not an $\mathrm{E}_3$-structure. However, it is not necessary to pass all the way to $\num{Z}$ to succeed: for example, it is possible to do so (and thus the proposition goes through) already over $\mathrm{MU}$.} There is a fiber sequence of $\Einfty$-spaces $\Pic(\num{Z}) \to \num{Z} \lblto{\kappa} \clspc^2(\num{Z}/2)$, coming from the Postnikov sequence at the level of spectra,
  \[
    \begin{tikzcd}
      \pi_1(\Picspt(\num{Z}))[1] \ar[r] &
      \tau_{\le 1}\Picspt(\num{Z}) \ar[r] &
      \tau_{\le 0} \Picspt(\num{Z}) \\
      \num{Z}/2[1] \ar[u, "\simeq"] \ar[r] &
      \Picspt(\num{Z}) \ar[u, "\simeq"] \ar[r] &
      \num{Z} \ar[u, "\simeq"],
    \end{tikzcd}
  \]
  where $\Picspt(\num{Z})$ denotes the connective delooping of $\Pic(\num{Z})$. The composite maps
  \[
    \num{Z}^\ds \lblto{\pm 2} \num{Z}^\ds \lblto{\kappa} \clspc^2(\num{Z}/2),
    \quad\text{or equivalently}\quad
    \num{Z}^\ds \lblto{\kappa} \clspc^2(\num{Z}/2) \lblto{\pm 2} \clspc^2(\num{Z}/2),
  \]
  are canonically nullhomotopic, inducing the desired maps $\pm\phi$. Note that the notation $-\phi$ is not abusive, as there is a homotopy of $\Einfty$-maps between the sum
  \[
    \num{Z}^\ds \lblto{(\phi,-\phi)} \Pic(\num{Z}) \times \Pic(\num{Z}) \lblto{\otimes} \Pic(\num{Z})
  \]
  and the zero map.

  We now use this construction to define the desired symmetric monoidal functors $[\pm 2*] : \Gr(\cat{C}) \to \Gr(\cat{C})$. To define a \emph{lax} symmetric monoidal functor of this form, it is equivalent by the universal property of the Day convolution structure to define a lax symmetric monoidal functor $\num{Z}^\ds \times \Gr(\cat{C}) \to \cat{C}$. Let us take this to be the composition of lax symmetric monoidal functors
  \[
    \num{Z}^\ds \times \Gr(\cat{C}) \lblto{(\pm\phi \circ \pi_1, \ev)} \Pic(\num{Z}) \times \cat{C} \lblto{\otimes} \cat{C},
  \]
  where the second map comes from the $\num{Z}$-linear structure on $\cat{C}$. Unwinding the definitions shows that the associated lax symmetric monoidal functor $\Gr(\cat{C}) \to \Gr(\cat{C})$ extends the desired construction on objects and is in fact symmetric monoidal.

  Finally, the t-exactness follows from the fact that the shift $[2n]$ carries $\cat{C}_{\le -n}$ into $\cat{C}_{\le n}$ and $\cat{C}_{\ge -n}$ into $\cat{C}_{\ge n}$, and it is clear that these functors are mutually inverse and induce the identity on the heart as claimed.
\end{proof}

We now discuss t-structures on $\Fil(\cat{C})$. The first one is analogous to the neutral t-structure on $\Gr(\cat{C})$.

\begin{construction}
  \label{gf--t--fil-neutral}
    We define a compatible t-structure $(\Fil(\cat{C})_{\ge 0}, \Fil(\cat{C})_{\le 0})$ on $\Fil(\cat{C})$ as follows: let $\Fil(\cat{C})_{\ge 0}$ be the full subcategory of $\Fil(\cat{C})$ consisting of those filtered objects $X^\star$ for which $X^n \in \cat{C}_{\ge 0}$ for all $n \in \num{Z}$, and let $\Fil(\cat{C})_{\le 0}$ be the full subcategory of $\Fil(\cat{C})$ consisting of those filtered objects $X^\star$ for which $X^n \in \cat{C}_{\le 0}$ for all $n \in \num{Z}$. We refer to this as the \emph{neutral t-structure on $\Fil(\cat{C})$}.
\end{construction}

The next t-structure on $\Fil(\cat{C})$ is related to the positive t-structure on $\Gr(\cat{C})$, and is motivated by the Postnikov filtration construction (\cref{gf--t--postnikov-fil}).

\begin{construction}
  \label{gf--t--postnikov-t}
  We define a compatible t-structure $(\Fil(\cat{C})^\post_{\ge 0}, \Fil(\cat{C})^\post_{\le 0})$ on $\Fil(\cat{C})$ as follows: let $\Fil(\cat{C})^\post_{\ge 0}$ be the full subcategory of $\Fil(\cat{C})$ consisting of those filtered objects $X^\star$ for which $X^n \in \cat{C}_{\ge n}$ for all $n \in \num{Z}$, and let $\Fil(\cat{C})^\post_{\le 0}$ be the full subcategory of $\Fil(\cat{C})$ consisting of those filtered objects $X^\star$ for which $X^n \in \cat{C}_{\le n}$ for all $n \in \num{Z}$. We refer to this as the \emph{Postnikov t-structure on $\Fil(\cat{C})$}.
\end{construction}

\begin{construction}
  \label{gf--t--postnikov-fil}
  Let $\delta : \cat{C} \to \Fil(\cat{C})$ denote the diagonal functor and let $\tau_{\ge 0}^\post : \Fil(\cat{C}) \to \Fil(\cat{C})_{\ge 0}^\post$ denote the connective cover functor with respect to the Postnikov t-structure (\cref{gf--t--postnikov-t}). We let $\tau_{\ge\star} : \cat{C} \to \Fil(\cat{C})_{\ge 0}^\post$ denote the composite
  \[
    \cat{C} \lblto{\delta} \Fil(\cat{C}) \lblto{\tau_{\ge 0}^\post} \Fil(\cat{C})_{\ge 0}^\post,
  \]
  and refer to this as the \emph{Postnikov filtration} functor; as the notation suggests, we have canonical identifications $\ev^n \circ \tau_{\ge \star} \iso \tau_{\ge n}$ for $n \in \num{Z}$. We note that there is a canonical lax symmetric monoidal structure on $\tau_{\ge\star}$, determined by the canonical lax symmetric monoidal structures on $\delta$ and $\tau_{\ge 0}^\post$.
\end{construction}

\begin{remark}
  \label{gf--t--postnikov-fully-faithful}
  Assuming that the t-structure on $\cat{C}$ is right complete, the Postnikov filtration functor $\tau_{\ge\star}$ of \cref{gf--t--postnikov-fil} is fully faithful. Indeed, it is right adjoint to the restriction of the colimit functor $\colim : \Fil(\cat{C}) \to \cat{C}$ to $\Fil(\cat{C})_{\ge 0}^\post$, and the counit transformation for this adjunctions is an equivalence under this completeness assumption.
\end{remark}

The last t-structure on $\Fil(\cat{C})$ that will be relevant in this paper is related to the negative t-structure on $\Gr(\cat{C})$: it is the \emph{Beilinson t-structure}, as discussed in \cite[\S5.1]{bms2} and employed in \cite{antieau--periodic}. We give an alternative explanation/construction of this t-structure using the results of \cref{gf--kd}.

\begin{notation}
  \label{gf--t--ch}
  Let $\DG(\cat{C}^\heart)$ denote the ordinary category of cochain complexes in the abelian category $\cat{C}^\heart$. We regard $\DG(\cat{C}^\heart)$ as a symmetric monoidal category in the standard manner (in particular with the Koszul sign rule imposed).
\end{notation}

\begin{construction}
  \label{gf--t--mod-t}
  Recall from \cref{gf--t--gr-t} that we have the negative t-structure on $\Gr(\cat{C})$, whose connective part $\Gr(\cat{C})_{\ge 0^-}$ consists of those graded objects $X^*$ such that $X^n \in \cat{C}_{\ge -n}$ for all $n \in \num{Z}$. Let $\num{D}_-$ be as in \cref{gf--kd--dminus}. Observe that $\num{D}_- \in \Gr(\cat{C})_{\ge 0^-}$, so that we have a full subcategory $\LMod_{\num{D}_-}(\Gr(\cat{C})_{\ge 0^-}) \subseteq \LMod_{\num{D}_-}(\Gr(\cat{C}))$; it follows from \cite[Proposition 1.4.4.11]{lurie--algebra} that this is the connective part of a t-structure on $\LMod_{\num{D}_-}(\Gr(\cat{C}))$.

  Suppose given a object $X \in \LMod_{\num{D}_-}(\Gr(\cat{C}))$ that is coconnective for this t-structure: this holds if and only if $\Map_{\LMod_{\num{D}_-}(\Gr(\cat{C}))}(Y[1],X) \iso 0$ for all connective objects $Y \in \LMod_{\num{D}_-}(\Gr(\cat{C})_{\ge 0^-})$. Applying this in the case that $Y = \num{D}_- \oast Y_0$ for $Y_0 \in \Gr(\cat{C})_{\ge 0^-}$, we deduce by adjunction that $\Map_{\Gr(\cat{C})}(Y_0[1],X) \iso 0$, implying that the underlying object of $X$ in $\Gr(\cat{C})$ is coconnective in the negative t-structure. The converse also holds: this follows from the fact that $\LMod_{\num{D}_-}(\Gr(\cat{C})_{\ge 0^-})$ is generated under colimits by objects of the form $\num{D}_- \oast Y_0$.

  That is, an object $\LMod_{\num{D}_-}(\Gr(\cat{C}))$ is connective or coconnective with respect to the t-structure just defined if and only if its underlying object in $\Gr(\cat{C})$ is connective or coconnective with respect to the negative t-structure. This shows that the above t-structure is a compatible one. In addition, the heart of the t-structure can be described similarly, and it follows that the heart identifies with $\LMod_{\pi_{0^-}(\num{D}_-)}(\Gr(\cat{C})^\heart_-)$, where $\pi_{0^-}$ is the zeroth homotopy group functor with respect to the negative t-structure. It then follows from the analysis of $\num{D}_-$ in \cref{gf--kd--dminus} that the equivalence $\Gr(\cat{C})^\heart_- \iso \Gr(\cat{C}^\heart)$ identifies this heart with the symmetric monoidal category $\DG(\cat{C}^\heart)$ of cochain complexes in $\cat{C}^\heart$ (that is, a $\pi_{0^-}(\num{D}_-)$-module structure corresponds to a differential).
\end{construction}

\begin{lemma}
  \label{gf--t--complete-beilinson}
  Let $\cpl\Fil(\cat{C})^\beil_{\ge 0}$ denote the full subcategory of $\cpl\Fil(\cat{C})$ spanned by those complete filtered objects $X$ satisfying $\gr(X) \in \Gr(\cat{C})_{\ge 0^-}$, i.e. $\gr^n(X) \in \cat{C}_{\ge -n}$ for all $n \in \num{Z}$; and let $\cpl\Fil(\cat{C})^\beil_{\le 0}$ denote the full subcategory of $\cpl\Fil(\cat{C})$ spanned by those complete filtered objects $X$ satisfying $\gr(X) \in \Gr(\cat{C})_{\le 0^-}$, i.e. $\gr^n(X) \in \cat{C}_{\le -n}$ for all $n \in \num{Z}$. Then $(\cpl\Fil(\cat{C})^\beil_{\ge 0}, \cpl\Fil(\cat{C})^\beil_{\le 0}$) defines a compatible t-structure on $\cpl\Fil(\cat{C})$, whose heart is canonically equivalent to the symmetric monoidal category $\DG(\cat{C}^\heart)$ of cochain complexes in $\cat{C}^\heart$.
\end{lemma}

\begin{proof}
  This follows from \cref{gf--t--mod-t}, using the equivalence $\cpl\Fil(\cat{C}) \iso \LMod_{\num{D}_-}(\Gr(\cat{C}))$ of \cref{gf--kd--gf}.
\end{proof}

\begin{lemma}
  \label{gf--t--beilinson-coconnective}
  Suppose that the t-structure on $\cat{C}$ is right separated, i.e. that any object in $\bigcap_{n \in \num{Z}_{\ge0}} \cat{C}_{\le -n}$ is zero. Then the following conditions on a filtered object $X \in \Fil(\cat{C})$ are equivalent :
  \begin{enumerate}
  \item \label{gf--t--beilinson-coconnective--cpl-gr}
    $X$ lies in the full subcategory $\cpl\Fil(\cat{C})^\beil_{\le 0} \subseteq \Fil(\cat{C})$ of \cref{gf--t--complete-beilinson};
  \item \label{gf--t--beilinson-coconnective--und}
    $\und(X) \in \Gr(\cat{C})_{\le 0^-}$, i.e. $X^n \in \cat{C}_{\le -n}$ for all $n \in \num{Z}$.
  \end{enumerate}
\end{lemma}

\begin{proof}
  Assume first that \cref{gf--t--beilinson-coconnective--cpl-gr} holds, and let us show for any $n \in \num{Z}$ that $X^n \in \cat{C}_{\le-n}$. By hypothesis, we have for each $m \ge n$ that $\gr^m(X) \in \cat{C}_{\le-m} \subseteq \cat{C}_{\le-n}$. It follows by induction that $\cofib(X^m \to X^n) \in \cat{C}_{\le -n}$ for each $m \ge n$. The canonical map $X^n \to \lim_{m \to \infty} \cofib(X^m \to X^n)$ is an equivalence: its fiber is $\lim_{m \to \infty} X^m$, which vanishes by our hypothesis that $X$ is complete. Since $\cat{C}_{\le-n}$ is closed under limits in $\cat{C}$, we conclude that $X^n \in \cat{C}_{\le-n}$.

  Assume now that \cref{gf--t--beilinson-coconnective--und} holds. This clearly implies that $\gr(X) \in \Gr(\cat{C})_{\le0^-}$. It remains to show that $X$ is complete, i.e. that $\lim_{m \to \infty} X^m \iso 0$. Since the t-structure is right separated, it suffices to show for all $n \in \num{Z}_{\ge 0}$ that $\lim_{m \to \infty} X^m \in \cat{C}_{\le-n}$. This follows from the hypothesis that $X^m \in \cat{C}_{\le-m} \subseteq \cat{C}_{\le-n}$ for each $m \ge n$ and the fact that $\cat{C}_{\le-n}$ is closed under limits in $\cat{C}$.
\end{proof}

\begin{proposition}
  \label{gf--t--beilinson}
  Suppose that the t-structure on $\cat{C}$ is right separated. Let $\Fil(\cat{C})^\beil_{\ge 0}$ denote the full subcategory of $\Fil(\cat{C})$ spanned by those filtered objects $X$ satisfying $\gr(X) \in \Gr(\cat{C})_{\ge 0^-}$, i.e. $\gr^n(X) \in \cat{C}_{\ge -n}$ for all $n \in \num{Z}$; and let $\Fil(\cat{C})^\beil_{\le 0}$ be the full subcategory of $\Fil(\cat{C})$ spanned by those filtered objects $X$ satisfying $\und(X) \in \Gr(\cat{C})_{\le 0^-}$, i.e. $X^n \in \cat{C}_{\le -n}$ for all $n \in \num{Z}$. Then $(\Fil(\cat{C})^\beil_{\ge 0}, \Fil(\cat{C})^\beil_{\le 0})$ defines a compatible t-structure on $\Fil(\cat{C})$, whose heart is canonically equivalent to the symmetric monoidal category $\DG(\cat{C}^\heart)$ of cochain complexes in $\cat{C}^\heart$.
\end{proposition}

\begin{proof}
  Let us first show that the subcategory $\Fil(\cat{C})^\beil_{\le 0}$ defined in the statement is the coconnective part of a t-structure on $\Fil(\cat{C})$. By \cite[Proposition 1.2.1.16]{lurie--algebra}, this is equivalent to showing that it is a localization of $\Fil(\cat{C})$ and closed under extensions. By \cref{gf--t--beilinson-coconnective}, this subcategory coincides with the subcategory $\cpl\Fil(\cat{C})^\beil_{\le 0}$ of \cref{gf--t--complete-beilinson}, so by that result is the coconnective part of a t-structure on $\cpl\Fil(\cat{C})$. That is, the subcategory is a localization of $\cpl\Fil(\cat{C})$ and is closed under extensions therein. Since $\cpl\Fil(\cat{C})$ itself is a localization of $\Fil(\cat{C})$ (\cref{gf--kd--cpl-loc}) that is closed under extensions, we deduce the desired claim.

  We next show that the connective part of the t-structure defined in the previous paragraph is the subcategory $\Fil(\cat{C})^\beil_{\ge 0}$ defined in the statement; we will invoke \cref{gf--kd--cpl-loc} again here. A filtered object $X \in \Fil(\cat{C})$ is connective for this t-structure if and only if $\Map_{\Fil(\cat{C})}(X,Y[-1]) \iso 0$ for all $Y \in \Fil(\cat{C})^\beil_{\le 0}$. Since $\Fil(\cat{C})^\beil_{\le 0} \subseteq \cpl\Fil(\cat{C})$, this is equivalent to the condition that $\Map_{\cpl\Fil(\cat{C})}(\cpl{X},Y[-1]) \iso 0$. By \cref{gf--t--complete-beilinson}, this holds if and only if $\gr(\cpl{X}) \in \Gr(\cat{C})_{\ge 0^-}$. Since the completion map $X \to \cpl{X}$ induces an equivalence on associated graded objects, this gives the desired description.

  Finally, let us note that it is visible from the definitions in the statement that the t-structure is indeed compatible, and by construction, its heart agrees with that of the t-structure of \cref{gf--t--complete-beilinson}.
\end{proof}

\begin{notation}
  \label{gf--t--beilinson-heart}
  Suppose that the t-structure on $\cat{C}$ is right separated. We refer to the t-structure $(\Fil(\cat{C})^\beil_{\ge 0}, \Fil(\cat{C})^\beil_{\le 0})$ of \cref{gf--t--beilinson} as the \emph{Beilinson t-structure on $\Fil(\cat{C})$}. The identification of the heart of this t-structure supplies a canonical lax symmetric monoidal functor $\DG(\cat{C}^\heart) \to \cpl\Fil(\cat{C})$. We will denote the image of a cochain complex $M \in \DG(\cat{C}^\heart)$ under this functor by $|M|^{\ge\star}$ and we let $|M| := \colim(|M|^{\ge\star}) \in \cat{C}$.
\end{notation}

\begin{remark}
  \label{gf--t--brutal}
  Suppose given $M \in \DG(\cat{C}^\heart)$. Let $M^{\ge\star} \in \Fun(\num{Z}^\op,\DG(\cat{C}^\heart))$ denote the \emph{(decreasing) brutal filtration} of $M$, so that, for $i \in \num{Z}$, the graded terms of $M^{\ge i}$ are given by
  \[
    (M^{\ge i})^j \iso
    \begin{cases}
      M^j & j \ge i \\
      0 & \text{otherwise}.
    \end{cases}
  \]
  Then, in the situation of \cref{gf--t--beilinson-heart}, there is a canonical natural equivalence $|M|^{\ge \star} \iso |M^{\ge\star}|$ of functors $\DG(\cat{C}^\heart) \to \Fil(\cat{C})$.
\end{remark}

\begin{example}
  \label{gf--t--beilinson-eg}
  Suppose that $\cat{C} = \Mod_A$ for $A$ a commutative ring. Then, for $M \in \DG(\Mod_A^\heart)$, the object $|M| \in \Mod_A$ is equivalent to the Eilenberg-MacLane object represented by $M$, and, by \cref{gf--t--brutal}, the filtration $|M|^{\ge\star} \in \cpl\Fil(\Mod_A)$ on $|M|$ is the brutal filtration.
\end{example}


\section{Derived commutative rings}
\label{dc}

The constructions and results of \cref{bi} allow us to reformulate the notion of an $\Einfty$-$\num{Z}$-algebra with $\cir$-action purely in terms of the group algebra $\num{Z}[\cir]$ or its dual $\num{Z}^\cir$. Namely, we have equivalences of $\infty$-categories
\[
  \Fun(\clspc\cir,\CAlg_{\num{Z}}) \iso \CAlg(\Mod_{\num{Z}[\cir]}) \iso \cMod_{\num{Z}^\cir}(\CAlg_{\num{Z}}),
\]
where in the middle term we regard $\Mod_{\num{Z}[\cir]}$ as a symmetric monoidal $\infty$-category using the cocommutative bialgebra structure on $\num{Z}[\cir]$, and in the right term we regard $\num{Z}^\cir$ as a commutative bialgebra over $\num{Z}$ (this will be spelled out more carefully in \cref{fc}). These alternative perspectives will be crucial for us later in the paper, but we will in fact need even more: in addition to $\Einfty$-algebra structures, we will be interested in simplicial commutative algebra structures. Thus, for example, we would like to similarly rewrite the $\infty$-category $\Fun(\clspc\cir,\CAlg^\Delta_{\num{Z}})$ of simplicial commutative rings with $\cir$-action in terms of the group algebra $\num{Z}[\cir]$.

It turns out that this is more naturally formulated with the dual $\num{Z}^\cir$: we would like $\num{Z}^{\cir}$ to be a ``simplicial commutative bialgebra'', i.e. a coalgebra object in $\CAlg^\Delta_{\num{Z}}$, and then for there to be an equivalence of $\infty$-categories $\Fun(\clspc\cir,\CAlg^\Delta_{\num{Z}}) \iso \cLMod_{\num{Z}^{\cir}}(\CAlg^\Delta_{\num{Z}})$.  However, this wish is nonsensical as stated: simplicial commutative rings are by definition connective objects, while $\num{Z}^{\cir}$ is not.

The goal of this section to introduce some theory that allows us to make sense of the above wish. Namely, we define the notion of a \emph{derived commutative ring}. We will make the definition in a certain axiomatic context in \cref{dc--df}, and then in \cref{dc--eg} explain the example contexts of interest, including the basic setting of $\num{Z}$-modules involved above as well as the graded and filtered settings needed in \cref{dg,fc}. The definition is formulated using the language of monads, some preliminary discussion of which will be necessary in \cref{dc--mo}. In \cref{dc--ct}, we give a definition of the cotangent complex in the setting of derived commutative rings, generalizing the cotangent complex from the setting of simplicial commutative rings; our main aim there is really to give a recharacterization of the latter that will be useful for studying the derived de Rham complex in \cref{dg--dr}. Finally, in \cref{dc--cn}, we discuss some key facts about coconnective derived commutative rings that will be needed to construct the graded and filtered analogues of $\num{Z}^{\cir}$ in \cref{dg,fc}.

Let us also repeat what was mentioned in \cref{in--ol}: while we are giving a complete account here of the definition of derived commutative rings, the theory we present is due to Mathew, building on an idea of Brantner (to be explicated in \cref{dc--df}), and related to earlier work of Illusie \cite[\textsection I.4]{illusie--cotangent-i}.


\subsection{Monad miscellany}
\label{dc--mo}

The following example of a monad will be a key player throughout much of this section:

\begin{construction}[Symmetric algebra monad]
  \label{dc--mo--sym}
  Let $\cat{C}$ be a presentable symmetric monoidal $\infty$-category. Then the forgetful functor $G : \CAlg(\cat{C}) \to \cat{C}$ admits a left adjoint $F : \cat{C} \to \CAlg(\cat{C})$. We let $\Sym_{\cat{C}} : \cat{C} \to \cat{C}$ denote the monad associated to this adjunction, and refer to it as the \emph{symmetric algebra monad on $\cat{C}$}. Recall that the functor $\Sym_{\cat{C}}$ preserves sifted colimits, and can be described by the formula $\Sym_{\cat{C}}(X) \iso \coprod_{j \ge 0} (X^{\otimes j})_{\Sigma_j}$ (where $(-)_{\Sigma_j}$ to denote homotopy orbits) \cite[Example 3.1.3.14]{lurie--algebra}. Recall also that $G$ is monadic (by the monadicity theorem, cf. \cite[Example 4.7.3.11]{lurie--algebra}), so that we have a canonical equivalence of $\infty$-categories $\CAlg(\cat{C}) \iso \LMod_{\Sym_{\cat{C}}}(\cat{C})$.
\end{construction}

In the situation of \cref{dc--mo--sym}, observe that the functor $\Sym_{\cat{C}}$ has a canonical filtration by functors $\Sym^{\le i}_{\cat{C}} : \cat{C} \to \cat{C}$ for $i \ge 0$, given by $\Sym^{\le i}_{\cat{C}}(X) \iso \coprod_{0 \le j \le i} (X^{\otimes j})_{\Sigma_j}$. In \cref{dc--df}, it will be useful to be able to control $\Sym_{\cat{C}}$ \emph{as a monad} in terms of this filtration. The bulk of this subsection is concerned with setting up the definitions and constructions needed to do this.  We begin with the following definition.

\begin{definition}
  \label{dc--mo--filmo}
  Let $\cat{C}$ be an $\infty$-category. A \emph{filtered monad on $\cat{C}$} is a lax monoidal functor $\num{Z}_{\ge0}^\times \to \End(\cat{C})$, where $\num{Z}_{\ge0}^\times$ denotes the partially ordered set of nonnegative integers regarded as a monoidal category via multiplication. More generally, if $\cat{E}$ is a monoidal full subcategory of $\End(\cat{C})$, we will refer to algebra objects of $\cat{E}$ as \emph{$\cat{E}$-monads} and to lax monoidal functors $\num{Z}_{\ge0}^\times \to \cat{E}$ as \emph{filtered $\cat{E}$-monads}.
\end{definition}

\begin{notation}
  \label{dc--mo--day-operad}
  Let $\cat{C}$ and $\cat{E} \subseteq \End(\cat{C})$ be as in \cref{dc--mo--filmo}. We let $\cat{E}^\circ \to \Assoc^\otimes$ denote the fibration of $\infty$-operads encoding the monoidal structure on $\cat{E}$, and we let $\Fun(\num{Z}_{\ge0}^\times,\cat{E}^\circ)^\star \to \Assoc^\otimes$ denote the fibration of $\infty$-operads obtained by applying the Day convolution construction (\cite[Construction 2.2.6.7]{lurie--algebra}) to the monoidal $\infty$-categories $\num{Z}_{\ge0}^\times$ and $\cat{E}$. We note that algebra objects of $\Fun(\num{Z}_{\ge0}^\times, \cat{E}^\circ)^\star$ can be identified with lax monoidal functors $\num{Z}_{\ge0}^\times \to \cat{E}$, i.e. filtered $\cat{E}$-monads \cite[Remark 2.2.6.8]{lurie--algebra}.\footnote{Beware however that $\Fun(\num{Z}_{\ge0}^\times, \cat{E}^\circ)^\star$ is not a monoidal $\infty$-category, due to composition $- \circ -$ not preserving colimits in the second variable.}
\end{notation}

\begin{proposition}
  \label{dc--mo--colim-operad}
  Let $\cat{C}$ and $\cat{E} \subseteq \End(\cat{C})$ be as in \cref{dc--mo--filmo}. Assume moreover that:
  \begin{enumerate}
  \item $\cat{C}$ admits all small colimits;
  \item $\cat{E}$ is closed under sequential colimits, so that we have a colimit functor $\colim : \Fun(\num{Z}_{\ge0},\cat{E}) \to \cat{E}$, left adjoint to the diagonal functor $\delta : \cat{E} \to \Fun(\num{Z}_{\ge0},\cat{E})$;
  \item each $F \in \cat{E}$ commutes with sequential colimits.
  \end{enumerate}
  Then the adjunction $\colim : \Fun(\num{Z}_{\ge0},\cat{E}) \fromto \cat{E} : \delta$ canonically lifts to a relative adjunction of $\infty$-operads $\Fun(\num{Z}_{\ge0}^\times,\cat{E}^\circ)^\star \fromto \cat{E}^\circ$ over $\Assoc^\otimes$ (in the sense of \cite[\textsection 7.3.2]{lurie--algebra}). In particular, the colimit of a filtered $\cat{E}$-monad is canonically an $\cat{E}$-monad.
\end{proposition}

\begin{proof}
  We begin with the right adjoint functor, i.e. the diagonal $\delta : \cat{E} \to \Fun(\num{Z}_{\ge0},\cat{E})$. Recall that $\delta$ corresponds under the equivalence $\Fun(\cat{E},\Fun(\num{Z}_{\ge0},\cat{E})) \iso \Fun(\cat{E} \times \num{Z}_{\ge0}, \cat{E})$ to the projection functor. Similarly, the projection map of $\infty$-operads $\cat{E}^\circ \times_{\Assoc^\otimes} \num{Z}_{\ge0}^\times \to \cat{E}^\circ$ corresponds uniquely under the universal property of Day convolution \cite[Definition 2.2.6.1]{lurie--algebra} to a map of $\infty$-operads $G : \cat{E}^\circ \to \Fun(\num{Z}_{\ge0}^\times,\cat{E}^\circ)^\star$ over $\Assoc^\otimes$, which lifts $\delta$.

  Now, to prove the desired claim, it suffices to show that $G$ admits a left adjoint relative to $\Assoc^\otimes$. It follows from \cite[Corollary 2.2.6.14]{lurie--algebra} that $\Fun(\num{Z}_{\ge0}^\times,\cat{E}^\circ)^\star \to \Assoc^\otimes$ is a locally cocartesian fibration of $\infty$-operads, so it will be enough to check that criterion (2) of \cite[Proposition 7.3.2.11]{lurie--algebra} is satisfied (note that the first criterion amounts in our situation to the existence of the colimit functor $\Fun(\num{Z}_{\ge0},\cat{E}) \to \cat{E}$ that we are studying). Unravelling the definitions, this follows from our hypothesis that each $F \in \cat{E}$ commutes with sequential colimits.
\end{proof}

We can now state the main goal of this subsection a bit more precisely: in the situation of \cref{dc--mo--sym}, we would like to construct a certain filtered monad $\Sym^{\le\star}_{\cat{C}}$ on $\cat{C}$ whose colimit recovers the symmetric algebra monad $\Sym_{\cat{C}}$. Doing so will require some preliminaries, involving the theory of symmetric sequences and the concomitant theory of operads. Let us begin by reviewing the basic notions (cf. \cite{trimble--lie,haugseng--symseq,brantner--thesis}).

\begin{recollection}
  \label{dc--mo--sseq}
  Let $\num{F}$ denote the groupoid whose objects are finite sets and whose morphisms are bijections, regarded as a symmetric monoidal category via disjoint union. For any presentable symmetric monoidal $\infty$-category $\cat{C}$, we let $\SSeq(\cat{C})$ denote the $\infty$-category $\Fun(\num{F},\cat{C})$ of \emph{symmetric sequences in $\cat{C}$}, regarded as a (presentable) symmetric monoidal $\infty$-category via Day convolution. For $A \in \SSeq(\cat{C})$ and $n \in \num{Z}_{\ge0}$, we let $A_n \in \cat{C}$ denote the value of $A$ on the finite set $\{1,\ldots,n\}$.

  Left Kan extension along the inclusion $\{\emptyset\} \inj \num{F}$ determines a fully faithful, colimit-preserving symmetric monoidal embedding $\iota : \cat{C} \to \SSeq(\cat{C})$. Composing the Yoneda embedding
  \[
    \num{F} \iso \num{F}^\op \to \Fun(\num{F},\Spc) = \SSeq(\Spc)
  \]
  with the unique map $\Spc \to \cat{C}$ in $\CAlg(\PrL)$, we obtain a functor $y : \num{F} \to \SSeq(\cat{C})$.

  There is an additional (non-symmetric) monoidal structure on $\SSeq(\cat{C})$, referred to as the \emph{composition monoidal structure}. This arises as follows. For any $\cat{D} \in \CAlg(\PrL)_{\cat{C}/}$, evaluation at $y(\{1\})$ determines an equivalence of $\infty$-categories
  \[
    \Fun_{\CAlg(\PrL)_{\cat{C}/}}(\SSeq(\cat{C}),\cat{D}) \isoto \cat{D},
  \]
  where the left-hand side denote the $\infty$-category of colimit-preserving symmetric monoidal functors under $\cat{C}$. Taking $\cat{D} = \SSeq(\cat{C})$, we obtain an equivalence of $\infty$-categories $\End_{\CAlg(\PrL)_{\cat{C}/}}(\SSeq(\cat{C})) \iso \SSeq(\cat{C})$. Transporting the composition monoidal structure on the left hand side across this equivalence defines a monoidal structure on $\SSeq(\cat{C})$, and we then take the reverse monoidal structure to obtain the composition monoidal structure on $\SSeq(\cat{C})$. The tensor product of the composition monoidal structure is referred to as the \emph{composition product}, denoted by $(A,B) \mapsto A \circ B$, and given concretely by the formula
  \[
    A \circ B \iso \coprod_{n \ge 0} (A_n \otimes B^{\oast n})_{\Sigma_n},
  \]
  where $\otimes$ denotes the tensoring of $\SSeq(\cat{C})$ over $\cat{C}$ (determined by the symmetric monoidal embedding $\iota$) and $\oast$ denotes the Day convolution product. We refer to algebra objects of $\SSeq(\cat{C})$ with respect to the composition monoidal structure as \emph{operads in $\cat{C}$}, and let $\Op(\cat{C})$ denote the $\infty$-category $\Alg(\SSeq(\cat{C})^\circ)$ of operads in $\cat{C}$.

  Finally, we let $\theta : \SSeq(\cat{C}) \to \End(\SSeq(\cat{C}))$ denote the monoidal functor determined by the composition product, sending $A \mapsto A \circ -$. We also denote the induced functor on algebra objects $\Op(\cat{C}) \to \Alg(\End(\SSeq(\cat{C}))$ by $\theta$. For any $A \in \SSeq(\cat{C})$, the endofunctor $\theta(A)$ preserve the essential image of the embedding $\iota : \cat{C} \to \SSeq(\cat{C})$, and hence restricts to an endofunctor on $\cat{C}$. 
\end{recollection}

\begin{construction}
  \label{dc--mo--sym-operad}
  \emph{Let $\cat{C}$ be a presentable symmetric monoidal $\infty$-category. Let $A \in \SSeq(\cat{C})$ denote the constant symmetric sequence with value the unit object $\unit \in \cat{C}$. We will construct an operad structure on $A$, together with an equivalence of monads $\theta(A) \iso \Sym_{\SSeq(\cat{C})}$ (where $\Sym_{\SSeq(\cat{C})}$ is defined using the Day convolution symmetric monoidal structure on $\SSeq(\cat{C})$).}

  We use ideas similar to those used in the proof of \cref{bi--du--linear-comonad}. As there, let us begin by recalling how the monad structure on $\Sym_{\SSeq(\cat{C})}$ is defined. Let $G : \CAlg(\SSeq(\cat{C})) \to \SSeq(\cat{C})$ denote the forgetful functor, and let $F$ denote its left adjoint, so that $\Sym_{\SSeq(\cat{C})}$ is given as a functor by the compositon $G \circ F$. Regarding the functor category $\Fun(\CAlg(\SSeq(\cat{C})),\SSeq(\cat{C}))$ as left tensored over the monoidal $\infty$-category $\End(\SSeq(\cat{C}))$ via postcomposition, one checks that $\Sym_{\SSeq(\cat{C})} = G \circ F$ is an endomorphism object for $G$, hence carries a canonical algebra (i.e. monad) structure.

  We now refine this construction. Let $\cat{E} := \End_{\CAlg(\PrL)_{\cat{C}/}}(\SSeq(\cat{C}))$. We regard $\cat{E}$ as a monoidal $\infty$-category, and we regard both $\SSeq(\cat{C})$ and $\CAlg(\SSeq(\cat{C}))$ as left tensored over $\cat{E}$ via evaluation, so that the forgetful functor $G$ is canonically $\cat{E}$-linear. It follows that the left adjoint $F$ is canonically oplax $\cat{E}$-linear, and it is easy to check that it is in fact $\cat{E}$-linear. Carrying out the same construction above in the $\cat{E}$-linear setting, we deduce that $\Sym_{\SSeq(\cat{C})}$ is canonically an \emph{$\cat{E}$-linear monad}, i.e. an algebra object in the $\infty$-category of $\cat{E}$-linear endofunctors of $\SSeq(\cat{C})$.

  Now recall from \cref{dc--mo--sseq} that $\cat{E} \iso \SSeq(\cat{C})$. Under this equivalence the left action of $\cat{E}$ on $\SSeq(\cat{C})$ identifies with the right action of $\SSeq(\cat{C})$ on itself via the composition monoidal structure. It follows that any $\cat{E}$-linear endofunctor (resp. monad) on $\SSeq(\cat{C})$ is given by $B \mapsto A' \circ B$ for some $A' \in \SSeq(\cat{C})$ (resp. $A' \in \Op(\cat{C})$). By inspection, the symmetric sequence underlying the operad giving $\Sym_{\SSeq(\cat{C})}$ is constant with value $\unit \in \cat{C}$.
\end{construction}

\begin{construction}[{cf. \cite[\textsection 4.1]{heuts--approx}}]
  \label{dc--mo--sseq-fil}
  Let $\cat{C}$ and $\SSeq(\cat{C})$ be as in \cref{dc--mo--sseq}. For $i \in \num{Z}_{\ge0}$, let $\SSeq^{\le i}(\cat{C}) \subseteq \SSeq(\cat{C})$ denote the full subcategory spanned by those symmetric sequences $A$ such that $A_j$ is an initial object of $\cat{C}$ for all $j > i$.

  Now consider the $\infty$-category $\Fun(\num{Z}_{\ge0},\SSeq(\cat{C}))$, and let $\SSeq^{\le\star}(\cat{C})$ denote the full subcategory of $\Fun(\num{Z}_{\ge0},\SSeq(\cat{C}))$ spanned by those functors $F$ such that $F(i) \in \SSeq^{\le i}(\cat{C})$ for all $i \in \num{Z}_{\ge0}$. The inclusion $\phi : \SSeq^{\le\star}(\cat{C}) \inj \Fun(\num{Z}_{\ge0},\SSeq(\cat{C}))$ admits a right adjoint $\psi : \Fun(\num{Z}_{\ge0},\SSeq(\cat{C})) \to \SSeq^{\le\star}(\cat{C})$, given by the formula
  \[
    \psi(F)(i)_j \iso
    \begin{cases}
      F(i) & j \le i \\
      \emptyset_{\cat{C}} & \text{otherwise},
    \end{cases}
  \]
  for $i,j \in \num{Z}_{\ge0}$, where $\emptyset_{\cat{C}}$ denotes the initial object of $\cat{C}$.

  Let $p : \Fun(\num{Z}_{\ge0}^\times,\SSeq(\cat{C})^\circ)^\otimes \to \Assoc^\otimes$ denote the fibration of $\infty$-operads obtained by applying the Day convolution construction for the multiplication monoidal structure on $\num{Z}_{\ge0}$ and the composition monoidal structure on $\SSeq(\cat{C})$. The composition product on $\SSeq(\cat{C})$ preserves sequential colimits in each variable, so that, by \cite[Corollary 2.2.6.14]{lurie--algebra}, $p$ is a locally cocartesian fibration. Concretely, this means that, for every finite linearly ordered set $S \iso \{1,\ldots,n\}$ (with $n = 0$ for $S = \emptyset$), we have a well-defined Day convolution product
  \[
    \bigotimes_{k=1}^n : \prod_{k=1}^n \Fun(\num{Z}_{\ge0},\SSeq(\cat{C})) \to \Fun(\num{Z}_{\ge0},\SSeq(\cat{C})),
  \]
  given by the formula
  \[
    (\bigotimes_{k=1}^n F_k)(i) \iso \colim_{i = \prod_{k=1}^n i_k } F_1(i_1) \circ \cdots \circ F_k(i_k)
  \]
  (see \cite[Remark 2.2.6.15]{lurie--algebra}). Observe that, for $n \ge 0$ and $F_1,\ldots,F_n \in \SSeq^{\le\star}(\cat{C})$, we have $\bigotimes_{k=1}^n F_k \in \SSeq^{\le\star}(\cat{C})$ as well.

  Let $\SSeq^{\le\star}(\cat{C})^\otimes \subseteq \Fun(\num{Z}_{\ge0}^\times,\SSeq(\cat{C})^\circ)^\otimes$ denote the full subcategory determined by $\SSeq(\cat{C}) \subseteq \Fun(\num{Z}_{\ge0},\SSeq(\cat{C}))$ (as defined in \cite[\S2.2.1]{lurie--algebra}). Noting that the statement and proof of \cite[Proposition 2.2.1.1]{lurie--algebra} go through with ``cocartesian fibration'' replaced by ``locally cocartesian fibration'', we deduce from the observation above that the functor $\psi$ canonically promotes to a map of $\infty$-operads $\psi' : \Fun(\num{Z}_{\ge0}^\times,\SSeq(\cat{C})^\circ)^\otimes \to \SSeq^{\le\star}(\cat{C})^\otimes$.

  Finally, we will denote the composition
  \[
    \SSeq(\cat{C})^\circ \lblto{\delta} \Fun(\num{Z}_{\ge0}^\times,\SSeq(\cat{C})^\circ)^\otimes \lblto{\psi'} \SSeq^{\le\star}(\cat{C})^\otimes \subseteq \Fun(\num{Z}_{\ge0}^\times,\SSeq(\cat{C})^\circ)^\otimes.
  \]
  by $(-)^{\le\star}$; here $\delta$ is the diagonal map, constructed as in the proof of \cref{dc--mo--colim-operad}. We will also denote by $(-)^{\le\star}$ the induced functor on algebra objects
  \[
    \Op(\cat{C}) = \Alg(\SSeq(\cat{C})^\circ) \to \Alg(\Fun(\num{Z}_{\ge0}^\times,\SSeq(\cat{C})^\circ)^\otimes) \iso \Fun^\lax((\num{Z}_{\ge0}^\times,\SSeq(\cat{C})^\circ)^\otimes).
  \]
\end{construction}

We finally come to the main construction of the subsection.

\begin{construction}[Symmetric powers filtered monad]
  \label{dc--mo--sym-fil}
  \emph{Let $\cat{C}$ be a presentable symmetric monoidal $\infty$-category. Let $\cat{E} \subseteq \End(\cat{C})$ denote the full subcategory spanned by those endofunctors of $\cat{C}$ that preserve sifted colimits. We will construct a filtered $\cat{E}$-monad $\Sym^{\le\star}_{\cat{C}}$, whose filtered pieces $\Sym^{\le i}_{\cat{C}}$ are given by the formula $\Sym^{\le i}_{\cat{C}}(X) \iso \coprod_{0 \le j \le i} (X^{\otimes j})_{\Sigma_j}$, together with an equivalence of $\cat{E}$-monads $\colim(\Sym^{\le *}_{\cat{C}}) \iso \Sym_{\cat{C}}$. We will refer to $\Sym^{\le\star}_{\cat{C}}$ as the \emph{symmetric powers filtered monad on $\cat{C}$}.}
  
  Let $\SSeq(\cat{C})$ be as in \cref{dc--mo--sseq} and let $\cat{E}'$ denote the (monoidal) full subcategory of $\End(\SSeq(\cat{C}))$ spanned by those endofunctors of $\SSeq(\cat{C})$ that preserve sifted colimits and that preserve the image of the embedding $\iota : \cat{C} \to \SSeq(\cat{C})$. We have a monoidal restriction functor $\rho : \cat{E}' \to \cat{E}$ that carries the $\cat{E}'$-monad $\Sym_{\SSeq(\cat{C})}$ to the $\cat{E}$-monad $\Sym_{\cat{C}}$.

  The monoidal functor $\theta : \SSeq(\cat{C}) \to \End(\SSeq(\cat{C}))$ factors through $\cat{E}'$. By \cref{dc--mo--sym-operad}, we have an operad $A \in \Op(\cat{C})$ with an equivalence of monads $\theta(A) \iso \Sym_{\SSeq(\cat{C})}$. Applying \cref{dc--mo--sseq-fil}, we obtain a lax monoidal functor $A^{\le\star} : \num{Z}_{\ge 0}^\times \to \SSeq(\cat{C})^\circ$. We let $\Sym^{\le\star}_{\cat{C}}$ denote the composite lax monoidal functor
  \[
    \num{Z}_{\ge 0}^\times \lblto{A^{\le\star}} \SSeq(\cat{C})^\circ \lblto{\theta} \cat{E}' \lblto{\rho} \cat{E}.
  \]
  This filtered $\cat{E}$-monad $\Sym^{\le\star}_{\cat{C}}$ fulfills the requirements delineated above.
\end{construction}

We close this subsection by recording two unrelated results about monads that will be needed later in the section.

\begin{proposition}
  \label{dc--mo--localization}
  Let $\tau : \cat{C} \fromto \cat{C}_0 : \iota$ be a localization of $\infty$-categories, i.e. an adjunction where the right adjoint $\iota$ is fully faithful. Let $T$ be a monad on $\cat{C}$ such that the unit transformation $\id_{\cat{C}} \to \iota \circ \tau$ induces an equivalence $\tau \circ T \isoto \tau \circ T \circ \iota \circ \tau$. Then there is an induced monad structure on the composite $T_0 := \tau \circ T \circ \iota \in \End(\cat{C}_0)$ and an induced localization $\tau : \LMod_T(\cat{C}) \fromto \LMod_{T_0}(\cat{C}_0) : \iota$, where the embedding $\iota : \LMod_{T_0}(\cat{C}_0) \to \LMod_T(\cat{C})$ identifies $\LMod_{T_0}(\cat{C}_0)$ with the fiber product $\LMod_T(\cat{C}) \times_{\cat{C}} \cat{C}_0$.
\end{proposition}

\begin{proof}
  Let $\cat{E}$ denote the full subcategory of $\End(\cat{C})$ spanned by those endofunctors satisfying the property that the natural transformation $\tau \circ T \to \tau \circ T \circ \iota \circ \tau$ is an equivalence. Then $\cat{E}$ is closed under composition in $\End(\cat{C})$ and contains $\id_{\cat{C}}$, so we may regard $\cat{E}$ as a monoidal $\infty$-category over which $\cat{C}$ is left tensored. It follows from the definition of $\cat{E}$ that the localization $\tau : \cat{C} \fromto \cat{C}_0 : \iota$ is compatible with this tensoring over $\cat{E}$, in the sense of \cite[Definition 2.2.1.6]{lurie--algebra} (for the case that the $\infty$-operad $\cat{O}^\otimes$ is the left module $\infty$-operad $\LM^\otimes$ of \S4.2.1 in op. cit.), determining a canonical left tensoring of $\cat{C}_0$ over $\cat{E}$. After unravelling definitons, this gives the claim.
\end{proof}

\begin{proposition}
  \label{dc--mo--alg-presentable}
  Let $\kappa$ be a regular cardinal and let $\cat{C}$ be a $\kappa$-presentable $\infty$-category. Let $T$ be a monad on $\cat{C}$ that commutes with sifted colimits. Then $\LMod_T(\cat{C})$ is a $\kappa$-presentable $\infty$-category.
\end{proposition}

\begin{proof}
  We first argue that $\LMod_T(\cat{C})$ admits small colimits. The hypothesis that $T$ commutes with sifted colimits ensures that $\LMod_T(\cat{C})$ admits sifted colimits (which are preserved by the forgetful functor $\LMod_T(\cat{C}) \to \cat{C}$), by \cite[Corollary 4.2.3.5]{lurie--algebra}. It must admit an initial object as the functor $T : \cat{C} \to \LMod_T(\cat{C})$ is a left adjoint and hence preserves initial objects. It's enough then to see that any pair of objects $A,B \in \LMod_T(\cat{C})$ admit a coproduct in $\LMod_T(\cat{C})$. Using the bar resolutions $A \iso \colim_{[n] \in \Delta^\op} T^{(n+1)}(A)$ and $B \iso \colim_{[n] \in \Delta^\op} T^{(n+1)}(B)$, where $T^{(k)}$ denotes the $k$-fold iterated functor, we may reduce to the case that $A = T(X)$ and $B = T(Y)$ for $X,Y \in \cat{C}$. But then we may again use the fact that $T : \cat{C} \to \LMod_T(\cat{C})$ is a left adjoint to conclude that $A$ and $B$ admit a coproduct.

  We now want to show that $\LMod_T(\cat{C})$ admits a small set of $\kappa$-compact generators. Let $\cat{C}_0$ denote the full subcategory of $\kappa$-compact objects in $\cat{C}$, and let $\cat{A}_0$ be the essential image of $\cat{C}_0$ under $T : \cat{C} \to \LMod_T(\cat{C})$. We make the following observations:
  \begin{itemize}
  \item Since $\cat{C}$ is $\kappa$-presentable, $\cat{C}_0$ is essentially small, which implies that $\cat{A}_0$ is essentially small.
  \item Since $T : \cat{C} \to \LMod_T(\cat{C})$ is left adjoint to the forgetful functor, and the forgetful functor preserves $\kappa$-filtered colimits, the objects of $\cat{A}_0$ are $\kappa$-compact in $\LMod_T(\cat{C})$.
  \item For any $A \in \LMod_T(\cat{C})$, the bar resolution gives us an equivalence
    $A \iso \colim_{[n] \in \Delta^\op} T^{(n+1)}(A)$ in $\LMod_T(\cat{C})$.
  \item Let $X \in \cat{C}$. Since $\cat{C}$ is $\kappa$-presentable, we may choose a $\kappa$-filtered diagram $\{X_\alpha\}$ in $\cat{C}_0$ with an equivalence $X \iso \colim_\alpha X_\alpha$ in $\cat{C}$. This induces an equivalence $T(X) \iso \colim_\alpha T(X_\alpha)$ in $\LMod_T(\cat{C})$, where each $T(X_\alpha)$ lies in $\cat{A}_0$.
  \end{itemize}
  We conclude that $\cat{A}_0$ is an essentially small subcategory of $\kappa$-compact objects generating $\LMod_T(\cat{C})$, as desired.
\end{proof}


\subsection{The general definition}
\label{dc--df}

Our goal in this subsection is to define \emph{derived commutative rings}. Before we begin, let us briefly outline the idea. As mentioned earlier, the goal here is to enlarge the $\infty$-category of simplicial commutative rings $\CAlg^\Delta_{\num{Z}}$ to allow for nonconnective objects. To do so, we look at $\CAlg^\Delta_{\num{Z}}$ from the vantage of the fact that its forgetful functor the $\infty$-category $\Mod_{\num{Z}}^\cn$ of connective $\num{Z}$-modules is \emph{monadic}: that is, it admits a left adjoint, and $\CAlg^\Delta_{\num{Z}}$ can be identified with $\infty$-category of modules over the monad associated to this adjunction. We denote this monad by $\LSym_{\num{Z}}$, and refer to it as the \emph{derived symmetric algebra monad} on $\Mod_{\num{Z}}^\cn$. It is uniquely characterized by the facts that it preserves sifted colimits and that on finite free $\num{Z}$-modules it is given by the formation of \emph{ordinary} symmetric/polynomial algebras (not of free $\Einfty$-algebras). Our strategy will be to construct an extension of the derived symmetric algebra monad from $\Mod_{\num{Z}}^\cn$ to all of $\Mod_{\num{Z}}$, and then define derived commutative rings as modules for this extended monad.

In this subsection, we will work not just in the setting of $\num{Z}$-modules, but in the general context of an $\infty$-category $\cat{C}$ subject to certain axioms. For our purposes in this paper, the point of making the definition in this generality is so that it may also be applied in the graded and filtered settings, i.e. not just for $\cat{C} = \Mod_{\num{Z}}$ but for $\cat{C} = \Gr(\Mod_{\num{Z}})$ and $\cat{C} = \Fil(\Mod_{\num{Z}})$ as well. We will discuss in \cref{dc--eg} how these settings fit into the framework of this subsection.

Let us now begin by formulating the axiomatic context in which we will work.

\begin{definition}
  \label{dc--df--context}
  A \emph{derived algebraic context} consists of a stable presentable symmetric monoidal $\infty$-category $\cat{C}$, a compatible t-structure $(\cat{C}_{\ge 0},\cat{C}_{\le 0})$ (\cref{gf--t--compatible}), and a small full subcategory $\cat{C}^0 \subseteq \cat{C}^\heart$, satisfying the following properties:
  \begin{enumerate}
  \item The t-structure is right complete. Note also that the compatibility of the t-structure implies that there is an induced presentable symmetric monoidal structure on $\cat{C}^\heart$, with tensor product given by $X \otimes_{\cat{C}^\heart} Y \iso \pi_0(X \otimes_{\cat{C}} Y)$.
  \item The subcategory $\cat{C}^0$ is a symmetric monoidal subcategory of $\cat{C}$ and is closed under the formation of $\cat{C}^\heart$--symmetric powers: that is, for $X \in \cat{C}^0$ and $n \ge 0$, we have that $\Sym^n_{\cat{C}^\heart}(X) \in \cat{C}^0$ as well.
  \item The subcategory $\cat{C}^0$ is closed under the formation of finite coproducts in $\cat{C}$ and its objects form a set of compact projective generators for $\cat{C}_{\ge 0}$.
  \end{enumerate}
  We will usually abusively refer to just $\cat{C}$ as a derived algebraic context, leaving the rest of the structure implicit.

  There is a natural notion of \emph{morphism between derived algebraic contexts $\cat{C}$ and $\cat{D}$}, namely a colimit-preserving, right t-exact symmetric monoidal functor $F : \cat{C} \to \cat{D}$ such that $F(\cat{C}^0) \subseteq \cat{D}^0$.
\end{definition}

\begin{remark}
  \label{dc--df--context-determined}
  The entire structure of a derived algebraic context $\cat{C}$ is uniquely determined by the constituent piece $\cat{C}^0$. Firstly, by the projective generation assumption, we have a canonical symmetric monoidal equivalence $\cat{C}_{\ge0} \iso \PSh_\Sigma(\cat{C}^0)$, where $\PSh_\Sigma(\cat{C}^0)$ inherits a symmetric monoidal structure from $\cat{C}^0$ by \cite[Proposition 4.8.1.10]{lurie--algebra}. And secondly, by the right completeness assumption, we have a canonical equivalence $\cat{C} \iso \Spt(\cat{C}_{\ge0}) \iso \Spt \otimes \cat{C}_{\ge0}$ \cite[Remark  C.3.1.5]{lurie--sag}, with the symmetric monoidal structure again inherited. These observations imply that $\cat{C}^\heart$ may be canonically identified with the category $\PSh_\Sigma(\cat{C}^0;\Set)$ of product-preserving functors $(\cat{C}^0)^\op \to \Set$ \cite[Remark  C.1.5.9]{lurie--sag}.
\end{remark}

For the remainder of this subsection, we fix a derived algebraic context $\cat{C}$.

\begin{notation}
  \label{dc--df--sigma}
  If $\cat{A}$ and $\cat{B}$ are $\infty$-categories admitting sifted colimits, we let $\Fun^\Sigma(\cat{A},\cat{B})$ denote the full subcategory of $\Fun(\cat{A},\cat{B})$ spanned by those functors that preserve sifted colimits.

  Recall from \cite[\S5.5.8]{lurie--topos} that $\cat{C}^0$ projectively generating $\cat{C}_{\ge 0}$ implies that, for any $\infty$-category $\cat{D}$ admitting sifted colimits, the restriction functor $\Fun^\Sigma(\cat{C},\cat{D}) \to \Fun(\cat{C}^0,\cat{D})$ is an equivalence of $\infty$-categories. We denote the image of a functor $F : \cat{C}^0 \to \cat{D}$ under the inverse to this equivalence by $\lder F : \cat{C}_{\ge 0} \to \cat{D}$, and refer to this as the \emph{left derived functor} of $F$.
\end{notation}

\begin{example}
  \label{dc--df--lsym-connective}
  Let $F : \cat{C}^0 \to \cat{C}_{\ge 0}$ denote the functor sending $X \mapsto \Sym_{\cat{C}^\heart}(X)$. We denote the left derived functor $\lder F$ by $\LSym_{\cat{C}_{\ge 0}} : \cat{C}_{\ge 0} \to \cat{C}_{\ge 0}$, and refer to this as the \emph{derived symmetric algebra functor on $\cat{C}_{\ge 0}$}. Replacing $\Sym_{\cat{C}^\heart}$ with $\Sym^i_{\cat{C}^\heart}$ for $i \ge 0$, we similarly obtain \emph{derived symmetric powers functors $\LSym^i_{\cat{C}_{\ge0}}$ on $\cat{C}_{\ge 0}$}.
\end{example}

As we will see below, the derived symmetric algebra functor $\LSym_{\cat{C}_{\ge 0}}$ has a canonical monad structure. Our goal is to construct an extension of this monad to $\cat{C}$. We will do so using the theory of functor calculus---this approach to extending functors is due to Brantner, and its application in this context is due to Mathew\footnote{There is also an earlier approach to extending the derived symmetric powers functors due to Illusie \cite[\textsection I.4]{illusie--cotangent-i}.}---so let us begin by reviewing the necessary aspects of this theory. It will be apparent further below, but let us note now that we are largely following \cite[\S3]{brantner-mathew--deformation} here.

\begin{definition}
  \label{dc--df--cube}
  For $n \ge 0$, let $\cube_n$ denote the power set of $\{0,\ldots,n\}$, and for $0 \le m \le n+1$, let $\cube_n^{\le m}$ (resp. $\cube_n^{\ge m}$) denote the subset of $\cube_n$ consisting of those subsets of $\{0,\ldots,n\}$ of cardinality at most (resp. at least) $m$. We regard $\cube_n$, $\cube_n^{\le m}$, $\cube_n^{\ge m}$ as categories via their partial ordering by inclusion.

  Given an $\infty$-category $\cat{A}$, an \emph{$n$-cube in $\cat{A}$} is a diagram $\chi : \cube_n \to \cat{A}$. We say an $n$-cube $\chi$ is:
  \begin{itemize}
  \item \emph{cocartesian} if it is a colimit diagram, i.e. exhibits $\chi(\{0,\ldots,n\})$ as a colimit of the diagram $\chi|_{\cube_n^{\le n}}$;
  \item \emph{cartesian} if it is a limit diagram, i.e. exhibits $\chi(\emptyset)$ a a limit of the diagram $\chi|_{\cube_n^{\ge 1}}$;
  \item \emph{strongly cocartesian} if it is left Kan extended from its restriction to $\cube_n^{\le 1}$.
  \end{itemize}
\end{definition}

\begin{remark}
  \label{dc--df--cube-stable}
  If $\cat{A}$ is a stable $\infty$-category, then an $n$-cube in $\cat{A}$ is cartesian if and only if it is cocartesian (\cite[Proposition 1.2.4.13]{lurie--algebra}).
\end{remark}

\begin{definition}
  \label{dc--df--exc-poly}
  Let $\cat{A}$ be an $\infty$-category admitting finite colimits and let $\cat{B}$ be a stable $\infty$-category.
  \begin{itemize}
  \item For $n \ge 0$, a functor $F : \cat{A} \to \cat{B}$ is called \emph{$n$-excisive} if it carries strongly cocartesian $n$-cubes in $\cat{A}$ to cocartesian (equivalently cartesian by \cref{dc--df--cube-stable}) $n$-cubes in $\cat{B}$.  We let $\Exc_n(\cat{A},\cat{B})$ denote the full subcategory of $\Fun(\cat{A},\cat{B})$ spanned by the functors that are $n$-excisive.
  \item A functor $F : \cat{A} \to \cat{B}$ is called \emph{excisively polynomial} if it is $n$-excisive for some $n \ge 0$. We let $\Fun_\epoly(\cat{A},\cat{B})$ denote the full subcategory of $\Fun(\cat{A},\cat{B})$ spanned by the functors that are excisively polynomial.
  \end{itemize}
\end{definition}

\begin{example}
  \label{dc--df--exc-poly-sym}
  Let $\cat{B}$ be a stable, presentable symmetric monoidal $\infty$-category. Then the functor $\Sym^{\le i}_{\cat{B}} : \cat{B} \to \cat{B}$ is $i$-excisive. This follows from the fact that each tensor power functor $(-)^{\otimes j} : \cat{B} \to \cat{B}$ is $j$-excisive, which is a consequence of the tensor product being exact in each variable (see e.g. \cite[Corollary 6.1.3.5]{lurie--algebra}).
\end{example}

\begin{remark}
  \label{dc--df--exc-poly-compose}
  Let $\cat{B},\cat{B}',\cat{B}''$ be stable $\infty$-categories. If $F : \cat{B} \to \cat{B}'$ is an $n$-excisive functor and $G : \cat{B}' \to \cat{B}''$ is an $m$-excisive functor, then the composition $GF$ is $mn$-excisive (see for example \cite[Lemma 7.5]{mccarthy--excisive}). It follows that the collection of excisively polynomial endofunctors on a stable $\infty$-category $\cat{B}$ is closed under composition.
\end{remark}

\begin{definition}
  \label{dc--df--add-poly}
  Let $\cat{A}$ and $\cat{B}$ be additive $\infty$-categories and assume $\cat{B}$ is idempotent complete. We inductively say that a functor $F : \cat{A} \to \cat{B}$ is:
  \begin{itemize}
  \item \emph{of degree $0$} if $F$ is constant;
  \item \emph{of degree $n$} for $n \ge 1$ if, for each $X \in \cat{A}$, the difference functor $D_XF : \cat{A} \to \cat{B}$ defined by $D_XF(Y) := \fib(F(X \oplus Y) \to F(Y))$ is of degree $n-1$ (note that this fiber is guaranteed to exist by the idemponent completeness assumption).
  \end{itemize}
  We say that a functor $F : \cat{A} \to \cat{B}$ is \emph{additively polynomial} if it is of degree $n$ for some $n \ge 0$, and let $\Fun_\apoly(\cat{A},\cat{B})$ denote the full subcategory of $\Fun(\cat{A},\cat{B})$ spanned by the additively polynomial functors.
\end{definition}

\begin{example}
  \label{dc--df--add-poly-sym}
  Let $\cat{B}$ be an additive, presentable symmetric monoidal $\infty$-category. Then the functor $\Sym^{\le i}_{\cat{B}} : \cat{B} \to \cat{B}$ is of degree $i$. This follows the fact that each tensor power functor $(-)^{\otimes j} : \cat{B} \to \cat{B}$ is of degree $j$ (see e.g. \cite[Example 2.6.3]{johnson-mccarthy--taylor}).
\end{example}

\begin{remark}
  \label{dc--df--add-poly-compose}
  Let $\cat{B},\cat{B}',\cat{B}''$ be additive $\infty$-categories. If $F : \cat{B} \to \cat{B}'$ is a functor of degree $n$ and $G : \cat{B}' \to \cat{B}''$ is a functor of degree $m$, then the composition $GF$ is of degree $mn$. It follows that the collection of additively polynomial endofunctors on an additive $\infty$-category $\cat{B}$ is closed under composition.
\end{remark}

\begin{notation}
  \label{dc--df--fun-ntn-addendum}
  The functor category decorations introduced in \cref{dc--df--sigma,dc--df--exc-poly,dc--df--add-poly} will also be applied to endomorphism categories, with the same meaning. They may also be applied in conjunction, with the obvious meaning. For example, $\End_\epoly^\Sigma(\cat{C})$ denotes the $\infty$-category of functors $F : \cat{C} \to \cat{C}$ that are excisively polynomial and preserve sifted colimits.
\end{notation}

\begin{proposition}
  \label{dc--df--degree-n-excisive}
  Let $\cat{D}$ be a stable $\infty$-category admitting small colimits and let $F : \cat{C}^0 \to \cat{D}$ be a functor of degree $n$ for some $n \ge 0$. Then the left derived functor $\lder F : \cat{C}_{\ge 0} \to \cat{D}$ is $n$-excisive.
\end{proposition}

\begin{proof}
  See (the proof of) \cite[Proposition 3.34]{brantner-mathew--deformation}, which goes back to Johnson--McCarthy \cite[Proposition 5.10]{johnson-mccarthy--taylor}.
\end{proof}

The following result is the key for extending functors defined on $\cat{C}_{\ge 0}$ to all of $\cat{C}$.

\begin{proposition}
  \label{dc--df--poly-cn}
  Let $\cat{D}$ be a stable $\infty$-category admitting small colimits. Then the restriction functor
  \[
    \Fun_\epoly^\Sigma(\cat{C}, \cat{D}) \to \Fun_\epoly^\Sigma(\cat{C}_{\ge 0}, \cat{D})
  \]
  is an equivalence of $\infty$-categories.
\end{proposition}

\begin{proof}
  It will suffice to show for each fixed $n \ge 0$ that the restriction functor
  \[
    \psi : \Exc_n^\Sigma(\cat{C},\cat{D}) \to \Exc_n^\Sigma(\cat{C}_{\ge 0},\cat{D})
  \]
  is an equivalence. We closely follow the proof of \cite[Theorem 3.35]{brantner-mathew--deformation}.

  We first claim that $\psi$ has a left adjoint $\phi$, which sends $F \in \Exc_n^\Sigma(\cat{C}_{\ge 0},\cat{D})$ to $P_n(F \circ \tau_{\ge 0})$, where $P_n : \Fun(\cat{C},\cat{D}) \to \Exc_n(\cat{C},\cat{D})$ denotes the $n$-excisive approximation functor. This is straightforward to see once we know that $P_n(F \circ \tau_{\ge 0})$ preserves sifted colimits, assuming $F$ does. It preserves filtered colimits because $\tau_{\ge 0}$ does and $P_n$ preserves this property (as a result of Goodwillie's explicit description of $P_n$). By \cite[Proposition 3.36]{brantner-mathew--deformation}, this implies it preserves sifted colimits.

  Now, the unit map $\id \to \psi \circ \phi$ is an equivalence, implying that $\phi$ is fully faithful. Thus, to prove that $\psi$ is an equivalence, it suffices to prove that $\psi$ is conservative. We can see this in two steps: firstly, restriction to $\cat{C}_{\ge -\infty} := \bigcup_{k \ge 0} \cat{C}_{\ge -k}$ is conservative, because our functors in particular preserve sequential colimits and, for any $X \in \cat{C}$, we have $X \iso \colim_{m \to \infty} \tau_{\ge -m}(X)$ by our assumption that the t-structure on $\cat{C}$ is right complete; secondly, restriction from $\cat{C}_{\ge -\infty}$ to $\cat{C}_{\ge 0}$ is conservative by the argument in the last paragraph of the proof of \cite[Theorem 3.35]{brantner-mathew--deformation}.
\end{proof}

\begin{remark}
  \label{dc--df--excisive-coconnective}
  Let $\cat{D}$ be a stable $\infty$-category admitting small colimits and let $F : \cat{C} \to \cat{D}$ be a functor that preserves filtered colimits and is excisively polynomial. By \cite[Proposition 3.36]{brantner-mathew--deformation}, $F$ preserves $n$-skeletal totalizations for any $n \ge 0$. This provides a way of understanding the behavior of $F$ on certain objects in $\cat{C}_{\le 0}$ in terms of its behavior on $\cat{C}^0$: namely, if $X \in \cat{C}_{\le 0}$ can be written as the totalization (in $\cat{C}$) of an $n$-skeletal cosimplicial diagram $X^\bul : \Delta \to \cat{C}^0$, then we have $F(X) \iso \lim_\Delta F(X^\bul)$. For example, in the case that $\cat{C} = \Mod_{\num{Z}}$, this may be applied when $X$ is any perfect $\num{Z}$-module of Tor-amplitude $\le 0$.
\end{remark}

We need just a bit more of preparation to construct the derived symmetric algebra monad on $\cat{C}$.

\begin{lemma}
  \label{dc--df--derived-pizero}
  Let $F : \cat{C}_{\ge 0} \to \cat{C}_{\ge 0}$ be a functor preserving sifted colimits. Then the truncation map $X \to \pi_0(X)$ induces an equivalence $\pi_0(F(X)) \isoto \pi_0(F(\pi_0(X)))$ for all $X \in \cat{C}_{\ge 0}$.
\end{lemma}

\begin{proof}
  Let $G := \pi_0F : \cat{C}_{\ge 0} \to \cat{C}^\heart$. Then $G$ preserves sifted colimits, hence is the left derived functor of its restriction $G_0 := G|_{\cat{C}^0}$. On the other hand, as we have $\cat{C}^\heart \iso \PSh_\Sigma(\cat{C}^0;\Set)$ (\cref{dc--df--context-determined}), there is a unique sifted colimit--preserving functor $G' : \cat{C}^\heart \to \cat{C}^\heart$ with restriction $G'|_{\cat{C}^0} \iso G_0$, and from this we obtain that $G' \circ \pi_0 : \cat{C}_{\ge 0} \to \cat{C}^\heart$ is also the left derived functor of $G_0$. We thus have $G \iso G'$, and the claim follows.
\end{proof}

\begin{remark}
  \label{dc--df--pizero-adjunction}
  Consider the following two full subcategories of $\End^\Sigma(\cat{C}_{\ge 0})$:
  \begin{itemize}
  \item $\End_0^\Sigma(\cat{C}_{\ge 0})$, spanned by those endofunctors that, in addition to preserving sifted colimits, preserve the subcategory $\cat{C}^0 \subseteq \cat{C}_{\ge 0}$;
  \item $\End_1^\Sigma(\cat{C}_{\ge 0})$, spanned by those endofunctors $F$ that, in addition to preserving sifted colimits, satisfy $\pi_0(F(X)) \in \cat{C}^0$ for all $X \in \cat{C}^0$.
  \end{itemize}
  The former is evidently a monoidal subcategory, and it follows from \cref{dc--df--derived-pizero} that the latter is a monoidal subcategory. We also have an inclusion $\End_0^\Sigma(\cat{C}_{\ge 0}) \subseteq \End_1^\Sigma(\cat{C}_{\ge 0})$.

  This inclusion of $\infty$-categories admits a left adjoint $\tau : \End^\Sigma_1(\cat{C}_{\ge 0}) \to \End_0^\Sigma(\cat{C}_{\ge 0})$: identifying sifted colimit--preserving functors $\cat{C}_{\ge 0} \to \cat{C}_{\ge 0}$ with functors $\cat{C}^0 \to \cat{C}_{\ge 0}$, the left adjoint $\tau$ is given by composition with $\pi_0$. The monoidal structure on the inclusion endows the left adjoint $\tau$ with an oplax monoidal structure, and it follows from \cref{dc--df--derived-pizero} that it is in fact strictly monoidal.
\end{remark}

We now come to the crucial constructions of this subsection. Note that we are forced to work with the filtrations $\Sym^{\le\star}$ and $\LSym^{\le\star}$, and the notion of filtered monads defined in \cref{dc--mo}, because the filtered pieces are excisively polynomial (hence fall under the purview of \cref{dc--df--poly-cn}) but their colimits $\Sym$ and $\LSym$ are not.

\begin{construction}[Derived symmetric powers filtered monad]
  \label{dc--df--lsym-fil}
  \emph{Let $\cat{E} \subseteq \End(\cat{C})$ denote the full subcategory spanned by those endofunctors that are excisively polynomial, preserve sifted colimits, and preserve the full subcategory $\cat{C}_{\ge0} \subseteq \cat{C}$. We will construct a filtered $\cat{E}$-monad $\LSym^{\le\star}_{\cat{C}}$, together with natural equivalences $\LSym^{\le i}_{\cat{C}}|_{\cat{C}^0} \iso \Sym^{\le i}_{\cat{C}^\heart}$ for all $i \ge 0$, as well as a map of filtered $\cat{E}$-monads $\theta^{\le\star} : \Sym^{\le\star}_{\cat{C}} \to \LSym^{\le\star}_{\cat{C}}$, where $\Sym^{\le\star}_{\cat{C}}$ is the symmetric powers filtered monad on $\cat{C}$ (\cref{dc--mo--sym-fil}). We will refer to the filtered monad $\LSym^{\le\star}_{\cat{C}}$ as the \emph{derived symmetric powers filtered monad on $\cat{C}$}.}
  
  Let $\cat{E}' \subseteq \End(\cat{C}_{\ge 0})$ denote the full subcategory spanned by those endofunctors that preserve sifted colimits and whose composition with the inclusion $\cat{C}_{\ge0} \subseteq \cat{C}$ is excisively polynomial; $\cat{E}'$ is a monoidal subcategory, i.e. contains the identity functor and is closed under composition. It follows from \cref{dc--df--poly-cn} that the (monoidal) restriction functor $\cat{E} \to \cat{E}'$ is an equivalence. Noting that this restriction carries $\Sym^{\le\star}_{\cat{C}}$ to $\Sym^{\le\star}_{\cat{C}_{\ge0}}$ (invoking \cref{dc--df--exc-poly-sym}), it will be enough to make the construction with $\cat{C}_{\ge 0}$ in place of $\cat{C}$.
 
  Now, using the notation of \cref{dc--df--pizero-adjunction}, we have $\Sym^{\le i}_{\cat{C}_{\ge 0}} \in \End_1^\Sigma(\cat{C}_{\ge 0})$ for all $i \ge 0$, and we define $\LSym^{\le\star}_{\cat{C}_{\ge 0}} := \tau(\Sym^{\le\star}_{\cat{C}_{\ge 0}})$. The monoidality of $\tau$ implies that $\LSym^{\le\star}_{\cat{C}_{\ge 0}}$ inherits a filtered monad structure from $\Sym^{\le\star}_{\cat{C}_{\ge0}}$, and the description of $\tau$ in \cref{dc--df--pizero-adjunction} gives natural equivalences
  \[
    \LSym^{\le i}_{\cat{C}_{\ge0}}(X) \iso \pi_0(\Sym^{\le i}_{\cat{C}_{\ge0}}(X)) \iso \Sym^{\le i}_{\cat{C}^\heart}(X)
    \quad\text{for}\ X \in \cat{C}^0\ \text{and}\ i \ge 0,
  \]
  as desired. Since the functors $\LSym^{\le i}_{\cat{C}_{\ge0}}$ preserve sifted colimits, they must be the left derived functors of the functors $\Sym^{\le i}_{\cat{C}^\heart}$; by \cref{dc--df--degree-n-excisive} and \cref{dc--df--add-poly-sym}, this implies that their compositions with the inclusion $\cat{C}_{\ge0} \subseteq \cat{C}$ are excisively polynomial. Thus, $\LSym^{\le\star}_{\cat{C}_{\ge0}}$ is indeed a filtered $\cat{E}'$-monad.

  Finally, the unit map of the adjunction defining $\tau$ gives us the desired map of filtered monads $\theta^{\le\star} : \Sym^{\le\star}_{\cat{C}_{\ge0}} \to \LSym^{\le\star}_{\cat{C}_{\ge0}}$, finishing the construction.
\end{construction}

\begin{construction}[Derived symmetric algebra monad]
  \label{dc--df--lsym}
  We let $\LSym_{\cat{C}}$ denote the colimit of the derived symmetric powers filtered monad $\LSym^{\le\star}_{\cat{C}}$ of \cref{dc--df--lsym-fil}. We regard $\LSym_{\cat{C}}$ as a (sifted colimit--preserving)  monad on $\cat{C}$ by \cref{dc--mo--colim-operad}, and refer to it as the \emph{derived symmetric algebra monad on $\cat{C}$}.
\end{construction} 

\begin{remark}
  \label{dc--df--lsym-sum}
  By \cref{dc--df--degree-n-excisive,dc--df--poly-cn}, for each $i \ge 0$, there is a unique sifted colimit--preserving, excisively polynomial functor $\LSym^i_{\cat{C}} : \cat{C} \to \cat{C}$ that restricts to the derived symmetric power functor $\LSym^i_{\cat{C}_{\ge 0}}$ (\cref{dc--df--lsym-connective}) on $\cat{C}_{\ge 0}$. Examining \cref{dc--df--lsym-fil,dc--df--lsym}, we see that the the equivalences $\Sym_{\cat{C}^\heart}^{\le i} \iso \bigoplus_{0 \le j \le i} \Sym_{\cat{C}^\heart}^i$ determine equivalences $\LSym_{\cat{C}}^{\le i} \iso \bigoplus_{0 \le j \le i} \LSym_{\cat{C}}^j$, and hence an equivalence $\LSym_{\cat{C}} \iso \bigoplus_{j \ge 0} \LSym_{\cat{C}}^j$.
\end{remark}

We are now able to make the main definition of this section.

\begin{definition}
  \label{dc--df--dalg}
  A \emph{derived commutative algebra object of $\cat{C}$} is a module over the derived symmetric algebra monad $\LSym_{\cat{C}}$ on $\cat{C}$ (\cref{dc--df--lsym}). We let $\DAlg(\cat{C})$ denote the $\infty$-category $\LMod_{\LSym_{\cat{C}}}(\cat{C})$ of derived commutative algebra objects of $\cat{C}$.
\end{definition}

\begin{remark}
  \label{dc--df--dalg-limits}
  By \cref{dc--mo--alg-presentable}, the $\infty$-category $\DAlg(\cat{C})$ is presentable, in particular admits small limits and colimits.
\end{remark}

\begin{remark}
  \label{dc--df--connective-generators}
  Let $\DAlg(\cat{C})^\cn$ denote the fiber product $\DAlg(\cat{C}) \times_{\cat{C}} \cat{C}_{\ge0}$, and let $\DAlg(\cat{C})^\heart$ denote the fiber product $\DAlg(\cat{C}) \times_{\cat{C}} \cat{C}^\heart$. These $\infty$-categories admit the following alternative descriptions:
  \begin{itemize}
  \item Let $\cat{D}^0$ denote the full subcategory of $\DAlg(\cat{C})^\cn$ spanned by the objects $\LSym_{\cat{C}}(X)$ for $X \in \cat{C}^0$. Then $\DAlg(\cat{C})^\cn$ is projectively generated by the objects of $\cat{D}^0$, so that there is a canonical equivalence $\DAlg(\cat{C})^\cn \iso \PSh_\Sigma(\cat{D}^0)$. This follows from \cite[Corollary 4.7.3.18]{lurie--algebra}.
  \item There is a canonical equivalence $\DAlg(\cat{C})^\heart \iso \CAlg(\cat{C}^\heart)$. To see this, consider the localization $\pi_0 : \cat{C}_{\ge 0} \fromto \cat{C}^\heart : \iota$. Using reasoning similar to that in the proof of \cref{dc--df--lsym-fil}, one can see that the monad $\LSym_{\cat{C}_{\ge 0}}$ on $\cat{C}_{\ge 0}$ satisfies the hypothesis of \cref{dc--mo--localization} for this localization, and that the induced monad on $\cat{C}^\heart$ canonically identifies with the symmetric algebra monad $\Sym_{\cat{C}^\heart}$ on $\cat{C}^\heart$. \cref{dc--mo--localization} then implies the desired equivalence
    \[
      \DAlg(\cat{C})^\heart \iso \LMod_{\Sym_{\cat{C}^\heart}}(\cat{C}^\heart) \iso \CAlg(\cat{C}^\heart).
    \]
  \end{itemize}
\end{remark}

\begin{remark}
  \label{dc--df--dalg-natural}
  The above definition of derived commutative algebra objects is natural in the following sense: a morphism of derived algebraic contexts $F : \cat{C} \to \cat{D}$ determines a colimit-preserving functor $F' : \DAlg(\cat{C}) \to \DAlg(\cat{D})$ such that the diagrams
  \[
    \begin{tikzcd}
      \DAlg(\cat{C}) \ar[r, "F'"] \ar[d, "U_{\cat{C}}", swap] &
      \DAlg(\cat{D}) \ar[d, "U_{\cat{D}}"] \\
      \cat{C} \ar[r, "F"] &
      \cat{D}
    \end{tikzcd}
    \quad\quad\text{and}\quad\quad
    \begin{tikzcd}
      \DAlg(\cat{C}) \ar[r, "F'"] &
      \DAlg(\cat{D})  \\
      \cat{C} \ar[r, "F"] \ar[u, "\LSym_{\cat{C}}"] &
      \cat{D} \ar[u, "\LSym_{\cat{D}}", swap]
    \end{tikzcd}
  \]
  canonically commute (where the vertical arrows in the first diagram are the forgetful functors, and in the second diagram are their left adjoints). Note that the commuting of the second diagram is equivalent to the commuting of the diagram of right adjoints
  \[
    \begin{tikzcd}
      \DAlg(\cat{C}) \ar[d, "U_{\cat{C}}", swap] &
      \DAlg(\cat{D}) \ar[l, "G'", swap] \ar[d, "U_{\cat{D}}"] \\
      \cat{C} &
      \cat{D} \ar[l, "G", swap]
    \end{tikzcd}
  \]
  ($F$ and $F'$ admit right adjoints $G$ and $G'$ by the adjoint functor theorem).

  More generally, similar statements hold when the morphism $F : \cat{C} \to \cat{D}$ is replaced by any diagram of derived algebraic contexts. This general form of naturality can be proved by running through the constructions and results of this section with a bit more care, replacing the single derived algebraic context $\cat{C}$ with any diagram of such.
\end{remark}

We now compare derived commutative algebra objects of $\cat{C}$ with $\Einfty$-algebra objects of $\cat{C}$.

\begin{notation}
  \label{dc--df--dalg-calg}
  Let $\theta : \Sym_{\cat{C}} \to \LSym_{\cat{C}}$ denote the map of monads obtained by taking the colimit of the map of filtered monads $\theta^{\le\star}$ of \cref{dc--df--lsym-fil}, using the equivalence $\Sym_{\cat{C}} \iso \colim(\Sym^{\le\star}_{\cat{C}})$ of \cref{dc--mo--sym-fil}. Then let $\Theta : \DAlg(\cat{C}) \to \CAlg(\cat{C})$ denote the restriction functor induced by $\theta$ on $\infty$-categories of modules, using the canonical equivalence $\CAlg(\cat{C}) \iso \LMod_{\Sym_{\cat{C}}}(\cat{C})$. Following \cite[\S25]{lurie--sag}, for $A \in \DAlg(\cat{C})$, we let $A^\circ$ denote $\Theta(A) \in \CAlg(\cat{C})$ and refer to this as the \emph{underlying $\Einfty$-algebra of $A$}. When less precision is acceptable, we will just write $A$ instead of $A^\circ$.
\end{notation}

\begin{proposition}
  \label{dc--df--dalg-colim}
  The functor $\Theta : \DAlg(\cat{C}) \to \CAlg(\cat{C})$ of \cref{dc--df--dalg-calg} preserves small limits and colimits.
\end{proposition}

\begin{proof}
  Limits and sifted colimits are preserved by $\Theta$ because they are computed in both $\DAlg(\cat{C})$ and $\CAlg(\cat{C})$ at the level of underlying objects in $\cat{C}$. The initial object is preserved by $\Theta$ because $\Sym_{\cat{C}}$ and $\LSym_{\cat{C}}$ preserve initial objects (being left adjoints) and we have $\Sym_{\cat{C}}(0) \iso \unit_{\cat{C}} \iso \LSym_{\cat{C}}$, where $\unit_{\cat{C}}$ denotes the unit object of $\cat{C}$ (the latter equivalence uses the hypothesis that $\unit_{\cat{C}} \in \cat{C}^0$). It now only remains to show that binary coproducts are preserved.

  Recall that the coproduct in $\CAlg(\cat{C})$ is given by the tensor product $\otimes_{\cat{C}}$, and let us temporarily denote the coproduct in $\DAlg(\cat{C})$ by $\boxtimes_{\cat{C}}$. We wish to show that, for $A,B \in \DAlg(\cat{C})$, the canonical map $\Theta(A) \otimes_{\cat{C}} \Theta(B) \to \Theta(A \boxtimes_{\cat{C}} B)$ is an equivalence. Remembering that $\DAlg(\cat{C})$ is the $\infty$-category of algebras for the monad $\LSym_{\cat{C}}$, we have the canonical bar resolutions $A \iso \colim_{[n] \in \Delta^\op} \LSym_{\cat{C}}^{(n+1)}(A)$ and $B \iso \colim_{[n] \in \Delta^\op} \LSym_{\cat{C}}^{(n+1)}(B)$, where $\LSym^{(k)}_{\cat{C}}$ denotes the $k$-fold iterated functor. As $\otimes_{\cat{C}}$ and $\boxtimes_{\cat{C}}$ commute with geometric realizations in each variable and $\Theta$ commutes with geometric realizations, this reduces us to the case that $A = \LSym_{\cat{C}}(X)$ and $B = \LSym_{\cat{C}}(Y)$ for $X,Y \in \cat{C}$.

  In this special case, we have a canonical equivalence of derived commutative algebras
  \[
    \LSym_{\cat{C}}(X) \boxtimes_{\cat{C}} \LSym_{\cat{C}}(Y) \isoto \LSym_{\cat{C}}(X \oplus Y),
  \]
  since $\LSym_{\cat{C}} : \cat{C} \to \DAlg(\cat{C})$ is a left adjoint and hence preserves coproducts. We thus wish to show that the canonical map $\LSym_{\cat{C}}(X) \otimes_{\cat{C}} \LSym_{\cat{C}}(Y) \to \LSym_{\cat{C}}(X \oplus Y)$ is an equivalence. One can show that this map may be naturally obtained as the colimit of a filtered map $\LSym_{\cat{C}}^{\le *}(X) \ostar_{\cat{C}} \LSym_{\cat{C}}^{\le *}(Y) \to \LSym_{\cat{C}}^{\le *}(X \oplus Y)$, where we regard $\LSym_{\cat{C}}^{\le *}$ as a functor $\cat{C} \to \Fun(\num{Z}_{\ge 0},\cat{C})$ and $\ostar_{\cat{C}}$ denotes the Day convolution tensor product in $\Fun(\num{Z}_{\ge 0},\cat{C})$. Since each filtered piece of $\LSym_{\cat{C}}^{\le *}$ is excisively polynomial and preserves sifted colimits, we may reduce first by \cref{dc--df--poly-cn} to the case that $X,Y \in \cat{C}_{\ge 0}$ and then to the case that $X,Y \in \cat{C}^0$. In this last case, the claim is immediate from the construction of $\LSym_{\cat{C}}$.
\end{proof}

\begin{notation}
  \label{dc--df--dalg-relative}
  Let $A$ be a derived commutative algebra object of $\cat{C}$.
  \begin{enumerate}[leftmargin=*]
  \item By an \emph{$A$-module} we mean an $A^\circ$-module object of $\cat{C}$, and we let $\Mod_A(\cat{C})$ denote the $\infty$-category $\Mod_{A^\circ}(\cat{C})$.
  \item A \emph{derived commutative $A$-algebra} is a derived commutative algebra object $B$ of $\cat{C}$ equipped with a map $\phi : A \to B$ in $\DAlg(\cat{C})$. We let $\DAlg_A(\cat{C})$ denote the $\infty$-category $\DAlg(\cat{C})_{A/}$ of derived commutative $A$-algebras, which is presentable (by \cite[Proposition 5.5.3.10]{lurie--topos}, since $\DAlg(\cat{C})$ is presentable).
  \item Let $\CAlg_{A^\circ}(\cat{C})$ denote the $\infty$-category $\CAlg(\cat{C})_{A^\circ/}$ of $\Einfty$-$A$-algebras in $\cat{C}$. The functor $\Theta$ induces a functor $\Theta_A : \DAlg_A(\cat{C}) \to \CAlg_{A^\circ}(\cat{C})$, which also preserves small limits and colimits. From this we obtain a forgetful functor $\DAlg_A(\cat{C}) \to \Mod_A(\cat{C})$ which preserves limits and sifted colimits; this forgetful functor admits a left adjoint $\LSym_A : \Mod_A(\cat{C}) \to \DAlg_A(\cat{C})$ by the adjoint functor theorem, and this adjunction is monadic by the monadicity theorem.
  \item We regard $\DAlg_A(\cat{C})$ as a symmetric monoidal $\infty$-category via the cocartesian symmetric monoidal structure. We denote the coproduct by $\otimes_A$, which is reasonable since the functor $\Theta_A : \DAlg_A(\cat{C}) \to \CAlg_A(\cat{C})$ preserves coproducts, so that, at the level of underlying $A$-modules or $\Einfty$-algebras, $\otimes_A$ is given by the relative tensor product $\otimes_{A^\circ}$. In other words, $\Theta_A$ and the forgetful functor $\DAlg_A(\cat{C}) \to \Mod_A(\cat{C})$ are canonically symmetric monoidal.
  \end{enumerate}
  When it is safe and convenient to do so, we will drop $\cat{C}$ from the notation established above; for example, we may denote $\Mod_A(\cat{C})$ simply by $\Mod_A$ and $\DAlg_A(\cat{C})$ by $\DAlg_A$.
\end{notation}

\begin{remark}
  \label{dc--df--sym-base-change}
  Let $A \to B$ be a map in $\DAlg(\cat{C})$. The resulting restriction functor $\DAlg_B(\cat{C}) \to \DAlg_A(\cat{C})$ has a left adjoint $B \otimes_A - : \DAlg_A(\cat{C}) \to \DAlg_B(\cat{C})$ given by coproduct in $\DAlg_A(\cat{C})$ (or equivalently pushout in $\DAlg(\cat{C})$) with $B$. Moreover, there is a canonical natural equivalence of derived commutative $B$-algebras $B \otimes_A \LSym_A(M) \iso \LSym_B(B \otimes_A M)$ for $M \in \Mod_A(\cat{C})$ (this is immediate from considering the right adjoint functors).
\end{remark}

We close this subsection by discussing the following definition, which was motivated at the beginning of the section.

\begin{definition}
  \label{dc--df--dcbi}
  A \emph{derived commutative} (resp. \emph{derived bicommutative}) \emph{bialgebra object of $\cat{C}$} is a coalgebra (resp. cocommutative coalgebra) object of $\DAlg(\cat{C})$. We set $\bAlg_\dcomm(\cat{C}) := \cAlg(\DAlg(\cat{C}))$ and $\bAlg_\dcomm^\comm(\cat{C}) := \cCAlg(\DAlg(\cat{C}))$, and refer to these as the \emph{$\infty$-category of derived commutative bialgebra objects of $\cat{C}$} and \emph{$\infty$-category of derived bicommutative bialgebra objects of $\cat{C}$} respectively.
\end{definition}

\begin{remark}
  \label{dc--df--dcbi-und}
  A derived commutative (resp. derived bicommutative) bialgebra object $A$ of $\cat{C}$ has an underlying commutative (resp. bicommutative) bialgebra object $A^\circ$ (which again we may just denote by $A$).
\end{remark}

Let us now explain why, given a derived commutative bialgebra object $A$ of $\cat{C}$, we may lift the derived symmetric algebra monad $\LSym_{\cat{C}}$ from $\cat{C}$ to the comodule category $\cLMod_A(\cat{C})$, allowing us to make sense of derived commutative algebra objects in this comodule category. This is analogous to what was discussed in \cref{bi--tn}, namely that the underlying commutative bialgebra structure on $A$ determines a symmetric monoidal structure on $\cLMod_A(\cat{C})$, allowing us to make sense of $\Einfty$-algebra objects in this comodule category.

\begin{construction}
  \label{dc--df--dcbi-comod}
  \emph{Let $A \in \bAlg_\dcomm(\cat{C})$, which we note has an underlying coalgebra object in $\cat{C}$. We construct a monad $\LSym_\cat{C}$ on the $\infty$-category $\cLMod_A(\cat{C})$ such that the diagram}
  \[
    \begin{tikzcd}
      \cLMod_A(\cat{C}) \ar[r, "\LSym_{\cat{C}}"] \ar[d] &
      \cLMod_A(\cat{C}) \ar[d] \\
      \cat{C} \ar[r, "\LSym_{\cat{C}}"] &
      \cat{C}
    \end{tikzcd}
  \]
  \emph{canonically commutes (the vertical arrows being the forgetful functors), together with an equivalence of $\infty$-categories $\LMod_{\LSym_{\cat{C}}}(\cLMod_A(\cat{C})) \iso \cLMod_A(\DAlg(\cat{C}))$.}
  
  Consider the free-forget adjunction $F : \cat{C} \fromto \DAlg(\cat{C}) : G$, whose associated monad is $\LSym_{\cat{C}}$ by definition of $\DAlg(\cat{C})$. Since the right adjoint $G$ is canonically symmetric monoidal, the left adjoint $F$ is canonically oplax symmetric monoidal, and there is an induced adjunction $F' : \cLMod_A(\cat{C}) \fromto \cLMod_A(\DAlg(\cat{C})): G'$. We define the monad $\LSym_{\cat{C}}$ on $\cLMod_A(\cat{C})$ to be the monad associated to this adjunction. The commutativity of the diagram in the claim is then clear, and since forgetful functors from comodule categories preserve colimits, the monadicity theorem implies that the adjunction $F' \dashv G'$ is also monadic, giving the desired equivalence
  $\LMod_{\LSym_{\cat{C}}}(\cLMod_A(\cat{C})) \iso \cLMod_A(\DAlg(\cat{C}))$.
\end{construction}

\begin{variant}
  \label{dc--df--dcbi-dual}
  Let $B$ be a dualizable cocommutative bialgebra object of $\cat{C}$. By \cref{bi--du--main-cor}, there is an induced commutative bialgebra structure on $\dual{B}$ for which we have an equivalence of symmetric monoidal $\infty$-categories $\alpha : \LMod_B(\cat{C}) \iso \cLMod_{\dual{B}}(\cat{C})$, where the symmetric monoidal structures are as defined in \cref{bi--tn}.

  Suppose we are given a lift of the commutative bialgebra structure on $\dual{B}$ to a derived commutative bialgebra structure. Then, using \cref{dc--df--dcbi-comod} and the equivalence $\alpha$, we obtain a derived symmetric algebra monad $\LSym_{\cat{C}}$ on $\LMod_B(\cat{C})$, compatible with that on $\cat{C}$. We set
  \[
    \DAlg(\LMod_B(\cat{C})) := \LMod_{\LSym_{\cat{C}}}(\LMod_B(\cat{C})),
  \]
  and refer to objects of this $\infty$-category as \emph{derived commutative algebra objects of $\LMod_B(\cat{C})$}. Note that we still have an underlying $\Einfty$-algebra functor:
  \[
    \DAlg(\LMod_B(\cat{C})) \iso \cLMod_{\dual{B}}(\DAlg(\cat{C})) \lblto{\Theta} \cLMod_{\dual{B}}(\CAlg(\cat{C})) \iso \CAlg(\LMod_B(\cat{C})),
  \]
  where the first equivalence is from \cref{dc--df--dcbi-comod} and the last equivalence is from \cref{bi--tn--calg}.
\end{variant}

The above discussion reformulated the $\infty$-category $\cLMod_A(\DAlg(\cat{C}))$ as the $\infty$-category $\DAlg(\LMod_B(\cat{C}))$ (where $A$ is a dualizable derived commutative bialgebra and $B$ is its dual). This rephrasing in terms of modules is not strictly necessary for us, but it is somewhat comforting psychologically, and makes the end of the following remark a bit clearer. (Conversely, we should emphasize that it is not clear how to define $\DAlg(\LMod_B(\cat{C}))$ without passing to the dual comodule perspective.)

\begin{remark}
  \label{dc--df--dcbi-natural}
  Let $F : \cat{C} \to \cat{D}$ be a morphism of derived algebraic contexts. Let $B$ be a dualizable cocommutative bialgebra object of $\cat{C}$ equipped with a lift of $\dual{B}$ to $\bAlg_\dcomm(\cat{C})$, as in \cref{dc--df--dcbi-dual}. Then $F(B)$ and $\dual{F(B)} \iso F(\dual{B})$ inherit the same structure in $\cat{D}$, and we have an induced functor $F_B : \LMod_B(\cat{C}) \to \LMod_{F(B)}(\cat{D})$. By \cref{dc--df--dalg-natural}, there is also an induced functor
  \[
    \DAlg(\LMod_B(\cat{C})) \iso \cLMod_{\dual{B}}(\DAlg(\cat{C})) \to
    \cLMod_{\dual{F(B)}}(\DAlg(\cat{D})) \iso \DAlg(\LMod_{\dual{F(B)}}(\cat{D})),
  \]
  which here we denote by $F_B'$, making the diagrams
  \[
    \begin{tikzcd}
      \DAlg(\LMod_B(\cat{C})) \ar[r, "F_B'"] \ar[d, "U_{\cat{C}}", swap] &
      \DAlg(\LMod_{F(B)}(\cat{D})) \ar[d, "U_{\cat{D}}"] \\
      \LMod_B(\cat{C}) \ar[r, "F_B"] &
      \LMod_{F(B)}(\cat{D})
    \end{tikzcd}
  \]
  and
  \[
    \begin{tikzcd}
      \DAlg(\LMod_B(\cat{C})) \ar[r, "F_B'"] &
      \DAlg(\LMod_{F(B)}(\cat{D}))  \\
      \LMod_B(\cat{C}) \ar[r, "F_B"] \ar[u, "\LSym_{\cat{C}}"] &
      \LMod_{F(B)}(\cat{D}) \ar[u, "\LSym_{\cat{D}}", swap]
    \end{tikzcd}
  \]
  commute. Note that we may pass to right adjoints in the second diagram, and since the right adjoint of $F_B$ is induced by the right adjoint $G : \cat{D} \to \cat{C}$ of $F$, this shows that the right adjoint of $F_B'$ is also induced by $G$.
\end{remark}


\subsection{Examples of contexts}
\label{dc--eg}

In this subsection, we catalogue the derived algebraic contexts that will be relevant later in the paper, and establish some abbreviated notation in these specific contexts for the notions defined in \cref{dc--df}.

\begin{example}
  \label{dc--eg--modz}
  The fundamental example of a derived algebraic context is obtained by taking $\cat{C} = \Mod_{\num{Z}}$, which we regard as equipped with the usual symmetric monoidal structure and t-structure, and taking $\cat{C}^0$ to be the full subcategory spanned by the finite free $\num{Z}$-modules. We refer to derived commutative algebra objects of $\Mod_{\num{Z}}$ simply as \emph{derived commutative rings}, let $\DAlg_{\num{Z}}$ denote $\DAlg(\Mod_{\num{Z}})$, and abbreviate the derived symmetric algebra monad $\LSym_{\Mod_{\num{Z}}}$ to $\LSym_{\num{Z}}$.
\end{example}

\begin{remark}
  \label{dc--eg--modz-initial}
  In addition to being the fundamental example, $\Mod_{\num{Z}}$ is the initial example of a derived algebraic context $\Mod_{\num{Z}}$: that is, for any derived algebraic context $\cat{C}$, there is a unique morphism of derived algebraic contexts $\Mod_{\num{Z}} \to \cat{C}$. This follows from \cref{dc--df--context-determined} and the fact that $\cat{C}^0$, being an additive, classical symmetric monoidal category, receives a unique additive symmetric monoidal functor from the category of finite free $\num{Z}$-modules.
\end{remark}

\begin{remark}
  \label{dc--eg--scr}
  It follows from \cref{dc--df--connective-generators} that the $\infty$-category
  \[
    \DAlg_{\num{Z}}^\cn := \DAlg_{\num{Z}} \times_{\Mod_{\num{Z}}} \Mod_{\num{Z}}^\cn
  \]
  of \emph{connective} derived commutative rings is canonically equivalent to the $\infty$-category of simplicial commutative rings, as described for example in \cite[\S25]{lurie--sag}.
\end{remark}

We now explain how to obtain derived algebraic contexts in the graded and filtered settings.

\begin{construction}
  \label{dc--eg--gf}
  Let $\cat{C}$ be a derived algebraic context. Then we regard the $\infty$-category $\Gr(\cat{C})$ of graded objects of $\cat{C}$ as a derived algebraic context by equipping it with the Day convolution symmetric monoidal structure structure (\cref{gf--df--gf}) and the neutral t-structure (\cref{gf--t--gr-t}), and taking $\Gr(\cat{C})^0 \subseteq \Gr(\cat{C})$ to be full subcategory spanned by finite coproducts of the objects $X(n) \iso \ins^n(X)$ for $X \in \cat{C}^0$ and $n \in \num{Z}$. We regard the $\infty$-category $\Fil(\cat{C})$ of filtered objects of $\cat{C}$ as a derived algebraic context in the analogous way, and similarly for the nonnegative variants $\Gr^{\ge 0}(\cat{C})$ and $\Fil^{\ge 0}(\cat{C})$ (\cref{gf--df--gf-nonnegative}).

  We will refer to derived commutative algebra objects of $\Gr(\cat{C})$ as \emph{graded derived commutative algebra objects of $\cat{C}$}, and denote $\DAlg(\Gr(\cat{C}))$ by $\Gr\DAlg(\cat{C})$. Similarly, we refer to derived commutative algebra objects of $\Fil(\cat{C})$ as \emph{filtered derived commutative algebra objects of $\cat{C}$}, and denote $\DAlg(\Fil(\cat{C}))$ by  $\Fil\DAlg(\cat{C})$. We similarly have \emph{nonnegatively} graded and filtered derived commutative objects of $\cat{C}$, the $\infty$-categories of which we denote by $\Gr^{\ge0}\DAlg(\cat{C})$ and $\Fil^{\ge0}\DAlg(\cat{C})$ respectively. In all these cases, we adopt the same notation for the relative notions of \cref{dc--df--dalg-relative}; moreover, as mentioned there, we will sometimes drop $\cat{C}$ from the notation in the relative situation.
\end{construction}

\begin{remark}
  \label{dc--eg--gf-lsym}
  In the situation of \cref{dc--eg--gf}, note that, for $X \in \cat{C}^0$ and $m,n \in \num{Z}$, we have $\Sym^m_{\Gr(\cat{C})^\heart}(X(n)) \iso \Sym^m_{\cat{C}^\heart}(X)(mn)$. It follows that, for all $X \in \cat{C}$, we have $\LSym^m_{\Gr(\cat{C})}(X(n)) \iso \LSym^m_{\cat{C}}(X)(mn)$. The same holds with $\Gr(\cat{C})$ replaced by $\Fil(\cat{C})$.
\end{remark}

\begin{remark}
  \label{dc--eg--gf-morphisms}
  The functor $\ins^0 : \cat{C} \to \Gr(\cat{C})$ is a morphism of derived algebraic contexts, with right adjoint $\ev^0 : \cat{C} \to \Gr(\cat{C})$. It follows from \cref{dc--df--dalg-natural} that there is an induced adjunction $\ins^0 : \DAlg(\cat{C}) \fromto \Gr\DAlg(\cat{C}) : \ev^0$. Informally speaking, this says that the zeroth graded piece of a graded derived commutative algebra object of $\cat{C}$ is canonically a derived commutative algebra object of $\cat{C}$, and conversely any derived commutative algebra object of $\cat{C}$ can be regarded as a graded derived commutative algebra in $\cat{C}$ concentrated in grading-degree $0$. We will often do the latter implicitly.

  Similar statements hold with $\Gr(\cat{C})$ replaced by $\Fil(\cat{C})$ or the nonnegative variants of these, as well as for the adjunction $\ins^{\ge 0} : \Gr^{\ge 0}(\cat{C}) \fromto \Gr(\cat{C}) : \ev^{\ge 0}$, and so on. The same can also be said for the colimit functor $\colim : \Fil(\cat{C}) \to \cat{C}$ and the associated graded functor $\gr : \Fil(\cat{C}) \to \Gr(\cat{C})$, which both determine morphisms of derived algebraic contexts.
\end{remark}

Before discussing the final construction of interest, we must spell out when a localization of a derived algebraic context inherits a derived algebraic context structure.

\begin{construction}
  \label{dc--eg--localization}
  Let $\cat{C}$ be a derived algebraic context and let $\cat{D}$ be a full subcategory of $\cat{C}$ satisfying the following conditions:
  \begin{enumerate}
  \item \label{dc--eg--inclusion--pr}
    $\cat{D}$ is stable and presentable, and the inclusion $\iota : \cat{D} \inj \cat{C}$ admits a left adjoint $\lambda : \cat{C} \to \cat{D}$.
  \item \label{dc--eg--inclusion--t}
    For $X \in \cat{C}$, if $X$ is contained in $\cat{D}$, then $\tau_{\ge 0}(X), \tau_{\le 0}(X)$ are also contained in $\cat{D}$, so that the t-structure on $\cat{C}$ restricts to a t-structure on $\cat{D}$. Moreover, $\lambda$ is right t-exact.
  \item \label{dc--eg--inclusion--mon}
    The left adjoint $\lambda$ is compatible with the symmetric monoidal structure on $\cat{C}$ (in the sense of \cite[Definition 2.2.1.6]{lurie--algebra}). It follows that $\cat{D}$ inherits a presentable symmetric monoidal structure for which $\lambda$ is canonically symmetric monoidal \cite[Proposition 2.2.1.9]{lurie--algebra}.
  \item \label{dc--eg--inclusion--gen}
    Set $\cat{D}^0 := \lambda(\cat{C}^0)$. Then the objects of $\cat{D}^0$ form a set of compact projective generators for $\cat{D}_{\ge 0}$.
  \end{enumerate}
  Then the above structure makes $\cat{D}$ a derived algebraic context and $\lambda$ a morphism of derived algebraic contexts, so that there is an induced localizing adjunction $\lambda : \DAlg(\cat{C}) \fromto \DAlg(\cat{D}) : \iota$.
\end{construction}

\begin{example}
  \label{dc--eg--gr-zeroone}
  Let $\cat{C}$ be a derived algebraic context, and regard $\Gr^{\ge 0}(\cat{C})$ as a derived algebraic context as in \cref{dc--eg--gf}. Let $\Gr^\zeroone(\cat{C})$ denote the $\infty$-category $\cat{C} \times \cat{C}$ of pairs $(X^0,X^1)$ of objects in $\cat{C}$. We have a fully faithful embedding $\ins^\zeroone : \Gr^\zeroone(\cat{C}) \inj \Gr^{\ge 0}(\cat{C})$ identifying $\Gr^\zeroone(\cat{C})$ with the full subcategory of $\Gr^{\ge0}(\cat{C})$ spanned by those graded objects $X^*$ for which $X^n \iso 0$ for all $n \notin \zeroone$. This embedding has a left adjoint given by the functor $\ev^\zeroone : \Gr^{\ge0}(\cat{C}) \to \Gr^\zeroone(\cat{C})$ that sends $X^* \mapsto (X^0,X^1)$. This localization satisfies the conditions of \cref{dc--eg--localization}, hence equips $\Gr^\zeroone(\cat{C})$ with the structure of a derived algebraic context. We note in particular that the tensor product on $\Gr^\zeroone(\cat{C})$ is given by the construction
  \[
    ((X^0,X^1),(Y^0,Y^1)) \mapsto (X^0 \otimes Y^0,(X^0 \otimes Y^1) \oplus (X^1 \otimes Y^0)).
  \]
  We let $\Gr^\zeroone\DAlg(\cat{C})$ denote the $\infty$-category $\DAlg(\Gr^\zeroone(\cat{C}))$, and similarly for the relative situation.
\end{example}


\subsection{The cotangent complex}
\label{dc--ct}

In this subsection, we continue to work with a fixed derived algebraic context $\cat{C}$. Our goal is to give a definition of relative cotangent complexes for derived commutative algebra objects of $\cat{C}$, generalizing and reformulating the usual cotangent complex in the setting of simplicial commutative rings in a manner that will be convenient in \cref{dg--dr}. We follow the usual pattern: we first define trivial square-zero extensions, then define derivations as maps into these, and finally define the cotangent complex as a representing object for derivations.

\begin{notation}
  \label{dc--ct--algmod}
  Let $\CAlgMod(\cat{C})$ denote the $\infty$-category of pairs $(A,M)$ where $A$ is an $\Einfty$-algebra in $\cat{C}$ and $M$ is an $A$-module in $\cat{C}$ (see \cite[Definition 3.3.3.8]{lurie--algebra}), and let $\DAlgMod(\cat{C}) \iso \DAlg(\cat{C}) \times_{\CAlg(\cat{C})} \CAlgMod(\cat{C})$ denote the $\infty$-category of pairs $(A,M)$ where $A$ is a derived commutative ring in $\cat{C}$ and $M$ is a $A$-module in $\cat{C}$.
\end{notation}

In the following two results, we regard $\Gr^\zeroone(\cat{C}) = \cat{C} \times \cat{C}$ as a derived algebraic context as in \cref{dc--eg--gr-zeroone}; we let $\Gr^\zeroone\CAlg(\cat{C})$ denote $\CAlg(\Gr^\zeroone(\cat{C}))$.

\begin{lemma}
  \label{dc--ct--zeroone-calg}
   Then there is a canonical equivalence of $\infty$-categories
  \[
    \alpha : \Gr^\zeroone\CAlg(\cat{C}) \isoto \CAlgMod(\cat{C})
  \]
  commuting with the forgetful functors to $\Gr^\zeroone(\cat{C})$.\footnote{This result goes through for any presentable symmetric monoidal $\infty$-category $\cat{C}$.}
\end{lemma}

\begin{proof}
  The embedding $\ins^\zeroone : \Gr^\zeroone(\cat{C}) \inj \Gr^{\ge0}(\cat{C})$ induces an embedding $\iota : \Gr^\zeroone\CAlg(\cat{C}) \to \Gr^{\ge 0}\CAlg(\cat{C})$. We will define a functor $\beta : \Gr^{\ge 0}\CAlg(\cat{C}) \to \CAlgMod(\cat{C})$ and then define the desired functor $\alpha$ to be the composition $\beta \circ \iota$.

  To define $\beta$, recall that $\Gr^{\ge 0}\CAlg(\cat{C})$ can be identified with the $\infty$-category of lax symmetric monoidal functors $X : \num{Z}_{\ge 0}^\ds \to \cat{C}$, where $\num{Z}_{\ge 0}^\ds$ is regarded as a symmetric monoidal category via addition (\cite[Example 2.2.6.9]{lurie--algebra}). Observe that $0 \in \num{Z}_{\ge 0}^\ds$ has a (unique) commutative algebra structure and $1 \in \num{Z}_{\ge 0}^\ds$ a (unique) module structure over $0$. It follows that a lax symmetric monoidal structure on a functor $X : \num{Z}_{\ge 0}^\ds \to \cat{C}$ determines a commutative algebra structure on $X^0$ and an $X^0$-module structure on $X^1$. This construction defines our functor $\beta : \Gr^{\ge 0}\CAlg(\cat{C}) \to \CAlgMod(\cat{C})$.

  Now consider the composition
  \[
    \alpha := \beta \circ \iota : \Gr^\zeroone\CAlg(\cat{C}) \to \CAlgMod(\cat{C}).
  \]
  This evidently commutes with the forgetful functors to $\Gr^\zeroone(\cat{C})$. We finish by showing that $\alpha$ is an equivalence of $\infty$-categories. Note that both forgetful functors admit left adjoints,
  \begin{align*}
    F = \Sym_{\Gr^\zeroone(\cat{C})} : \Gr^\zeroone(\cat{C}) \to \Gr^\zeroone\CAlg(\cat{C}), \quad
    F' : \Gr^\zeroone(\cat{C}) \to \CAlgMod(\cat{C}),
  \end{align*}
  where $F'(X^0,X^1) \iso (\Sym_{\cat{C}}(X^0), \Sym_{\cat{C}}(X^0) \otimes_{\cat{C}} X^1)$; moreover, both adjunctions are monadic. Unravelling the definition of the symmetric powers $\Sym^n_{\Gr^\zeroone(\cat{C})}$, one can see that the canonical map $F' \to \alpha \circ F$ is an equivalence.
  It then follows from \cite[Corollary 4.7.3.16]{lurie--algebra} that $\alpha$ is an equivalence.
\end{proof}

\begin{lemma}
  \label{dc--ct--zeroone-dalg}
  There is a canonical equivalence of $\infty$-categories
  \[
    \alpha' : \Gr^\zeroone\DAlg(\cat{C}) \isoto \DAlgMod(\cat{C})
  \]
  commuting with the forgetful functors to $\Gr^\zeroone(\cat{C})$.
\end{lemma}

\begin{proof}
  By our definition of modules over derived commutative rings (\cref{dc--df--dalg-relative}), we have an equivalence of $\infty$-categories $\DAlgMod(\cat{C}) \iso \DAlg(\cat{C}) \times_{\CAlg(\cat{C})} \CAlgMod(\cat{C})$. In terms of this equivalence, we define the functor $\alpha'$ via the functor $\Gr^\zeroone\DAlg(\cat{C}) \to \DAlg(\cat{C})$ induced by $\ev^0 : \Gr^\zeroone(\cat{C}) \to \cat{C}$ (see \cref{dc--eg--gf-morphisms}, which applies also with $\Gr(\cat{C})$ replaced by $\Gr^\zeroone(\cat{C})$) and the composition
  \[
    \Gr^\zeroone\DAlg(\cat{C}) \lblto{\Theta} \Gr^\zeroone\CAlg(\cat{C}) \lblto{\alpha} \CAlgMod(\cat{C}),
  \]
  where $\alpha$ is the equivalence of \cref{dc--ct--zeroone-calg}. That $\alpha'$ is an equivalence follows from an argument similar to the one used in \cref{dc--ct--zeroone-calg} to show that $\alpha$ is an equivalence.
\end{proof}

\begin{notation}
  \label{dc--ct--D}
  Let $\dual{\num{D}} \in \Gr\CAlg(\cat{C}^\heart)$ denote the ordinary trivial square-zero commutative algebra $\unit_{\cat{C}} \oplus \unit_{\cat{C}}(-1)$, where $\unit_{\cat{C}}$ denotes the unit object of $\cat{C}$.\footnote{Our choice of notation here is made to be compatible with notation to be introduced in \cref{dg}, namely the objects $\num{D}_{\pm}$ of \cref{dg--pm--d}.} By \cref{dc--df--connective-generators}, we have a canonical embedding $\Gr\CAlg(\cat{C}^\heart) \inj \Gr\DAlg(\cat{C})$, via which we regard $\dual{\num{D}}$ as an object of $\Gr\DAlg(\cat{C})$.
\end{notation}

\begin{construction}[Trivial square-zero extensions]
  \label{dc--ct--square-zero}
  Let $U : \DAlgMod(\cat{C}) \to \DAlg(\cat{C})$ and $U' : \DAlgMod(\cat{C}) \to \cat{C}$ denote the functors given by $U(A,M) := A$ and $U'(A,M) := M$. Define $G : \DAlgMod(\cat{C}) \to \DAlg(\cat{C})$ to be the composite functor
  \[
    \DAlgMod(\cat{C})
    \lblto{(\alpha')^{-1}} \Gr^\zeroone\DAlg(\cat{C})
    \inj \Gr\DAlg(\cat{C})
    \lblto{- \ostar \dual{\num{D}}} \Gr\DAlg(\cat{C})
    \lblto{\ev^0} \DAlg(\cat{C}).
  \]
  Then we have:
  \begin{itemize}
  \item natural transformations $\eta : U \to G$ and $\epsilon : G \to U$, induced by the unit and augmentation maps $\eta_{\dual{\num{D}}} : \unit_{\Gr(\cat{C})} \to \dual{\num{D}}$ and $\epsilon_{\dual{\num{D}}} : \dual{\num{D}} \to \unit_{\Gr(\cat{C})}$ in $\CAlg(\Gr(\cat{C})^\heart) \subseteq \Gr\DAlg(\cat{C})$
  \item a homotopy $\epsilon \circ \eta \iso \id_U$, induced by the equality $\epsilon_{\dual{\num{D}}} \circ \eta_{\dual{\num{D}}} = \id_{\unit_{\Gr(\cat{C})}}$.
  \item a natural equivalence $\fib(\epsilon) \iso U'$, induced by the equivalence $\fib(\epsilon_{\dual{\num{D}}}) \iso \unit_{\cat{C}}(-1)$.
  \end{itemize}
  These data determine a natural equivalence $G(A,M) \iso A \oplus M$ in $\cat{C}$ for $(A,M) \in \DAlgMod(\cat{C})$. We will thus denote $G(A,M) \in \DAlg(\cat{C})$ itself by $A \oplus M$, and refer to it as the \emph{trivial square-zero extension of $A$ by $M$}.
\end{construction}

\begin{notation}
  \label{dc--ct--dalgmod-rel}
  For $A$ a derived commutative algebra object of $\cat{C}$, we let $\DAlgMod_A$ denote the $\infty$-category $\DAlgMod(\cat{C})_{(A,0)/}$ of pairs $(B,M)$ where $B$ is a derived commutative $A$-algebra and $M$ is a $B$-module. Since the functors $U$ and $G$ in \cref{dc--ct--square-zero} both send $(A,0)$ to $A$, they induce functors $\DAlgMod_A \to \DAlg_A$, which we abusively also denote by $U$ and $G$.
\end{notation}

\begin{definition}
  \label{dc--ct--derivation}
  Let $A$ be a derived commutative algebra object $\cat{C}$. For $B$ a derived commutative $A$-algebra and $M$ a $B$-module, we define an \emph{$A$-linear derivation of $B$ into $M$} to be a map $\delta : B \to B \oplus M$ in $(\DAlg_A)_{/B}$. We often abusively identify a derivation $\delta$ with the map of $A$-modules $d : B \to M$ obtained by composing with the projection $B \oplus M \to M$.
\end{definition}

\begin{proposition}
  \label{dc--ct--square-zero-adjoint}
  Let $A$ be a derived commutative algebra object of $\cat{C}$. Then the trivial square-zero extension functor $G : \DAlgMod_A \to \DAlg_A$ preserves limits and sifted colimits, hence admits a left adjoint $L$.
\end{proposition}

\begin{proof}
  It suffices to show that the composite of $G$ with the forgetful functor $\DAlg_A \to \cat{C}$ preserves limits and sifted colimits. But we know from \cref{dc--ct--square-zero} that this composite is given by the construction $(A,M) \mapsto A \oplus M$, so this is clear. It then follows from the adjoint functor theorem that $G$ admits a left adjoint.
\end{proof}

\begin{remark}
  \label{dc--ct--square-zero-adjoint-underlying}
  In the situation of \cref{dc--ct--square-zero-adjoint}, we have a natural equivalence $U \circ L \iso \id_{\DAlg_A}$, where $U : \DAlgMod_A \to \DAlg_A$ is as defined in \cref{dc--ct--square-zero}. To see this, note that $U$ admits a right adjoint $Z : \DAlg_A \to \DAlgMod_A$, given on objects by $Z(B) \iso (B,0)$. The natural transformation $\eta$ of \cref{dc--ct--square-zero} induces an natural equivalence $\id_{\DAlg_A} \isoto G \circ Z$, and passing to left adjoints gives the desired natural equivalence $U \circ L \isoto \id_{\DAlg_A}$.
\end{remark}

\begin{construction}[Cotangent complex]
  \label{dc--ct--ct}
  Let $A$ be a derived commutative algebra object of $\cat{C}$ and let $B$ be a derived commutative $A$-algebra. It follows from \cref{dc--ct--square-zero-adjoint-underlying} that there is a unique $B$-module $\cot_{B/A} \in \Mod_B$ equipped with an equivalence $(B,\cot_{B/A}) \iso L(B)$ in $\DAlgMod_A$, where $L$ is the left adjoint functor of \cref{dc--ct--square-zero-adjoint}. We refer to $\cot_{B/A}$ as the \emph{cotangent complex of $B$ over $A$}. By definition, $\cot_{B/A}$ comes equipped with an $A$-linear derivation $d : B \to \cot_{B/A}$; we refer to $d$ as the \emph{universal $A$-linear derivation of $B$} (or just the \emph{universal derivation}, if context makes the abbreviation appropriate). As the terminology suggests, for any $B$-module $M$, composition with $d$ defines an equivalence between $\Map_{\Mod_B}(\cot_{B/A},M)$ and the space of $A$-linear derivations of $B$ into $M$ (this follows from the definition of $\cot_{B/A}$ in terms of the left adjoint $L$).
\end{construction}

\begin{example}
  \label{dc--ct--free}
  Let $A$ be a derived commutative algebra object of $\cat{C}$, let $M \in \Mod_A$, and let $B := \LSym_A(M) \in \DAlg_A$. Then there is a canonical equivalence of $B$-modules $\cot_{B/A} \iso B \otimes_A M$.
\end{example}

\begin{remark}
  \label{dc--ct--square-zero-cn}
  Let $\DAlgMod(\cat{C})^\cn$ denote the full subcategory of $\DAlgMod(\cat{C})$ spanned by those pairs $(A,M)$ where $A$ and $M$ are connective. This $\infty$-category is projectively generated by the full subcategory $\DAlgMod(\cat{C})^0$ spanned by those pairs $(A,M)$ where $A \iso \LSym_{\cat{C}}(X) \iso \Sym_{\cat{C}^\heart}(X)$ for some $X \in \cat{C}^0$ and $M \iso A^{\oplus n}$ for some $n \ge 0$. Since the trivial square-zero extension functor $G$ of \cref{dc--ct--square-zero} preserves sifted colimits, the restriction $G|_{\DAlgMod(\cat{C})^\cn}$ is the left derived functor of the restriction $G|_{\DAlgMod(\cat{C})^0}$. This latter restriction is valued in the full subcategory $\CAlg(\cat{C}^\heart) \iso \DAlg(\cat{C})^\heart \subseteq \DAlg(\cat{C})$, and it is easy to see that this canonically identifies with the ordinary trivial square-zero extension construction.

  It follows that, in the case $\cat{C} = \Mod_{\num{Z}}$, the notions of trivial square-zero extension, derivation, and cotangent complex agree with the usual notions for simplicial commutative rings (as described in \cite[\S25]{lurie--sag}), under the equivalence $\DAlg_{\num{Z}}^\cn \iso \CAlg_{\num{Z}}^\Delta$ of \cref{dc--eg--scr}.
\end{remark}


\subsection{Connectivity in negative degrees}
\label{dc--cn}

In \cref{dc--df}, we extended the derived symmetric algebra monad $\LSym_{\num{Z}}$ from connective $\num{Z}$-modules to all $\num{Z}$-modules, and thereby defined the notion of possibly nonconnective derived commutative rings. In order to construct the nonconnective derived commutative rings of interest in \cref{dg,fc}, we will need to know a little bit about how this extended monad behaves on nonconnective modules. The purpose of this subsection is to record the relevant facts regarding this behavior, the main results being \cref{dc--cn--gr,dc--cn--fil}.\footnote{One can also analyze the behavior of $\LSym_{\num{Z}}$ on nonconnective modules using \cref{dc--df--excisive-coconnective}, which relates this behavior to the theory of cosimplicial commutative rings. Regarding the latter, see \cite{priddy--sym}.}

\begin{lemma}
  \label{dc--cn--cube}
  Let $\cat{C}$ be a stable $\infty$-category equipped with a t-structure $(\cat{C}_{\ge 0}, \cat{C}_{\le 0})$. Let $\chi : \cube_n \to \cat{C}$ be a cartesian $n$-cube in $\cat{C}$ for some $n \ge 0$ (see \cref{dc--df--cube}). Suppose that there is some integer $m$ such that $\chi(S)$ is $m$-connective for every nonempty $S \in \cube_n$. Then $\chi(\emptyset)$ is $(m-n)$-connective.
\end{lemma}

\begin{proof}
  We proceed by induction on $n$, the case $n=0$ being trivial. For $n \ge 1$, we may identify $\cube_n$ with $\cube_{n-1} \times \Delta^1$ and thereby regard $\chi$ as a natural transformation $\alpha : \chi' \to \chi''$ of $(n-1)$-cubes. By \cite[Lemma 1.2.4.15]{lurie--algebra}, $\chi$ being a cartesian $n$-cube implies that the $\cofib(\alpha)$ is a cartesian $(n-1)$-cube, and the connectivity assumption on $\chi$ implies that $\cofib(\alpha)(T)$ is $m$-connective for all nonempty $T \in \cube_{n-1}$. We deduce from the inductive hypothesis that $\cofib(T)(\emptyset)$ is $(m-n+1)$-connective. It follows that $\chi(\emptyset)$ is the fiber of a map from an $m$-connective object to an $(m-n+1)$-connective object, implying that it must be $(m-n)$-connective, as desired.
\end{proof}

\begin{lemma}
  \label{dc--cn--sym-lsym}
  For $k \le 0$ and $n \ge 0$, the canonical map $\theta^n : \Sym^n_{\num{Z}}(\num{Z}[k]) \to \LSym^n_{\num{Z}}(\num{Z}[k])$ is $(nk+1)$-connective.
\end{lemma}

\begin{proof}
  Fix $n \ge 0$ and let $F_k$ denote the fiber of $\theta^n : \Sym^n_{\num{Z}}(\num{Z}[k]) \to \LSym^n_{\num{Z}}(\num{Z}[k])$ for $k \le 0$. We wish to show that $F_k$ is $(nk+1)$-connective. The case $k=0$ is clear. We may then deduce the claim for $k \le -1$ by (descending) induction, using \cref{dc--cn--cube} and the fact that $\Sym^n_{\num{Z}}$ and $\LSym^n_{\num{Z}}$ are $n$-excisive.
\end{proof}

\begin{lemma}
  \label{dc--cn--gr-preserve}
  The graded derived symmetric algebra monad $\LSym_{\num{Z}} : \Gr(\Mod_{\num{Z}}) \to \Gr(\Mod_{\num{Z}})$ preserves the full subcategory $\Gr^{\le 0}(\Mod_{\num{Z}})_{\ge 0^+} \subseteq \Gr(\Mod_{\num{Z}})$ consisting of those graded modules $X^*$ where $X^n \iso 0$ for $n > 0$ and $X^n$ is $n$-connective for $n \le 0$.
\end{lemma}

\begin{proof}
  The subcategory $\Gr^{\le 0}(\Mod_{\num{Z}})_{\ge 0^+}$ is closed under colimits and tensor products and has compact projective generators given by finite coproducts of the objects $\num{Z}[k](k)$ for $k \le 0$. Since $\LSym_{\num{Z}}$ preserves sifted colimits and sends coproducts to tensor products, it therefore suffices to check that $\LSym(\num{Z}[k](k))$ lies in $\Gr^{\le 0}(\Mod_{\num{Z}})_{\ge 0^+}$ for $k \le 0$.

  By virtue of \cref{dc--df--lsym-sum,dc--eg--gf-lsym}, this amounts to showing that $\LSym_{\num{Z}}^n(\num{Z}[k]) \in \Mod_{\num{Z}}$ is $nk$-connective for each $k \le 0$ and $n \ge 0$. This is true when $\LSym_{\num{Z}}^n$ is replaced by $\Sym_{\num{Z}}^n$, and it then follows from \cref{dc--cn--sym-lsym} that it is true for $\LSym_{\num{Z}}^n$ as well.
\end{proof}

\begin{notation}
  \label{dc--cn--gr-ntn}
  Let
  \[
    [*] : \Gr(\Mod_{\num{Z}}^\heart) \fromto \Gr(\Mod_{\num{Z}})^\heart_+ : [-*]
  \]
  denote the equivalence discussed in \cref{gf--t--gr-heart} (so the functor $[*]$ sends $\{X^n\}_{n \in \num{Z}} \mapsto \{X^n[n]\}_{n \in \num{Z}}$). The fully faithful embedding $\iota : \Gr^{\le 0}(\Mod_{\num{Z}}^\heart) \to \Gr^{\le 0}(\Mod_{\num{Z}})_{\ge 0^+}$ given by the composition
  \[
    \Gr^{\le 0}(\Mod_{\num{Z}}^\heart) \lblto{[*]} \Gr^{\le 0}(\Mod_{\num{Z}})^\heart_+ \inj \Gr^{\le 0}(\Mod_{\num{Z}})_{\ge 0^+}
  \]
  admits a left adjoint $\tau : \Gr^{\le 0}(\Mod_{\num{Z}})_{\ge 0^+} \to \Gr(\Mod_{\num{Z}}^\heart)$, given by the composition
  \[
    \Gr^{\le 0}(\Mod_{\num{Z}})_{\ge 0^+} \lblto{\tau_{\le 0^+}} \Gr(\Mod_{\num{Z}})^\heart_+ \lblto{[-*]} \Gr(\Mod_{\num{Z}}^\heart).
  \]
\end{notation}

\begin{lemma}
  \label{dc--cn--gr-truncate}
  The following diagram of $\infty$-categories canonically commutes:
  \[
    \begin{tikzcd}
      \Gr^{\le 0}(\Mod_{\num{Z}})_{\ge 0^+} \ar[r, "\LSym_{\num{Z}}"] \ar[d, "\tau"] &
      \Gr^{\le 0}(\Mod_{\num{Z}})_{\ge 0^+} \ar[d, "\tau"] \\
      \Gr^{\le0}(\Mod_{\num{Z}}^\heart) \ar[r, "\CSym_{\num{Z}}"] &
      \Gr^{\le0}(\Mod_{\num{Z}}^\heart);
    \end{tikzcd}
  \]
  here the vertical arrows are as defined in \cref{dc--cn--gr-ntn}, the functor $\LSym_{\num{Z}}$ is the restriction of the graded derived symmetric algebra monad provided by \cref{dc--cn--gr-preserve}, and $\CSym_{\num{Z}}$ denotes the ordinary graded symmetric algebra functor (i.e. the symmetric algebra functor for the Koszul symmetric monoidal structure $\Gr(\Mod_{\num{Z}}^\heart)^\koz$).
\end{lemma}

\begin{proof}
  We argue in a fashion similar to the one used to prove \cref{dc--cn--gr-preserve}. For $M \in \Gr^{\le 0}(\Mod_{\num{Z}})_{\ge 0^+}$, the unit map $M \to \LSym_{\num{Z}}(M)$ induces a natural map $\tau(M) \to \tau(\LSym_{\num{Z}}(M))$. Observing that $\tau$ is symmetric monoidal (for the Koszul symmetric monoidal structure on the target), the natural $\Einfty$-algebra structure on $\LSym_{\num{Z}}(M)$ induces a natural graded commutative algebra structure on $\tau(\LSym_{\num{Z}}(M))$, and hence the previous map extends to a natural map $\CSym_{\num{Z}}(\tau(M)) \to \tau(\LSym_{\num{Z}}(M))$. We wish to show that this is an equivalence. The source and target functors preserve sifted colimits and send direct sums to tensor products, reducing us to the case that $M = \num{Z}[k](k)$ for some $k \le 0$. Now, the statement holds with $\LSym_{\num{Z}}$ replaced by $\Sym_{\num{Z}}$, and we conclude by applying \cref{dc--cn--sym-lsym}.
\end{proof}
  
\begin{proposition}
  \label{dc--cn--gr}
  The functors $\iota$ and $\tau$ of \cref{dc--cn--gr-ntn} induced an equivalence between the following two $\infty$-categories:
  \begin{itemize}
  \item the full subcategory of $\Gr\DAlg_{\num{Z}}$ spanned by those graded derived commutative rings $X^*$ such that $X^n \iso 0$ for $n > 0$ and $X^n$ has homotopy concentrated in degree $n$ for $n \le 0$;
  \item the full subcategory of $\CAlg(\Gr(\Mod_{\num{Z}}^\heart)^\koz)$ spanned by those ordinary graded commutative rings $Y^*$ such that $Y^n \iso 0$ for $n > 0$.
  \end{itemize}
\end{proposition}

\begin{proof}
  By \cref{dc--cn--gr-truncate}, we have a natural equivalence $\tau \circ {\LSym_{\num{Z}}} \circ \iota \iso \CSym_{\num{Z}}$. It follows from \cref{dc--mo--localization} that the monad structure on $\LSym_{\num{Z}}$ determines a monad structure on $\CSym_{\num{Z}}$, and it is easy to see that this agrees with the usual graded symmetric algebra monad. Thus, \cref{dc--mo--localization} also tells us that $\tau \dashv \iota$ induces a localizing adjunction on module categories over the monads $\LSym_{\num{Z}}$ and $\CSym_{\num{Z}}$. Unravelling definitions, this gives the claim.
\end{proof}

\cref{dc--cn--gr-preserve} has the following analogue in the filtered setting, which is proved in exactly the same way.

\begin{lemma}
  \label{dc--cn--fil-preserve}
  The filtered derived symmetric algebra monad $\LSym_{\num{Z}} : \Fil(\Mod_{\num{Z}}) \to \Fil(\Mod_{\num{Z}})$ preserves the full subcategory $\Fil^{\le 0}(\Mod_{\num{Z}})_{\ge 0}^\post \subseteq \Fil(\Mod_{\num{Z}})$ consisting of those filtered modules $X^\star$ where $X^n \iso 0$ for $n > 0$ and $X^n$ is $n$-connective for $n \le 0$.
\end{lemma}

And from this we deduce the following analogue of \cref{dc--cn--gr} in the filtered setting.

\begin{proposition}
  \label{dc--cn--fil}
  Let $\Mod_{\num{Z}}^\ccn$ denote the full subcategory of $\Mod_{\num{Z}}$ spanned by the coconnective $\num{Z}$-modules, and let $\DAlg_{\num{Z}}^\ccn$ be defined similarly. Then the Postnikov filtration functor  $\tau_{\ge\star} : \Mod_{\num{Z}} \to \Fil(\Mod_{\num{Z}})$ (\cref{gf--t--postnikov-fil}) induces a fully faithful embedding $\tau_{\ge\star} : \DAlg_{\num{Z}}^\ccn \inj \Fil\DAlg_{\num{Z}}$.
\end{proposition}

\begin{proof}
  Let $\Fil^{\le0}(\Mod_{\num{Z}})$ be the full subcategory of $\Fil(\Mod_{\num{Z}})$ spanned by those filtered modules $X^\star$ such that $X^n \iso 0$ for $n > 0$, and let $\Fil^{\le 0}(\Mod_{\num{Z}})_{\ge 0}^\post$ be as in \cref{dc--cn--fil-preserve}. We then have adjunctions
  \[
    \begin{tikzcd}
      \Fil^{\le0}(\Mod_{\num{Z}})_{\ge0}^\post \ar[r, shift left=0.5ex, hook] &
      \Fil^{\le0}(\Mod_{\num{Z}}) \ar[r, "\colim", shift left=0.5ex] \ar[l, "\tau_{\ge0}^\post", shift left=0.5ex] &
      \Mod_{\num{Z}}. \ar[l, "\ins^0", shift left=0.5ex]
    \end{tikzcd}
  \]

  Let us regard $\Fil^{\le0}(\Mod_{\num{Z}})$ as a derived algebraic context in the same manner as $\Fil(\Mod_{\num{Z}})$ and $\Fil^{\ge0}(\Mod_{\num{Z}})$ (\cref{dc--eg--gf}). Then the colimit functor is a morphism of derived algebraic contexts, so the right hand adjunction induces one on $\infty$-categories of derived commutative algebras (\cref{dc--df--dalg-natural}). By \cref{dc--cn--fil-preserve}, the same is true for the left hand adjunction, in the sense that the derived symmetric algebra monad on $\Fil^{\le0}(\Mod_{\num{Z}})$ restricts to the subcategory $\Fil^{\le0}(\Mod_{\num{Z}})_{\ge0}^\post$, and thus there is an induced adjunction on $\infty$-categories of modules over the monad.

  Composing right adjoints, we deduce that we have an induced functor $\tau_{\ge0}^\post \circ \ins^0 : \DAlg_{\num{Z}} \to \Fil\DAlg_{\num{Z}}$, and since the counit map for the composite adjunction is an equivalence, this functor is fully faithful. To finish, we note that, upon restriction to the coconnective subcategory $\Mod_{\num{Z}}^\ccn$, the composite $\tau_{\ge0}^\post \circ \ins^0$ identifies with the Postnikov filtration functor $\tau_{\ge\star}$.
\end{proof}


\section{Homotopy-coherent cochain complexes}
\label{dg}

As discussed in \cref{in--dR}, one of the main goals of this paper is to formulate generalizations of the notions of cochain complex and (strictly) commutative differential graded algebra to the world of higher algebra, and to use these generalizations to characterize the derived de Rham complex by a universal property. We will accomplish this goal in this section, proving \cref{in--dR--thm}. Before beginning, let us briefly motivate the path we will take.

The data comprising an ordinary cochain complex of abelian groups may be broken into two pieces:
\begin{enumerate}
\item \label{dg--gr--mot--ab} an underlying graded abelian group;
\item \label{dg--gr--mot--d} a degree-one endomorphism of this graded abelian group that squares to zero.
\end{enumerate}
It is straightforward to generalize the first piece to the higher setting: we replace the graded abelian group with, say, a graded spectrum. It is also straightforward to contemplate a degree-one endomorphism of a graded spectrum. However, in the higher setting, it no longer makes good sense to consider the \emph{property} that such a map square to zero. Rather, we must consider the further \emph{data} of a nullhomotopy of the squared map, followed by an infinite cascade of further coherence data.

To encode all this data systematically, we adopt an alternative perspective on \cref{dg--gr--mot--d}: a differential on a graded abelian group $X^*$ is equivalent to a graded module structure over the graded ring $\num{Z}[\epsilon]$ with $\epsilon$ in grading $1$ and $\epsilon^2=0$. As discussed in \cref{bi}, from this perspective, the tensor product of cochain complexes arises from the canonical cocommutative bialgebra structure on $\num{Z}[\epsilon]$. Our strategy will be to locate an appropriate version of this graded bialgebra in the homotopical setting, and then analogously define homotopy-coherent cochain complexes as modules over it.

In fact, we already located a version of this graded bialgebra in \cref{gf--kd}, namely the object $\num{D}_-$ of \cref{gf--kd--dminus} (see \cref{gf--kd--unravel}). This was defined in the context of any stable presentable symmetric monoidal $\infty$-category $\cat{C}$, in particular the universal example $\cat{C} = \Spt$, i.e. ``over the sphere spectrum''. While this object is certainly of interest to us, it is not the end of the story. Recall that a $\num{D}_-$-module structure on a graded object $X^*$ encodes differentials of the form $X^i[-1] \to X^{i+1}$, i.e. which decrease homotopy-degree by $1$; we will refer to these as \emph{\hminus-differentials}. On the other hand, circle actions naturally give rise to (non-graded) differentials of the form $X[1] \to X$, i.e. that increase homotopy-degree by $1$; we will refer to these as \emph{\hplus-differentials}. Thus, in order to make contact with circle actions, we would like a variant bialgebra $\num{D}_+$ whose modules come equipped with \hplus-differentials. It is actually not possible to construct $\num{D}_+$ over $\num{S}$ (as a cocommutative bialgebra), but it is fairly simple to do so over $\num{Z}$, which suffices for the purposes of this paper. (There is also a simpler construction of $\num{D}_-$ over $\num{Z}$.)

Furthermore, it turns out that it is in the context of \hplus-differentials that the universal property of the derived de Rham complex is naturally formulated. (This ``coincidence'' is in some sense the main subject of this paper.) This is related to the fact that $\num{D}_+$ carries more structure than $\num{D}_-$: we will see that the dual $\dual{\num{D}_+}$ is canonically a derived commutative bialgebra, allowing us to formulate the notion of an \emph{\hplus-differential graded derived commutative algebra}, our derived analogue of strictly commutative differential graded algebras.

This section is organized as follows. In \cref{dg--pm}, we will define the basic notions of \hplus-\ and \hminus-differential graded objects, and explain how to shift between the two worlds. In \cref{dg--co}, we will define the ``homotopy-coherent cohomology'' of our homotopy-coherent cochain complexes, and use the Koszul dual perspective on \hminus-differentials from \cref{gf--kd} to understand this. Finally, in \cref{dg--dr}, we define and analyze the derived de Rham complex; this goes through in the generality of any derived algebraic context $\cat{C}$, and we explain how it recovers the existing notion of Hodge-completed derived de Rham cohomology in the case $\cat{C}=\Mod_{\num{Z}}$.


\subsection{H-plus and h-minus differentials}
\label{dg--pm}

We begin by defining the bialgebras over $\num{Z}$ encoding the notions of differential graded objects that we will consider. For this, we recall from \cref{gf--t--gr-heart} that we have identifications of $\Gr(\Mod_{\num{Z}}^\heart)$ with the hearts of the positive and negative t-structures on $\Gr(\Mod_{\num{Z}})$ (\cref{gf--t--gr-t}), determining embeddings $\iota_\pm : \Gr(\Mod_{\num{Z}}^\heart) \inj \Gr(\Mod_{\num{Z}})$, which are canonically lax symmetric monoidal with respect to the Koszul symmetric monoidal structure $\Gr(\Mod_{\num{Z}}^\heart)^\koz$ on the source.

\begin{construction}
  \label{dg--pm--d}
  \emph{We construct a pair of bicocommutative bialgebra objects $\num{D}_+$ and $\num{D}_-$ in $\Gr(\Mod_{\num{Z}})$, with underlying objects given by}
  \[
    \num{D}_\pm \iso \num{Z} \oplus \num{Z}[\pm 1](1).
  \]

  First consider the trivial square-zero algebra $\num{Z}[\epsilon] = \num{Z}[\epsilon]/(\epsilon^2)$, which we regard as a bicocommutative bialgebra in $\Gr(\Mod_{\num{Z}}^\heart)^\koz$ with the element $\epsilon$ in grading-degree $1$ and the comultiplication sending $\epsilon \mapsto \epsilon \otimes 1 - 1 \otimes \epsilon$. We then define
  \[
    \num{D}_\pm := \iota_\pm(\num{Z}[\epsilon]).
  \]
  These have underlying graded objects as written above by definition of $\iota_\pm$, and inherit the bicommutative bialgebra structure from $\num{Z}[\epsilon]$ since the lax symmetric monoidal structures on the functors $\iota_\pm$ are strictly symmetric monoidal when restricted to the full subcategory of $\Gr(\Mod_{\num{Z}}^\heart)$ spanned by those graded abelian groups $X^*$ such that $X^i$ is free for all $i \in \num{Z}$.
\end{construction}

\begin{remark}
  \label{dg--pm--dminus-agree}
  We have now constructed two cocommutative bialgebra objects in $\Gr(\Mod_{\num{Z}})$ named $\num{D}_-$: one just above in \cref{dg--pm--d}, and one earlier in \cref{gf--kd--dminus}. However, there is no ambiguity: there is a unique equivalence between the two. This follows from the fact that the earlier one also has underlying object $\num{Z} \oplus \num{Z}[-1](1)$, hence lies in the essential image of the embedding $\iota_-$, and can thereby be uniquely identified with $\num{Z}[\epsilon]$ in $\Gr(\Mod_{\num{Z}}^\heart)$ (with no possible ambiguity in the bilagebra structure).
\end{remark}

\begin{notation}
  \label{dg--pm--dplus-image}
  If $\cat{C}$ is any $\num{Z}$-linear stable presentable symmetric monoidal $\infty$-category, the structure map $\Mod_{\num{Z}} \to \cat{C}$ (which sends $\num{Z}$ to the unit object of $\cat{C}$ and preserves colimits) induces a symmetric monoidal functor $\Gr(\Mod_{\num{Z}}) \to \Gr(\cat{C})$. We will abusively denote the image of $\num{D}_\pm$ under this functor also by $\num{D}_\pm$.
\end{notation}

With the bialgebras $\num{D}_\pm$ in hand, we may formulate the desired notions of homotopy-coherent cochain complex:

\begin{definition}
  \label{dg--pm--dg}
  Let $\cat{C}$ be a $\num{Z}$-linear stable presentable symmetric monoidal $\infty$-category. We define
  \[
    \DG_+(\cat{C}) \ce \Mod_{\num{D}_+}(\Gr(\cat{C})),
    \qquad
    \DG_-(\cat{C}) \ce \Mod_{\num{D}_-}(\Gr(\cat{C})),
  \]
  which we regard as presentable symmetric monoidal $\infty$-categories using the cocommutative bialgebra structurs on $\num{D}_+$ and $\num{D}_-$ (see \cref{bi--tn--main-cor-op}). We refer to objects of $\DG_+(\cat{C})$ (resp. $\DG_-(\cat{C})$) as \emph{\hplus\ (resp. \hminus) cochain complexes in $\cat{C}$} or \emph{\hplus-differential (resp. \hminus-differential)  graded objects of $\cat{C}$}. We will sometimes use the notation $X^\bul$ to refer to an \hplus\ or \hminus\ cochain complex and then write $X^*$ for its underlying graded object.
\end{definition}

\begin{remark}
  \label{dg--pm--dg-sphere}
  The above definition of \hminus\ cochain complexes does not rely on the $\num{Z}$-linearity assumption on $\cat{C}$. That is, the same definition can be made in general using the construction of $\num{D}_-$ from \cref{gf--kd}, although in this generality we should write $\DG_- := \LMod_{\num{D}_-}(\Gr(\cat{C}))$, as $\num{D}_-$ does not have a commutative structure. With this notation, the Koszul duality result \cref{gf--kd--gf} says that we have $\cpl\Fil(\cat{C}) \iso \DG_-(\cat{C})$.
\end{remark}

The symmetric monoidal structures on $\DG_\pm(\cat{C})$ allow one to consider homotopy-coherent analogues of commutative differential graded algebras, namely commutative algebra objects in these $\infty$-categories. We now explain how, in the $\hplus$ setting, we may formulate a stronger notion, analogous to \emph{strictly} commutative differential graded algebras, using the theory of derived commutative algebras from \cref{dc}.

\begin{notation}
  \label{dg--pm--dplusdual}
  Observe that $\num{D}_+$ is dualizable as an object of $\Gr(\Mod_{\num{Z}})$. We let $\dual{\num{D}_+} \iso \num{Z} \oplus \num{Z}[-1](-1)$ denote its dual, and regard this as a bicommutative bialgebra object in $\Gr(\Mod_{\num{Z}})$ by \cref{bi--du--main-cor}.
\end{notation}

\begin{proposition}
  \label{dg--pm--dplusdual-derived}
  There is a unique derived bicommutative bialgebra structure (\cref{dc--df--dcbi}) on $\dual{\num{D}_+}$ in $\Gr(\Mod_{\num{Z}})$ promoting its bicommutative bialgebra structure.
\end{proposition}

\begin{proof}
  This is an immediate consequence of \cref{dc--cn--gr}.
\end{proof}

\begin{remark}
  \label{dg--pm--dplusdual-derived-addendum}
  It follows from \cref{dg--pm--dplusdual-derived} that, for any derived algebraic context $\cat{C}$, we have a canonical derived bicommutative bialgebra structure on $\dual{\num{D}_+}$ in $\Gr(\cat{C})$. Moving forward, we will regard $\dual{\num{D}_+}$ as equipped with this structure.
\end{remark}

\begin{remark}
  \label{dg--pm--dp-loop}
  Let $\dual{\num{D}}$ denote the unshifted trivial square-zero algebra $\num{Z} \oplus \num{Z}(-1)$, regarded as an object of $\Gr\DAlg_{\num{Z}}$ as in \cref{dc--ct--D}. Ignoring the coalgebra structure on $\dual{\num{D}_+}$, there is a canonical equivalence $\dual{\num{D}_+} \iso \num{Z} \times_{\dual{\num{D}}} \num{Z}$ in $\Gr\DAlg_{\num{Z}}$: the two objects clearly have equivalent underlying graded $\num{Z}$-modules, and it is easy to promote this to an equivalence of graded derived commutative rings because both objects lie in the essential image of $\iota_+$ and hence can be compared in the ordinary category of graded commutative rings.
\end{remark}

The following is the key definition for formulating the universal property of the derived de Rham complex (see \cref{dg--dr}).

\begin{definition}
  \label{dg--pm--dga}
  Let $\cat{C}$ be a derived algebraic context and let $A$ be a derived commutative algebra object of $\cat{C}$. Let $\Gr\DAlg_A$ denote the $\infty$-category of graded derived commutative $A$-algebras (\cref{dc--eg--gf}). We let $\DG_+\DAlg_A$ denote the $\infty$-category $\cMod_{\dual{\num{D}_+}}(\Gr\DAlg_A)$, and refer to objects of this $\infty$-category as \emph{\hplus-differential graded derived commutative $A$-algebras in $\cat{C}$}.
\end{definition}

\begin{remark}
  \label{dg--pm--dga-alt}
  In the context of \cref{dg--pm--dga}, the discussion of \cref{dc--df--dcbi-dual} allows us to rewrite the $\infty$-category $\DG_+\DAlg_A$ as $\DAlg(\DG_+(\Mod_A))$; that is, we may think of \hplus-differential graded derived commutative $A$-algebras as derived commutative algebra objects of the $\infty$-category of \hplus\ cochain complexes of $A$-modules.
\end{remark}

Finally, let us comment on passing between the $\hplus$ and $\hminus$ settings.

\begin{remark}
  \label{dg--pm--shift}
  Recall from \cref{gf--t--shift} that we have symmetric monoidal equivalences $[\pm 2*] : \Gr(\Mod_{\num{Z}}) \to \Gr(\Mod_{\num{Z}})$. It follows from the statement in loc. cit. regarding the hearts that these equivalences send the bicommutative bialgebra $\num{D}_\mp$ to the bicommutative bialgebra $\num{D}_\pm$. This implies that, for $\cat{C}$ a $\num{Z}$-linear presentable symmetric monoidal $\infty$-category, the symmetric monoidal equivalence $[\pm 2*] : \Gr(\cat{C}) \to \Gr(\cat{C})$ induces a symmetric monoidal equivalence $[\pm 2*] : \DG_\mp(\cat{C}) \to \DG_\pm(\cat{C})$.
\end{remark}

\begin{remark}
  \label{dg--pm--t}
  Let $\cat{C}$ be a $\num{Z}$-linear stable presentable symmetric monoidal $\infty$-category equipped with a compatible t-structure $(\cat{C}_{\ge 0},\cat{C}_{\le 0})$. Then the fact that $\num{D}_\pm$ lies in the heart of $\Gr(\Mod_{\num{Z}})_\pm$, corresponding under the identification of the heart with $\Gr(\Mod_{\num{Z}}^\heart)$ (\cref{gf--t--gr-heart}) to the ordinary graded bialgebra $\num{D} = \num{Z}[\epsilon]$, implies that the t-structures of $\Gr(\cat{C})_\pm$ induces t-structures $(\DG_\pm(\cat{C})_{\ge 0}, \DG_\pm(\cat{C})_{\le 0})$ on $\DG_\pm(\cat{C})$ with the following properties:
  \begin{enumerate}
  \item An $\hpm$\ cochain complex $X^* \in \DG_\pm(\cat{C})$ is connective if and only if its underlying graded object is connective in $\Gr(\cat{C})_\pm$, i.e. $X^n \in \cat{C}_{\ge \pm n}$ for all $n \in \num{Z}$.
  \item There is a canonical symmetric monoidal identification of $\DG_\pm(\cat{C})^\heart$ with the category of ordinary cochain complexes in $\cat{C}^\heart$ (induced by the same functors as in the identification between $\Gr(\cat{C})^\heart_\pm$ with $\Gr(\cat{C}^\heart)$, described in \cref{gf--t--gr-heart}).
  \item The equivalence $[\pm 2*] : \DG_\mp(\cat{C}) \to \DG_\pm(\cat{C})$ discussed in \cref{dg--pm--shift} is t-exact.
  \end{enumerate}
  In the $\hminus$ setting, this t-structure was discussed earlier (\cref{gf--t--mod-t}).
\end{remark}


\subsection{Cohomology}
\label{dg--co}

Given an ordinary cochain complex of abelian groups $M$, one is often interested in its \emph{cohomology groups} $\H^i(M) = \ker(d^i : M^i \to M^{i+1})/\im(d^{i-1} : M^{i-1} \to M^i)$. Our goal in this subsection is to study the analogous construction for the homotopy-coherent cochain complexes we introduced in \cref{dg--pm}.

Let us first observe that, for $M$ an ordinary cochain complex as above, we can rewrite $\H^i(M)$ as $\coker(\nu_M^i)$, where $\nu^i_M$ denotes the canonical map $\coker(d^{i-1}) \to \ker(d^i)$. Moreover, if we regard $M$ as a graded module over $\num{D} := \num{Z}[\epsilon]$, then we can reinterpret this as a graded map
\[
  \nu_M : (M \otimes_{\num{D}} \num{Z})(-1) \to \Hom_{\num{D}}(\num{Z},M).
\]
One can check that this map agrees with the norm map $\Nm_M$ defined in \cref{bi--ta--norm} (over the bialgebra $\num{D}$), so that taking its cokernel can be regarded as an analogue in ordinary algebra of the Tate construction. It is thus natural to consider the Tate construction in the homotopy-coherent setting. This will be relevant later on, when we want to understand the associated graded of the Tate construction for filtered circle actions.

For the remainder of the subsection, we let $\cat{C}$ be a $\num{Z}$-linear stable presentable symmetric monoidal $\infty$-category. Note however that it is only the statements about \hplus\ cochain complexes in this subsection that rely on the $\num{Z}$-linearity assumption; all statements and proofs concerning only \hminus\ cochain complexes go through over $\num{S}$.

\begin{remark}
  \label{dg--co--dual}
   The cocommutative bialgebras $\num{D}_\pm$ in $\Gr(\cat{C})$ satisfy the assumptions of \cref{bi--ta}, so we have a norm map and Tate construction for objects of $\LMod_{\num{D}_\pm}(\Gr(\cat{C})) \iso \DG_\pm(\cat{C})$. Indeed, $\num{D}_\pm$ is dualizable in $\Gr(\cat{C})$ and we have a canonical equivalence of $\num{D}_\pm$-modules $\dual{\num{D}_\pm} \iso \num{D}_\pm[\mp 1](-1)$, so that we can take $\omega_{\num{D}_\pm} = \unit[\pm 1](1)$ (in the notation of loc. cit.).
\end{remark}

\begin{definition}
  \label{dg--co--cohomology-groups}
   Let $X \in \DG_\pm(\cat{C})$. We let $\H^*(X) \in \Gr(\cat{C})$ denote the Tate construction $X^{\tate \num{D}_\pm}$ of \cref{bi--ta--norm}, and refer to $\H^*(X)$ as the \emph{cohomology} of $X$.
\end{definition}

\begin{remark}
  \label{dg--co--cohomology-groups-lax}
  By \cref{bi--ta--univ}, the functor $\H^*(-) : \DG_\pm(\cat{C}) \to \Gr(\cat{C})$ is canonically lax symmetric monoidal. In particular, given a commutative algebra in $\DG_\pm(\cat{C})$, its cohomology $\H^*(X)$ is canonically a commutative algebra in $\Gr(\cat{C})$.
\end{remark}

We now seek to understand this notion of cohomology. Recall that the cohomology of an ordinary cochain complex of abelian groups $M$ depends only on the object it represents in the derived category, or equivalently the associated Eilenberg-MacLane object $|M| \in \Mod_{\num{Z}}$; namely, we have $\H^i(M) \iso \pi_{-i}(|M|)$. Analogous to the homotopy type of a topological space, let us refer to $|M|$ as the \emph{cohomology type} of $M$. We will understand the cohomology of homotopy-coherent cochain complexes $X$ by observing an analogous phenomenon: that is, there is an underlying ``cohomology type'' $|X|$ that contains less information than $X$ itself but completely determines the cohomology $\H^*(X)$.

\begin{definition}
  \label{dg--co--cohomology-type}
  Let $X \in \DG_-(\cat{C})$. We let $|X|^{\ge *} \in \cpl\Fil(\cat{C})$ denote the image of $X$ under the inverse to the equivalence $\o\gr : \cpl\Fil(\cat{C}) \isoto \DG_-(\cat{C})$ of \cref{gf--kd--gf}, and we define $|X| := \colim(|X|^{\ge *}) \in \cat{C}$. We then translate these definitions to objects of $\DG_+(\cat{C})$ using the equivalence $[-2*] : \DG_+(\cat{C}) \iso \DG_-(\cat{C})$ (\cref{dg--pm--shift}): that is, for $X \in \DG_+(\cat{C})$, we define $|X|^{\ge \star} := |X[-2*]|^{\ge \star} \in \cpl\Fil(\cat{C})$ and $|X| := |X[-2*]| \in \cat{C}$. In both cases, we refer to $|X|$ as the \emph{cohomology type} of $X$ and $|X|^{\ge \star}$ as the \emph{brutal filtration} on the cohomology type of $X$.
\end{definition}

\begin{remark}
  \label{dg--co--cohomology-type-recall}
  The constructions in \cref{dg--co--cohomology-type} are a generalization of extracting from an ordinary cochain complex the object it represents in the derived category and its brutal filtration (see \cref{gf--t--beilinson-heart,gf--t--beilinson-eg}).
\end{remark}

We can now state the main result of this subsection.

\begin{proposition}
  \label{dg--co--tate}
  Let $\delta_\gr : \cat{C} \to \Gr(\cat{C})$ denote the diagonal functor. Then:
  \begin{enumerate}
  \item \label{dg--co--tate--minus} For $X \in \DG_-(\cat{C})$, there is a natural equivalence $\H^*(X) \iso \delta_\gr(|X|)$.
  \item \label{dg--co--tate--plus} For $X \in \DG_+(\cat{C})$, there is a natural equivalence $\H^*(X) \iso \delta_\gr(|X|)[2*]$, where $[2*] : \Gr(\cat{C}) \to \Gr(\cat{C})$ is as defined in \cref{gf--t--shift}.
  \end{enumerate}
\end{proposition}

It will be useful to set some notation and isolate some steps of the proof into a lemma, as we will need to refer to them again in \cref{fc--ta}.

\begin{notation}
  \label{dg--co--other-brutal}
  Let $X \in \DG_\pm(\cat{C})$. Let $\delta_\fil : \cat{C} \to \Fil(\cat{C})$ denote the diagonal functor. We have a canonical map $|X|^{\ge\star}(-1) \to \delta_\fil(|X|)$ in $\Fil(\cat{C})$, and we let $|X|^{\le\star} \in \Fil(\cat{C})$ denote the cofiber of this map, so that $|X|^{\le i} \iso \cofib(|X|^{\ge i+1} \to |X|)$.
\end{notation}

\begin{lemma}
  \label{dg--co--orbits-fixed}
  For $X \in \DG_-(\cat{C})$, there are natural equivalences
  \[
    X^{\num{D}_-} \iso \und(|X|^{\ge \star})
    \quad\text{and}\quad
    X_{\num{D}_-} \iso \und(|X|^{\le \star}),
  \]
  where $X_{\num{D}_-}, X^{\num{D}_-} \in \Gr(\cat{C})$ are as defined in \cref{bi--ta--orbits-fixed} and $\und : \Fil(\cat{C}) \to \Gr(\cat{C})$ as in \cref{gf--df--gf}\cref{gf--df--gf--und}.
\end{lemma}

\begin{proof}
  The functors $(-)_{\num{D}_-}$ and $(-)^{\num{D}_-}$ are defined as the left and right adjoints to the restriction functor $\rho : \Gr(\cat{C}) \to \LMod_{\num{D}_-}(\Gr(\cat{C})) \iso \DG_-(\cat{C})$. It follows from \cref{gf--kd--mod-natural} that the diagram
  \[
    \begin{tikzcd}
      \Gr(\cat{C}) \ar[dr, "\rho"] \ar[d, "{\unit^\gr[t] \oast -}", swap] &
      \\
      \Mod_{\unit^\gr[t]}(\Gr(\cat{C})) \ar[r, "F"] &
      \LMod_{\num{D}_-}(\Gr(\cat{C}))
    \end{tikzcd}
  \]
  commutes, where $F$ is the symmetric monoidal functor of \cref{gf--kd--mod} (as applied in the proof of \cref{gf--kd--gf}). Under the equivalence $\Fil(\cat{C}) \iso \Mod_{\unit^\gr[t]}(\Gr(\cat{C}))$ of \cref{gf--kd--und}, the forgetful functor $\Mod_{\unit^\gr[t]}(\Gr(\cat{C})) \to \Gr(\cat{C})$ identifies with the functor $\und : \Fil(\cat{C}) \to \Gr(\cat{C})$, so the base change functor $\unit^\gr[t] \oast - : \Gr(\cat{C}) \to \Mod_{\unit^\gr[t]}(\Gr(\cat{C}))$ identifies with the left adjoint $\spl : \Gr(\cat{C}) \to \Fil(\cat{C})$. We thus obtain a commutative diagram
  \[
    \begin{tikzcd}
      \Gr(\cat{C}) \ar[drr, "\rho"] \ar[d, "\spl", swap] &
      \\
      \Fil(\cat{C}) \ar[r, "\cpl{(-)}", swap]  &
      \cpl\Fil(\cat{C}) \ar[r, "\o\gr", swap] &
      \DG_-(\cat{C}),
    \end{tikzcd}
  \]
  where the bottom row is as in \cref{gf--kd--gf}. Passing to right adjoints, we deduce that $X^{\num{D}_-} \iso \und(|X|^{\ge \star})$ for $X \in \DG_-(\cat{C})$, as desired.

  We now address $X_{\num{D}_-}$. Fix any $i \in \num{Z}$ and $Y \in \cat{C}$. We have
  \begin{align*}
    \Map_\cat{C}(X_{\num{D}_-}^i,Y)
    &\iso \Map_{\Gr(\cat{C})}(X_{\num{D}_-},\ins^i_\gr(Y)) \\
    &\iso \Map_{\DG_-(\cat{C})}(X,\rho(\ins^i_\gr(Y))) \\
    &\iso \Map_{\Fil(\cat{C})}(|X|^{\ge\star},\cpl{\spl(\ins^i_\gr(Y))}) \\
    &\iso \Map_{\Fil(\cat{C})}(|X|^{\ge\star},\ins^i_\fil(Y)),
  \end{align*}
  where the penultimate equivalence comes from the commutative diagram above, and, for the sake of clarity, we are using $\ins^i_\gr$ and $\ins^i_\fil$ rather than just $\ins^i$ to denote the respective insertion functors $\cat{C} \to \Gr(\cat{C})$ and $\cat{C} \to \Fil(\cat{C})$, which we recall are given by left Kan extension along the inclusions $\{i\} \inj \num{Z}^\ds$ and $\{i\} \inj \num{Z}^\op$. We now use the canonical cofiber sequence
  \[
    \ins^i_\fil(Y) \to \delta_\fil(Y) \to \coins^{i+1}_\fil(Y),
  \]
  where $\coins^{i+1}_\fil : \cat{C} \to \Fil(\cat{C})$ does not refer to currency but rather denotes \emph{right} Kan extension along the inclusion $\{i+1\} \inj \num{Z}^\op$, described concretely by the formula
  \[
    \coins^{i+1}_\fil(Y)^j \iso
    \begin{cases}
      Y & j \ge i+1 \\
      0 & \text{otherwise}.
    \end{cases}
  \]
  Combining this cofiber sequence with the above equivalences, we find
  \begin{align*}
    \Map_\cat{C}(X_{\num{D}_-}^i,Y)
    &\iso \fib(\Map_{\Fil(\cat{C})}(|X|^{\ge\star},\delta_\fil(Y)) \to \Map_{\Fil(\cat{C})}(|X|^{\ge\star},\coins^{i+1}_\fil(Y))) \\
    &\iso \fib(\Map_{\cat{C}}(|X|,Y) \to \Map_{\cat{C}}(|X|^{\ge i+1},Y)).
  \end{align*}
  This gives us the desired equivalence $X_{\num{D}_-}^i \iso \cofib(|X|^{\ge i+1} \to |X|) \iso |X|^{\le i}$, finishing the proof.
\end{proof}

\begin{proof}[Proof of \cref{dg--co--tate}]
  First note that \cref{dg--co--tate--plus} follows from \cref{dg--co--tate--minus}, since the Tate construction transports across the equivalence $[2*] : \DG_-(\cat{C}) \isoto \DG_+(\cat{C})$, by \cref{bi--ta--natural}. Now let us prove \cref{dg--co--tate--minus}. Let $X \in \DG_-(\cat{C})$. The Tate construction $\H^*(X) = X^{\tate \num{D}_-}$ is given by the cofiber of the norm map $\Nm_X : \omega_{\num{D}_-} \oast X_{\num{D}_-} \to X^{\num{D}_-}$. Recall from \cref{dg--co--dual} that $\omega_{\num{D}_-} \iso \unit[-1](1)$. We claim that, under the equivalences $X_{\num{D}_-} \iso \und(|X|^{\le\star})$ and $X^{\num{D}_-} \iso \und(|X|^{\ge\star})$ of \cref{dg--co--orbits-fixed}, $\Nm_X$ is given by applying $\und : \Fil(\cat{C}) \to \Gr(\cat{C})$ to the canonical map
  \[
    |X|^{\ge\star}[-1](1) \iso \fib\l(|X|^{\ge\star} \to \delta_\fil(|X|)\r) \to |X|^{\ge\star}.
  \]
  Note that the cofiber of this map is canonically equivalent to $\und(\delta_\fil(|X|)) \iso \delta_\gr(|X|)$, as desired.

  To prove the claim made above, it suffices, by \cref{bi--ta--univ}, to prove that the cohomology type functor $|-| : \DG_-(\cat{C}) \to \cat{C}$ vanishes on induced $\num{D}_-$-modules in $\Gr(\cat{C})$. In view of the equivalence $\num{D}_- \iso \omega_{\num{D_-}} \oast \dual{\num{D}_-}$, it furthermore suffices to show that $|\dual{\num{D}_-} \oast Y| \iso 0$ for any $Y \in \Gr(\cat{C})$. We have a commutative diagram
  \[
    \begin{tikzcd}
      \DG_-(\cat{C}) \ar[dr, "U", swap] &
      &
      \cpl\Fil(\cat{C}) \ar[ll, "\o\gr", swap] \ar[dl, "\gr"] \\
      &
      \Gr(\cat{C})
    \end{tikzcd}
  \]
  where $U$ denotes the forgetful functor. Passing to right adjoints, we obtain a commutative diagram
  \[
    \begin{tikzcd}
      \DG_-(\cat{C}) \ar[rr, "|-|^{\ge\star}"] &
      &
      \cpl\Fil(\cat{C})  \\
      &
      \Gr(\cat{C}), \ar[ul, "\dual{\num{D}_-} \oast -"]  \ar[ur, "\zeta", swap]
    \end{tikzcd}
  \]
  where $\zeta$ sends a graded object $Y^*$ to the filtered object
  \[
    \cdots \to Y^2 \lblto{0} Y^1 \lblto{0} Y^0 \lblto{0} Y^{-1} \lblto{0} Y^{-2} \to \cdots
  \]
  (\cref{gf--kd--zeta}). Noting that the colimit of such a diagram with zero transitition maps is zero, we conclude, as desired, that
  \[
    |\dual{\num{D}_-} \oast Y| = \colim(|\dual{\num{D}_-} \oast Y|^{\ge\star}) \iso \colim(\zeta(Y)) \iso 0. \qedhere
  \]
\end{proof}

\begin{remark}
  \label{dg--co--compare-lax}
  The equivalence $(-)^{\num{D}_-} \iso \und(|-|^{\ge\star})$ constructed in the proof of \cref{dg--co--orbits-fixed} is one of lax symmetric monoidal functors (as it was obtained by passing to right adjoints from an equivalence of symmetric monoidal functors). It follows from the uniqueness statement for the lax symmetric monoidal structure on the Tate construction (\cref{bi--ta--univ}) that the equivalences of \cref{dg--co--tate} also respect lax symmetric monoidal structures.
\end{remark}


\subsection{The derived de Rham complex}
\label{dg--dr}

We now formulate the universal property of the derived de Rham complex. Throughout this subsection, we work over a fixed derived commutative algebra $A$ object of a fixed derived algebraic context $\cat{C}$. The reader should feel free to imagine $\cat{C} = \Mod_{\num{Z}}$ to fix ideas.

\begin{notation}
  \label{dg--dr--nonneg-dga}
  We let $\DG_+^{\ge0}\DAlg_A$ denote the fiber product $\DG_+\DAlg_A \times_{\Gr\DAlg_A} \Gr^{\ge0}\DAlg_A$, i.e. the $\infty$-category of \emph{nonnegative \hplus-differential graded derived commutative $A$-algebras} (see \cref{dg--pm--dga,dc--eg--gf} for the definitions of the constituent terms of the fiber product).
\end{notation}

\begin{proposition}
  \label{dg--dr--adj}
  The composite of the forgetful functor $U : \DG_+^{\ge0}\DAlg_A \to \Gr^{\ge0}\DAlg_A$ with the evaluation functor $\ev^0 : \Gr^{\ge0}\DAlg_A \to \DAlg_A$ admits a left adjoint.
\end{proposition}

\begin{proof}
  Recalling from \cref{dc--eg--gf-morphisms} that $\ev^0$ admits a left adjoint, namely $\ins^0$, we need only show that $U$ admits a left adjoint. We will show that $U$ preserves limits and colimits, which suffices by the adjoint functor theorem. Since the embedding $\Gr^{\ge0}\DAlg_A \inj \Gr\DAlg_A$ preserves limits and colimits, we may instead show that the forgetful functor $\DG_+\DAlg_A \to \Gr\DAlg_A$ preserves limits and colimits. It preserves colimits because it is a forgetful functor from a comodule category. For limits, note that the forgetful functors $\DG_+\DAlg_A \to \DG_+(\Mod_A)$ and $\Gr\DAlg_A \to \Gr(\Mod_A)$ preserve limits and are conservative, so that it suffices to check that the forgetful functor $U' : \DG_+(\Mod_A) \to \Gr(\Mod_A)$ preserves limits. This then follows from (the categorical dual of) \cite[Corollary 4.2.3.5]{lurie--algebra}, as $\dual{\num{D}}_+$ is dualizable and hence tensoring with $\dual{\num{D}}_+$ preserves limits in $\Gr(\Mod_A)$.
\end{proof}

\begin{definition}
  \label{dg--dr--ntn}
  We shall denote the left adjoint functor given by \cref{dg--dr--adj} by
  \[
    \LOmega^{+\bul}_{-/A} : \DAlg_A \to \DG_+^{\ge0}\DAlg_A.
  \]
  That is, given a derived $A$-algebra $B$, its image under this functor is denoted $\LOmega^{+\bul}_{B/A}$; we refer to this object as the \emph{derived de Rham complex of $B$ over $A$}.
\end{definition}

\begin{remark}
  \label{dg--dr--graded-linearity}
  Let $B$ be a derived commutative $A$-algebra. Consider $\LOmega^{+*}_{B/A}$, the image of the derived de Rham complex $\LOmega^{+\bul}_{B/A}$ under the forgetful functor $\DG_+\DAlg_A \to \Gr\DAlg_A$. By definition of $\LOmega^{+\bul}_{B/A}$, we have a canonical map $B \to \LOmega^{+0}_{B/A}$ in $\DAlg_A$. By the adjunction $\ins^0 : \DAlg_A \fromto \Gr\DAlg_A : \ev^0$, this induces a map $B \to \LOmega^{+*}_{B/A}$ in $\Gr\DAlg_A$. This promotes $\LOmega^{+*}_{B/A}$ to an object of $\Gr\DAlg_B$, and we will generally regard it as such.
\end{remark}

We now come to a key result (\cref{dg--dr--und-gr} below): it tells us that the universal object we have defined to be the derived de Rham complex deserves this name, i.e. looks the way we expect it to.

\begin{notation}
  \label{dg--dr--ext}
  Let $B$ be a derived commutative $A$-algebra and let $M$ be a $B$-module. For $i \ge 0$, we define
  \[
    \textstyle{\bigwedge^i_B M := \LSym^i_B(M[1])[-i]}
  \]
  and refer to this as the \emph{$i$-th derived exterior power of $M$ over $B$}. We recall from \cite[Proposition 25.2.4.2]{lurie--sag} that, for $B$ and $M$ connective, this agrees with the left derived functor of the classical exterior power functor. In particular, if $B$ is an ordinary commutative ring and $M$ is a projective $B$-module, then $\bigwedge^i_B M$ recovers the classical exterior power construction.

  This extension of derived exterior powers to possibly nonconnective $B$ and $M$ could equivalently be done by performing the extension procedure employed in \cref{dc--df} directly to exterior powers, rather than defining them in terms of derived symmetric powers.
\end{notation}

\begin{theorem}
  \label{dg--dr--und-gr}
  For $B$ a derived commutative $A$-algebra, there is a canonical equivalence of graded derived commutative $B$-algebras
  \[
    \LOmega^{+*}_{B/A} \iso \LSym_B(\cot_{B/A}[1](1)),
  \]
  where $\cot_{B/A}$ denotes the cotangent complex of $B$ over $A$ (\cref{dc--ct--ct}). In particular, there are canonical equivalences of $B$-modules $\LOmega^{+i}_{B/A} \iso (\bigwedge^i_B \cot_{B/A})[i]$ for $i \ge 0$. Furthermore, under these equivalences, the first differential
  \[
    B \iso \LOmega^{+0}_{B/A} \to \LOmega^{+1}_{B/A}[-1] \iso \cot_{B/A}
  \]
  of the $\hplus$ cochain complex $\LOmega^{+\bul}_{B/A}$ is given by the universal $A$-linear derivation of $B$.
\end{theorem}

\begin{remark}
  \label{dg--dr--ff}
  \cref{dg--dr--und-gr} in particular tells us that, for $B \in \DAlg_A$, the unit map $B \to \ev^0(\LOmega^{+\bul}_{B/A})$ is an equivalence of derived commutative $A$-algebras. This implies that the derived de Rham complex functor $\LOmega^{+\bul}_{-/A} : \DAlg_A \to \DG_+^{\ge0}\DAlg_A$ is fully faithful.\footnote{I thank Ben Antieau for suggesting this remark.}
\end{remark}

Before proving \cref{dg--dr--und-gr}, let us see how we may apply it to recover Hodge-completed derived de Rham cohomology from our derived de Rham complex. We first recall the definition of the former (going back to work of Illusie \cite{illusie--cotangent-ii}, and studied in particular by Bhatt \cite{bhatt--completed-ddr}).

\begin{recollection}
  \label{dg--dr--classical}
  Suppose that $\cat{C} = \Mod_{\num{Z}}$ and that $A$ is an ordinary commutative ring. Let $\cpl{\Fil}\CAlg_A$ denote the full subcategory of $\Fil\CAlg_A$ (the $\infty$-category of filtered $\Einfty$-$A$-algebras) spanned by the complete filtered objects. We recall the definition of the functor
  \[
    \cplge\dR\star_{-/A} : \DAlg_A^\cn \to \cpl{\Fil}\CAlg_A
  \]
  taking a connective derived commutative $A$-algebra (equivalently simplicial commutative $A$-algebra) to its \emph{Hodge-filtered Hodge-completed derived de Rham cohomology}.

  Recall that $\DAlg_A^\cn \iso \CAlg_A^\Delta$ is projectively generated by the full subcategory $\Poly_A$ spanned by the finitely generated polynomial $A$-algebras (\cref{dc--df--connective-generators}). We have a functor
  \[
    \Omega^\bul_{-/A} : \Poly_A \to \CAlg(\DG(\Mod_A^\heart))
  \]
  assigning to a finitely generated polynomial $A$-algebra $B$ the algebraic de Rham complex $\Omega^\bul_{B/A}$, regarded as an ordinary commutative differential graded $A$-algebra. Composing with the functor $|-|^{\ge\star} : \DG(\Mod_A^\heart) \to \cpl\Fil(\Mod_A)$ of \cref{gf--t--beilinson-heart} (which is lax symmetric monoidal, hence induces a functor on commutative algebra objects), we obtain a functor $F = |\Omega^\bul_{-/A}|^{\ge\star} : \Poly_A \to \cpl{\Fil}\CAlg_A$. The functor $\cplge\dR\star_{-/A} : \DAlg_A^\cn \to \cpl{\Fil}\CAlg_A$ is defined to be the left derived functor of $F$: that is, $\cplge\dR\star_{-/A}$ is the unique such functor preserving sifted colimits and restricting to $F$ on $\Poly_A$.

  \emph{Hodge-completed derived de Rham cohomology} is then defined by taking colimits: for $B \in \DAlg_A^\cn$, we set $\cpl\dR_{B/A} := \colim(\cplge\dR\star_{B/A}) \in \CAlg_A$.
\end{recollection}

In the following statement, we apply the cohomology type $|-|$ and brutal filtration $|-|^{\ge\star}$ constructions of \cref{dg--co--cohomology-type} to the derived de Rham complex.

\begin{corollary}
  \label{dg--dr--classical-compare}
  Suppose that $\cat{C} = \Mod_{\num{Z}}$ and that $A$ is an ordinary commuative ring. Then there is a canonical natural equivalence between the functors $|\LOmega^{+\bul}_{-/A}|^{\ge\star}, \cplge\dR\star_{-/A} : \DAlg_A^\cn \to \cpl{\Fil}\CAlg_A$, and hence between their colimits $|\LOmega^{+\bul}_{-/A}|, \cpl\dR_{-/A} : \DAlg_A^\cn \to \CAlg_A$.
\end{corollary}

\begin{proof}
  Observe that $|\LOmega^{+\bul}_{-/A}|^{\ge\star}$ preserves sifted colimits, in fact all colimits, as it is the composition of the colimit-preserving functors
  \[
    \DAlg_A \lblto{\LOmega^{+\bul}_{-/A}} \DG_+\DAlg_A \to \DG_+\CAlg_A \lblto{|-|^{\ge\star}} \cpl\Fil\CAlg_A
  \]
  (where $\DG_+\CAlg_A$ denotes $\CAlg(\DG_+(\Mod_A))$ and the unlabelled arrow is the forgetful functor). It therefore suffices to construct such an equivalence on the restriction of these functors to $\Poly_A \subseteq \DAlg_A^\cn$. Passing through the (symmetric monoidal) equivalence $\cpl\Fil(\Mod_A) \iso \DG_-(\Mod_A)$ of \cref{gf--kd--gf}, it suffices to produce a natural equivalence between the compositions
  \[
    \Poly_A \lblto{\LOmega^{+\bul}_{-/A}} \CAlg(\DG_+(\Mod_A)) \lblto{[-2*]} \CAlg(\DG_-(\Mod_A)),
  \]
  \[
    \Poly_A \lblto{\Omega^\bul_{-/A}} \CAlg(\DG(\Mod_A^\heart)) \lblto{[-*]} \CAlg(\DG_-(\Mod_A)),
  \]
  where the functor $[-2*]$ is the induced by the equivalence $[-2*] : \DG_+(\Mod_A) \isoto \DG_-(\Mod_A)$ of \cref{dg--pm--shift}), and the functor $[-*]$ is induced by the fully faithful embedding $[-*] : \DG(\Mod_A^\heart) \inj \DG_-(\Mod_A)$ whose essential image is the heart of the t-structure on the target (\cref{dg--pm--t}). Now, since $\cot_{B/A}$ is a free $A$-module for $B \in \Poly_A$ (and $A$ is discrete), the equivalences $\LOmega^{+i}_{B/A} \iso (\bigwedge^i_B \cot_{B/A})[i]$ of \cref{dg--dr--und-gr} imply that the first composition factors through the latter embedding via the composition
  \[
    \Poly_A \lblto{\LOmega^{+\bul}_{-/A}} \CAlg(\DG_+(\Mod_A)) \lblto{[-*]} \CAlg(\DG(\Mod_A^\heart)).
  \]
  It remains to produce a natural equivalence between this last composite and $\Omega^\bul_{-/A}$. But this is immediate from the identification of $\LOmega^{+*}_{B/A}$ and the first differential of $\LOmega^{+\bul}_{B/A}$ in \cref{dg--dr--und-gr}.
\end{proof}

\begin{remark}
  \label{dg--dr--cohomology}
  In light of \cref{dg--dr--classical-compare}, if $\cat{C}$ is any derived algebraic context, $A$ is any derived commutative algebra object of $\cat{C}$, and $B$ is any derived commutative $A$-algebra, we set
  \[
    \cpl\dR_{B/A} := |\LOmega^{+\bul}_{B/A}| \in \CAlg_A
    \quad\text{and}\quad
    \cplge\dR\star_{B/A} := |\LOmega^{+\bul}_{B/A}|^{\ge\star} \in \cpl{\Fil}\CAlg_A,
  \]
  and refer to these as the \emph{Hodge-completed derived de Rham cohomology of $B$ over $A$} and the \emph{Hodge filtration} thereon.

  We note that these objects have previous definitions more generally when $A$ is any connective derived commutative ring and $B$ a connective derived commutative $A$-algebra. That is, a bit more care allows one to extend the constructions of \cref{dg--dr--classical} to this case (from the more special case that $A$ is an ordinary commutative ring). Similarly, a bit more care would extend \cref{dg--dr--classical-compare} to this case.
\end{remark}

We now turn to the proof of \cref{dg--dr--und-gr}. We begin with some preliminaries, which are essentially concerned with the functor $\ev^\zeroone : \Gr\DAlg_A \to \Gr^\zeroone\DAlg_A$ extracting the zeroth and first pieces of a graded derived commutative algebra. The first result concerns the left adjoint construction, freely generating a graded derived commutative algebra from the data of these two pieces.

\begin{lemma}
  \label{dg--dr--und-gr-lem1}
  The composite functor
  \[
    \Gr^{\ge0}\DAlg_A \lblto{\ev^\zeroone} \Gr^\zeroone\DAlg_A \lblto{\alpha'} \DAlgMod_A
  \]
  admits a left adjoint $F : \DAlgMod_A \to \Gr^{\ge0}\DAlg_A$, which on objects sends $(B,M) \mapsto \LSym_B(M(1))$; here $\Gr^\zeroone\DAlg_A$ is as defined in \cref{dc--eg--gr-zeroone} and $\alpha'$ denotes the equivalence of \cref{dc--ct--zeroone-dalg}.
\end{lemma}

\begin{proof}
  For $(B,M) \in \DAlgMod_A$, there is an equivalence
  \[
    (B,M) \isoto (\alpha' \circ \ev^\zeroone)(\LSym_B(M(1)))
  \]
  in $\DAlgMod_A$, inducing a commutative diagram of spaces
  \[
    \begin{tikzcd}
      \Map_{\Gr^{\ge0}\DAlg_A}(\LSym_B(M(1)),C) \ar[rr, "f"] \ar[dr, "p", swap] &[-3em]
      &[-3em]
      \Map_{\DAlgMod_A}((B,M),(C^0,C^1)) \ar[dl, "q"] \\
      &
      \Map_{\DAlg_A}(B,C^0).
    \end{tikzcd}
  \]
  for any $C \in \Gr^{\ge0}\DAlg_A$. We wish to show that $f$ is an equivalence. This follows from the fact that $f$ induces an equivalence on the fibers of $p$ and $q$ over any $\phi \in \Map_{\DAlg_A}(B,C^0)$: this map on fibers identifies with the canonical equivalence
  \[
    \Map_{\Gr^{\ge0}\DAlg_B}(\LSym_B(M(1)),C) \isoto \Map_{\Mod_B}(M,C^1). \qedhere
  \]
\end{proof}

We next unpack some of the structure contained in an \hplus-differential graded derived commutative algebra, specifically the structure visible in its zeroth and first graded pieces.

\begin{notation}
  \label{dg--dr--dg-cofree}
  Let $W : \DG_+\DAlg_A \to \Gr\DAlg_A$ denote the forgetful functor. This admits a right adjoint $V : \Gr\DAlg_A \to \DG_+\DAlg_A$, forming the cofree comodule over $\dual{\num{D}_+}$, given by the formula $W(B) \iso \dual{\num{D}_+} \oast B$ (recall that $\oast$ denotes the Day convolution tensor product in $\Gr(\cat{C})$). By an abuse of notation, we denote by $\ev^0 : \DG_+\DAlg_A \to \DAlg_A$ the composite functor
  \[
    \DG_+\DAlg_A \lblto{W} \Gr\DAlg_A \lblto{\ev^0} \DAlg_A.
  \]
\end{notation}

\begin{lemma}
  \label{dg--dr--und-gr-lem2}
  \begin{enumerate}[leftmargin=*]
  \item \label{dg--dr--und-gr-lem2--zeroone}
    The natural transformation between the composites
    \begin{align*}
      & \Gr\DAlg_A \lblto{V} \DG_+\DAlg_A \lblto{\ev^0} \DAlg_A, \\
      \Gr\DAlg_A \lblto{\ev^\zeroone} \Gr^\zeroone\DAlg_A \lblto{\ins^\zeroone} & \Gr\DAlg_A \lblto{V} \DG_+\DAlg_A \lblto{\ev^0} \DAlg_A
    \end{align*}
    induced by the unit transformation $\id \to {\ins^\zeroone} \circ \ev^\zeroone$ is an equivalence.
  \item \label{dg--dr--und-gr-lem2--sqzero}
    There is a canonical natural equivalence between the following two composite functors:
    \begin{align*}
      &\Gr^\zeroone\DAlg_A \lblto{\ins^\zeroone} \Gr\DAlg_A \lblto{V} \DG_+\DAlg_A \lblto{\ev^0} \DAlg_A, \\
      &\Gr^\zeroone\DAlg_A \lblto{\alpha'} \DAlgMod_A \lblto{\Omega} \DAlgMod_A \lblto{G} \DAlg_A;
    \end{align*}
    here $G$ denotes the trivial square-zero extension functor (\cref{dc--ct--square-zero}) and $\Omega$ denotes the functor $(B,M) \mapsto (B,M[-1])$.
  \end{enumerate}
\end{lemma}

\begin{proof}
  \begin{enumerate}[leftmargin=*]
  \item This follows from the natural equivalence $V(B)^0 \iso (B \oast \dual{\num{D}_+})^0 \iso B^0 \oplus B^1[-1]$ in $\Mod_A$.

  \item Recall from \cref{dg--pm--dp-loop} that we have an equivalence of graded derived commutative $A$-algebras $\dual{\num{D}_+} \iso A \times_{\dual{\num{D}}} A$, and that $G$ is defined as the composite
    \[
      \DAlgMod_A \lblto{(\alpha')^{-1}} \Gr^\zeroone\DAlg_A \lblto{\ins^\zeroone} \Gr\DAlg_A \lblto{- \otimes \dual{\num{D}}} \Gr\DAlg_A \lblto{\ev^0} \DAlg_A.
    \]
    From this we obtain a natural equvialence between the first composite in the statement with the functor $\ev^0 \times_{(G \circ \alpha')} \ev^0$. That this functor identifies with the second composite in the statement follows from the natural equivalences
    \[
      G(B,M[-1]) \iso G((B,0) \times_{(B,M)} (B,0)) \iso G(B,0) \times_{G(B,M)} G(B,0)) \iso B \times_{G(B,M)} B
    \]
    for $(B,M) \in \DAlgMod_A$; here we have used that $G$ preserves pullbacks (\cref{dc--ct--square-zero-adjoint}), and the equivalence $G(B,0) \iso B$ is as in \cref{dc--ct--square-zero-adjoint-underlying}). \qedhere
  \end{enumerate}
\end{proof}

\begin{remark}
  \label{dg--dr--unit-counit}
  Combining statements \cref{dg--dr--und-gr-lem2--zeroone} and \cref{dg--dr--und-gr-lem2--sqzero} of \cref{dg--dr--und-gr-lem2}, we obtain an equivalence between the composition $\ev^0 \circ V : \Gr\DAlg_A \to \DAlg_A$ and the composition
  \[
    \Gr\DAlg_A \lblto{\ev^\zeroone} \Gr^\zeroone\DAlg_A \lblto{\alpha'} \DAlgMod_A \lblto{\Omega} \DAlgMod_A \lblto{G} \DAlg_A.
  \]
  On the other hand, and more simply, the functor $\ev^0 : \Gr\DAlg_A \to \DAlg_A$ itself is equivalent to the composition
  \[
    \Gr\DAlg_A \lblto{\ev^\zeroone} \Gr^\zeroone\DAlg_A \lblto{\alpha'} \DAlgMod_A \lblto{\Omega} \DAlgMod_A \lblto{U} \DAlg_A,
  \]

  Let $\eta : \id_{\Gr\DAlg_A} \to W \circ V$ and $\epsilon : W \circ V \to \id_{\Gr\DAlg_A}$ denote the maps induced by the unit and counit maps of $\num{D}_+$. (Note that $\epsilon$ is the counit map for the adjunction $W \dashv V$, but $\eta$ is evidently not the unit map for this adjunction.) These induce maps $\epsilon^0 : \ev^0 \circ V = \ev^0 \circ W \circ V \to \ev^0$ and $\eta^0 : \ev^0 \to \ev^0 \circ W \circ V = \ev^0 \circ V$. Under the equivalences of the previous paragraph, $\epsilon^0$ and $\eta^0$ identify with the maps induced by the natural transformations $\epsilon : G \to U$ and $\eta : U \to G$ of \cref{dc--ct--square-zero} (this is immediate from the definitions and inspection of the proof of \cref{dg--dr--und-gr-lem2}).
\end{remark}

\begin{proof}[Proof of \cref{dg--dr--und-gr}]
  We will contemplate the following diagram of $\infty$-categories:
  \[
    \begin{tikzcd}[row sep=large]
      \DG_+^{\ge 0}\DAlg_A \ar[r, "\ins^{\ge 0}", hook, shift left=0.5ex] \ar[dd, "\ev^0", shift left=0.5ex] &
      \DG_+\DAlg_A \ar[rr, "W", shift left=0.5ex] \ar[l, "\ev^{\ge 0}", shift left=0.5ex]
      &
      &
      \Gr\DAlg_A \ar[ll, "V", shift left=0.5ex] \ar[d, "\ev^{\ge 0}", shift left=0.5ex] \\
      &
      \DAlgMod_A \ar[r, "\Sigma", shift left=0.5ex] \ar[dl, "G", shift left=0.5ex]
      &
      \DAlgMod_A \ar[l, "\Omega", shift left=0.5ex] \ar[r, "F", shift left=0.5ex] \ar[dr, shift left =0.5ex]
      &
      \Gr^{\ge0}\DAlg_A \ar[d, "\ev^\zeroone", shift left=0.5ex] \ar[u, "\ins^{\ge 0}", hook, shift left=0.5ex] \ar[l, shift left=0.5ex]
      \\
      \DAlg_A \ar[ur, "L", shift left=0.5ex] \ar[uu, "\LOmega^{+\bul}_{-/A}", shift left=0.5ex] &
      \DG_+\DAlg_A \ar[l, "\ev^0"] &
      \Gr\DAlg_A  \ar[l, "V"] &
      \Gr^\zeroone\DAlg_A \ar[ul, "\alpha'", shift left=0.5ex] \ar[u, shift left=0.5ex] \ar[l, "\ins^\zeroone", hook]
    \end{tikzcd}
  \]
  Let us explain what's going on in the diagram:
  \begin{itemize}
  \item The adjacent pairs of arrows denote adjunctions, with the left-hand or upper arrow of the pair denoting the left adjoint.
  \item The adjoint pair $L \dashv G$ is that of \cref{dc--ct}, i.e. $G$ is the trivial square-zero extension functor and $L$ is the cotangent complex functor $B \mapsto (B, \cot_{B/A})$.
  \item The functor $\alpha'$ is the equivalence of \cref{dc--ct--zeroone-dalg} (so its adjoint is the inverse equivalence), and the functor $F$ is the left adjoint of \cref{dg--dr--und-gr-lem1}.
  \item The adjoint pair $W \dashv V$ is that of \cref{dg--dr--dg-cofree}.
  \item As in \cref{dg--dr--und-gr-lem2}\cref{dg--dr--und-gr-lem2--sqzero}, the functor $\Omega$ is given by $(B,M) \mapsto (B, M[-1])$, and so its left adjoint $\Sigma$ is given by $(B,M) \mapsto (B,M[1])$.
  \end{itemize}

  Now, we find the functor $\LOmega_{-/A}^{+*} : \DAlg_A \to \Gr\DAlg_A$ by following the left adjoints in the diagram from the bottom left to the top left to the top right. And, by the description of $F$ in \cref{dg--dr--und-gr-lem1}, the functor $\DAlg_A \to \Gr\DAlg_A$ given by following the left adjoints from the bottom left to the center of the diagram and then to the top right sends $B$ to $\LSym_B(\cot_{B/A}[1](1))$. We will produce a natural equivalence between these two functors by producing one between the corresponding composites of right adjoints. To do so, it suffices to see that the diagram commutes when we consider only the right adjoints in the adjoint pairs: the outer rectangle commutes by \cref{dg--dr--und-gr-lem2}\cref{dg--dr--und-gr-lem2--zeroone}; the inner triangle commutes by definition of $F$; and the inner trapezoid commutes by \cref{dg--dr--und-gr-lem2}\cref{dg--dr--und-gr-lem2--sqzero}.

  In sum, we have produced an equivalence $\LOmega^{+*}_{B/A} \iso \LSym_B(\cot_{B/A}[1](1))$ in $\Gr\DAlg_A$. It remains to promote this to an equivalence in $\Gr\DAlg_B$ and to identify the first differential of $\smash{\LOmega^{+\bul}_{B/A}}$ (once the former is done, the equivalences $\smash{\LOmega^{+i}_{B/A}} \iso (\bigwedge^i_B \cot_{B/A})[i]$ follow by \cref{dc--eg--gf-lsym} and the definition of the derived exterior powers (\cref{dg--dr--ext})). We will address both of these matters by considering the unit map
  \[
    u : B \to (\ev^0 \circ V \circ W)(\LOmega^{+\bul}_{B/A}) \iso \LOmega^{+0}_{B/A} \oplus \LOmega^{+1}_{B/A}[-1]
  \]
  of the adjunction obtained by going up around the left of the diagram. Note that this is equivalent to the composition of the unit maps associated to the adjunctions $\smash{\LOmega^{+\bul}_{B/A} \dashv \ev^0}$ and $W \dashv V$ (we have implicitly used that the unit map of the adjunction $\ins^{\ge0} \dashv \ev^{\ge0}$ is an equivalence). The displayed splitting is induced by the maps $\epsilon^0 : \ev^0 \circ V \to \ev^0$ and $\eta^0 : \ev^0 \to \ev^0 \circ V$ of \cref{dg--dr--unit-counit}. Since $\epsilon^0$ is induced by the counit map of the adjunction $W \dashv V$, it follows from one of the triangle identities that the composition of $\epsilon^0 \circ u$, i.e. the first component of $u$, recovers the unit map of the adjunction $\LOmega^{+\bul}_{B/A} \dashv \ev^0$, which gives the $B$-algebra structure of $\LOmega^{+*}_{B/A}$ (\cref{dg--dr--graded-linearity}). On the other hand, the second component of $u$ is the differential of $\LOmega^{+\bul}_{B/A}$ of interest.

  Now, since the diagram commutes, the unit map $u$ identifies with the unit map of the adjunction obtained by going up through the middle of the diagram. Again using that the unit map of the adjunction $\ins^{\ge0} \dashv \ev^{\ge0}$ is an equivalence, and moreover that the same is true for the adjunctions $\Sigma \dashv \Omega$ and $F \dashv (\alpha' \circ \ev^\zeroone)$, we find that this identifies with the unit map
  \[
    u' : B \to (G \circ L)(B) \iso B \oplus \cot_{B/A}
  \]
  of the adjunction $L \dashv G$. Here the splitting is that induced by the maps $\epsilon : G \to U$ and $\eta : U \to G$ of \cref{dc--ct--square-zero}, so the first component is the identity map (\cref{dc--ct--square-zero-adjoint-underlying}) and the second component is the universal $A$-linear derivation of $B$ (\cref{dc--ct--ct}). By \cref{dg--dr--unit-counit}, this splitting agrees with the that in the previous paragraph under our identification between the targets of $u$ and $u'$, so this finishes the proof.
\end{proof}


\section{Filtered circle actions}
\label{fc}

Let $X$ be a $\num{Z}$-module with $\cir$-action. Consider the \emph{Postnikov filtration} $\tau_{\ge \star}(X)$,
\[
  \cdots \to \tau_{\ge 2}(X) \to \tau_{\ge 1}(X) \to \tau_{\ge 0}(X) \to \tau_{\ge -1}(X) \to \tau_{\ge -2}(X) \to \cdots.
\]
By functoriality of this construction, the filtered $\num{Z}$-module $\tau_{\ge \star}(X)$ inherits an $\cir$-action, or in other words, the filtration is $\cir$-equivariant. In the case that $X = \HH(B/A)$ for $A$ a commutative ring and $B$ a smooth commutative $A$-algebra, the above filtration is the HKR filtration, and this $\cir$-equivariant structure has been observed and used in previous work, e.g. in \cite{bms2,antieau--periodic}.

In this section, we study a more refined $\cir$-equivariant structure in the filtered setting, referred to here as a \emph{filtered $\cir$-action}, which in particular exists on $\tau_{\ge \star}(X)$ in the situation above, and hence on the HKR filtration on Hochschild homology. The idea is that the circle action ``increases the filtration degree'', analogous to the differential of a cochain complex increasing the grading degree. More precisely, we adopt the perspective that the $\cir$-action on $X$ is equivalent to a module structure on $X$ over the group algebra $\num{Z}[\cir]$, and we observe that the Postnikov filtration construction carries this to a filtered module structure on $\tau_{\ge \star}(X)$ over the filtered algebra $\tau_{\ge \star}(\num{Z}[\cir])$. We refer to this last object as the \emph{($\num{Z}$-linear) filtered circle}, and the notion of filtered $\cir$-action mentioned above is by definition taken to mean a filtered module structure over the filtered circle.

This section is organized as follows. In \cref{fc--fc}, we flesh out the construction of the filtered circle and definition of filtered circle actions sketched above. In \cref{fc--hh}, we use this definition to characterize HKR-filtered Hochschild homology by a universal property, proving \cref{in--HH--thm}. In \cref{fc--ta}, we analyze the orbits, fixed points, and Tate constructions in the setting of filtered circle actions, and use these to construct filtrations on
cyclic, negative cyclic, and periodic cyclic homology from HKR-filtered Hochschild homology. Finally, in \cref{fc--ad}, we explain how our construction of the filtrations on Hochschild, cyclic, negative cyclic, and periodic cyclic homology allows for a simple analysis of their interaction with the Adams operations that exist on these invariants.


\subsection{The $\num{Z}$-linear filtered circle}
\label{fc--fc}

Before filtering it, let us first discuss the $\num{Z}$-linear circle itself.

\begin{notation}
  \label{fc--fc--c}
  To simplify notation, we will let $\lincir$ denote the group algebra $\num{Z}[\cir]$, i.e. the image of $\cir$ under the unique colimit-preserving symmetric monoidal functor $\Spc \to \Mod_{\num{Z}}$. As discussed in \cref{bi--tn--rep}, the group structure on $\cir$ induces a bicommutative bialgebra structure on $\lincir$.

  Note that $\lincir$ is dualizable: the dual $\dual\lincir$ is given by the function spectrum $\num{Z}^\cir$, which can be described as the limit of the constant diagram $\cir \to \Mod_{\num{Z}}$ with value $\num{Z}$, or as the image of $\cir$ under the unique limit-preserving functor $\num{Z}^{(-)}: \Spc^\op \to \Mod_{\num{Z}}$ sending $\pt \mapsto \num{Z}$. By \cref{bi--du--main-cor}, the bicommutative bialgebra structure on $\lincir$ determines a bicommutative bialgebra structure on $\dual\lincir$.

  If $\cat{C}$ is any $\num{Z}$-linear stable presentable symmetric monoidal $\infty$-category, we will abusively denote the images of $\lincir$ and $\dual\lincir$ under the structure map $\Mod_{\num{Z}} \to \cat{C}$ also by $\lincir$ and $\dual\lincir$. We regard these images also as bicommutative bialgebras; by \cref{bi--tn--rep}, we have a canonical identification of symmetric monoidal $\infty$-categories $\Mod_\lincir(\cat{C}) \iso \Fun(\clspc\cir,\cat{C})$.
\end{notation}

\begin{construction}
  \label{fc--fc--cir-derived}
  \emph{We construct a derived bicommutative bialgebra structure (\cref{dc--df--dcbi}) on $\dual\lincir$ in $\Mod_{\num{Z}}$ promoting its bicommutative bialgebra structure.}

  By definition of the bicommutative bialgebra structure on $\lincir$ and the fact that dualization determines an equivalence $\bAlg^\comm_\comm((\Mod_{\num{Z}})_\fd) \iso \bAlg^\comm_\comm((\Mod_{\num{Z}})_\fd)^\op$ (see \cref{bi--du--main-equiv}), $\dual\lincir$ can be described, with its bicommutative bialgebra structure, as the image of the commutative monoid $\cir$ under the finite limit--preserving symmetric monoidal functor $\num{Z}^{(-)} : \Spc_\fin^\op \to  \CAlg_{\num{Z}}$, where $\Spc_\fin \subseteq \Spc$ is the full subcategory spanned by the finite spaces. This description/construction lifts immediately along the limit-preserving symmetric monoidal forgetful functor $\Theta : \DAlg_{\num{Z}} \to \CAlg_{\num{Z}}$, proving the claim.
\end{construction}

\begin{remark}
  \label{fc--fc--cir-image-derived}
  It follows from \cref{fc--fc--cir-derived} that, for any derived algebraic context $\cat{C}$, we have a canonical derived bicommutative bialgebra structure on $\dual\lincir$ in $\cat{C}$.

  This allows us to realize the wish described at the start of \cref{dc}. Namely, for any derived commutative algebra object $A$ of $\cat{C}$, we may construct an equivalence of $\infty$-categories $\Fun(\clspc\cir,\DAlg_A) \iso \cMod_{\dual\lincir}(\DAlg_A)$, as follows.

  We begin with the constant functor $f : (\clspc\cir)^\op \to \cMod_{\dual\lincir}(\DAlg_A)$ with value $A$. This extends uniquely to colimit-preserving functor $f' : \Fun(\clspc\cir,\Spc) \to \cMod_{\dual\lincir}(\DAlg_A)$. Next, observe that the unique group map $\cir \to \pt$ induces a map of derived commutative bialgebras $\unit_{\cat{C}} \to \dual\lincir$, determining a colimit-preserving corestriction functor $g : \DAlg_A \to \cMod_{\dual\lincir}(\DAlg_A)$. Together, $f'$ and $g$ induce a colimit-preserving functor $F : \Fun(\clspc\cir,\DAlg_A) \iso \Fun(\clspc\cir,\Spc) \otimes \DAlg_A \to \cMod_{\dual\lincir}(\DAlg_A)$. One can check that $F$ commutes with the two forgetful functors from the source and target to $\DAlg_A$, and one can then conclude that $F$ is an equivalence using (the categorical dual of) \cite[Corollary 4.7.3.16]{lurie--algebra}.
\end{remark}

We now construct the filtered circle, and then use it to define filtered circle actions.

\begin{notation}
  \label{fc--fc--fc}
  Let $\tau_{\ge\star} : \Mod_{\num{Z}} \to \Fil(\Mod_{\num{Z}})$ denote the Postnikov filtration functor (\cref{gf--t--postnikov-fil}). We set $\lincir_\fil := \tau_{\ge\star}(\lincir) \in \Fil(\Mod_{\num{Z}})$ and $\dual\lincir_\fil := \tau_{\ge\star}(\dual\lincir) \in \Fil(\Mod_{\num{Z}})$.
\end{notation}

\begin{theorem}
  \label{fc--fc--fc-unique}
  \begin{enumerate}[leftmargin=*]
  \item \label{fc--fc--fc-unique--einfty} There exist unique bicommutative bialgebra structures on $\lincir_\fil$ and $\dual\lincir_\fil$ promoting the bicommutative bialgebra structures on $\lincir$ and $\dual\lincir$.
  \item \label{fc--fc--fc-unique--derived} There exists a unique derived bicommutative bialgebra structure on $\dual\lincir_\fil$ promoting the derived bicommutative bialgebra structure on $\dual\lincir$.
  \end{enumerate}
\end{theorem}

To be clear, in the above statement, by ``promoting'', we mean recovering the latter upon applying the functor $\colim : \Fil(\Mod_{\num{Z}}) \to \Mod_{\num{Z}}$.

\begin{proof}
  Recall from \cref{gf--t--postnikov-fil} that the functor $\tau_{\ge\star} : \Mod_{\num{Z}} \to \Fil(\Mod_{\num{Z}})$ is canonically lax symmetric monoidal and from \cref{gf--t--postnikov-fully-faithful} that $\tau_{\ge\star}$ is fully faithful. Let $\Mod_{\num{Z}}^\qf$ denote the full subcategory of $\Mod_{\num{Z}}$ spanned by the \emph{quasi-free} objects, i.e. those of the form $\bigoplus_{i \in I} \num{Z}[n_i]$. Note that this is a symmetric monoidal subcategory, and observe that the restriction of $\tau_{\ge\star}$ to $\Mod_{\num{Z}}^\qf$ is in fact strictly symmetric monoidal. We deduce that $\tau_{\ge\star}$ embeds $\Mod_{\num{Z}}^\qf$ as a symmetric monoidal full subcategory of $\Fil(\Mod_{\num{Z}})$. Statement \cref{fc--fc--fc-unique--einfty} follows, since $\lincir$ and $\dual\lincir$ are quasi-free.

  Statement \cref{fc--fc--fc-unique--derived} follows from similar reasoning. Namely, if we let $\DAlg_{\num{Z}}^{\qf,\ccn}$ denote the full subcategory of $\DAlg_{\num{Z}}$ spanned by those derived commutative rings that are coconnective and whose underlying $\num{Z}$-module are quasi-free, then it follows from  \cref{dc--cn--fil} that $\tau_{\ge\star}$ induces a symmetric monoidal embedding $\DAlg_{\num{Z}}^{\qf,\ccn} \inj \Fil\DAlg_{\num{Z}}$. This implies the claim, since $\dual\lincir \in \DAlg_{\num{Z}}^{\qf,\ccn}$.
\end{proof}

\begin{notation}
  \label{fc--fc--fc-image}
  If $\cat{C}$ is any $\num{Z}$-linear stable presentable symmetric monoidal $\infty$-category, the structure map $\Mod_{\num{Z}} \to \cat{C}$ induces a symmetric monoidal functor $\Fil(\Mod_{\num{Z}}) \to \Fil(\cat{C})$. We will denote the images of $\lincir_\fil$ and $\dual\lincir_\fil$ under this functor also by $\lincir_\fil$ and $\dual\lincir_\fil$, and regard these as bicommutative bialgebras by virtue of \cref{fc--fc--fc-unique}\cref{fc--fc--fc-unique--einfty}. The uniqueness statement in loc. cit. implies that the bicommutative bialgebra structures on $\lincir_\fil$ and $\dual\lincir_\fil$ are dual in the sense of \cref{bi--du--main-cor}.
\end{notation}

\begin{definition}
  \label{fc--fc--fc-action}
  Let $\cat{C}$ be a $\num{Z}$-linear stable presentable symmetric monoidal $\infty$-category. We let $\Fil_{\cir}(\cat{C})$ denote the $\infty$-category $\Mod_{\lincir_\fil}(\Fil(\cat{C}))$, and refer to objects of this $\infty$-category as \emph{filtered objects of $\cat{C}$ with filtered $\cir$-action}. We regard $\Fil_{\cir}(\cat{C})$ as a symmetric monoidal $\infty$-category by \cref{bi--tn--main}, using the cocommutative bialgebra structure on $\lincir_\fil$. We refer to commutative algebra objects of $\Fil_{\cir}(\cat{C})$ as \emph{filtered commutative algebras in $\cat{C}$ with filtered $\cir$-action}, and we denote the $\infty$-category $\CAlg(\Fil_{\cir}(\cat{C}))$ by $\Fil_{\cir}\CAlg(\cat{C})$.
\end{definition}

The following is the key definition for formulating the universal property of HKR-filtered Hochschild homology (see \cref{fc--hh}).

\begin{definition}
  \label{fc--fc--dalg-fc-action}
  Let $\cat{C}$ be a derived algebraic context. By \cref{fc--fc--fc-unique}\cref{fc--fc--fc-unique--derived}, we may regard $\dual\lincir_\fil$ as a derived bicommutative bialgebra in $\Fil(\cat{C})$. Let $A$ be a derived commutative algebra object of $\cat{C}$ and let $\Fil\DAlg_A$ denote the $\infty$-category of filtered derived commutative algebra objects of $\cat{C}$ (\cref{dc--eg--gf}). Then we let $\Fil_{\cir}\DAlg_A$ denote the $\infty$-category $\cMod_{\dual\lincir_\fil}(\Fil\DAlg_A)$, and we refer to objects of this $\infty$-category as \emph{filtered derived commutative $A$-algebras with filtered $\cir$-action}.
\end{definition}

\begin{remark}
  \label{fc--fc--dalg-fc-action-alt}
  In the context of \cref{fc--fc--dalg-fc-action}, the discussion of \cref{dc--df--dcbi-dual} allows us to rewrite the $\infty$-category $\Fil_\cir\DAlg_A$ as $\DAlg(\Fil_\cir(\Mod_A))$; that is, we may think of filtered derived commutative $A$-algebras with filtered $\cir$-action as derived commutative algebra objects of the $\infty$-category of filtered $A$-modules with filtered $\cir$-action.
\end{remark}

To understand the structure of a filtered $\cir$-action on a filtered object $X^\star$, it is helpful to examine what structure is inherited by the underlying object $\colim(X^\star)$ and the associated graded object $\gr(X^\star)$; we close the subsection by discussing this.

\begin{remark}
  \label{fc--fc--colim}
  Let $\cat{C}$ be a stable presentable $\num{Z}$-linear symmetric monoidal $\infty$-category. Consider the colimit functor $\colim : \Fil(\cat{C}) \to \cat{C}$, with right adjoint given by the diagonal functor $\delta : \cat{C} \to \Fil(\cat{C})$. By construction, we have an equivalence of bicommutative bialgebras $\colim(\lincir_\fil) \iso \lincir$ in $\cat{C}$. It follows that there is an induced adjunction
  \[
    \begin{tikzcd}
      \Fil_{\cir}(\cat{C}) \ar[r, "\colim", shift left=0.5ex] &
      \Mod_\lincir(\cat{C}) \iso \Fun(\clspc\cir,\cat{C}) \ar[l, "\delta", shift left=0.5ex]
    \end{tikzcd}
  \]
  with the left adjoint symmetric monoidal so that this further determines an adjunction on commutative algebra objects
  \[
    \begin{tikzcd}
      \Fil_{\cir}\CAlg(\cat{C}) \ar[r, "\colim", shift left=0.5ex] &
      \CAlg(\Mod_\lincir(\cat{C})) \iso \Fun(\clspc\cir,\CAlg(\cat{C})). \ar[l, "\delta", shift left=0.5ex]
    \end{tikzcd}
  \]
  If $\cat{C}$ is a derived algebraic context, then we moreover have an equivalence of derived bicommutative bialgebras $\colim(\dual\lincir_\fil) \iso \dual\lincir$ in $\cat{C}$, and it follows from \cref{dc--df--dcbi-natural} that, for any $A \in \DAlg(\cat{C})$, we also obtain an induced adjunction
  \[
    \begin{tikzcd}
      \Fil_{\cir}\DAlg_A \ar[r, "\colim", shift left=0.5ex] &
      \DAlg(\Mod_\lincir(\Mod_A)) \iso \Fun(\clspc\cir,\DAlg_A), \ar[l, "\delta", shift left=0.5ex]
    \end{tikzcd}
  \]
  where the last equivalence is as discussed in \cref{fc--fc--cir-image-derived}.

  We can summarize the above discussion informally as follows: a filtered object with filtered $\cir$-action has an underlying object with $\cir$-action and the constant filtration of an object with $\cir$-action carries a filtered $\cir$-action, and moreover both of these constructions are compatible with commutative or derived commutative algebra structures.
\end{remark}

\begin{remark}
  \label{fc--fc--postnikov}
  Slightly modifying the discussion in \cref{fc--fc--colim} allows us to make precise the motivating example of filtered $\cir$-actions sketched at the beginning of the section, as follows. Note that $\lincir_\fil$ is connective for the Postnikov t-structure on $\Fil(\cat{C})$ (\cref{gf--t--postnikov-t}). The restricted colimit functor $\colim : \Fil(\cat{C})_{\ge 0}^\post \to \cat{C}$ has right adjoint given by the Postnikov filtration functor $\tau_{\ge\star} : \cat{C} \to \Fil(\cat{C})_{\ge 0}^\post$ (\cref{gf--t--postnikov-fil}). Setting $\Fil_{\cir}(\cat{C})_{\ge 0}^\post := \Fil_{\cir}(\cat{C}) \times_{\Fil(\cat{C})} \Fil(\cat{C})_{\ge 0}^\post$, we obtain an adjunction
  \[
    \begin{tikzcd}
      \Fil_{\cir}(\cat{C})_{\ge 0}^\post \ar[r, "\colim", shift left=0.5ex] &
      \Mod_\lincir(\cat{C}) \iso \Fun(\clspc\cir,\cat{C}) \ar[l, "\tau_{\ge\star}", shift left=0.5ex]
    \end{tikzcd}
  \]
  with the left adjoint symmetric monoidal so that this further determines an adjunction on commutative algebra objects, as in \cref{fc--fc--colim}. This shows that, for $X \in \Fun(\clspc\cir,\cat{C})$, the Postnikov filtration $\tau_{\ge\star}(X)$ carries a canonical filtered $\cir$-action, and that this construction is compatible with commutative algebra structures.

  On the other hand, the Postnikov filtration construction does \emph{not} in general preserve derived commutative algebra structures (though it does in certain circumstances, e.g. \cref{dc--cn--fil}). Thus, the fact that HKR-filtered Hochschild homology nevertheless does end up carrying this structure (as we will see in \cref{fc--hh}) is a special circumstance, not a priori clear.
\end{remark}

\begin{remark}
  \label{fc--fc--gr}
  Let $\cat{C}$ be a stable presentable $\num{Z}$-linear symmetric monoidal $\infty$-category. Consider the associated graded functor $\gr : \Fil(\cat{C}) \to \Gr(\cat{C})$, which admits a right adjoint $\zeta : \Gr(\cat{C}) \to \Fil(\cat{C})$  sending a graded object $X^*$ to the filtered object
  \[
    \cdots \to X^2 \lblto{0} X^1 \lblto{0} X^0 \lblto{0} X^{-1} \lblto{0} X^{-2} \to \cdots
  \]
  (\cref{gf--kd--zeta}). There is a canonical equivalence of bicommutative bialgebras $\gr(\lincir_\fil) \iso \num{D}_+$ in $\Gr(\cat{C})$; this is immediate from the construction of $\num{D}_+$ (\cref{dg--pm--d}). It follows that there is an induced adjunction
  \[
    \begin{tikzcd}
      \Fil_{\cir}(\cat{C}) \ar[r, "\gr", shift left=0.5ex] &
      \DG_+(\cat{C}) \ar[l, "\zeta", shift left=0.5ex]
    \end{tikzcd}
  \]
  with the left adjoint symmetric monoidal so that this further determines an adjunction on commutative algebra objects
  \[
    \begin{tikzcd}
      \Fil_{\cir}\CAlg(\cat{C}) \ar[r, "\gr", shift left=0.5ex] &
      \DG_+\CAlg(\cat{C}). \ar[l, "\zeta", shift left=0.5ex]
    \end{tikzcd}
  \]
  If $\cat{C}$ is a derived algebraic context, then $\gr$ is a morphism of derived algebraic contexts and it follows from \cref{dg--pm--dplusdual-derived} that we moreover have an equivalence of derived bicommutative bialgebras $\gr(\dual\lincir_\fil) \iso \dual{\num{D}}_+$ in $\Gr(\cat{C})$. It then follows from \cref{dc--df--dcbi-natural} that, for any $A \in \DAlg(\cat{C})$, we also obtain an induced adjunction
  \[
    \begin{tikzcd}
      \Fil_{\cir}\DAlg_A \ar[r, "\gr", shift left=0.5ex] &
      \DG_+\DAlg_A. \ar[l, "\zeta", shift left=0.5ex]
    \end{tikzcd}
  \]

  We can summarize the above discussion informally as follows: the associated graded of a filtered object with filtered $\cir$-action is canonically an \hplus\ cochain complex, compatibly so with commutative or derived commutative algebra structures, and conversely for the right adjoint construction $\zeta$.  
\end{remark}


\subsection{HKR-filtered Hochschild homology}
\label{fc--hh}

In this subsection, we explain how the constructions and definitions of \cref{fc--fc} supply a direct construction, and universal characterization, of Hochschild homology together with its HKR filtration. Let us begin by recalling the definition of Hochschild homology, here stated in the setting of derived commutative $A$-algebras for a fixed derived commutative algebra $A$ in a fixed derived algebraic context $\cat{C}$.

\begin{definition}
  \label{fc--hh--hh}
  Consider the $\infty$-category $\Fun(\clspc\cir,\DAlg_A)$ of derived commutative $A$-algebras with $\cir$-action. We have a functor $U : \Fun(\clspc\cir,\DAlg_A) \to \DAlg_A$ forgetting the $\cir$-action, given by restriction to the basepoint $b : \pt \to \clspc\cir$. This functor admits a left adjoint, given by left Kan extension along $b$; we denote this left adjoint by $\HH(-/A) : \DAlg_A \to \Fun(\clspc\cir,\DAlg_A)$. For $B \in \DAlg_A$, we refer to $\HH(B/A)$ as the \emph{Hochschild homology of $B$ over $A$}.
\end{definition}

\begin{remark}
  \label{fc--hh--hh-formula}
  Let $B \in \DAlg_A$. It follows from the pointwise formula for left Kan extension that, after forgetting the $\cir$-action, the Hochschild homology $\HH(B/A)$ is given by the colimit of the constant diagram $\cir \to \DAlg_A$ with value $B$. Equivalently, it is given by the image of $\cir$ under the unique colimit preserving functor $B^{\otimes_A -}: \Spc \to \DAlg_A$ sending $\pt \mapsto B$. From this perspective, the $\cir$-action on $\HH(B/A) \iso B^{\otimes_A \cir}$ is induced by the $\cir$-action on $\cir$ itself. Note also that the equivalence $\cir \iso \pt \amalg_{\mathrm{S}^0} \pt$ in $\Spc$ determines an equivalence $\HH(B/A) \iso B \otimes_{B \otimes_A B} B$ in $\DAlg_A$.
\end{remark}

We now define a filtered enhancement of Hochschild homology by replacing $\cir$-actions with filtered $\cir$-actions.

\begin{notation}
  \label{fc--hh--nonneg}
  We let $\Fil_\cir^{\ge 0}\DAlg_A$ denote the fiber product $\Fil_\cir\DAlg_A \times_{\Fil\DAlg_A} \Fil^{\ge0}\DAlg_A$, i.e. the $\infty$-category of \emph{nonnegative filtered derived commutative $A$-algebras with filtered $\cir$-action} (see \cref{fc--fc--dalg-fc-action,dc--eg--gf} for the definitions of the constituent terms of the fiber product).
\end{notation}

\begin{proposition}
  \label{fc--hh--adj}
  The composition of the forgetful functor $U : \Fil_\cir^{\ge0}\DAlg_A \to \Fil^{\ge0}\DAlg_A$ with the evaluation functor $\ev^0 : \Fil^{\ge0}\DAlg_A \to \DAlg_A$ admits a left adjoint.
\end{proposition}

\begin{proof}
  Similar to the proof of \cref{dg--dr--adj}.
\end{proof}

\begin{definition}
  \label{fc--hh--hh-fil}
  We shall denote the left adjoint functor given by \cref{fc--hh--adj} by
  \[
    \HH_\fil(-/A) : \DAlg_A \to \Fil_\cir^{\ge0}\DAlg_A.
  \]
  For $B \in \DAlg_A$, we refer to $\HH_\fil(B/A)$ as the \emph{HKR-filtered Hochschild homology of $B$ over $A$}.
\end{definition}

This definition of $\HH_\fil(-/A)$ should be viewed as an interpolation between the definition of Hochschild homology (\cref{fc--hh--hh}) and the definition of the derived de Rham complex (\cref{dg--dr--ntn}). The following (completely formal) result says that it is an interpolation in a precise mathematical sense too.

\begin{theorem}
  \label{fc--hh--conceptual-hkr}
  For $B \in \DAlg_A$, there are canonical natural equivalences
  \begin{align*}
    \ev^0(\HH_\fil(B/A)) \iso \HH(B/A) &\quad\text{in}\ \Fun(\clspc\cir,\DAlg_A),\\
    \gr(\HH_\fil(B/A)) \iso \LOmega^{+\bul}_{B/A} &\quad\text{in}\ \DG_+\DAlg_A.
  \end{align*}
\end{theorem}

\begin{proof}
  We contemplate the following diagram of $\infty$-categories:
  \[
    \begin{tikzcd}
      \Fun(\clspc\cir,\DAlg_A) \ar[r, "\delta"', shift right=0.5ex] \ar[rd, "U", shift left=0.5ex] &[6em]
      \Fil_\cir^{\ge0}\DAlg_A \ar[r, "\gr", shift left=0.5ex] \ar[l, "\ev^0"', shift right=0.5ex] \ar[d, "\ev^0", shift left=0.5ex] &[5em]
      \DG_+^{\ge0}\DAlg_A \ar[l, "\zeta", shift left=0.5ex] \ar[dl, "\ev^0", shift left=0.5ex] \\[6ex]
      &
      \DAlg_A \ar[u, "\HH_\fil(-/A)", shift left=0.5ex] \ar[ul, "\HH(-/A)", shift left=0.5ex] \ar[ur, "\LOmega^{+\bul}_{-/A}", shift left=0.5ex]
    \end{tikzcd}
  \]
  Here, the adjoint pair $\HH(-/A) \dashv U$ is as defined in \cref{fc--hh--hh}, the adjoint pair $\HH_\fil(-/A) \dashv \ev^0$ is as defined in \cref{fc--hh--hh-fil}, the adjoint pair $\LOmega^{+\bul}_{-/A} \dashv \ev^0$ is as defined in \cref{dg--dr--ntn}, and the two horizontal adjoint pairs are as described in \cref{fc--fc--colim,fc--fc--gr} (note that $\ev^0 \iso \colim$ on the nonnegative filtered category $\Fil_\cir^{\ge0}\DAlg_A$). The claim amounts to showing that the diagram of left adjoints commutes, but it is easy to see that the diagram of right adjoints commutes.
\end{proof}

We now address completeness of HKR-filtered Hochschild homology (as defined above).

\begin{lemma}
  \label{fc--hh--free-split}
  Let $M \in \Mod_A$ and let $B := \LSym_A(M) \in \DAlg_A$. Then the filtered object $\HH_\fil(B/A)$ is split (\cref{gf--df--gf}\cref{gf--df--gf--spl}).
\end{lemma}

\begin{proof}
  Set $\Fil_\cir^{\ge0}(\Mod_A) := \Fil_\cir(\Mod_A) \times_{\Fil(\Mod_A)} \Fil^{\ge0}(\Mod_A)$. Consider the commutative diagram of $\infty$-categories
  \[
    \begin{tikzcd}
      \Fil_\cir^{\ge0}(\Mod_A) \ar[d, "\ev^0", swap] &
      \Fil_\cir^{\ge0}\DAlg_A \ar[l] \ar[d, "\ev^0"] \\
      \Mod_A  &
      \DAlg_A, \ar[l]
    \end{tikzcd}
  \]
  where the leftward arrows are the forgetful functors. Passing to left adjoints, we obtain the commutative diagram
  \[
    \begin{tikzcd}
      \Fil_\cir^{\ge0}(\Mod_A) \ar[r, "\LSym_A"] &
      \Fil_\cir^{\ge0}\DAlg_A \\
      \Mod_A \ar[r, "\LSym_A"]  \ar[u, "\lincir_\fil \otimes \ins^0(-)"] &
      \DAlg_A. \ar[u, "\HH_\fil(-/A)", swap]
    \end{tikzcd}
  \]
  As $\lincir_\fil \iso \unit_{\cat{C}} \oplus \unit_{\cat{C}}[1](1)$, this gives us an equivalence
  \begin{align*}
    \HH_\fil(B/A)
    &\iso \LSym_A(M \oplus M[1](1)) \\
    &\iso \LSym_A(M) \otimes_A \LSym_A(M[1](1)) \\
    &= B \otimes_A \LSym_A(M[1](1)) \\
    &\iso \LSym_B((B \otimes_A M)[1](1))
  \end{align*}
  in $\Fil^{\ge0}\DAlg_A$, which shows that $\HH_\fil(B/A)$ is split (keeping in mind \cref{dc--eg--gf-lsym}).
\end{proof}

\begin{warning}
  \label{fc--hh--split-warning}
  It does not follow from \cref{fc--hh--free-split} that $\HH_\fil(B/A)$ is split for all $B \in \DAlg_A$. The splitting produced in the proof of \cref{fc--hh--free-split} depends on the identification $B \iso \LSym_A(M)$, hence is functorial in $M \in \Mod_A$ but not in $B \in \DAlg_A$.
\end{warning}

\begin{proposition}
  \label{fc--hh--complete}
  Suppose that $A$ is connective and that the t-structure on $\cat{C}$ is left separated. Then, for every connective derived commutative $A$-algebra $B$, the filtered object $\HH_\fil(B/A)$ is complete.
\end{proposition}

\begin{proof}
  We will show, for such $A$ and $B$, that $\HH_\fil(B/A)^i$ is $i$-connective for all $i \ge 0$, i.e. $\HH_\fil(B/A)$ is connective in the Postnikov t-structure on $\Fil(\Mod_A)$; this implies the claim by the left separatedness hypothesis. Let $\cat{X} \subseteq \DAlg_A^\cn$ denote the full subcategory spanned by those $B$ for which $\HH_\fil(B/A)$ satisfies this connectivity condition. By \cref{fc--hh--conceptual-hkr,dg--dr--und-gr}, we know that $\grop^i \HH_\fil(B/A)$ is $i$-connective for all $B \in \DAlg_A^\cn$. Combining this with \cref{fc--hh--free-split}, we see that $\cat{X}$ contains $\LSym_A(M)$ for all $M \in \Mod_A^\cn$. Since $\cat{X}$ is closed under colimits, this implies that $\cat{X} = \DAlg_A^\cn$, as any object of $\DAlg_A^\cn$ can be written as a geometric realization of such free objects (e.g. by the bar resolution).
\end{proof}

Combining \cref{fc--hh--complete} with \cref{fc--hh--conceptual-hkr} (and specializing to the context $\cat{C} = \Mod_{\num{Z}}$), we have now fully proved \cref{in--HH--thm}. We may also verify that our definition of the HKR filtration agrees with the usual one:

\begin{remark}
  \label{fc--hh--hkr}
  Suppose that $\cat{C} = \Mod_{\num{Z}}$ and that $A$ is an ordinary commutative ring. Consider the composition of functors
  \begin{equation}
    \label{fc--hh--hkr--functor} \tag{$*$}
    \DAlg_A^\cn \inj \DAlg_A \lblto{\HH_\fil(-/A)} \Fil_\cir^{\ge0}\DAlg_A \lblto{U} \Fil\CAlg_A,
  \end{equation}
  where $U$ denotes the forgetful functor. Each of these functors preserves colimits, so the composite is the left derived functor of its restriction to the full subcategory $\Poly_A \subseteq \DAlg_A^\cn$ spaned by the finitely generated polynomial $A$-algebras. By virtue of \cref{fc--hh--conceptual-hkr,fc--hh--complete}, the restriction to $\Poly_A$ factors through the full subcategory $(\Fil\CAlg_A)_{\ge 0}^\post \subseteq \Fil\CAlg_A$ spanned by those objects that are connective for the Postnikov t-structure. Using the adjunction
  \[
    \colim : (\Fil\CAlg_A)_{\ge 0}^\post \fromto \CAlg_A : \tau_{\ge\star},
  \]
  we obtain a natural comparison map $\alpha : \HH_\fil(B/A) \to \tau_{\ge\star}(\colim(\HH_\fil(B/A)))$ for $B \in \Poly_A$. It follows from \cref{fc--hh--conceptual-hkr} that the target can be identified with $\tau_{\ge\star}(\HH(B/A))$. Invoking also \cref{dg--dr--und-gr} and the fact that $\cot_{B/A}$ is a free $B$-module for $B \in \Poly_A$, we find that $\gr^i(\HH_\fil(B/A))$ has homotopy concentrated in degree $i$ for all $i \ge 0$; this implies that $\alpha$ must be an equivalence. We conclude that the functor \cref{fc--hh--hkr--functor} agrees with the usual definition of the HKR filtration on Hochschild homology for simplicial commutative $A$-algebras.
\end{remark}


\subsection{Filtered orbits, fixed points, and Tate construction}
\label{fc--ta}

In this subsection, we specialize the constructions of \cref{bi--ta} to obtain notions of orbits, fixed points, and the Tate construction in the setting of filtered $\cir$-actions, and discuss their behavior upon taking colimits and associated graded objects. Applying these constructions to HKR-filtered Hochschild homology, we will obtain filtrations on cyclic, negative cyclic, and periodic cyclic homology.

\begin{remark}
  \label{fc--ta--circle-dual}
  Recall that there is a canonical equivalence of $\lincir$-modules $\dual\lincir \iso \lincir[-1]$. The (lax symmetric monoidal) Postnikov filtration functor $\tau_{\ge\star}$ carries this to an equivalence of $\lincir_\fil$-modules $\dual\lincir_\fil \iso \lincir_\fil[-1](-1)$. It follows that, for any stable presentable $\num{Z}$-linear symmetric monoidal $\infty$-category $\cat{C}$, the cocommutative bialgebra $\lincir_\fil$ in $\Fil(\cat{C})$ satisfies the assumptions of \cref{bi--ta}, where we take $\omega_{\lincir_\fil} = \unit_{\cat{C}}[1](1)$ for $\unit_{\cat{C}}$ the unit object of $\cat{C}$. Thus, we have a norm map and Tate construction for filtered objects of $\cat{C}$ with filtered $\cir$-action.
\end{remark}

\begin{notation}
  \label{fc--ta--orbfixtate}
  Let $\cat{C}$ be a stable presentable $\num{Z}$-linear symmetric monoidal $\infty$-category. For $X \in \Fun(\clspc\cir,\cat{C}) \iso \Mod_\lincir(\cat{C})$, we let
  \[
    X_{\lincir}, \quad X^{\lincir}, \quad X^{\tate \lincir} \quad\in\ \cat{C}
  \]
  denote the orbits, fixed points, and Tate construction, using the definitions of \cref{bi--ta} for the cocommutative bialgebra $\lincir$ in $\cat{C}$; these agree with the usual homotopy $\cir$-orbits, homotopy $\cir$-fixed points, and $\cir$-Tate construction. For $X \in \Fil_\cir(\cat{C})$, we let
  \[
    X_{\lincir_\fil}, \quad X^{\lincir_\fil}, \quad X^{\tate \lincir_\fil}
    \quad\in\ \Fil(\cat{C})
  \]
  denote the orbits, fixed points, and Tate construction, using the definitions of \cref{bi--ta} for the cocommutative bialgebra $\lincir_\fil$ in $\Fil(\cat{C})$; we refer to these as the \emph{filtered $\cir$-orbits, filtered $\cir$-fixed points, and filtered $\cir$-Tate construction}.
\end{notation}

\begin{proposition}
  \label{fc--ta--gr}
  Let $\cat{C}$ be a stable presentable $\num{Z}$-linear symmetric monoidal $\infty$-category. Then, for $X \in \Fil_{\cir}(\cat{C})$, there are canonical natural equivalences
  \begin{align*}
    \gr(X_{\lincir_\fil}) &\iso \und(|\gr(X)|^{\le\star})[2*], \\
    \gr(X^{\lincir_\fil}) &\iso \und(|\gr(X)|^{\ge\star})[2*], \\
    \gr(X^{\tate \lincir_\fil}) &\iso \delta(|\gr(X)|)[2*]
  \end{align*}
  in $\Gr(\cat{C})$, where:
  \begin{itemize}
  \item we regard $\gr(X)$ as an object of $\DG_+(\cat{C})$ by \cref{fc--fc--gr};
  \item the functors $|-|^{\le\star}, |-|^{\ge\star} : \DG_+(\cat{C}) \to \Fil(\cat{C})$ and $|-| : \DG_+(\cat{C}) \to \cat{C}$ are as defined in \cref{dg--co--cohomology-type};
  \item the functor $\und : \Fil(\cat{C}) \to \Gr(\cat{C})$ is as defined in \cref{gf--df--gf}\cref{gf--df--gf--und};
  \item $\delta$ denotes the diagonal functor $\cat{C} \to \Gr(\cat{C})$;
  \item the functor $[2*] : \Gr(\cat{C}) \to \Gr(\cat{C})$ is as defined in \cref{gf--t--shift}.
  \end{itemize}

\end{proposition}

\begin{proof}
  The symmetric monoidal functor $\gr : \Fil(\cat{C}) \to \Gr(\cat{C})$ preserves all small limits and colimits, so, by \cref{bi--ta--natural}, we have canonical natural equivalences
  \[
    \gr(X_{\lincir_\fil}) \iso \gr(X)_{\num{D}_+}, \quad
    \gr(X^{\lincir_\fil}) \iso \gr(X)^{\num{D}_+}, \quad
    \gr(X^{\tate\lincir_\fil}) \iso \gr(X)^{\tate\num{D}_+}.
  \]
  The claim now follows from \cref{dg--co--tate,dg--co--orbits-fixed}.
\end{proof}

\begin{remark}
  \label{fc--ta--gr-lax}
  In the situation of \cref{fc--ta--gr}, it follows from \cref{dg--co--compare-lax} that the natural equivalence $\gr(X^{\lincir_\fil}) \iso \und(|\gr(X)|^{\ge\star})[2*]$ is canonically one of lax symmetric monoidal functors, and that the same goes for the equivalence $\gr(X^{\tate \lincir_\fil}) \iso \delta(|\gr(X)|)[2*]$.
\end{remark}

\begin{proposition}
  \label{fc--ta--colim}
  Let $\cat{C}$ be a stable presentable $\num{Z}$-linear symmetric monoidal $\infty$-category. Then, for $X \in \Fil_\cir(\cat{C})$, there are a canonical natural equivalence
  \[
    \colim(X_{\lincir_\fil}) \iso \colim(X)_{\lincir}
  \]
  and canonical natural transformations
  \[
    \colim(X^{\lincir_\fil}) \to \colim(X)^{\lincir},
    \quad
    \colim(X^{\tate \lincir_\fil}) \to \colim(X)^{\tate \lincir},
  \]
  where here we regard $\colim(X)$ as an object of $\Fun(\clspc\cir,\cat{C}) \iso \Mod_\lincir(\cat{C})$ by \cref{fc--fc--colim}. In the situation that $\cat{C}$ is equipped with a right separated t-structure $(\cat{C}_{\ge 0},\cat{C}_{\le 0})$, the latter two transformations are equivalences whenever $X$ satisfies the following two conditions:
  \begin{enumerate}
  \item \label{fc--ta--colim--nonneg} $X$ is nonnegatively filtered, i.e. $\grop^i(X) \iso 0$ for all $i < 0$;
  \item \label{fc--ta--colim--coconn} there exists an $a \ge 0$ such that $\grop^i(X) \in \cat{C}_{\le 2i+a}$ for all $i \ge 0$.
  \end{enumerate}
\end{proposition}

\begin{proof}
  Everything except the last sentence is immediate from \cref{bi--ta--natural}. Let us address the last sentence. Suppose given a right separated t-structure on $\cat{C}$ and $X \in \Fil_\cir(\cat{C})$ satisfying \cref{fc--ta--colim--nonneg} and \cref{fc--ta--colim--coconn}. Set $Y := \colim(X) \iso X^0$. It suffices to show that the natural map $\colim(X^{\lincir_\fil}) \to Y^\lincir$ is an equivalence (the claim for the Tate construction follows in light of its definition in terms of orbits and fixed points); we will show that its cofiber $Z$ vanishes. By \cref{bi--ta--fixed-formula}, we have
  \[
    Y^\lincir \iso \lim(\cosimpl{Y}{Y \oplus Y[-1]}{Y \oplus (Y[-1])^{\oplus 2} \oplus Y[-2]} )
  \]
  and, for each $n \ge 0$, we have
  \[
    (X^{\lincir_\fil})^{-n} \iso \lim(\cosimpl{X^{-n}}{X^{-n} \oplus X^{-n+1}[-1]}{X^{-n} \oplus (X^{-n+1}[-1])^{\oplus 2} \oplus X^{-n+2}[-2]} ).
  \]
  Consider the cofiber $Z^{-n}$ of the natural map $(X^{\lincir_\fil})^{-n} \to Y^{\lincir}$, which in terms of the above formulas is induced by the canonical maps $\theta^k : X^k \to Y$. Our assumptions on $X$ imply that $\theta^k$ is an equivalence for $k \le 0$ and that $\cofib(\theta^k)$ is $(2(k-1)+a)$-coconnective for $k \ge 1$. It follows that $Z^{-n}$ is the limit of a cosimplicial diagram which is zero up through level $n$, and at level $(n+i)$ for $i \ge 1$ is $(i+a-n-2)$-coconnective (here $i+a-n-2$ appears as $2(i-1)+a - (n+i)$). Examining the Bousfield--Kan spectral sequence, we deduce that the limit $Z^{-n}$ is $(a-2n-2)$-coconnective. Hence $Z \iso \colim_{n \to \infty} Z^{-n} \iso 0$, as desired.
\end{proof}

\begin{remark}
  \label{fc--ta--colim-lax}
  In the situation of \cref{fc--ta--colim}, recall from \cref{bi--ta--natural} that the natural transformation $\colim(X^{\lincir_\fil}) \to \colim(X)^{\lincir}$ is canonically one of lax symmetric monoidal functors, and the same goes for the transformation $\colim(X^{\tate\lincir_\fil}) \to \colim(X)^{\tate\lincir}$.
\end{remark}

\begin{proposition}
  \label{fc--ta--complete}
  Let $\cat{C}$ be a stable presentable $\num{Z}$-linear symmetric monoidal $\infty$-category. Let $X \in \Fil_\cir(\cat{C})$. Then:
  \begin{enumerate}
  \item \label{fc--ta--complete--fixed} If the underlying filtered object of $X$ is complete, the same holds for $X^{\lincir_\fil} \in \Fil(\cat{C})$.
  \end{enumerate}
  Assuming that $\cat{C}$ is equipped with a t-structure $(\cat{C}_{\ge 0},\cat{C}_{\le0})$, we also have the following:
  \begin{enumerate}[resume]
  \item \label{fc--ta--complete--orbits} If $X^i$ is $i$-connective for all $i \ge 0$, the same holds for $X_{\lincir_\fil} \in \Fil(\cat{C})$; in particular, if the t-structure on $\cat{C}$ is left separated, this condition implies that $X_{\lincir_\fil}$ is complete, which in combination with \cref{fc--ta--complete--fixed} implies that $X^{\tate\lincir_\fil}$ is complete.
  \end{enumerate}
\end{proposition}

\begin{proof}
  Statement \cref{fc--ta--complete--fixed} follows from the formula in \cref{bi--ta--fixed-formula} for $X^{\lincir_\fil}$ and the fact that completeness is stable under the formation of limits. Statement \cref{fc--ta--complete--orbits} follows similarly using the formula for the relative tensor product $X_{\lincir_\fil} = X \otimes_{\lincir_\fil} \unit_{\Fil(\cat{C})}$ as a geometric realization and the fact that $i$-connectivity is stable under the formation of colimits for any $i \ge 0$.
\end{proof}

We now explain how these results supply a proof of \cref{in--HH--fils}\cref{in--HH--fils--HP}.

\begin{example}
  \label{fc--ta--hh}
  Let $\cat{C}$ be a derived algebraic context. Let $A \in \DAlg(\cat{C})$ and let $B$ be a derived commutative $A$-algebra. Then \emph{cyclic, negative cyclic, and periodic cyclic homology of $B$ over $A$} are defined as follows:
  \[
    \HC(B/A) := \HH(B/A)_\lincir, \qquad
    \HN(B/A) := \HH(B/A)^\lincir, \qquad
    \HP(B/A) := \HH(B/A)^{\tate\lincir},
  \]
  where $\HH(B/A)$ denotes the Hochschild homology of $B$ over $A$ (\cref{fc--hh--hh}). Now let $\HH_\fil(B/A)$ denote HKR-filtered Hochschild homology (\cref{fc--hh--hh-fil}), and define
  \[
    \HC_\fil(B/A) := \HH_\fil(B/A)_{\lincir_\fil}, \quad
    \HN_\fil(B/A) := \HH_\fil(B/A)^{\lincir_\fil}, \quad
    \HP_\fil(B/A) := \HH_\fil(B/A)^{\tate \lincir_\fil}.
  \]
  We may understand these filtered objects using the results of this subsection, together with the identifications $\colim(\HH_\fil(B/A)) \iso \HH(B/A)$ and $\gr(\HH_\fil(B/A)) \iso \LOmega^\bul_{B/A}$ of \cref{fc--hh--conceptual-hkr}.

  Firstly, \cref{fc--ta--colim} supplies a canonical equivalence and maps
  \[
    \colim(\HC_\fil(B/A)) \iso \HC(B/A),
  \]
  \[
    \colim(\HN_\fil(B/A)) \to \HN(B/A), \quad
    \colim(\HP_\fil(B/A)) \to \HP(B/A),
  \]
  and says that the latter two maps are equivalences whenever there exists an $a \ge 0$ such that $\gr^i(\HH_\fil(B/A)) \iso \LOmega^{+i}_{B/A} \iso (\bigwedge^i_B \cot_{B/A})[i]$ is $(2i+a)$-truncated for all $i \ge 0$. In the case $\cat{C} = \Mod_{\num{Z}}$, one can check that this condition holds whenever $A$ and $B$ are connective, $B$ is truncated (i.e. $\pi_k(B) \iso 0$ for $k \gg 0$), and $\cot_{B/A}$ has Tor-amplitude contained in $[0,1]$; for example, this includes the case where $A$ is an ordinary commutative ring and $B$ is quasi-smooth over $A$ (in the sense of \cite[Definition 3.4.15]{lurie--thesis}).

  Secondly, \cref{fc--ta--gr}, together with the equivalence $\gr(\HH_\fil(B/A)) \iso \LOmega_{B/A}^{+\bul}$, supplies canonical identifications
  \[
    \gr(\HC_\fil(B/A)) \iso \cplle{\dR_{B/A}}{*}[2*], \quad
    \gr(\HN_\fil(B/A)) \iso \cplge{\dR_{B/A}}{*}[2*], \quad
    \gr(\HP_\fil(B/A)) \iso \cpl\dR_{B/A}[2*],
  \]
  where $\cpl\dR_{B/A}$ and $\cplge{\dR_{B/A}}{*}$ denote Hodge-completed derived de Rham cohomology of $B$ over $A$ and the Hodge filtration thereon (see \cref{dg--dr--classical-compare,dg--dr--cohomology}); here we have left the underlying graded functor ($\und$) and the diagonal functor ($\delta$) implicit.

  Finally, \cref{fc--ta--complete}, together with \cref{fc--hh--complete}, implies that the filtered objects $\HC_\fil(B/A)$, $\HN_\fil(B/A)$, and $\HP_\fil(B/A)$ are complete whenever $A$ and $B$ are connective.
\end{example}


\subsection{Application: Adams operations}
\label{fc--ad}

Our goal in this subsection will be to analyze the effect of the Adams operations on the graded pieces of the filtrations on Hochschild homology and negative cyclic homology discussed in the previous two subsections. The same ideas can be applied to cyclic and periodic cyclic homology, but we shall leave this to the interested reader.

Let us begin by recalling how Adams operations are defined on Hochschild homology (here we follow the perspective of \cite{mccarthy--operations}, as in \cite{msv--THH,abghlm--norm}; see \cite{loday--operations,gerstenhaber-schack--hodge} for earlier accounts). Throughout this subsection, we fix an integer $\ell \ge 2$, a derived algebraic context $\cat{C}$ such that $\ell$ is invertible in $\pi_0\End_{\cat{C}}(\unit_{\cat{C}})$ (where $\unit_{\cat{C}}$ denotes the unit object of $\cat{C}$), and a map of derived commutative algebras $A \to B$ in $\cat{C}$.

\begin{notation}
  \label{fc--ad--setup}
  Let $\cat{D}$ be a symmetric monoidal $\infty$-category and let $T$ be a bicommutative bialgebra in $\cat{D}$. We let $[\ell] : T \to T$ denote the \emph{$\ell$-th power map} of bicommutative bialgebras, obtained by composing the $\ell$-fold comultiplication map $T \to T^{\otimes \ell}$ with the $\ell$-fold multiplication map $T^{\otimes \ell} \to T$.

  The cases of interest to us are when $\cat{D} = \Spc$ and $T = \cir$, when $\cat{D} = \cat{C}$ and $T = \lincir$, when $\cat{D} = \Fil(\cat{C})$ and $T = \lincir_\fil$, and when $\cat{D} = \Gr(\cat{C})$ and $T = \num{D}_+$.
\end{notation}

\begin{notation}
  \label{fc--ad--reparam}
  Let $\cat{D}$ be a symmetric monoidal $\infty$-category, let $T$ be a bicommutative bialgebra in $\cat{D}$, and let $\phi : T \to T$ be a map of bicommutative bialgebras in $\cat{D}$. Then we denote the corresponding restriction functor $\Mod_A(\cat{D}) \to \Mod_A(\cat{D})$ by $M \mapsto M^\phi$. Recall from \cref{bi--tn--functorial} that this restriction functor is canonically symmetric monoidal.

  In the situation that $T$ is dualizable, we have from \cref{bi--du} a dual bicommutative bialgebra structure on $\dual{T}$ and a symmetric monoidal equivalence $\Mod_T(\cat{D}) \iso \cMod_{\dual{T}}(\cat{D})$. Under this equivalence, restriction in $\phi$ corresponds to corestriction in the dual map of bicommutative bialgebras $\dual{\phi} : \dual{T} \to \dual{T}$. If $\cat{D}$ is in fact a derived algebraic context, $\dual{T}$ a derived bicommutative bialgebra in $\cat{D}$, and $\dual{\phi}$ a map of such, there is in addition an induced restriction functor on derived commutative algebra objects, $(-)^\phi : \DAlg(\Mod_A(\cat{D})) \to \DAlg(\Mod_A(\cat{D}))$ (our notation here is as in \cref{dc--df--dcbi-dual}). This applies in the last three examples listed in \cref{fc--ad--setup}.
\end{notation}

\begin{construction}
  \label{fc--ad--HH-ad}
  Restricting $\HH(B/A) \in \Fun(\clspc\cir,\DAlg_A)$ along the $\ell$-th power map $[\ell] : \cir \to \cir$ gives an object $\HH(B/A)^{[\ell]} \in \Fun(\clspc\cir,\DAlg_A)$, whose underlying derived commutative algebra is the same as that of $\HH(B/A)$ (i.e. the difference is only in the $\cir$-actions). Thus, the canonical map of derived commutative algebras $A \to \HH(B/A)$ can also be regarded as a map of derived commutative algebras $A \to \HH(B/A)^{[\ell]}$. By the definition/universal property of $\HH(B/A)$, this extends to a unique map $\psi^\ell : \HH(B/A) \to \HH(B/A)^{[\ell]}$ in $\Fun(\clspc\cir,\DAlg_A)$, which we refer to as the \emph{$\ell$-th Adams operation on $\HH(B/A)$}.
\end{construction}

We can run the same construction in the filtered setting to see that the Adams operation is compatible with the HKR filtration on Hochschild homology:

\begin{proposition}
  \label{fc--ad--HH-fil-ad}
  The $\ell$-th Adams operation $\psi^\ell : \HH(B/A) \to \HH(B/A)^{[\ell]}$ promotes uniquely to a map $\psi^\ell_\fil : \HH_\fil(B/A) \to \HH_\fil(B/A)^{[\ell]}$ in $\Fil_\cir^{\ge0}\DAlg_A$.
\end{proposition}

\begin{proof}
  This is immediate from the universal property of $\HH_\fil(B/A)$, or in other words from the equivalences
  \begin{align*}
    \Map_{\Fil_\cir^{\ge0}\DAlg_A}(\HH_\fil(B/A),\HH_\fil(B/A)^{[\ell]})
    &\iso \Map_{\DAlg_A}(B,\ev^0(\HH_\fil(B/A)^{[\ell]})) \\
    &\iso \Map_{\DAlg_A}(B,\HH(B/A)^{[\ell]}) \\
    &\iso \Map_{\Fun(\clspc\cir,\DAlg_A)}(\HH(B/A),\HH(B/A)^{[\ell]}).\ \qedhere
  \end{align*}
\end{proof}

We now explain why the Adams operations on Hochschild homology induce Adams operations on negative cyclic homology.

\begin{lemma}
  \label{fc--ad--restrict-fixed}
  Let $\cat{D}$ be a presentable symmetric monoidal $\infty$-category, let $T$ be a cocommutative bialgebra in $\cat{D}$, and let $\phi : T \to T$ be an equivalence of cocommutative bialgebras in $\cat{D}$. Then there is a canonical natural lax symmetric monoidal equivalence $M^T \iso (M^\phi)^T$ for $M \in \LMod_T(\cat{C})$, where $(-)^T$ denotes the fixed point functor of \cref{bi--ta--orbits-fixed}.
\end{lemma}

\begin{proof}
  This is an immediate consequence of $\phi$ being an equivalence.
\end{proof}

\begin{lemma}
  \label{fc--ad--ell-equiv}
  The $\ell$-th power maps
  \[
    [\ell] : \lincir \to \lincir, \qquad
    [\ell] : \lincir_\fil \to \lincir_\fil, \qquad
    [\ell] : \num{D}_+ \to \num{D}_+
  \]
  are all equivalences, in $\cat{C}$, $\Fil(\cat{C})$, and $\Gr(\cat{C})$ respectively.
\end{lemma}

\begin{proof}
  We will just argue for the $\ell$-th power map on $\num{T}$; the others can be addressed by the same reasoning (or alternatively follow immediately from this first case). Let $R := \num{Z}[1/\ell]$. Since we have assumed that $\ell$ is invertible in $\pi_0\End_{\cat{C}}(\unit_{\cat{C}})$, the unique morphism of derived algebraic contexts $\Mod_{\num{Z}} \to \cat{C}$ factors through $\Mod_R$. It is thus enough to consider the case $\cat{C} = \Mod_R$. Here we have $\lincir \iso R[\cir]$, and the map $[\ell] : \lincir \to \lincir$ is an equivalence because it induces isomorphisms on homotopy groups (namely the identity on $\pi_0$ and multiplication by $\ell$ on $\pi_1$).
\end{proof}

\begin{construction}
  \label{fc--ad--HC-ad}
  Let $\psi^\ell : \HH(B/A) \to \HH(B/A)^{[\ell]}$ be the Adams operation of \cref{fc--ad--HH-ad}. Passing to $\cir$-fixed points, we obtain a map
  \[
    \HN(B/A) = \HH(B/A)^\lincir \to (\HH(B/A)^{[\ell]})^\lincir \iso \HH(B/A)^\lincir = \HN(B/A)
  \]
  in $\CAlg_A$, where the equivalence comes from \cref{fc--ad--restrict-fixed,fc--ad--ell-equiv}. We denotes this map also by $\psi^\ell$, and refer to it as the \emph{$\ell$-th Adams operation on $\HN(B/A)$}.
\end{construction}

Again, we may repeat the above construction in the filtered setting to see that the Adams operation is compatible with the filtration on negative cyclic homology.

\begin{construction}
  \label{fc--ad--HC-fil-ad}
  Let $\psi^\ell_\fil : \HH_\fil(B/A) \to \HH_\fil(B/A)^{[\ell]}$. Applying filtered $\cir$-fixed points, we obtain a map $\psi^\ell_\fil : \HN_\fil(B/A) \to \HN_\fil(B/A)$, which is compatible with the Adams operation $\psi^\ell : \HN(B/A) \to \HN(B/A)$ of \cref{fc--ad--HC-ad} in the sense that the following diagram commutes:
  \[
    \begin{tikzcd}[column sep = huge]
      \colim(\HN_\fil(B/A)) \ar[r, "\colim(\psi^\ell_\fil)"] \ar[d] &
      \colim(\HN_\fil(B/A)) \ar[d] \\
      \HN(B/A) \ar[r, "\psi^\ell"] &
      \HN(B/A).
    \end{tikzcd}
  \]
\end{construction}

We would now like to identify the effect of the maps $\psi^\ell_\fil$ of \cref{fc--ad--HH-fil-ad,fc--ad--HC-fil-ad} upon passage to associated graded objects. What we will show is that, for each $i \in \num{Z}$, the effect on $\gr^i$ is given by multiplication by $\ell^i$. It will be useful to first formulate this graded multiplication map in a precise and structured way.

\begin{construction}
  \label{fc--ad--ellstar}
  \emph{Let $u$ be a unit in $\num{Z}[1/\ell]$ (the relevant examples being $u = \ell^{\pm 1}$). We will construct a natural isomorphism of symmetric monoidal functors}
  \[
    \{u^*\} : \id_{\Gr(\cat{C})} \to \id_{\Gr(\cat{C})},
  \]
  \emph{whose evaluation on any graded object $X^* \in \Gr(\cat{C})$ is the endomorphism given by multiplication by $u^i$ on $X^i$ for each $i \in \num{Z}$.}

  Consider the \emph{pointwise} tensor product functor $\otimes : \Gr(\cat{C}) \times \Gr(\cat{C}) \to \Gr(\cat{C})$, given by the formula $(X^* \otimes Y^*)^n \iso X^n \otimes Y^n$. We claim that this has a canonical lax symmetric monoidal structure, where we still regard (each copy of) $\Gr(\cat{C})$ as equipped with the Day convolution symmetric monoidal structure. By the universal property of Day convolution, it is equivalent to produce a lax symmetric monoidal structure on the corresponding functor $\num{Z}^\ds \times \Gr(\cat{C}) \times \Gr(\cat{C}) \to \cat{C}$, which we may obtain by factoring it as the following composition of lax symmetric monoidal functors:
  \begin{align*}
    \num{Z}^\ds \times \Gr(\cat{C}) \times \Gr(\cat{C})
    &\lblto{\Delta \times \id \times \id}
    \num{Z}^\ds \times \num{Z}^\ds \times \Gr(\cat{C}) \times \Gr(\cat{C}) \\
    &\iso \num{Z}^\ds \times \Gr(\cat{C}) \times \num{Z}^\ds \times \Gr(\cat{C}) \\
    &\lblto{\ev \times \ev} \cat{C} \times \cat{C} \\
    &\lblto{\otimes} \cat{C}
  \end{align*}

  The preceding claim implies that, for any commutative algebra object $X$ of $\Gr(\cat{C})$, the functor $X \otimes - : \Gr(\cat{C}) \to \Gr(\cat{C})$ given by pointwise tensor product with $X$ is canonically lax symmetric monoidal, and that this extends to functor
  \[
    \tau : \Gr\CAlg(\cat{C}) \to \End^\lax(\Gr(\cat{C})),
  \]
  where the right-hand side denotes the $\infty$-category of lax symmetric monoidal endofunctors of $\Gr(\cat{C})$. Let $A \in \Gr\CAlg(\cat{C})$ denote the constant graded object with value $\unit_{\cat{C}}$, regarded as a commutative algebra in the canonical manner. Then we have a canonical equivalence $\tau(A) \iso \id_{\cat{\Gr(C)}}$, so that $\tau$ determines a map of spaces
  \[
    \End_{\Gr\CAlg(\cat{C})}(A) \to \End_{\End^\lax(\Gr(\cat{C}))}(\id_{\Gr(\cat{C})}).
  \]
  Now recall (see \cite[Example 2.2.6.9]{lurie--algebra}) that $\Gr\CAlg(\cat{C})$ is canonically equivalent to the $\infty$-category $\Fun^\lax(\num{Z}^\ds,\cat{C})$ of lax symmetric monoidal functors $\num{Z}^\ds \to \cat{C}$. Under this equivalence, $A$ corresponds to the constant lax symmetric monoidal functor with value $\unit_{\cat{C}}$. It follows that there is an induced equivalence of spaces
  \[
    \End_{\Gr\CAlg(\cat{C})}(A) \iso \Map_{\CAlg(\Spc)}(\num{Z}^\ds,\End_{\cat{C}}(\unit_{\cat{C}})),
  \]
  where, on the right-hand side, $\End_{\cat{C}}(\unit_{\cat{C}})$ is regarded as an $\Einfty$-space via the multiplicative structure determined by the symmetric monoidal structure on $\cat{C}$. We thus obtain a map
  \[
    \sigma : \Map_{\CAlg(\Spc)}(\num{Z},\End_{\cat{C}}(\unit_{\cat{C}})) \to \End_{\End^\lax(\Gr(\cat{C}))}(\id_{\Gr(\cat{C})}).
  \]

  The preceding map allows us to make the desired construction. Let $R := \num{Z}[1/\ell]$. As used in the proof of \cref{fc--ad--ell-equiv}, we have an $R$-linear structure on $\cat{C}$, determining a map of $\Einfty$-spaces
  \[
    \lambda : R^\times \iso \End_{\Mod_R}(R) \to \End_{\cat{C}}(\unit_{\cat{C}}).
  \]
  Precomposing this with the map of commutative monoids $\num{Z} \to R^\times$ sending $1 \mapsto u$ and then applying the map $\sigma$ above produces the desired natural isomorphism $\{u^*\} \in \End_{\End^\lax(\Gr(\cat{C}))}(\id_{\Gr(\cat{C})})$.
\end{construction}

\begin{remark}
  \label{fc--ad--ellstar-dalg}
  As the natural isomorphism $\{u^*\}$ of \cref{fc--ad--ellstar} is one of symmetric monoidal functors, if $X$ is a commutative algebra object or bicommutative bialgebra object of $\Gr(\cat{C})$, the endomorphism $\{u^\star\} : X \to X$ is canonically a map of such objects. The same statements go through for derived commutative algebra objects and derived bicommutative bialgebra objects by virtue of the naturality stated in \cref{dc--df--dalg-natural}.
\end{remark}

\begin{remark}
  \label{fc--ad--ell-gr}
  The $\ell$-th power map $[\ell]: \num{D}_+ \to \num{D}_+$ is equivalent to the endomorphism $\{\ell^*\} : \num{D}_+ \to \num{D}_+$ of \cref{fc--ad--ellstar} as maps of bicommutative bialgebras, and dually $[\ell]: \dual{\num{D}_+} \to \dual{\num{D}_+}$ is equivalent to $\{\ell^*\}^{-1} \iso \{(\ell^{-1})^*\} : \dual{\num{D}_+} \to \dual{\num{D}_+}$ as maps of derived bicommuative bialgebras. These identifications reduce to straightforward calculations since $\num{D}_+$ and $\dual{\num{D}_+}$ lie in the ordinary category $\Gr(\cat{C})_+^\heart$ (and using \cref{dc--cn--gr}).

  These statements imply that there is an induced natural isomorphism of symmetric monoidal functors $\{\ell^*\} : \id_{\DG_+(\cat{C})} \to \id_{\DG_+(\cat{C})}$, which moreover respects derived commutative algebra structures.
\end{remark}

\begin{proposition}
  \label{fc--ad--gr-ad}
  Let $\psi^\ell_\fil : \HH_\fil(B/A) \to \HH_\fil(B/A)^{[\ell]}$ and $\psi^\ell_\fil : \HN_\fil(B/A) \to \HN_\fil(B/A)$ be the maps of \cref{fc--ad--HH-fil-ad,fc--ad--HC-fil-ad}. Then the associated graded maps
  \[
    \gr(\psi^\ell_\fil) : \gr(\HH_\fil(B/A)) \to \gr(\HH_\fil(B/A)^{[\ell]}) \iso \gr(\HH_\fil(B/A))^{[\ell]},
  \]
  \[
    \gr(\psi^\ell_\fil) : \gr(\HN_\fil(B/A)) \to \gr(\HN_\fil(B/A))
  \]
  are each homotopic to the map $\{\ell^*\}$.
\end{proposition}

\begin{proof}
  We first prove the statement for Hochschild homology. By \cref{fc--hh--conceptual-hkr}, we have an equivalence $\gr(\HH_\fil(B/A)) \iso \LOmega_{B/A}^{+\bul}$, through which we may identify the map $\gr(\psi^\ell_\fil) : \gr(\HH_\fil(B/A)) \to \gr(\HH_\fil(B/A))^{[\ell]}$ with a map $\phi : \LOmega^{+\bul}_{B/A} \to (\LOmega^{+\bul}_{B/A})^{[\ell]}$. By construction of $\psi^\ell_\fil$, the map $\phi$ is one of differential graded derived commutative $A$-algebras, which on zeroth graded pieces is the identity map $\id_B : B \to B$. By the definition/universal property of $\LOmega^{+\bul}_{B/A}$, this uniquely characterizes $\phi$. On the other hand, it follows from \cref{fc--ad--ell-gr} that $\{\ell^\star\}$ is also such a map, so we must have $\phi \iso \{\ell^\star\}$, as desired.

  We now prove the statement for negative cyclic homology. Recall that the map $\psi^\ell_\fil : \HN_\fil(B/A) \to \HN_\fil(B/A)$ was obtained from the map $\psi^\ell_\fil : \HH_\fil(B/A) \to \HH_\fil(B/A)^{[\ell]}$ by passing to filtered $\cir$-fixed points. It follows that the map $\gr(\psi^\ell_\fil) : \gr(\HN_\fil(B/A)) \to \gr(\HN_\fil(B/A))$ is obtained from the map $\gr(\psi^\ell_\fil) : \gr(\HH_\fil(B/A)) \to \gr(\HH_\fil(B/A))^{[\ell]}$ by passing to $\num{D}_+$-fixed points, hence is the unique such map $\phi'$ fitting into a commutative diagram
  \[
    \begin{tikzcd}
      \rho(\gr(\HN_\fil(B/A))) \ar[rr] \ar[d, "\rho(\phi')"] & &
      \gr(\HH_\fil(B/A)) \ar[d, "\phi"] \\
      \rho(\gr(\HN_\fil(B/A))) \ar[r, "\sim"] &
      \rho(\gr(\HN_\fil(B/A)))^{[\ell]} \ar[r] &
      \gr(\HH_\fil(B/A))^{[\ell]}
    \end{tikzcd}
  \]
  in $\DG_+\CAlg(\cat{C})$, where $\rho : \Gr(\cat{C}) \to \DG_+(\cat{C})$ is the restriction functor (along the counit $\num{D}_+ \to \unit_{\cat{C}}$) and the displayed equivalence is that of \cref{fc--ad--restrict-fixed}. We have seen that $\phi \iso \{\ell^\star\}$, and it follows from $\{\ell^\star\}$ being a natural transformation of symmetric monoidal functors that we must have $\phi' \iso \{\ell^\star\}$ as well, finishing the proof.
\end{proof}


\printbibliography

\end{document}